\documentclass[12pt]{amsart}
\usepackage{amsmath,amsthm,amsfonts,amscd,amssymb}
\usepackage[left=2cm,right=2cm]{geometry}

\usepackage[colorlinks = true,
linkcolor = blue,
urlcolor  = blue,
citecolor = purple,
anchorcolor = blue]{hyperref}
\usepackage[capitalize]{cleveref}
\usepackage{tikz-cd}
\usepackage{mathtools}
\usepackage{thmtools}
\usepackage{thm-restate}
\usepackage{xcolor,float}
\usepackage{graphics,graphicx}
\graphicspath{ {./Diagrams_Prelim/} }
\usepackage[colorinlistoftodos]{todonotes}
\usepackage{commath} 
\usepackage{tkz-fct}  
\usepackage{enumerate}  
\usepackage{enumitem}
\usepackage{comment}
\usepackage{lipsum}
\usepackage{caption}
\usepackage{subcaption}
\usepackage{mdframed}
\usepackage{lipsum}

\theoremstyle{plain}
\newtheorem{thm}{Theorem}[section]
\newtheorem{proposition}[thm]{Proposition}
\newtheorem{lemma}[thm]{Lemma}

\newtheorem{theorem}[thm]{Theorem}
\newtheorem{corollary}[thm]{Corollary}
\newtheorem*{theorem*}{Theorem}

\theoremstyle{plain}
\newtheorem*{definition*}{Definition}
\newtheorem*{proposition*}{Proposition}
\newtheorem*{lem*}{Lemma}

\theoremstyle{plain}
\newtheorem{cor}[thm]{Corollary} 

\theoremstyle{plain}

\theoremstyle{plain}
\newtheorem{remark}[thm]{Remark}

\theoremstyle{definition}
\newtheorem{definition}[thm]{Definition}

\theoremstyle{definition}

\newtheorem{exmp}[thm]{Example}
\newtheorem{ex}[thm]{Example}

\newtheorem{theo}{Theorem}
\newenvironment{ftheo}
  {\begin{mdframed}[innertopmargin = 3pt, innerbottommargin=3pt,skipabove=5pt,skipbelow=5pt,linewidth=0.25pt,nobreak=true,align=center]\begin{theo}}
  {\end{theo}\end{mdframed}}

\theoremstyle{definition}

\theoremstyle{definition}

\newenvironment{RC}{\color{blue}{\bf RC:} \footnotesize}{}

\newenvironment{YJN}{\color{red}{\bf Nick:} \footnotesize}{}

\definecolor{DarkGreen}{RGB}{1,50,32}

\makeatletter
\def\@tocline#1#2#3#4#5#6#7{\relax
  \ifnum #1>\c@tocdepth % then omit
  \else
    \par \addpenalty\@secpenalty\addvspace{#2}%
    \begingroup \hyphenpenalty\@M
    \@ifempty{#4}{%
      \@tempdima\csname r@tocindent\number#1\endcsname\relax
    }{%
      \@tempdima#4\relax
    }%
    \parindent\z@ \leftskip#3\relax \advance\leftskip\@tempdima\relax
    \rightskip\@pnumwidth plus4em \parfillskip-\@pnumwidth
    #5\leavevmode\hskip-\@tempdima
      \ifcase #1
       \or\or \hskip 1em \or \hskip 2em \else \hskip 3em \fi%
      #6\nobreak\relax
    \dotfill\hbox to\@pnumwidth{\@tocpagenum{#7}}\par
    \nobreak
    \endgroup
  \fi}
\makeatother
\newcommand{\sse}{\subset}

\newcommand{\R}{\mathbb R}
\newcommand{\C}{\mathbb C}

\newcommand{\N}{\mathbb N}

\newcommand{\CP}{\mathbb{CP}}
\newcommand{\sS}{\mathbb{S}}

\renewcommand{\Re}{\text{Re}}
\renewcommand{\Im}{\text{Im}}

\newcommand{\CF}{CF^\bullet}

\newcommand{\s}{\sigma}

\DeclareMathOperator{\Sh}{Sh}

\DeclareMathOperator{\cone}{cone}

\DeclareMathOperator{\Hom}{Hom}

\DeclareMathOperator{\sgn}{sgn}

\newcommand{\loct}{\mbox{Loc}^\dagger}
\newcommand{\ff}{\mathbb{F}}

\newcommand\ram[1]{\text{ram}(#1)}
\newcommand\inner[2]{\langle #1, #2 \rangle}

\newcommand{\lknot}{\Lambda} 
\newcommand{\radial}{r}
\newcommand{\scurve}{\Sigma}
\newcommand{\snetwork}{\EuScript{W}}
\newcommand{\fnetwork}{\EuScript{F}}
\newcommand{\primitive}{W}

\newcommand{\trunparam}{\delta}

\newcommand{\injradius}{\mbox{inj}(g)}
\newcommand{\sheetgap}{\rho}
\newcommand{\heightR}{R}

\newcommand{\bis}{{\bf S}}
\newcommand{\bic}{C}
\newcommand{\obic}{\widetilde{C}}
\newcommand{\tbis}{\bf \overline{S}}
\newcommand{\obis}{S}

\newcommand{\Loc}{\mbox{Loc}}
\newcommand{\slit}{\xi}
\newcommand{\soliton}[1]{{\mathfrak{s}(#1)}}

\newcommand{\bdir}{x}
\newcommand{\ldir}{\bar{x}}
\newcommand{\Rcs}{\mathcal{R}}
\newcommand{\Bspts}{\mathcal{B}}
\newcommand{\bspt}{b}
\newcommand{\bsptl}{\bar{b}}
\newcommand{\pts}{\Phi}
\newcommand{\mkpts}{\boldsymbol{m}}
\newcommand{\iregc}{\Theta}
\newcommand{\esym}{\mathfrak{c}}

\newcommand{\treegraph}{\Gamma}
\newcommand{\treemap}{F}
\newcommand{\Laggerm}{L}
\newcommand{\multigraph}{L}
\newcommand{\caustL}{K_{\multigraph}}

\newcommand{\sheet}{f}
\newcommand{\holsheet}{\lambda}
\newcommand{\irreg}{Q}
\newcommand{\irregsheet}{q}
\newcommand{\energy}{Z}
\newcommand{\sphsc}{\mathring{L}}
\newcommand{\sphb}{\mathring{S}}
\newcommand{\basepath}{\alpha}
\newcommand{\cotpath}{{\alpha}}

\newcommand{\locsys}{\textbf{V}}
\newcommand{\spins}{\widetilde{\mathfrak{s}}}
\newcommand{\spinb}{\mathfrak{s}}
\newcommand{\precptset}{K}
\newcommand{\nbh}{U}

\newcommand{\smallcst}{\hbar}

\newcommand{\fibre}[1]{T^{\ast}_{#1}\obis}

\newcommand{\halfdisk}{A}
\newcommand{\stdm}{\triangle_m}

\newcommand{\defac}{J_{cyl}}

\newcommand{\strip}{\mathcal{Z}}

\newcommand{\elon}{l}
\newcommand{\dist}{d}

\newcommand{\dcgraph}{\treegraph}

\newcommand{\dfs}{D_4^{-}}
\newcommand{\wcfl}[1]{\mathbb{F}(#1)}
\newcommand{\wcfle}[1]{\mathbb{F}_{\e}(#1)}
\newcommand{\wcflm}{\mathbb{F}}

\newcommand{\jbis}{J^1\bis}

\newcommand{\wfc}{\mathfrak{W}(T^{\ast}\obis)}
\newcommand{\kmod}{\mbox{mod}_k}
\newcommand{\pwfc}{\mathfrak{W}(T^{\ast}\obis,\Lambda)}
\newcommand{\dset}{\overrightarrow{P}}%decoratedposet
\newcommand{\wto}{\rightsquigarrow}

\newcommand{\lksat}{\lknot_{\beta}^{+}}

\newcommand{\blag}{L}
\newcommand{\tlag}{P}

\DeclareMathAlphabet\EuScript{U}{eus}{m}{n}
\SetMathAlphabet\EuScript{bold}{U}{eus}{b}{n}

\newcommand{\std}{\text{st}}
\newcommand{\La}{\Lambda}

\newcommand{\la}{\lambda}
\newcommand{\e}{\varepsilon}

\newcommand{\T}{\mathbb{T}}
\newcommand{\w}{\mathfrak{w}}
\newcommand{\reeb}{\mathfrak{r}}
\renewcommand{\tt}{\mathfrak{t}}
\newcommand{\dd}{\partial}
\newcommand{\SL}{\mbox{SL}}
\newcommand{\GL}{\mbox{GL}}
\newcommand{\lr}{\longrightarrow}

\newcommand{\ppp}{\rho}

\newcommand{\ccc}{\mathring{\mathfrak{s}}}

\newcommand{\tdfd}{\Phi_{\snetwork}(V)}
\newcommand{\wall}{\mathfrak{w}}

\newcommand{\pneck}{\blag^{\dagger}}
\newcommand{\neckg}{g^{\dagger}}

\newcommand{\floor}[1]{\lfloor #1 \rfloor }
\newcommand{\ylag}{\EuScript{Y}(\blag)}
\newcommand{\ylagv}{\EuScript{Y}_{(\blag,V)}}

\newcommand{\ww}{\mathfrak{w}}
\newcommand{\ldisk}{\EuScript{D}}
\newcommand{\lcyl}{\blag_\circ}
\newcommand{\ycyl}{\EuScript{Y}(\blag_\circ,V)}
\newcommand{\ycylv}{\EuScript{Y}_{(\blag_\circ,V)}}
\newcommand{\bcyl}{\blag_{\circ}}
\newcommand{\SA}{\mathcal{A}}
\newcommand{\D}{\mathcal{D}}
\newcommand{\Z}{\mathbb{Z}}

\newcommand{\inva}{\bar{\alpha}}

\usepackage[style=alphabetic]{biblatex}
\bibliography{main}
\graphicspath{ {./figures/} }

\begin{document}
%%%%%%%%%%%%%%%%%%%%%%%%%%%%%%%%%%%%%%%%%%%%%%%%%%%%%%%%%%%%%%%%%%%%%%%%
%%%%%%%%%%%%%%%%%%%%%%%%%%%%%%%%%%%%%%%%%%%%%%%%%%%%%%%%%%%%%%%%%%%%%%%%
%%%%%%%%%%%%%%%%%%%%%%%%%%%%%%%%%%%%%%%%%%%%%%%%%%%%%%%%%%%%%%%%%%%%%%%%
\title{Spectral Networks and Betti Lagrangians}
\author{Roger Casals}
	\address{University of California Davis, Dept. of Mathematics, USA}
	\email{casals@ucdavis.edu}
\author{Yoon Jae Nho}
	\address{University of California Davis, Dept. of Mathematics, USA}
	\email{yjnho@ucdavis.edu}
\maketitle

\begin{abstract}
We introduce and develop the theory of spectral networks in real contact and symplectic topology. First, we establish the existence and pseudoholomorphic characterization of spectral networks for Lagrangian fillings in the cotangent bundle of a smooth surface. These are proven via analytic results on the adiabatic degeneration of Floer trajectories and the explicit computation of continuation strips. Second, we construct a Family Floer functor for Lagrangian fillings endowed with a spectral network and prove its equivalence to the non-abelianization functor. In particular, this implies that both the framed 2d-4d BPS states and the Gaiotto-Moore-Neitzke non-abelianized parallel transport are realized as part of the $A_\infty$-operations of the associated 4d partially wrapped Fukaya categories. To conclude, we present a new construction relating spectral networks and Lagrangian fillings using Demazure weaves, and show the precise relation between spectral networks and augmentations of the Legendrian contact dg-algebra.
\end{abstract}

\tableofcontents
%%%%%%%%%%%%%%%%%%%%%%%%%%%%%%%%%%%%%%%%%%%%%%%%%%%%%%%%%%%%%%%%%%%%%%%%
%%%%%%%%%%%%%%%%%%%%%%%%%%%%%%%%%%%%%%%%%%%%%%%%%%%%%%%%%%%%%%%%%%%%%%%%
%%%%%%%%%%%%%%%%%%%%%%%%%%%%%%%%%%%%%%%%%%%%%%%%%%%%%%%%%%%%%%%%%%%%%%%%
\section{Introduction}\label{section_introduction}
    %%%%%%%%%%%%%%%%%%%%%%%%%%%%%%%%%%%%%%%%%%%%%%%%%%%%%%%%%%%%%%%%%%%%%%%%
%%%%%%%%%%%%%%%%%%%%%%%%%%%%%%%%%%%%%%%%%%%%%%%%%%%%%%%%%%%%%%%%%%%%%%%%
%%%%%%%%%%%%%%%%%%%%%%%%%%%%%%%%%%%%%%%%%%%%%%%%%%%%%%%%%%%%%%%%%%%%%%%%

The object of this article is to develop the theory of spectral networks within the context of real contact and symplectic topology. First, we provide foundational results on the existence and characterization of spectral networks for Lagrangian fillings in the cotangent bundle of a smooth surface. In particular, we characterize the walls of spectral networks in terms of pseudoholomorphic strips. These are proven via new results on the adiabatic degeneration of Floer trajectories, the behavior of the associated flowlines, and explicit computations of continuation strips. Second, we construct a Family Floer functor for Betti Lagrangians endowed with a spectral network and show that the non-abelian parallel transport in Gaiotto-Moore-Neitzke's supersymmetric $N=4$ context is both generalized and made rigorous by this Floer-theoretic parallel transport. Third, the article brings forth new constructions and applications relating spectral networks and Lagrangian fillings, including the study of spectral networks from Demazure weaves, their relation to the Legendrian contact dg-algebra, and an interpretation of framed 2d-4d BPS states and the non-abelianization maps as part of the $A_\infty$-operations and module categories of partially wrapped Fukaya categories.\\

A central theme of this work is that spectral networks can be studied and characterized in terms of the Floer theory of real Lagrangians in the real cotangent bundle of a smooth surface. This allows for a key generalization of the standard physical theory of spectral networks, which itself focuses on meromorphic spectral curves in the holomorphic cotangent bundle of a Riemann surface. This generalization is needed to define and study the Betti moduli space in the wild case for smooth asymptotic data, not necessarily meromorphic Stokes data. In particular, it allows us to tackle asymptotics associated to any positive braid, including and generalizing all braid varieties, and perform computations on their associated cluster structures, BPS monodromy and Donaldson-Thomas theory. Indeed, the Floer-theoretic framework and results we establish provide a precise bridge to apply the new developments on cluster algebras in 4-dimensional symplectic topology to the study of spectral networks.\\

The introduction of spectral networks in mathematical physics has lead to an outstanding number of breakthroughs in string theory. The present work also serves the purpose of establishing a rigorous mathematical foundation for many beautiful insights and claims from the physics community, in our case within the context of contact and symplectic topology. A number of such conjectured properties and examples are discussed and proven throughout this manuscript, including the characterization of spectral networks in terms of pseudo-holomorphic disks and the study of higher-order Stokes phenomena, including the Berk-Nevins-Roberts example and beyond.%, and new and precise mathematical questions are posed, allowing for modern symplectic methods to be used in developing class-S theories.
\\

%%%%%%%%%%%%%%%%%%%%%%%%%%%%%%%%%%%%%%%%%%%%%%%%%%%%%%%%%%%%%%%%%%%%%%%%
%%%%%%%%%%%%%%%%%%%%%%%%%%%%%%%%%%%%%%%%%%%%%%%%%%%%%%%%%%%%%%%%%%%%%%%%
%%%%%%%%%%%%%%%%%%%%%%%%%%%%%%%%%%%%%%%%%%%%%%%%%%%%%%%%%%%%%%%%%%%%%%%%

\subsection{Scientific context} In 2012, within the context of supersymmetric gauge theories in mathematical physics, D.~Gaiotto, G.~Moore and A.~Neitzke introduced and developed the study of spectral networks in \cite{GMN13_Framed,GNMSN,GMN14_Snakes}. These contributions to physical theories of class $S$ allowed for the computation of different BPS degeneracies in the associated 4-dimensional $N=2$ theories (e.g.~coupling to surface defects) and a better understanding of how the wall-crossing phenomenon in coupled 2d-4d systems \cite{GMN12_WallCrossCoupled,WallcrossingHitchinSystemWKBapproximation} related to the BPS spectrum. Since then, an abundance of exciting work illustrates that spectral networks are part of a fertile and productive ground for mathematical physics. To wit, further explorations of different types of BPS states were achieved in \cite{MR3688368,MR3395151,MR3555651,MR3640256,MR3769247} and \cite{MR4089181,MR4370932,MR3638316,MR4097010}, non-abelianization and quantum holonomies were studied in \cite{MR3613514,MR3500424,MR4166626,MR4568006,MR4190271} and see also the surveys \cite{MR3263304,MR4380684}.\\

In pure mathematics, spectral networks have started to be used in a variety of areas, including exact WKB analysis \cite{MR4814083,MR4640921,MR4372666,MR4567381} (see also the preceding works \cite{AokiKawaiTakei01_StokesCurves,HowlsLangmanOlde04_HigherOrderStokes}), the study of moduli of Higgs bundles and their Hitchin fibrations \cite{MR4214341,MR4324961,MR3322389,MR3551264}, and that of spaces of stability conditions \cite{MR3349833,MR4808258,MR3416110}; see also \cite{MR4748489,kuwagaki2024genericexistencewkbspectral,nho2024familyfloertheorynonabelianization,MR4843307}. Throughout these contributions, there is a focus towards the case $G=\mbox{SL}_2(\C)$, where the spectral curve is a 2-fold branch cover of the base. This case is rather well-understood by now, in part due to its direct relation to the theory of quadratic differentials on Riemann surfaces and the combinatorics of their associated foliations. The higher rank case, for $G=\mbox{SL}_n(\C)$, remains an interesting subject of exploration, cf.~\cite{MR3322389,MR4340928}, and a main focus of this article.

In the $\mbox{SL}_2(\C)$ case, a first connection between the Floer theory of real Lagrangian surfaces and framed 2d-4d BPS states was hinted by the first author in \cite{CasalsDGAcubic}, being further developed in \cite{nho2024familyfloertheorynonabelianization} by the second author. Subsequently, the introduction of Legendrian weaves \cite{legendrianweaves} has provided key insights to establish the general connection between Floer theory and spectral networks in all higher ranks, significantly beyond the theory of quadratic differentials. Equipped with new tools in symplectic topology \cite{casals2022conjugate,legendrianweaves,GPSCV,GPSSD}, the present manuscript now greatly generalizes the results in \cite{CasalsDGAcubic,nho2024familyfloertheorynonabelianization} to arbitrary rank, and crystallizes a number of further insights relating spectral networks to both partially wrapped Fukaya categories \cite{GPSCV,GPSSD,MR3911570} and the construction of cluster structures on braid varieties \cite{casals2024algebraicweavesbraidvarieties,casals2024clusterstructuresbraidvarieties,casals2023microlocal}. From the Betti perspective on wild character varieties, a key difference between rank $n=2$ and higher rank spectral networks is that, in the latter case, asymptotics are not constrained to be meromorphic. Namely, the putative spectral curves can have asymptotics dictated by non-algebraic links, and thus they are not necessarily associated to meromorphic curves. Thus far, all higher rank explorations in the literature focus on the algebraic case, on the so-called WKB spectral networks, where the spectral curve is a holomorphic subvariety of the holomorphic cotangent bundle of a Riemann surface. The results we present go much beyond such algebraic constraints.\\

A key conceptual advance of this article is that the technique of spectral networks also works for real Lagrangian fillings of Legendrian links, regardless of whether the Lagrangian surface and their asymptotic Legendrian links have algebraic origin or not. For this, we introduce a more general definition of spectral network than \cite{GNMSN}, where only an almost complex structure is used and the necessary analytical properties of pseudo-holomorphic strips and their adiabatic limits can be proven. An advantage of intertwining spectral networks with some of the current tools from real contact and symplectic topology is that structural results known within the latter translate into theorems for spectral networks. An instance of this is the invariance of BPS soliton classes under real Hamiltonian isotopies, and not just holomorphic deformations. Another example is the functorial behavior of Floer-theoretic invariants for Legendrian links under Lagrangian cobordisms, which readily explains how to glue and compute spectral networks associated to Demazure weaves. These new classes of spectral networks, built from Demazure weaves, include and generalize many guiding examples in the literature and allow for direct computations of the BPS monodromy generator by using explicit reddening sequences from their associated cluster algebras.

%%%%%%%%%%%%%%%%%%%%%%%%%%%%%%%%%%%%%%%%%%%%%%%%%%%%%%%%%%%%%%%%%%%%%%%%
%%%%%%%%%%%%%%%%%%%%%%%%%%%%%%%%%%%%%%%%%%%%%%%%%%%%%%%%%%%%%%%%%%%%%%%%
%%%%%%%%%%%%%%%%%%%%%%%%%%%%%%%%%%%%%%%%%%%%%%%%%%%%%%%%%%%%%%%%%%%%%%%%

\subsection{Main results} Let ${\bf S}$ be a smooth closed oriented surface and $\mkpts\sse {\bf S}$ a finite collection of marked points. A triple $({\bf S},\mkpts;{\bf \lknot})$ is said to be a Betti surface of rank $n$ if ${\bf \lknot}$ is a collection of Legendrian links in the ideal contact boundary $(T^\infty S,\xi_\std)$, each of which is associated to an $n$-stranded positive braid over the Legendrian fiber of a marked point in $\mkpts$. Let us denote by $S:={\bf S}\setminus\mkpts$ the complement of the marked points. By depicting the wavefronts of the links in ${\bf \lknot}$ onto $S\sse {\bf S}$, a Betti surface is pictorially drawn as in Figure \ref{fig:BettiSurface}.

\begin{center}
	\begin{figure}[h!]
		\centering
		\includegraphics[scale=0.6]{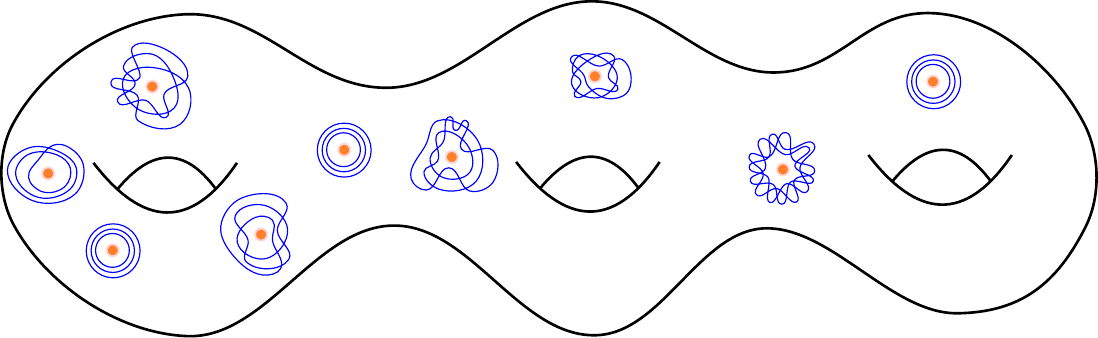}
		\caption{A Betti surface ${\bf S}$ of genus $3$ and rank $3$, with 9 marked points $\mkpts=\{m_1,\ldots,m_9\}$, in orange, and their associated Legendrian links in ${\bf \lknot}$, in blue. The Legendrian links associated to marked points $m_1,m_5,m_8$ have concentric circles as their fronts, which corresponds to a regular singularity. A non-concentric Legendrian isotopy class corresponds to an irregular singularity.}
		\label{fig:BettiSurface}
	\end{figure}
\end{center}

The Legendrian links in $\lknot$ are the contact topological incarnation of the Stokes data from the theory of irregular singularities of meromorphic connections, cf.~\cite[Remark 2.17]{MR2483750}. In particular, any meromorphic spectral curve in the sense of \cite{GNMSN} yields a well-defined Betti surface of the corresponding rank, where the fronts of the Legendrians in $\lknot$ are the associated Stokes diagrams. Details are provided in Section \ref{section_setup}, cf.~also \cite{Boalch21_TopologyStokes,STWZ} and \cite[Appendix B]{Su25_DualBoundary}. At core, a Betti Lagrangian $L\sse(T^*S,\omega_\std)$ is a Lagrangian submanifold whose projection onto $S$ is an $n$-fold branched cover and whose asymptotics are dictated by the Legendrian links in ${\bf \lknot}$; this is detailed in Definition \ref{definition:tamespectralcurve}. There are two important classes of such Betti Lagrangians, both discussed in \cref{subsection:scurves}:\\

\begin{itemize}
    \item[(i)] Exact Betti Lagrangians, where the restriction of the Liouville form $\la_\std$ to $L$ is an exact 1-form, i.e.~all real periods of $\la_\std$ on $L$ vanish. This class is defined with $\obis$ a smooth punctured surface and it yields Lagrangian fillings $L\sse(T^*S,\la_\std)$ of the Legendrian link given by the union of all Legendrian links $\lknot$, after cylindrization near the contact boundary; cf.~\cref{subsection:PWFC} and \cite{CasalsGao22,CasalsGao24,legendrianweaves,EHKexactLagrangiancobordisms}.\\
    
    \item[(ii)] If $S$ is endowed with a Riemann surface structure $\sS$, we can consider meromorphic Betti Lagrangians. These are the Betti Lagrangians that appear as the real part of a holomorphic Lagrangian curve $\scurve\sse (\T^*\sS,\omega_\C)$ in the holomorphic cotangent bundle $(\T^*\sS,\omega_\C)$ with asymptotics dictated by the irregular classes at $\mkpts$. These are the spectral curves studied in the literature, including those in \cite{GMN13_Framed,GNMSN,GMN14_Snakes} and those featuring in the irregular Riemann-Hilbert correspondence, see e.g.~\cite{BiquardBoalch04_WildNAH,Boalch21_TopologyStokes}.\\
\end{itemize}

The theory of spectral networks in \cite{GMN13_Framed,GNMSN,GMN14_Snakes} is devoted to Betti Lagrangians of type $(ii)$, in Riemann surfaces $\sS$. The resulting spectral networks satisfy a number of additional properties, implied by the underlying holomorphicity: we refer to those spectral networks in \cite{GMN13_Framed,GNMSN,GMN14_Snakes} as WKB spectral networks. In order to develop the theory of spectral networks in the real setting of $(T^*S,\la_\std)$, so as to include the important class $(i)$ above, Section \ref{section_mnetwork} introduces a generalization of those WKB spectral networks, which we refer to as a Morse spectral networks. In a nutshell, a Morse spectral network $\snetwork\sse S$ compatible with $L$ is a set of gradient flow trajectories in $S$ given by the difference functions of the sheets of $L\sse T^*S$ above $S$ satisfying a set of non-trivial {\it interaction and asymptotic constraints}. These constraints exclude a number of pathological behaviors. For instance, it is proven in \cref{section_mnetwork} that Morse spectral networks do not allow for ill-behaved asymptotics at infinity, such as trajectories reaching $\mkpts$ in finite time, nor trajectories getting arbitrarily close to $\mkpts$ and then steering away from $\mkpts$. In fact, trajectories should converge to Reeb chords of $\La$, which generalize the points of maximal decay in Stokes diagrams. In the interior, Morse spectral networks also disallow for high-valence vertex interactions and tangencies between trajectories and sink points inside of $S$, among others. In general, these are all possible behaviors for such difference-function gradient trajectories that occur without a priori constraints on the geometry of $L$. In addition, a Morse spectral network must have a compatible counting function  $\mu$ on such trajectories, generalizing the count of framed 2d-4d BPS states.\footnote{This is in line with the 4d BPS invariants of \cite[Definition 2.1]{Bridgeland19_RH_from_DT}. Nevertheless, we do not consider a central charge in this general real setting; that might be the focus of future work.} These interaction constraints and the necessary properties of $\mu$, both crucial to the concept of a Morse spectral network, are specified in Section \ref{subsubsection:mnetwork}: a depiction of the key allowed local models for a Morse spectral network $\snetwork$ is in Figure \ref{fig:vertices_spec_net}.

\begin{center}
	\begin{figure}[h!]
		\centering
		\includegraphics[scale=1.1]{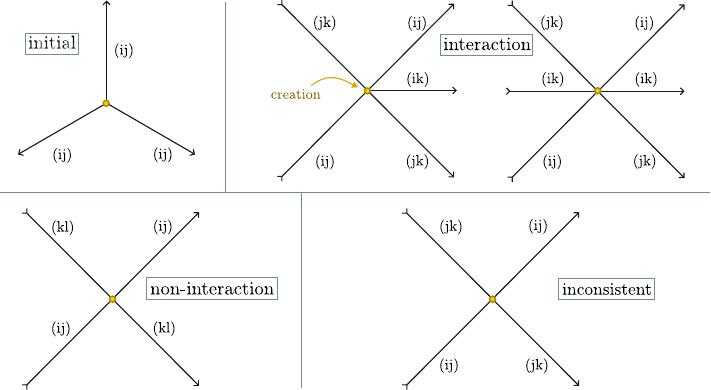}
		\caption{The local models for vertices in spectral networks: initial, interaction and non-interaction. There are two types of interaction vertices: creation and 6-valent type. The inconsistent vertex type is allowed for pre-spectral networks, but is {\it not} allowed for spectral networks.}\label{fig:vertices_spec_net}
    \end{figure}
\end{center}

%%%%%%%%%%%%%%%%%%%%%%%%%%%%%%%%%%%%%%%%%%%%%%%%%%%
%%%%%%%%%%%%%%%%%%%%%%%%%%%%%%%%%%%%%%%%%%%%%%%%%%%

{\bf A.} The first natural question that arises is the existence of such Morse spectral networks compatible with an exact Betti Lagrangian $L\sse(T^*S,\omega_\std)$. A main difficulty is that the gradient flow trajectories given by the sheets of $L$ over $S$ might not satisfy the necessary asymptotic and interaction constraints. Also, it might not be clear how to generally assign the counting function on trajectories given an arbitrary such Lagrangian $L\sse(T^*S,\omega_\std)$. Similarly, in the holomorphic setting, it has not been established that a finite-energy WKB spectral network compatible with a given a meromorphic Betti Lagrangian $\scurve\sse (\T^*\sS,\omega_\C)$ actually does exist. In practice, it is also important to guarantee that there are finitely many trajectories in a spectral network, given an energy cut-off (or equivalently, a mass cut-off in physics literature). In fact, due to its conceptual importance, we include this condition  into the definition of the WKB spectral network. The first result in this manuscript resolves the matter of existence, including the (energy)-finiteness constraint:

\begin{ftheo}[\color{blue} Existence of Spectral Networks\color{black}]\label{thm:existence} Let $S$ be a Betti surface.

\begin{itemize}
    \item[(i)] Let $L\sse(T^*S,\la_\std)$ be an exact Betti Lagrangian and $(S,g)$ a metric adapted to $L$.
    \color{black}Then there exists a finite Morse spectral network compatible with $L$ and $(S,g)$.\color{black}\\
   
    \item[(ii)] Let $\sS$ a Riemann surface structure for $S$, $(\T^*\sS,\omega_\C)$ its holomorphic cotangent bundle, and $\scurve\sse (\T^*\sS,\omega_\C)$ a meromorphic spectral curve with $O(-1)$ ends. Then\color{black}, for dense subset of phases $\theta\in S^1$,
    \color{black}there exists a generic WKB spectral network compatible with the real Lagrangian $e^{i\theta}\scurve\sse(T^*S,\la_\std)$. \color{black}
\end{itemize}

\end{ftheo}
\vspace{0.5cm}
\noindent In the statement of Theorem \ref{thm:existence}, the dense set of phases in $S^1$ is a countable intersection of open dense sets. The notation $e^{i\theta}\scurve$ stands for the real part of the holomorphic subvariety $\scurve$ after being acted by a rotation of phase $\theta$, see Section \ref{section_setup} for details. In the real case $(i)$, our proof of Theorem \ref{thm:existence} actually yields a compatible Morse spectral network whose only interaction vertices are creation vertices, which is in itself an interesting difference with the WKB case. For the proof of Theorem \ref{thm:existence}, presented in \cref{section_mnetwork}, both the trapping lemma in Section \ref{subsection:trappinglemma} and the preliminary transversality established in Section \ref{subsection:preliminarytransversalitycondition} are important technical aspects, controlling the behavior of flow trajectories and ensuring necessary asymptotics and interaction constraints. The proof of Theorem \ref{thm:existence} in the WKB case $(ii)$ nicely intertwines with the heuristic construction proposed in \cite{GNMSN}: indeed, we prove that their proposed procedure gives an algorithm that both terminates and outputs an object which is indeed a WKB spectral network, satisfying the energy-finiteness condition, or rather, a stronger condition that we call \textit{gapped} (c.f \cref{def:spectralnetwork}). In this latter WKB case, see also \cite[Theorem 1.2]{kuwagaki2024genericexistencewkbspectral}, which discusses the construction of certain WKB spectral networks via sheaf quantization. Note that in the WKB case $(ii)$, the real Lagrangian $e^{i\theta}\scurve\sse(T^*S,\la_\std)$ is typically not an exact Lagrangian, as the periods of $\scurve$ might not all have the same phase $\theta+\pi/2$.\\

%%%%%%%%%%%%%%%%%%%%%%%%%%%%%%%%%%%%%%%%%%%%%%%%%%%
%%%%%%%%%%%%%%%%%%%%%%%%%%%%%%%%%%%%%%%%%%%%%%%%%%%

{\bf B.} Floer theory, broadly understood as the study of symplectic topology through pseudo-holomorphic methods, infinite-dimensional variational analysis on path spaces and their homological incarnations, is a pillar of contact and symplectic topology. See for instance \cite{AbbondandoloSchlenk19_FloerSurvey,MR1826267,Floer88_MorseLagrangianIntersections,Floer89_SymplecticFixedPoints,Gromov85_Pseudoholomorphic,PSS96,Viterbo99_FunctorsFloerHomology} and the monographs \cite{FOOO1,FOOO2,McduffSalamon,SeidelZurich}, to name some key references. The context for Theorem \ref{thm:existence} certainly belongs to real symplectic topology: a real Lagrangian submanifold $L$ inside the cotangent bundle $(T^*S,\omega_\std)$. Yet, spectral networks $\snetwork\sse S$ are defined in terms of gradient trajectories on $S$: even if they are trajectories of functions describing $L$, there are no immediate Floer-theoretic aspects to the notion of a spectral network. Based on examples, e.g.~ when certain 4d BPS states correspond to holomorphic disks \cite[Section 6.2]{GNMSN}, it has been hoped by the mathematical physics community that spectral networks $\snetwork\sse S$ can in fact be given a characterization entirely in Floer-theoretic terms. Given the effectiveness of Floer theory and all its known functorial properties, establishing a rigorous connection between Floer theory and spectral networks seems particularly desirable from the viewpoint of pure mathematics as well. Our second result proves such characterization:

\begin{ftheo}[\color{blue}Floer-theoretic characterization of Spectral Networks\color{black}]\label{thm:characterization} Let $S$ be a Betti surface, $\scurve\sse(T^*S,\la_\std)$ a Betti Lagrangian and $\snetwork\sse S$ a spectral network compatible with $\scurve$. Then the following characterization holds:\\

\color{purple}
{\small
\noindent {\it A point $p\in S$ belongs to an active wall of $\snetwork$ $\Longleftrightarrow$ $\exists$ broken pseudo-holomorphic strip in $(T^*S,\la_\std)$ }
\begin{flushright}  {\it between the fiber $T^*_pS$ and $\scurve$ in the adiabatic limit.\\}\end{flushright}
%{\RC Add active wall for meromorphic case.}
}

\color{black}
\vspace{0.2cm}
\noindent Furthermore, the following two properties are satisfied:
\begin{itemize}
    \item[(i)] If $p\in\snetwork$, then the relative homology class in $\scurve$ associated any such pseudo-holomorphic strip is uniquely determined by $\snetwork$, and such strip must adiabatically degenerate to a subset of $\snetwork$.
    \item[(ii)] For $\scurve$ an exact Betti Lagrangian, the characterization holds for (genuine) pseudoholomorphic strips, without being potentially broken.%{\RC For exact, every wall is active.}
\end{itemize}
\end{ftheo}
\vspace{0.5cm}
%For $\scurve$ an exact Betti Lagrangian, we actually prove that Theorem \ref{thm:characterization} holds for genuine holomorphic strips, without being potentially broken.
When $\scurve$ is non-exact, we demand $\scurve$ to be meromorphic outside some compact subset of $T^{\ast}\obis$, for some technical reasons. All known interesting examples of non-exact Betti Lagrangians  fall into such a category.
A first important aspect of Theorem \ref{thm:characterization} is the presence of an adiabatic limit: it refers to the limit $\varepsilon\to 0$ when the Betti Lagrangian $\scurve\sse(T^*S,\la_\std)$ is rescaled by a factor of $\varepsilon\in\R_{+}$ under the fiberwise radial scaling action of the cotangent bundle.\footnote{In the literature in geometric analysis, adiabatic limit refers to the process of degenerating a Riemannian metric in certain directions. In symplectic topology, one typically starts with the symplectic 2-form and a compatible almost complex structure, from which a Riemannian metric is produced. The fiberwise rescaling action can be seen as a rescaling of the almost complex structure, which itself is then seen as rescaling the metric. In this case, the limit metrics when $\varepsilon\to 0$ degenerate to be supported in the zero section of $T^*S$.} Since even in the simplest cases pseudo-holomorphic strips are in bijective correspondence with flow-trees only in the adiabatic limit, the adiabatic hypothesis in \cref{thm:characterization} is needed in any such characterization. A second relevant fact is that, for exact Betti Lagrangians, all the walls of the spectral networks produced in \cref{thm:existence} are active. Therefore \cref{thm:characterization} gives a complete characterization of points in a spectral network, without even considering framed 2d-4d BPS indices.\\

Appropriately understood, Theorem \ref{thm:characterization} generalizes the analysis relating flow-trees and pseudo-holomorphic strips from \cite{Morseflowtree}. The result in ibid.~are nevertheless far from concluding \cref{thm:characterization}: a key new ingredient is establishing the convergence from pseudo-holomorphic strips to trees while satisfying the initial edge constraints, which is crucial to control the behavior of $\dfs$-trees near branch points. This convergence result is proven in \cref{section_adg}. In addition, the flowtree results in the literature do not apply to non-exact Lagrangians (and so would not be able to deal with any of \cite{GNMSN,GMN13_Framed,GMN14_Snakes}) and also require generic swallowtails front singularities. Since $D_4^-$-singularities are necessarily present in a branch cover and the study of spectral curves, we also further develop the machinery of Morse flowtrees for this context.\\

Theorem \ref{thm:characterization} is proven at the end of \cref{section_Floer}, after developing the necessary analytical results in the bulk of Sections \ref{section_adg} and \ref{section_Floer}. At core, it is the combination of two different of results. First, we establish no-go theorems in Section \ref{section_adg}, cf.~ Theorem \ref{thm:snetworkadg}, using adiabatic degeneration. We are able to use the adiabatic limit to show that certain pseudoholomorphic strips do {\it not} exist, as stated in Theorem \ref{thm:characterization}: it allows us to prove that points in the complement of the spectral networks cannot support such a pseudoholomorphic strip and that, for any point in the spectral network $\snetwork$, the only possible homology class for a potential pseudoholomorphic strip through that point must be the soliton class determined by $\snetwork$ and the point. Second, we prove that there exist such pseudo-holomorphic strips for points in $\snetwork$, as stated in Theorem \ref{thm:characterization}, through the explicit computations established in Section \ref{section_Floer}.\\

%{\RC Insert picture here}

%%%%%%%%%%%%%%%%%%%%%%%%%%%%%%%%%%%%%%%%%%%%%%%%%%%
%%%%%%%%%%%%%%%%%%%%%%%%%%%%%%%%%%%%%%%%%%%%%%%%%%%

{\bf C.} In Theorems \ref{thm:existence} and \ref{thm:characterization}, the more foundational aspects of existence and characterizations are addressed. A central aspect of spectral networks in the supersymmetric theories of \cite{GNMSN} is the non-abelianization functor, cf.~\cite[Section 10]{GMN12_WallCrossCoupled} and \cite[Section 5]{GMN14_Snakes}. In the supersymmetric $N=4$ context, it is a construction from the infrared to the ultraviolet that carries $\GL_1(\C)$-local systems on the Seiberg-Witten curve to $\GL_n(\C)$-local system on the space of ultraviolet parameters. In topological terms, each spectral network $\snetwork\sse S$ partially defines a functor $\Phi_\snetwork:\loct(L)\dashrightarrow \loct(S)$
from rank-$n$ twisted local systems on $L$ to rank-$1$ twisted local systems in $S$,
i.e.~from a Betti Lagrangian $L\sse(T^*S,\la_\std)$ to the zero section $S$. The functor $\Phi_\snetwork$ is constructed combinatorially from $\snetwork$, introducing detour paths that act as instanton corrections and abide by formulas similar to the Hori-Vafa wall-crossing formulas. Technically, given $V\in\loct(\blag)$, the physics framework is such that $\Phi_\snetwork(V)$ has parallel transport only defined for paths transverse to $\snetwork$ and endpoints in the complement $\snetwork^c$, thus our use of the dashed arrows in $\Phi_\snetwork:\loct(L)\dashrightarrow \loct(S)$, cf.~\cref{subsubsection:pathdetourclasses}.\\

In this paper we develop Floer theory for exact Betti Lagrangians in the presence of a spectral network $\snetwork$, in \cref{section_Floer}, constructing Floer-type cochain complexes with differentials recording certain counts of (actual non-adiabatic) pseudo-holomorphic disks. As a consequence, we are able to construct an actual functor $\ff:\loct(L)\lr\loct(S)$ using Family Floer techniques. The functor $\ff$ is defined for all paths and built as a direct limit of functors $\ff_\e$ depending on the adiabatic parameter. This is presented in \cref{subsection:familyfloer}; here $\ff$ stands both for Family and for Floer. Then we establish the comparison between these two functors in \cref{ssec:detour_are_continuation_strips}. The results can be summarized as follows:\\

\begin{ftheo}[\color{blue}Family Floer and Non-abelianization\color{black}]\label{thm:FamilyFloer_NonAbelianization} Let $S$ be a Betti surface, $\blag\sse(T^*S,\la_\std)$ an exact Betti Lagrangian and $\snetwork\sse S$ a spectral network compatible with $\scurve$. Then:\\

\begin{enumerate}
    \item Family Floer with respect to $L$ can be defined and yields a functor
$$\ff:\loct(L)\lr\loct(S),$$
invariant under compactly supported Hamiltonian isotopies of $L$. Specifically, it is built from a type of Floer complexes $CF^\ast(T^\ast_z S,L)$ between cotangent fibers and $L$, which can be constructed and shown to be well-defined.\\

\item In the adiabatic limit $\e\to0$, the two functors
$$\ff,\Phi_\snetwork:\loct(\e L)\lr\loct(S),$$
are equivalent.
\end{enumerate}
\end{ftheo}
\vspace{0.5cm}

\noindent As said above, the Floer-theoretic functor $\ff$ in \cref{thm:FamilyFloer_NonAbelianization}.(1) is defined as a direct limit of Family Floer functors $\ff_\e$ that depends on the adiabatic parameter $\e$: at core, $\ff$ uses the spectral network $\snetwork$ in order to correct the technical issues that often obstruct when applying Family Floer methods, including caustics. The comparison in \cref{thm:FamilyFloer_NonAbelianization}.(2) does require scaling the Betti Lagrangian and thus the comparison $\ff=\Phi_\snetwork$ only makes sense in the adiabatic limit. \cref{thm:FamilyFloer_NonAbelianization} is proven in \cref{subsection:familyfloer} for Part (1) and \cref{ssec:detour_are_continuation_strips} for Part (2). Note that the case of quadratic differentials, with $\GL_2(\C)$, can be argued from \cite{nho2024familyfloertheorynonabelianization}, whereas \cref{thm:FamilyFloer_NonAbelianization} works for any higher rank exact Betti Lagrangians. We emphasize that the functor $\ff$ is well-defined whereas, technically, $\Phi_\snetwork$ cannot parallel transport starting at $\snetwork$: \cref{thm:FamilyFloer_NonAbelianization}.(2) is thus an equality where $\Phi_\snetwork$ is defined. More conceptually, it shows how the framework of Floer theory not only captures but generalizes that of spectral networks: the former being able to be used to define and extend the latter. An interesting question is whether \cref{thm:FamilyFloer_NonAbelianization} could be extended to general Betti Lagrangians. For such Lagrangians, we are currently only able to explicitly compute parallel transport maps for very short line segments. Such a computation is the key component behind the proof of \cref{thm:characterization}. We leave this question for the future research.\\

A corollary of \cref{thm:FamilyFloer_NonAbelianization} is that we can extract a symplectic invariant of the exact Betti Lagrangians $L\sse T^*S$ from $\snetwork$, up to compactly supported Hamiltonian isotopies. Indeed, \cref{thm:FamilyFloer_NonAbelianization} implies that the set of soliton classes in $\snetwork$ associated to Reeb chords of $\dd_\infty L$ is an invariant of $L$. This is proven in \cref{thm:pathdetourclassesareHaminvariant}.\footnote{It is an effective invariant, as it at least matches the information contained in the augmentation induced by $L$ for the Legendrian contact dg-algebra of $\dd_\infty L$, when the latter is defined. Thus it already distinguishes Lagrangians before and after a surgery.} Note that the cluster algebra structures built in \cite{casals2024clusterstructuresbraidvarieties}, corresponding to a Betti surface $\obis=S^2$ with one irregular singularity, have cluster variables defined by Floer parallel transport along relative cycles. In combination with that work, \cref{thm:FamilyFloer_NonAbelianization} then formalizes the cluster-like coordinates sketched in the supersymmetric context in \cite[Section 10]{GNMSN} and \cite{MR3263304}.\\

{\bf D.} Fukaya $A_\infty$-categories, in their many forms, are a central object in Floer theory and the study of symplectic topology via pseudo-holomorphic methods, see e.g.~\cite{FOOO1,GPSCV,SeidelZurich}. The Floer complexes we constructed for \cref{thm:FamilyFloer_NonAbelianization} and the results developed in Sections \ref{section_adg} and \ref{section_Floer} can be enhanced to an $A_\infty$-categorical level. This is achieved in \cref{section_wfc}, as follows:\\

\begin{itemize}
    \item[(i)] First, due to its asymptotic ends, an exact Betti Lagrangian $L\sse T^*S$ does not define an object of $\wfc$ or any partially wrapped version thereof. That said, we construct in \cref{subsection:WFC} an $A_\infty$-module $\ylagv$ of $\wfc$ that behaves as if it was the Yoneda module for such an object. More crucially, we then prove in \cref{subsec:prooffamilyfloerwrapped} that its $\mu_{1\vert1}$ $A_\infty$-operation precisely captures the non-abelian parallel transport.\\

    \item[(ii)] Second, in \cref{subsection:PWFC} we geometrically modify $L$ via cylindrization so as to obtain a Lagrangian $\blag_{\circ}$ which defines an object in the partially wrapped Fukaya category $\pwfc$. We also establish that its associated $\mu_{2}$ $A_\infty$-map recovers non-abelian parallel transport, now in the partially wrapped case.\\

    \item[(iii)] Third, we compare the $A_\infty$-modules presented in (i) and (ii) above under the inclusion functor $i_*:\wfc\to\pwfc$ and completely describe the Yoneda $A_\infty$-module of $\blag_{\circ}$. This description of this $A_\infty$-module of $\pwfc$ associated to the cylindrized Betti Lagrangian $\blag_\circ$ is achieved in \cref{subsection:partialyonneda} by proving a generation result for $\pwfc$, cf.~\cref{prop:gen_internal_infinity_fibers}, showing that its generated by internal and infinity fibers, without linking disks being in the generating set.
\end{itemize} 

\noindent These are all established in detail in \cref{section_wfc}. See \cref{prop:A-inftymodule}, \cref{thm:Familyfloerwrapped} and \cref{cor:familyfloerwrapped} for (i), \cref{sssec:cylindrize_BettiLag} and \cref{thm:pwfcyonneda} for (ii), and \cref{cor:pwfcyonnedamodule} and \cref{prop:gen_internal_infinity_fibers} for (iii). We summarize some of the key outcomes in the following result:

\begin{ftheo}[\color{blue} Spectral curves and 4d Partially Wrapped Fukaya Categories\color{black}]\label{thm:specnet_fuk} Let $L\sse (T^*S,\la_\std)$ be an exact Betti Lagrangian endowed with a local system $V\in\Loc(L)$ and $\snetwork\sse S$ a compatible Morse spectral network. Then:

\begin{itemize}
    \item[(i)] There exists an $A_{\infty}$-module
$\ylagv:\wfc\lr\kmod$ with cohomologies
$$H^{\ast}(\ylagv(\tlag))\cong HF^\ast(\tlag,(\blag,V)),$$
where $\tlag\sse T^*S$ is any exact cylindrical Lagrangian with compact horizontal support. Furthermore, the minimal geodesic generator $[\alpha_{zw}]\in \mbox{Hom}_{\wfc}(\fibre{z},\fibre{w})$ of an $\snetwork$-adapted pair $z,w\in S$ satisfies
\begin{equation*}\label{eq:mu11_paralleltrans}
\mu_{1\vert 1}([\alpha_{zw}],\cdot)=\tdfd(\alpha_{zw}).
\end{equation*}
In particular, there exists an $A_\infty$-quasi-equivalence $\wfc\mbox{-mod}\cong C_{-\ast}(\Omega_zS)$-mod,
$$\ylagv\simeq\Phi_\snetwork(V),$$
as $A_{\infty}$-modules over $\wfc$.\\

\item[(ii)] Consider the Yoneda $A_\infty$-module $\ycylv:\pwfc\lr\kmod$ of the cylindrized Betti Lagrangian $\blag_\circ\sse T^*S$ and its $\mu_2$ $A_\infty$-operation. Then the minimal geodesic $\alpha_{zw}\in\Omega_{z,w}$ satisfies
\begin{equation*}\label{eq:mu2_paralleltrans}
\mu_2(i_{\ast}[\alpha_{zw}],\cdot )=\tdfd(\alpha_{zw})
\end{equation*}
and there is a homotopy
$$\ycylv\circ i_{\ast}\simeq \ylagv$$
of $A_\infty$-modules over $\wfc$.
\end{itemize}

\end{ftheo}
\vspace{0.5cm}

At core, \cref{thm:specnet_fuk} establishes that the non-abelian parallel transport of Gaiotto-Moore-Neitzke studied in \cite{GMN12_WallCrossCoupled,GNMSN,GMN14_Snakes}, and the associated counts of framed 2d-4d BPS states, can be realized as parts of the $A_\infty$-operations in either wrapped or partially wrapped 4-dimensional Fukaya categories. \cref{thm:specnet_fuk}.(i) does so in the wrapped category $\wfc$ and \cref{thm:specnet_fuk}.(ii) in the partially wrapped category $\pwfc$, also establishing the comparison. At a technical level, we use in \cref{section_wfc} the modern construction of partially wrapped categories in \cite{GPSCV} and establish our results by indeed using their definitions in terms of homotopy colimits, carefully defining the class associated to a minimal geodesic $\alpha_{zw}$ by studying the necessary continuation maps, cf.~ Sections \ref{subsec:prooffamilyfloerwrapped} and \ref{sssec:proof_blag_cyl_mu2}.\\

\begin{comment}

\noindent In Theorem \ref{thm:specnet_fuk}, it is important to note that the presence of conical ends prevent Betti Lagrangians $L$ from directly defining objects in a partially wrapped Fukaya category. These conical ends are key for the analysis of spectral networks undergone in the previous sections. Now, Theorem \ref{thm:specnet_fuk}.(i) in particular states that such an $L$ still behaves as if it defined an object, e.g.~there exists a module that behaves as the Yoneda module of such putative object. Also, it is alternatively possible to modify $L$ via cylindrization so as it defines an object in a partially wrapped Fukaya category, but then we must ensure the such an object still reflects the properties of the initial spectral network for $L$, and viceversa: in part, this is the statement about $\mu_2$ in Theorem \ref{thm:specnet_fuk}.(ii); cf.~Section \ref{section_Floer} for further details.\\

\end{comment}
%%%%%%%%%%%%%%%%%%%%%%%%%%%%%%%%%%%%%%%%%%%%%%%%%%%
%%%%%%%%%%%%%%%%%%%%%%%%%%%%%%%%%%%%%%%%%%%%%%%%%%%

{\bf E.} Theorems \ref{thm:existence} and \ref{thm:characterization} establish existence and characterizations of spectral networks, and Theorems \ref{thm:FamilyFloer_NonAbelianization} and \ref{thm:specnet_fuk} develop the Floer-theoretic results needed to rigorously define and establish non-abelianization in spectral networks, first analytically and then categorically, respectively. The last result we present is a new construction of spectral networks, providing both novel examples and crystallizing connections to cluster algebras.

The input of the construction is a Legendrian weave $\w\sse S$, as introduced in \cite{legendrianweaves}, specifically a Demazure weave, cf.~\cref{ssec:weaves_Betti_surfaces}, \cite[Section 4.1]{casals2024algebraicweavesbraidvarieties} or \cite[Section 4.1]{casals2024clusterstructuresbraidvarieties}. Such a weave $\w\sse S$ defines an exact Betti Lagrangian $L_\w\sse(T^*S,\la_\std)$ with asymptotics determined by a Legendrian $\La_{\dd\w}\sse(T^\infty S,\xi_\std)$, depending only on the boundary $\dd\w$ of the weave. In \cref{ssec:networks_Demazure_weaves} we define a certain object $\snetwork_\w\sse S$, associated to a weave $\w$, entirely in combinatorial terms: we refer to it as the {\it augmentation forest} of $\w$. These $\snetwork_\w\sse \R^2$ are built combinatorially from $\w$, in contrast to the intrinsically analytic nature of any spectral network associated to $L_\w$ and its associated augmentation $\e_\ww:\SA(\La_{\dd\w})\lr k[H_1(L_\ww)]$ of the Legendrian contact dg-algebra $\SA(\La_{\dd\w})$. Figure \ref{fig:intro_specnet_weaves} illustrates two examples of weaves and their associated spectral networks.
\begin{center}
	\begin{figure}[h!]
		\centering
		\includegraphics[scale=1.3]{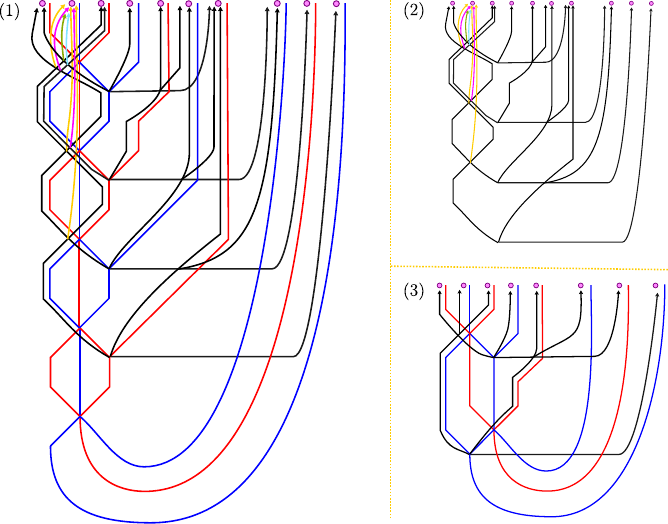}
		\caption{(1) A weave $\ww$ and the spectral network $\snetwork_\ww$ constructed in \cref{ssec:networks_Demazure_weaves}. (2) The spectral network $\snetwork_\ww$ in (1), drawn on its own, without the weave $\ww$. (3) Another weave with its associated spectral network. In this case, the spectral network recovers precisely the Berk-Nevins-Roberts network from \cite{BerkNevinsRoberts82_NewStokes} by adiabatically degenerating the rigid pseudoholomorphic strips contributing to the augmentation associated to the Lagrangian filling of the weave.}
        \label{fig:intro_specnet_weaves}
	\end{figure}
\end{center}
For simplicity, we focus on $S=\R^2$ and their $\snetwork_\w\sse \R^2$, as this already yields many interesting new examples; see ~\cite{CasalsNg22} for specifics on the Legendrian dg-algebra $\SA(\La_{\dd\w})$. In \cref{section_weave}, after providing the construction of the augmentation forest $\Theta_\w$, we conclude the following result and use it to explicitly compute several examples, cf.~\cref{ssec:explicit_comp}:
%a wide class of Betti exact Lagrangians and many new spectral networks.

\begin{ftheo}[\color{blue} Augmentations, weaves and Spectral Networks\color{black}]\label{thm:aug_graphs} Let $\w\sse \R^2$ be a Demazure weave and $L_\w\sse (T^*\R^2,\la_\std)$ its associated exact Betti Lagrangian. Then:
\begin{enumerate}
    \item The augmentation forest $\snetwork_\w\sse \R^2$ is a Morse spectral network compatible with $L_\w$.\\
    In addition, it is a finite and creative spectral network.\\
    
    \item The flowlines of the spectral network $\snetwork_\w$ are exactly given by the union of $\dfs$-trees obtained by adiabatically degenerating the rigid pseudoholomorphic strips contributing to the augmentation $\e_\ww$ induced by $L_\ww$.
\end{enumerate}
\end{ftheo}
\vspace{0.5cm}
\noindent

Theorem \ref{thm:aug_graphs} has two advantages. First, there are many weaves with the same boundary $\dd\w$, e.g.~\cite[Section 2]{CasalsGao22} and \cite[Section 7.1]{legendrianweaves}, and thus Theorem \ref{thm:aug_graphs} provides many examples of spectral networks with the {\it same} asymptotic conditions: the spectral curves $L_\w$ for different such weaves $\w$ are typically {\it not} Hamiltonian isotopic, but they are all exact Lagrangian fillings of the same $\La_{\dd\w}$. The associated spectral networks $\Theta_\w$ are themselves also different: in fact, the notion of (Lagrangian disk) mutation for Betti Lagrangians aligns with the notion of mutation of spectral networks, corresponding to the appearance of a 4d BPS state in a degenerate spectral curve, cf.~\cref{sssec:2strands_examples_2}. Second, since the construction of $\Theta_\w$ is combinatorial, not requiring to solve any differential equation, Theorem \ref{thm:aug_graphs} is an effective tool for understanding which spectral networks are associated to irregular data. In particular, given a higher-order linear differential equation -- e.g.~with new Stokes curves as in \cite{BerkNevinsRoberts82_NewStokes} -- one can qualitatively draw spectral networks associated to it once the Legendrian associated to the Stokes data is computed, which only requires the asymptotic range. \cref{ssec:explicit_comp} gives several such examples and computations.\\

\noindent{\bf Acknowledgements}: We thank Merlin Christ, Ian Le, Andrew Neitzke and Daping Weng for useful discussions. Specifically, the construction in \cref{ssec:networks_Demazure_weaves} was first discussed by R.C.~with I.~Le and D.~Weng in relation to another project and we are thankful for their input and generosity. R.C.~is also particularly grateful to M.~Aganagi\'{c} and P.~Zhou for their interest and encouragement on this work during his visit to their String Theory-Math Seminar at UC Berkeley. Y.J.N.~is grateful to Yong-Geun Oh for his interest and hospitality at the Institute for Basic Science (IBS). R.C.~is supported by the NSF CAREER DMS-1942363, a Sloan Research Fellowship of the {\color{black} Alfred P. Sloan Foundation} and a UC Davis College of L\&S Dean's Fellowship.  \hfill$\Box$

    %\newpage
%%%%%%%%%%%%%%%%%%%%%%%%%%%%%%%%%%%%%%%%%%%%%%%%%%%%%%%%%%%%%%%%%%%%%%%%
%%%%%%%%%%%%%%%%%%%%%%%%%%%%%%%%%%%%%%%%%%%%%%%%%%%%%%%%%%%%%%%%%%%%%%%%
%%%%%%%%%%%%%%%%%%%%%%%%%%%%%%%%%%%%%%%%%%%%%%%%%%%%%%%%%%%%%%%%%%%%%%%%

\section{Betti Lagrangians and Flowline Trajectories}\label{section_setup}
 %%%%%%%%%%%%%%%%%%%%%%%%%%%%%%%%%%%%%%%%%%%%%%%%%%%%%%%%%%%%%%%%%%%%%%%%
%%%%%%%%%%%%%%%%%%%%%%%%%%%%%%%%%%%%%%%%%%%%%%%%%%%%%%%%%%%%%%%%%%%%%%%%
%%%%%%%%%%%%%%%%%%%%%%%%%%%%%%%%%%%%%%%%%%%%%%%%%%%%%%%%%%%%%%%%%%%%%%%%

This section introduces Betti surfaces in Section \ref{subsection:Betti_surface} and Betti Lagrangians in \ref{subsection:scurves}, where exact Betti Lagrangians and meromorphic spectral curves are also discussed. Real and holomorphic flowlines are discussed in Sections \ref{subsection:Morselines} and \ref{subsection:holMorse} respectively. Accordingly, the trapping lemma is proved in Section \ref{subsection:trappinglemma} and the asymptotics of WKB trajectories are studied in Section \ref{subsection:asymptoteWKB}. We start with a brief discussion on real and holomorphic cotangent bundles in \cref{subsection:prelim_cotangent}.

%%%%%%%%%%%%%%%%%%%%%%%%%%%%%%%%%%%%%%%%%%%%%%%%%%%%%%%%%%%%%%%%%%%%%%%%
%%%%%%%%%%%%%%%%%%%%%%%%%%%%%%%%%%%%%%%%%%%%%%%%%%%%%%%%%%%%%%%%%%%%%%%%

\subsection{Real and holomorphic cotangent bundles}\label{subsection:prelim_cotangent}

Let $S$ be a smooth real surface, we denote by $(T^*S,\la_\std)$ its real cotangent bundle endowed with the canonical Liouville 1-form $\la_\std\in\Omega^1(T^*S)$. It is the unique 1-form such that $\eta^*\la_\std=\eta$ for any $\eta\in\Omega^1(S)$, and its differential is denoted by $\omega_\std=d\la_\std$. The ideal contact boundary \cite[Prop.~2]{Giroux20_Ideal} of this Liouville domain is denoted $(T^\infty S,\xi_\std)$. In this article, we focus on real Lagrangian surfaces $L\sse (T^*S,\la_\std)$ with boundary on a specified Legendrian $\La\sse(T^\infty S,\xi_\std)$. The Legendrian boundary $\La$ will always be smooth, and $L$ will typically be embedded, though we also consider immersed, and more singular, Lagrangian subsets in certain parts. We refer to \cite{ArGi,ArnoldSing} or \cite{Geiges08,McDuffSalamon98_Intro} for initiating matters on contact and symplectic topology. \\

Let $S$ be endowed with a Riemann surface structure $\sS$, so that $\sS$ is a complex 1-dimensional manifold and $J_\sS:TS\lr TS$ is its defining (almost) complex structure. The holomorphic cotangent bundle will be denoted by $(\T^*\sS,\omega_\C)$, and its canonical (complex-valued) holomorphic symplectic form by $\omega_\C\in\Omega_\C^2(\T^*\sS)=\Omega^{2,0}(T^*S,J_\sS)$. Here the splitting of $\Omega^{\bullet}(T^*S)$ into its $\Omega^{i,j}(T^*S,J_\sS)$ parts is determined by the eigenspaces of $J_\sS$. As in the real case, this 2-form can be written as $d\omega_\C=\la_\C$, where the primitive $\la_\C\in\Omega_\C^1(\T^*\sS)=\Omega^{1,0}(T^*S,J_\sS)$ is the canonical holomorphic Liouville form. In this holomorphic context, holomorphic Lagrangian submanifolds will be denoted by $\scurve\sse (\T^*\sS,\omega_\C)$. Given $(\T^*\sS,\omega_\C)$, we can consider its underlying real rank-2 bundle $T^*S$ and the real-valued smooth symplectic 2-form $\Re{(e^{i\theta}{\omega_\C})}$, for a choice of $\theta\in S^1$. We refer to this pair $(T^*S,\Re{(e^{i\theta}\omega_\C)})$ over $S$ as the $\theta$-part $(\T^*\sS,\omega_\C)$; when $\theta=0$, this is also said to be the real part of $(\T^*\sS,\omega_\C)$.

\begin{remark} (i). Let $\theta\in S^1$ and $\scurve\sse (\T^*\sS,\omega_\C)$ be a holomorphic Lagrangian submanifold. Then, then underlying real submanifold $L\sse T^*S$ of $\scurve$ is a real Lagrangian submanifold of $(T^*S,\Re{(e^{i\theta}\omega_\C)})$, where we can identify $TL$ and $T\scurve$ as real rank-2 bundles. Indeed, $\omega_\C|_{T\scurve}\equiv 0$ if and only if $e^{i\theta}\omega_\C|_{T\scurve}\equiv 0$, which implies that $\Re{(e^{i\theta}\omega_\C|_{T\scurve})}\equiv 0$.\\

(ii) Studying the real Lagrangian $L$ in $(i)$ in the $\theta$-part $(T^*S,\Re{(e^{i\theta}\omega_\C)})$ is equivalent to studying the real Lagrangian underlying $e^{i\theta}\scurve$ for $(T^*S,\omega_\std)$, where $e^{i\theta}\scurve$ is the result of applying fiberwise multiplication to $\scurve$ by the scalar $e^{i\theta}$. We thus denote such real Lagrangians as $L=e^{i\theta}\scurve$.\\

(iii) For a generic $\theta\in S^1$, the real Lagrangian $e^{i\theta}\scurve$ submanifold of $(T^*S,\omega_\std)$ is not exact, i.e. $\la_\std$ defines a closed but non-exact real 1-form when restricted to $e^{i\theta}\scurve$. For instance, the real part of $\scurve$ is exact if and only if all the complex periods of $\scurve$ are purely imaginary. Nevertheless, note also that for generic $\theta\in S^1$, $e^{i\theta}\scurve\sse (T^*S,\omega_\std)$ does not bound any holomorphic disks either (in an appropriate adiabatic limit) and, in many aspects, behaves similarly to an exact Lagrangian.\hfill$\Box$
\end{remark}

\begin{center}
	\begin{figure}[h!]
		\centering
		\includegraphics[scale=0.7]{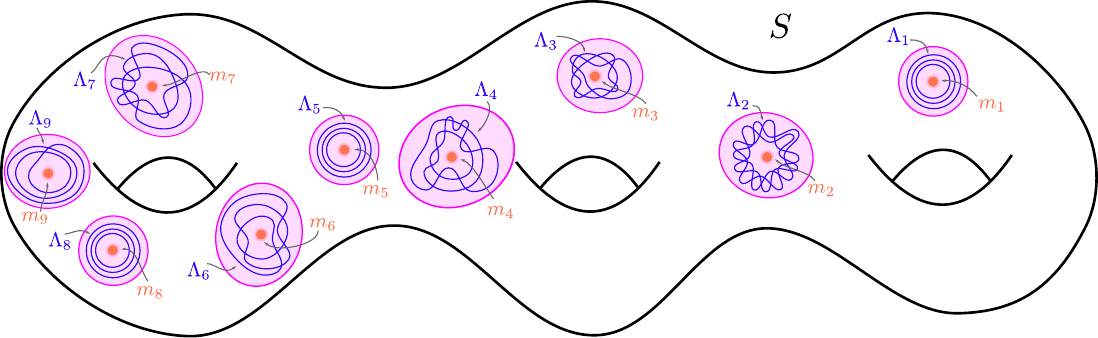}
		\caption{A smooth surface ${\bf S}$ of genus $3$ with 9 marked points $\mkpts=\{m_1,\ldots,m_9\}$, in orange, and fronts for their associated Legendrian links in ${\bf \lknot}=\{\La_1,\ldots,\La_9\}$, in blue. The annular neighborhoods $A_i\sse S$ for each marked point are shaded in pink. The positive braid words $\beta_1,\beta_5,\beta_8$ are the empty braid words in 3-strands and, for instance, we can choose $\beta_6=\s_2^2\s_1\s_2\s_1^2$ and $\beta_9=\s_1\s_2$.}
		\label{fig:BettiSurface_Setup}
	\end{figure}
\end{center}

%%%%%%%%%%%%%%%%%%%%%%%%%%%%%%%%%%%%%%%%%%%%%%%%%%%%%%%%%%%%%%%%%%%%%%%%
%%%%%%%%%%%%%%%%%%%%%%%%%%%%%%%%%%%%%%%%%%%%%%%%%%%%%%%%%%%%%%%%%%%%%%%%
%%%%%%%%%%%%%%%%%%%%%%%%%%%%%%%%%%%%%%%%%%%%%%%%%%%%%%%%%%%%%%%%%%%%%%%%

\subsection{Betti surfaces}\label{subsection:Betti_surface} The asymptotics of Betti Lagrangians will be constrained by Legendrian links. Given a marked oriented closed surface $(\bf S,\mkpts)$, denote $S=\bf S\setminus\mkpts$ and consider a set $\{A_1,\ldots,A_{|\mkpts|}\}$ of small annular neighborhoods $A_i\sse S$, each obtained by considering a small disk neighborhood of the $i$th marked point in $S$ and removing such marked point. The required asymptotics of $L\sse(T^*S,\la_\std)$ will be prescribed by a set ${\bf \lknot}=\{\La_1,\ldots,\La_{|\mkpts|}\}$ of Legendrians, each in the ideal contact boundary $\La_i\sse(T^\infty A_i,\xi_\std)$. Since projection of the ideal contact boundary onto the zero section is a Legendrian fibration, the front for $\La_i$ can (and will) be drawn in $A_i\sse S$. See Figure \ref{fig:BettiSurface_Setup} for a depiction of this setup. Given a collection $\beta=\{\beta_1,\ldots,\beta_{|\mkpts|}\}$ of positive braid words $\beta_i\in\mbox{Br}_n^+$ in $n$-strands, $i\in[1,|\mkpts|]$, consider the Legendrian link $\La_\beta\sse(T^\infty S,\xi_\std)$ whose front is given by drawing the braid word $\beta_i$ circularly around $A_i$, according to the orientation of $S$; cf.~\cite[Section 2.2]{CasalsNg22} and see again Figure \ref{fig:BettiSurface_Setup}. The Legendrian isotopy class of $\La_\beta$ is independent of the braid word representative of its associated braid element in the braid group.

\begin{remark}\label{rem:positivebraid}
The braids in $\beta$ need not be algebraic braids, i.e.~they are not constrained to be certain iterated cables of the unknot. See the proof of \cite[Corollary 1.6]{CasalsGao22} for infinitely many examples of non-cable satellite knots and hyperbolic knots that are associated to positive braids (and necessarily not algebraic). In the meromorphic setting, as in the literature on spectral curves \cite{GMN13_Framed,GNMSN,GMN14_Snakes}, the braids are necessarily algebraic.\hfill$\Box$
\end{remark}

\noindent Formalizing this, a marked oriented closed surface $(\bf S,\mkpts)$ with such Legendrians links leads to the following:

\begin{definition}[Betti surfaces]\label{def:bis}
A Betti surface of rank $n$ is a triple $({\bf S},\mkpts;{\bf \lknot})$ such that ${\bf \lknot}$ is a collection of Legendrian links in the ideal contact boundary $(T^\infty S,\xi_\std)$, each of which is associated to an $n$-stranded positive braid circularly drawn around a marked point in $\mkpts$.\hfill$\Box$
\end{definition}

See Figure \ref{fig:BettiSurface_Setup} for an example of Definition \ref{def:bis}. For each $\La_{\beta_i}\sse(T^\infty A_i,\xi_\std)$, a Reeb chord is said to be positive if, when seen as a crossing in the Lagrangian projection, the upper strand travels from northwest to southeast with respect to the given orientations on the strands; else, we say it is \textit{negative}. For notation purposes, $\La_{\beta}$ is said to be Reeb-positive if all the Reeb chords of all $\La_i$ are positive. After a $C^\infty$-small Legendrian isotopy, we can and do assume that each of the Reeb chords and transverse double points (of the fronts in $A_i$) are at different angles. The angles in the core $S^1\sse A_i$ where there is a Reeb chord for $\La_{\beta_i}$ are said to be \textit{crossing points} for that core circle.

\begin{remark}
By \cite[Section 2.2]{CasalsNg22}, if $\beta_i$ contains a half-twist $w_0$ as a subword, then the braid for its front can be drawn such that Reeb chords are in bijection with the crossings of $\beta$. Specifically, the front can be drawn so that there is exactly one Reeb chord (arbitrarily close and) preceding each crossing, and these are all the Reeb chords.\hfill$\Box$
\end{remark}
For Floer-theoretic purposes, it is useful to regard the marked points $\mkpts$ as giving the \textit{horizontal infinity} of $\obis$, e.g.~as opposed to a 1-point compactification of $S$. In detail, for each $m_i\in\mkpts$, we identify the associated annular neighborhood $A_i\sse S$ with $A_i\cong S^1\times (1,\infty)$ with coordinates $(e^{i\theta},r)\in S^1\times (1,\infty)$, so that any point converges to the point $m_i$ in $\bis$ as $r\to \infty$. For a Reeb chord with a given crossing point $\theta_0$, we call the ray $(e^{i\theta_0},r)\sse A_i$ its \textit{Reeb chord ray}. 

\begin{ex}[Betti surfaces from Stokes data]\label{example:BettiSurfaces_StokesData}
As hinted in \cite[Remark 2.17]{MR2483750}, the data of the irregular singularities of a meromorphic connection on a Riemann surface $\sS$ leads to a Betti surface. This is beautifully crystallized in \cite{Boalch21_TopologyStokes}, specifically in the notion of an irregular curve $(\sS,\mkpts,\Theta)$, as defined in \cite[Definition 8.1]{BoalchStokesunramified} or \cite[Definition 5.5]{Boalch21_TopologyStokes}; cf.~also \cite[Part I]{Sabbah13_StokesBook}, \cite[Section 3.3]{STWZ}, \cite[Appendix B]{Su25_DualBoundary}. In a nutshell, an irregular curve is a marked Riemann surface $(\sS,\mkpts)$ with the data $\Theta$ of an irregular class at each marked point: the irregular class codifies the formal type of the pole of a meromorphic connection, up to formal gauge transformation. In practice, this can be written as a finite Puiseux series $Q$, representing an orbit of the Galois action of the profinite group $\hat{\mathbb{Z}}$ on $\mathbb{C}((z^{1/\infty}))/z^{-1}\mathbb{C}[[z^{1/\infty}]]:=\bigcup_{N\in \mathbb{N}}\mathbb{C}((z^{1/N}))/z^{-1}\mathbb{C}[[z^{1/N}]]$. Here we will work with representatives $Q$ of the form
\[\irreg=\sum_{k_j\in \mathbb{Q}_{>0}} c_{k_j}z^{-k_j},\quad c_{k_j}\in\mathfrak{t}_{reg}\sse\mathfrak{g}\mathfrak{l}_n,\] 
%{\RC HOW IS THIS NOT MATRIX VALUED?}
modulo the action of $c_{k_j}\to c_{k_j}e^{2\pi i n/\ram{\irreg}}, n\in[1,\ram{\irreg}]$, where $\ram{\irreg}$ is the smallest number such that $\irreg\in \mathbb{C}((z^{1/\ram{\irreg}}))$. The irregular part of the formal type of the meromorphic connection is then given by $d-dQ$. Since $Q$ is a finite Puiseux series with pole orders strictly greater than $0$, it can be seen as a multivalued function on $\C^*$. The Stokes diagram associated to such irregular data is given by the graphs of the multivalued functions $\Re( \irreg(re^{i\theta}))$ on $S_\theta^1$, for fixed $r$ large enough. For each marked point $m_i$, the Stokes diagram in $A_i$ is precisely the circular closure of a positive braid, and thus can be interpreted as a front for a Legendrian of the type $\La_{\beta_i}$ above.

\noindent In summary, the Betti surface associated to an irregular curve $(\sS,\mkpts,\Theta)$ is precisely the underlying smooth surface $({\bf S},\mkpts,{\bf\La}_\Theta)$, where ${\bf\La}_\Theta$ is the Legendrian link in $(T^\infty S,\xi_\std)$ defined by the Stokes diagrams of the irregular data $\Theta$, understood as fronts in $S$. In these examples, the points of maximal decay for the Stokes diagram of an irregular type give the Reeb rays for the associated Legendrian link; these are also referred to as singular directions or anti-Stokes directions of the Stokes diagram. By construction, the Stokes directions (a.k.a.~oscillatory directions) of the Stokes diagram are precisely the directions where there is a crossing in the front of the associated Legendrian link, i.e.~the angles at which we draw the crossings of the associated braid word.\hfill$\Box$
\end{ex}

%%%%%%%%%%%%%%%%%%%%%%%%%%%%%%%%%%%%%%%%%%%%%%%%%%%%%%%%%%%%%%%%%%%%%%%%
%%%%%%%%%%%%%%%%%%%%%%%%%%%%%%%%%%%%%%%%%%%%%%%%%%%%%%%%%%%%%%%%%%%%%%%%
%%%%%%%%%%%%%%%%%%%%%%%%%%%%%%%%%%%%%%%%%%%%%%%%%%%%%%%%%%%%%%%%%%%%%%%%

\subsection{Betti Lagrangians}\label{subsection:scurves}
Given a Betti surface $(\bis,\mkpts,{\bf \La})$, we introduce Betti Lagrangians: the Lagrangian surfaces in $(T^{\ast}\obis,\la_\std)$ for which we construct and study spectral networks. These Lagrangians are our main object of study. There are two important classes, \textit{exact Betti Lagrangians} and \textit{meromorphic spectral curves}, discussed respectively in Subsections \ref{subsubsection:bLagandebLag} and \ref{subsubsection:iregcurvesandmeroscurves}. Many of the Lagrangian fillings of Legendrian links studied in the literature, e.g.~\cite{CasalsGao22,casals2023microlocal,legendrianweaves,EHKexactLagrangiancobordisms}, provide plenty of examples of this first class of exact Betti Lagrangians. The irregular curves studied in the context of wild character varieties and the irregular Riemann-Hilbert or non-Abelian Hodge correspondences give many examples of the second class, e.g.~cf.~\cite{BiquardBoalch04_WildNAH,BoalchStokesunramified,BoalchStokesunramified,MR4324961,GMN12_WallCrossCoupled,GNMSN,GMN14_Snakes,MR3322389,Sabbah13_StokesBook}.\\

To define Betti Lagrangians, we first discuss the notion of Lagrangian multigraphs on a smooth manifold $M$, as follows. By definition, a properly embedded Lagrangian submanifold $\multigraph\subset T^{\ast}M$ is said to be a \textit{Lagrangian multi-graph} if outside a codimension-$1$ properly closed subset $\caustL\sse\multigraph$, the projection $\pi: \multigraph\lr M$ is a smooth immersion outside $\caustL$ and it is a covering over $M-\pi(\caustL)$. %The set $\pi(\caustL)$ is called the set of \textit{singularities} of (the projection restricted to) $\multigraph$.
Given a point $m\in M-\pi(\caustL)$, a neighborhood of the preimage $\pi^{-1}(m)\cap (\multigraph-\caustL)$ in $M$ consists of $n$ disjoint Lagrangian disks, the \textit{smooth sheets} of $\multigraph$ over $m$. Each such smooth sheet can be realized as the Lagrangian graph of the differential of a smooth function $\sheet_i:\lr\R$, $i\in[1,n]$, with each $f_i$ well-defined up to constant. 

\begin{definition}[Betti Lagrangians]\label{definition:tamespectralcurve}
Let $(\bis,\mkpts,{\bf \La})$ be a Betti surface. A Betti Lagrangian $L\subset T^{\ast}\obis$ is a weakly bounded Lagrangian multigraph $L\subset T^{\ast}\obis$ such that the restriction of the projection $\pi:L\to \obis$ onto the zero section $S$ is a degree-$n$ simple branched cover with finitely many branched points.\hfill$\Box$
 \end{definition}

In Definition \ref{definition:tamespectralcurve} we have used the notation {\it weakly bounded}, which we use to encode a number of Riemannian properties of the Lagrangian multigraph $L$. Its precise meaning is as follows:

\begin{definition}\label{eq:weaklyboundedmultigraph}
A Lagrangian multigraph $L\sse T^*M$ is said to be \textit{weakly bounded} if there exists a Riemannian metric $(M,g)$, constants $\injradius,\sheetgap\in\R_+$ and a precompact subset $K\subset M$ containing $\pi(\caustL)$ such that:
\begin{itemize}
\item[(i)] The Riemannian metric $g$ is complete, the $C^\infty$-norm of all the derivatives of the curvature tensor are uniformly bounded (a.k.a.~geometrically bounded), and the minimal injectivity radius of $g$ is bounded below by $\injradius>0$. 
\item[(ii)]  Given a point $m\in M-N_{2\injradius}(K)$, at distance at least $2r$ from $K$, and the functions $\sheet_i:M\lr\R$ describing the sheets of $\multigraph$ over $B_{\injradius}(m)$ as $\mbox{gr}(df_i)$, then we must have
$$\inf_{B_{\injradius}(m)}\abs{d\sheet_i-d\sheet_j}>\sheetgap,\quad \forall q\in B_{\injradius}(m).$$

\end{itemize}
By definition, a weakly bounded Lagrangian multigraph $L$ is said to be \emph{uniformly bounded} if $L$ lies in the $R$-disk bundle $D_{\heightR}T^{\ast}\obis\sse T^*S$ for some $\heightR\in\R_+$ and for the smooth sheets of its fiberwise $\varepsilon$-scaled image $\varepsilon L\sse T^*S$ over $B_{\injradius}(m)$, $m\in M-N_{2\injradius}(K)$, there exist uniform bounds $C_k$ such that $\|\nabla^kf_i\|_g\leq C_k$ for all $i\in[1,n]$, where $\nabla$ is Levi-Civita connection of $(M,g)$ and there is no dependence on $\varepsilon$.\hfill$\Box$
\end{definition}

\noindent Intuitively, the second condition for a uniformly bounded Lagrangian multigraph in Definition \ref{eq:weaklyboundedmultigraph} is that away from the preimage of the locus $K$, the sheets of the rescaled Lagrangian $\varepsilon L\sse T^*S$ uniformly $C^{\infty}$-converge with respect to $g$ to the zero section as $\varepsilon\to 0$.\\\\
%%%%%%%%%%%%%%%%%%%%%%%%%%%%%%%%%%%%
It is worth providing a local description of the Betti Lagrangian near its branch point. The germ $\{w^2-z=0\}$ in $\mathbb{C}^2=T^{\ast}\mathbb{C}$ is called the  holomorphic cusp singularity, and we call its real part the $D_4^{-}$-singularity. Near its branch points, $L$ is always locally equivalent to the $D_4^{-}$-singularity germ. For the rest of the paper, We will assume that given a branch point $b$, there exists some local conformal coordinate $z$ near $b$, and some constant holomorphic covector $c$ such that near the ramification locus $\pi^{-1}(b)$, $L$ is equal to the real part of $ \{(\la_{\C}-c)^2=zdz^2\}\subset T^{\ast}_{\C}\mathbb{C}$ near $z=0$. We will furthermore assume that the smooth sheets of $L$ are all holomorphic. 

%%%%%%%%%%%%%%%%%%%%%%%%%%%%%%%%%%%%%%%%%%%%%%%%%%%%%%%%%%%%%%%%%%%%%%%%

\subsubsection{Exact Betti Lagrangians}\label{subsubsection:bLagandebLag} For exact Betti Lagrangians, the analysis uses the conical ends, as with the framework of Lagrangian fillings from \cite[Section 2.2]{EHKexactLagrangiancobordisms}. Consider the Legendrian fiber in $(T^\infty {\bf S},\xi_\std)$ above a marked point $m_i\in\mkpts$ and identify a neighborhood of it with $(J^1S^1,\xi_\std)$, where the fiber is mapped to the zero section. Suppose that the Legendrian link $\La_i\sse (T^\infty S,\xi_\std)$ is parametrically given by $(x(\theta),y(\theta),z(\theta))$ where $x(\theta)=\theta$ and the contact 1-form is $dz-ydx$, $(x,y)\in T^*S^1$ and $z\in\R$. Then a conical end is defined as follows:

\begin{definition}[Conical ends]\label{def:conicalend}
Let $\Laggerm\sse (T^*D^2,\la_\std)$ be the germ of a Lagrangian surface and $\La\sse (J^1S_\theta^1,\xi_\std)$ be parametrized by $(x(\theta),y(\theta),z(\theta))$, where $S^1$ is identified with the Legendrian fiber $T^\infty_0 D^2$ of $(T^\infty D^2,\xi_\std)$ above $0$. By definition, $\Laggerm$ has a conical end of type $\lknot$ if $\Laggerm$ admits a parametrization of the form
	\begin{align}\label{eq:LCend}
		(r,\theta)\longmapsto(x(\theta),f(\radial) y(\theta),\radial, f'(\radial)z(\theta))\in T^*D^2,
	\end{align}
 for some strictly increasing positive real function $f(r):[1,\infty)\longrightarrow \R$ which is linear at infinity. In other words, $L$ is the multigraph associated to the multi-valued function $z(\theta)f(\radial):S^1\times[1,\infty)\longrightarrow\R$.\hfill$\Box$ 
\end{definition}

\noindent The conical ends in Definition \ref{def:conicalend} allow us to introduce the following:

\begin{definition}[Exact Betti Lagrangians]\label{definition:topologicalexactspectralcurve}
Let $(\bis,\mkpts,{\bf \La})$ be a Betti surface. An \emph{exact Betti Lagrangian} $L\subset (T^{\ast}\obis,\la_\std)$ of rank $n$ is an exact Lagrangian submanifold $L\subset (T^{\ast}\obis,\la_\std)$ such that
	\begin{itemize}
		\item[(i)] The restriction $\pi:L\to \obis$ of the projection $T^*S\lr S$ onto the zero section is a degree-$n$ simple branched cover, with finitely many branch points.
		\item[(ii)] $L$ has conical ends near each of the marked points in $\mkpts$.\hfill$\Box$
	\end{itemize}
\end{definition}

\noindent Let $(\obis,g)$ be a Riemannian metric such that $g$ can be written as $g=d\radial^2+\radial^2d\theta^2$ near each of the horizontal infinities of the Betti surface $S$. Then the conical condition in Definition \ref{definition:topologicalexactspectralcurve}.(ii) implies that exact Betti Lagrangians are uniformly bounded with respect to such Riemannian metric $g$. Therefore, exact Betti Lagrangians are Betti Lagrangians, according to Definition \ref{definition:tamespectralcurve}.

%%%%%%%%%%%%%%%%%%%%%%%%%%%%%%%%%%%%%%%%%%%%%%%%%%%%%%%%%%%%%%%%%%%%%%%%

\subsubsection{Meromorphic spectral curves}\label{subsubsection:iregcurvesandmeroscurves}
Let us introduce a second important class of Betti Lagrangians, meromorphic spectral curves, associated to an irregular curve $(\sS,\mkpts,\Theta)$; see \cite[Definition 5.5]{Boalch21_TopologyStokes} and Example \ref{example:BettiSurfaces_StokesData} above for irregular curves. Intuitively, these spectral curves are given by meromorphic multi-graphs on $\sS$ with their behavior at a pole in $\mkpts$ specified by the irregular data $\Theta$. Locally, the behavior at poles can be described as follows.\\

\noindent Consider a germ of a holomorphic Lagrangian multi-graph $\scurve\sse T^{\ast}\mathbb{C}^{\ast}$ over $\C^*$ without singularities. Since the projection of $\scurve$ to the base $\C^*$ defines an $n$-sheeted covering, we can consider the elementary symmetric polynomial $\esym_i$, $i\in[1,n]$, over the distinct sheets $\lambda_1,\ldots,\lambda_n$ of $\scurve$ and obtain holomorphic functions $\esym_i(\lambda_1,\ldots,\lambda_n)$ on $\mathbb{C}^{\ast}$, $i\in[1,n]$. By definition, $\scurve$ is said to extend to a \textit{meromorphic multi-graph} over $\C$ if the elementary symmetric polynomials extend over $\mathbb{C}$ as meromorphic functions with poles over the origin. After possibly passing to a ramified cover, under $z\to z^r$ for some $r\in\N$, we regard the sheets $\lambda_1,\ldots,\lambda_n$ of $\scurve$ as giving holomorphic functions on $\mathbb{C}^{\ast}$. By Lagrange's upper bound $\displaystyle\max\{1,\sum_{i=1}^n \abs{\frac{{\esym_i}}{{\esym_n}}}\}$ on the norm of complex roots, cf.~\cite[Chapter IV]{MR439547}, there exists $N\in\N$ large enough such that the norm of $z^N\lambda_i(z^r)$ is uniformly bounded for all $i\in[1,n]$. Therefore, the functions $\lambda_i(z^r):\C^*\lr\C$ all extend to germs of meromorphic functions on $\C$ near the origin and we can regard the sheets of $\scurve$ as the Lagrangian graphs of Puiseux series with finite polar parts.\\

Given a meromorphic multi-graph over $\C$, as above, we consider the differential operator 
%\begin{equation}\label{eq:opereq}
$$\dd^n_z-\sum_{i=1}^{n}\esym_i(\lambda_1,\ldots,\lambda_n)\dd^{n-i}_z,\quad z\in\C.$$
%\end{equation}
By reducing this operator to a system of order $1$ differential operators, it defines a meromorphic connection on the rank-$n$ trivial bundle over $\mathbb{C}$ with a single pole at the origin. The irregular data is then given by considering the irregular class of such meromorphic connection. Intuitively, these Galois orbits %{\RC which are themselves multi-functions, because of branched cuts and roots of unity} 
are given by a local primitive of the holomorphic Liouville form $\lambda_\std$ over the distinct components of the multi-graph $\scurve$, after dropping the $o(\log(z))$ terms.

\begin{definition}[Meromorphic Spectral Curves]\label{def:meromorphicspectralcurves}
Let $\big(\sS,\mkpts, \iregc\big)$ be an irregular curve. A holomorphic submanifold $\scurve\subset T^{\ast}S$ is said to be a \emph{meromorphic spectral curve} if
\begin{itemize}
    \item[(i)] the germ of $\scurve$ is a meromorphic multi-graph with poles on $\mkpts$,
    \item[(ii)] the corresponding formal Puiseux series has the same orbits as $\iregc$, with the same multiplicities.
\end{itemize}
In addition, a meromorphic spectral curve is said to have $O(-1)$-ends if the sheets $\lambda_1,\ldots, \lambda_n$ and their difference functions have growth bounded below by $z^{-1}$, locally in $z\in\C^*$ with a pole at $z=0$.\hfill$\Box$
\end{definition}

For the purposes of our results, including Theorems \ref{thm:existence}, \ref{thm:characterization}, \ref{thm:aug_graphs} and \ref{thm:specnet_fuk}, we always consider meromorphic spectral curve with $O(-1)$-ends. First, $O(-1)$-ends are necessary because, without such hypothesis, one could add a holomorphic perturbation to force the difference between the sheets to blow-up slower than $z^{-1}$ as $z\to 0$. Second, having $O(-1)$-ends implies that the $\theta$-real part of a meromorphic spectral curve is a Betti Lagrangian as in Definition \ref{definition:tamespectralcurve}, as it implies that the underlying multi-graph is weakly bounded. In practice, here and in all situations we shall choose a K\"ahler metric on $\sS$ which coincides with $|z|^{-2}dz\otimes dz$ near the marked points.
\begin{comment}
e.g.~by adding $\log(z)$ terms corresponding to a regular summand of the meromorphic connection
N: z^{-1} still satisfies O(-1)-end condition
\end{comment}

\begin{remark} (i) If an irregular class in the irregular data $\Theta$ has multiplicity, the associated components of this same type coincide in $\mathbb{C}((z^{1/\infty}))/z^{-1}\mathbb{C}[[z^{1/\infty}]]$ and thus differ by a finite Puiseux series in $z^{-1}\mathbb{C}[[z^{1/\infty}]]$. In contact topological terms, the Legendrian link associated to an irregular class with multiplicity $m\in\N$ corresponds to the $m$-copy Reeb push-off of the Legendrian link associated to the irregular class (without multiplicity). This choice guarantees that the $O(-1)$-condition can be satisfied.\\

\noindent (ii) It follows from the construction of Stokes diagrams, see Example \ref{example:BettiSurfaces_StokesData}, that the condition that the difference of the sheets grows faster or equal than $z^{-1}$ is equivalent to requiring that the order of the poles of the multi-valued 1-form $dQ_i-dQ_j$, $i,j\in[1,n]$, are all greater or equal than $1$. That is, the (difference of the) polar terms of $Q$ themselves are of degree greater than $0$, e.g.~ no holomorphic parts.\\

(iii) An alternative viewpoint of meromorphic multi-graph is as follows, cf.~\cite{ionita2021spectralnetworksnonabelianization,kuwagaki2024genericexistencewkbspectral}. Given a rank $n$ meromorphic multi-graph, its cameral cover is given by taking all the possible permutations of the sheets. Then, each pair $i,j\in[1,n]$, $i<j$, induces a covering of $\mathbb{C}^{\ast}$ over which the germ of the cameral cover is realized as the germ of a rank $2$ quadratic differential near its pole. Then the condition that the difference of the sheets grow faster or equal than $z^{-1}$ is equivalent to requiring that the order of the poles of this quadratic differential are all greater or equal than $2$.\hfill$\Box$\\
\end{remark}

\begin{ex}\label{ex:irregularclass_Legendrian}
(1) {\it Airy Equation}. As an explicit instance of Example \ref{example:BettiSurfaces_StokesData}, consider the trivial $\C^2$-bundle over the Riemann surface $\sS=\C$ associated to the Airy differential equation $f''(x)=xf(x)$, $x\in\C$. It has an irregular singularity at $x=\infty$, which we take to be the unique marked point of the associated irregular curve $\CP^1=\sS\cup\{\infty\}$; this is the equation originally studied by G.~Stokes \cite{Stokes_2009}. In the $x$-coordinate, the associated meromorphic connection can be written as
$$\nabla_m:=d- \begin{bmatrix}
0 &  1\\
x &  0\\
\end{bmatrix}dx=d-dQ,\quad Q=\begin{bmatrix}
0 &  x\\
x^{2}/2 &  0\\
\end{bmatrix}, \quad x\in\C.$$
The formal irregular type of $\nabla_m$ at $x=\infty$ is given by $\langle x^{3/2}\rangle$, as the eigenvalues of $Q$ are $e^{\pm x^{3/2}/\sqrt{2}}$. Thus we have the irregular curve $(\sS,\mkpts,\Theta)=(\CP^1,\infty,\langle x^{3/2}\rangle)$. Its associated Stokes diagram at $x=\infty$ is described by the braid word $\beta=\sigma_1^3$, so the Betti surface is $(S,\mkpts,\lknot)=(S^2,\infty,\La_{\s_1^3})$. After satelliting the Legendrian zero section of $(J^1S^1,\xi_\std)$ to the max-tb Legendrian unknot in $(\R^3,\xi_\std)$, where $S^1$ is the Legendrian fiber of $T^\infty S^2\lr S^2$ at infinity, the Legendrian knot $\La_{\s_1^3}$ satellites to the max-tb unknot itself, presented in the front as the $(-1)$-closure of $\beta=\sigma_1^3$. Any $\theta$-part of the associated meromorphic spectral curve $\scurve=\{(x,\eta)\in T^*\C:\eta^2=x\}\sse T^*\C$, $\scurve\cong\C$, thus gives a Betti Lagrangian disk filling of the max-tb unknot. In terms of weaves, giving the front of the Legendrian lift of the $\theta$-part of $\scurve$, this particular meromorphic spectral curve gives a 2-weave with a unique trivalent vertex.

For coherence with previous notation, we can rewrite the above in terms of the coordinate $z=x^{-1}$, locally around the pole, as the coordinate $z$ is the one used in Example \ref{example:BettiSurfaces_StokesData} and previously in this subsection. Then the Airy equation transforms to $z^4f''(z)+2z^3f'(z)=f(z)z^{-1}$: at $z\neq0$, this is equivalent to $f''(z)+2z^{-1}f'(z)=f(z)z^{-5}$, which has an irregular singularity at $z=0$ because the coefficient $z^{-5}$ of $f(z)$ has a pole greater than 2 at $z=0$. Its analysis of the irregular type at $z=0$ coincides with that of the Airy equation at $z^{-1}=x=\infty$.\\

%if $m\in\N$ is odd, and $\langle z^{-{m}/2},-z^{-{m}/2}\rangle$ if $m\in\N$ is even. Integrating this, we get $\pm z^{-\frac{k}{2}}$ which recovers the $(k,2)$-braid link.

\noindent (2) {\it Irreducible isolated plane curve singularities}. The discussion in (1) readily generalizes for differential operators $P(x,\dd_x)$ on $\CP^1\setminus\{\infty\}$, $x\in\C$ whose spectral curve $\scurve=\{(x,\eta)\in T^*\C:P(x,\eta)=0\}$ has an irreducible isolated singularity at the origin $(x,\eta)=0$. In these cases, the irregular curve $(\sS,\mkpts,\Theta)=(\CP^1,\infty, \Theta_P)$ has as irregular data $\Theta_P$ a Stokes diagram given by a braid $\beta(P)$ that represents -- upon satelliting along the unknot -- the link of the singularity $P:\C^2\lr\C$. The Betti surface is thus of the form $(S,\mkpts,\lknot)=(S^2,\infty,\La_{\beta(P)})$, where the smooth type of $\La_{\beta(P)}$, as a link in $S^3=\dd(T^*\C)$, is that of the link of the singularity. The theory of algebraic links allows to readily compute such smooth type using a Puiseux series for $P$ (via Newton exponents), expressing this link as an iterated cable link with positive enough $(p_i,q_i)$-coefficients; see e.g.~\cite[Chapter I \& II]{EisenbudNeumann85_Book} or \cite[Chapters 17 \& 18]{Ghys17_Promenade}. By \cite[Prop.~2.2]{Casals22_PlaneCurveSing}, algebraic links have a unique max-tb Legendrian representative, and thus the Legendrian isotopy type of $\La_{\beta(P)}$ is uniquely determined by the irregular data $\Theta_P$.

This construction leads to a large class of Betti Lagrangians in $(\R^4,\omega_\std)$ which are Lagrangian fillings of max-tb Legendrian representatives of algebraic links.\footnote{Technically, as given, they are asymptotically Lagrangian fillings, but this can be corrected as in \cite[Section 5.1]{MR2580429}.} At the same time, their associated meromorphic connections are also important objects of study: these include the differential equations of Airy, Bessel, Clifford, Weber, Whittaker and general differential operators of Airy type (cf.~\cite{HohlJakob22_AiryOperators,Katz87_AiryOperators}), among others.\footnote{See 48min in P.~Boalch's talk ``First Steps in Global Lie Theory'' at the Simons Center during the program ``The Stokes Phenomenon and its applications in Mathematics and Physics'' as well as his {\it spirograph Stokes diagrams} applet in his website.} By construction, these spectral curves are singular. In order to obtain smooth (embedded) Lagrangian fillings for such links $\La_{\beta(P)}$ one must morsify the singularity $P$: as explained in \cite{Casals22_PlaneCurveSing}, any real morsification will give rise to an exact embedded Lagrangian filling. Furthermore, since the front singularities of holomorphic Legendrians are real codimension-2, these Lagrangian fillings can all be described by weaves \cite{legendrianweaves}. Building on (1) above, a simple instance is that, for generic $\theta$, the $\theta$-part of the meromorphic spectral curve associated to a morsification of $f''(x)=x^3f(x)$ gives Betti Lagrangians for the Betti surface $(S,\mkpts,\lknot)=(S^2,\infty,\La_{\s_1^5})$: as we vary $\theta\in S^1$, this actually recovers the 5 known Lagrangian fillings of the max-tb trefoil. In the general case, there is a moduli for the space of meromorphic connections on $\C$ with such fixed irregular data at $x=\infty$, including such morsifications, and it is isomorphic to the moduli of sheaves with singular support on these Legendrian links. See e.g.~\cite{casals2023microlocal,Su25_DualBoundary}. The $\theta$-parts of different such meromorphic spectral curves give rise to different Lagrangian fillings: both varying $\theta$ and varying the meromorphic connection typically changes the Hamiltonian isotopy type of the Betti Lagrangian, see e.g.~\cite{MR4651703} for how varying $\theta\in S^1$ for a fixed morsification of $f''(x)=x^mf(x)$ leads to different orbits of Lagrangian fillings for $(2,m)$-torus links.\\

\noindent (3) {\it BNR Equation}. Consider the Berk-Nevins-Roberts differential equation, cf.~\cite{BerkNevinsRoberts82_NewStokes}. This is the order 3 differential equation
$$\dd_x^3f(x)-3\dd_xf(x)+xf(x)=0,\quad x\in\C,$$
known for being one of the first studied examples with higher rank Stokes phenomenon. In the coordinate $z=x^{-1}$, the equation is transformed to
$$z^6\dd_z^3f(z)+6z^5\dd_z^2f(z)+(6z^4-3z^2)\dd_zf(z)-z^{-1}f(z)=0,\quad z\in\C$$
which has an irregular singularity at $z=0$, e.g.~because the coefficient of $f(z)$ has a pole of order $7$. A direct computation of the irregular data gives that the irregular curve is $$(\sS,\mkpts,\Theta)=(\CP^1,\mkpts,\langle x^{4/3}\rangle).$$
The positive 3-stranded braid word associated to this irregular data is $\beta=(\s_1\s_2)^2$ and thus the Betti surface is $(S,\mkpts,\lknot)=(S^2,\infty,\La_{(\s_1\s_2)^2})$. The Betti Lagrangian $L\sse (T^*\R^2,\la_\std)$ associated to this meromorphic spectral curve $\scurve\sse T^*\C$ consists of a unique embedded Lagrangian disk, as it is graphical in $x\in\C$. Indeed, it is a 3-weave that can be parametrized by taking the real part of the holomorphic Legendrian lift of
$$\scurve=\{(x,\eta)\in T^*\C:\eta^3-3\eta+x=0\}\sse T^*\C,$$
understood as a holomorphic Lagrangian. Such $\scurve$ is a morsification of the function $x(\eta)=\eta^3$, so it fits within the description of (2) above, for a smooth germ in this case. In real coordinates $x=u+iv$, $u,v\in\R^2$, the front for this real Legendrian lift is parametrized by
$$\sigma:\R^2\lr\R^3,\quad \sigma(u,v)=\left(u^3-3uv^2,3u^2v-v^3,3(u^3v-uv^3-uv)\right),$$
which readily exhibits the boundary braid $\beta=(\s_1\s_2)^2$ and the embedded Lagrangian disk, as there are no Reeb chords in this front. The contact topological proof that the Berk-Nevins-Roberts equation has no moduli is then simply the observation that $\La_{(\s_1\s_2)^2}$, understood as a Legendrian link in $(S^3,\xi_\std)$ after satelliting the Legendrian (unit) cotangent fiber along the max-tb Legendrian unknot, is actually Legendrian isotopic to the max-tb Legendrian unknot. Note that, in accordance to this, taking different $\theta$-parts leads to a rotation of the 3-weave and the Hamiltonian isotopy type of the Betti Lagrangian remains the same.
\hfill$\Box$
\end{ex}

%%%%%%%%%%%%%%%%%%%%%%%%%%%%%%%%%%%%%%%%%%%%%%%%%%%%%%%%%%%%%%%%%%%%%%%%
%%%%%%%%%%%%%%%%%%%%%%%%%%%%%%%%%%%%%%%%%%%%%%%%%%%%%%%%%%%%%%%%%%%%%%%%
%%%%%%%%%%%%%%%%%%%%%%%%%%%%%%%%%%%%%%%%%%%%%%%%%%%%%%%%%%%%%%%%%%%%%%%%

\subsection{Real Morse flowlines}\label{subsection:Morselines} Let $(M,g)$ be a Riemannian manifold and $L\sse T^*M$ a Betti Lagrangian. Consider $m\in M-\pi(\caustL)$ and let $\{\sheet_1,\ldots,\sheet_n\}$ define the smooth sheets of $\multigraph$ over $m$, via the graphs of $d\sheet _i$.

\begin{definition}\label{def:flowline}
A smooth map
$\gamma:(-\epsilon,\epsilon)\lr M-\caustL$ is said to be an $ij$-gradient flowline for $L$ if it satisfies
\begin{equation}\label{eq:ij-gradientflow}
	\gamma'(s)+\nabla_g(\sheet_{i}-\sheet_{j})(\gamma(s))=0,\quad \gamma(0)=m,\quad s\in(-\epsilon,\epsilon),
\end{equation}
for the two local sheets $\sheet_{i},\sheet_j$ of $\multigraph$ along $\gamma$, $i,j\in[1,n]$. The domain of definition of a flowline can be extended as long as it does not converge to a point in $\pi(\caustL)$: a flowlines obtained by extending the domain of definition are called \textit{continuations}. Maximal continuations of solutions of the gradient flowline equation (\ref{eq:ij-gradientflow}) are said to be \textit{trajectories}.\hfill$\Box$
\end{definition} 

\noindent Given an $ij$-gradient flowline $\gamma:(-\epsilon,\epsilon)\lr M-\caustL$, consider the lift
$$\bar{\gamma}^{i}:(-\epsilon,\epsilon)\lr L,\quad \bar{\gamma}^{i}(s)=(\gamma(s),d(\sheet_i\circ\gamma)(s))$$
of $\gamma$ to the $i$th sheet of $\multigraph$, and note the following energy inequality:

\begin{lemma}\label{lemma:flow-energy} Let $(M,g)$ be a Riemannian manifold, $L\sse (T^*M,d\lambda)$ a Betti Lagrangian with $\la$ the Liouville form, and $\gamma:(a,b)\lr M$ an $ij$-flowline for $L$. Then
	\[\int_{\bar{\gamma}^{i}-\bar{\gamma}^{j}} \lambda=-\int_{a}^{b}\abs{\nabla(\sheet_{i}-\sheet_{j})}^2=(\sheet_{i}-\sheet_{j})(\gamma(b))-(\sheet_{i}-\sheet_{j})(\gamma(a))\]
\end{lemma}
\begin{proof}
This follows from $(\bar{\gamma}^i-\bar{\gamma}^j)^{\ast}\lambda=d(\sheet_{i}-\sheet_{j})(\gamma'(s))=\partial_s f$ and Equation (\ref{eq:ij-gradientflow}).
\end{proof}

\noindent Lemma \ref{lemma:flow-energy} implies that the relative homology class $[\bar{\gamma}^{j}-\bar{\gamma}^{i}]\in H_1(L,\tt)$ has positive $\lambda$-length, where $\tt$ is the appropriate set of points in $L$ giving the boundaries of the lifts $\bar{\gamma}^{i}$. This relative homology class $\bar{\gamma}^{i}-\bar{\gamma}^j$ is said to be the \textit{cotangent $ij$-lift} of $\gamma$, and its $\lambda$-length above is said to be its \textit{flow-energy}, a.k.a.~its symplectic area. Note that, by Definition \ref{def:flowline}, the difference $\sheet_{i}-\sheet_{j}$ strictly \textit{decreases} along the flowline. There are two important classes of flowlines in our context:

\begin{enumerate}
    \item The flowline equation in $(S^1,g_{flat})$ for an immersed Betti Lagrangian $L=L_\La\sse T^*S^1$ given by the Lagrangian projection of a Legendrian link $\lknot:=\{(\theta,y(\theta),z(\theta)\}\subset J^1S^1$ whose front in $S^1\times\R$ has no cusps.\\

    \item The flowline equation in a disk $\R^2$ for a Betti Lagrangian $L\sse T^*\R^2$ near a $D_4^{-}$-singularity. That is, near a branch point of the projection $L\to \R^2$ that, in the spatial wavefront in $\R^2\times\R$ of the Legendrian lift of $L$, is a non-generic $D_4^-$ singularity.
\end{enumerate} 

\noindent {\bf For case (1)}, we choose a base coordinate $\theta\in S^1$ and the $ij$-flowline equation for $\gamma:(-\epsilon,\epsilon)\lr S^1$ reads
\begin{equation}\label{eq:braidflowline}
	\gamma'(s)+ \left[(y_i(\theta)-y_j(\theta))\partial_\theta\right]_{|_{\gamma(s)}} =0.
\end{equation}
Here we denoted $\gamma'(s)=\gamma_*(\dd_s)$, so writing $\gamma'(s)=\theta'(s)\dd_\theta$ expresses (\ref{eq:braidflowline}) as $\theta'(s)+ (y_i(\theta(s))-y_j(\theta(s)))=0$. Near a Reeb chord of the Legendrian link $\La$, i.e.~an immersed point of $L_\La$, the local configuration of sheets reduces to that of the two Lagrangian strands
$$L_1:=(x,-x)\in T^*(-1,1)\mbox{ and } L_2:=(x,x)\in T^*(-1,1)\quad\mbox{meeting at }x=0.$$
In these local coordinates, the difference $y_1(x)-y_2(x)=2x$ of the sheets $y_1,y_2$ gives the $12$-flowline equation $x'(s)-2x(s)=0$ and the $21$-flowline equation, which is $x'(s)+2x(s)=0$. These have solutions $e^{2s}$ and $e^{-2s}$, respectively, and it follows that the $12$-flowline equation has $x=0$ as its unstable manifold, and the $21$-flowline equation has $x=0$ as its stable manifold. Assuming that $L_1$ and $L_2$ are all oriented in the direction $\partial_x$, the Reeb chord is positive if $L_1$ is above $L_2$, and negative otherwise. This settles the local behavoir near a Reeb chord. Globally on $L_\La\sse T^*S^1$, the trajectories must begin and end at some Reeb chord: Lemma \ref{lemma:flow-energy} then implies that the difference function $z_{i}(\theta)-z_{j}(\theta)$ decreases along the trajectory. Here the \textit{sign} of the difference function \textit{cannot be} both negative at $s=-\infty$ and positive at $s=+\infty$. This describes the qualitative behavior of flowlines for case (1), which we summarize in the following:
%{\RC Either add proof, or cite somewhere, or at least cite somewhere with the techniques to have it.}
\begin{lemma}\label{lemma:minimalescapetime}
	Let $L\sse T^*S^1$ be an immersed Betti Lagrangian $L_\La\sse T^*S^1$ given by the Lagrangian projection of a Legendrian link $\lknot:=\{(\theta,y(\theta),z(\theta)\}\subset J^1S^1$ whose front in $S^1_\theta\times\R_z$ has no cusps. Consider a pair of points $c,c'\in S^1$ above each of which lies one crossing of $L_\La$. Then:
    \begin{enumerate}
        \item There exist finitely many gradient trajectories beginning at $c$ and terminating at $c'$, up to translation. 
        % $I(c,c')$
        \item Suppose each such $ij$-gradient trajectory is maximally well-defined an open interval $(a,b)=(a(c,c'),b(c,c'))\sse\R$ in the universal cover $\mathbb{R}\lr S^1$, up to $\mathbb{Z}$-translation. Then $\exists\epsilon=\epsilon(c,c')\in\R_+$ such that
        $$z_{i}-z_{j}\neq 0 \text{ if }\theta\in(a,a+\epsilon)\cup (b-\epsilon,b).$$
        
        \item $\exists T=T(\epsilon)\in\R_+$ a minimal escape time such that for any point $u\in (a+\frac{\epsilon}{2},b-\frac{\epsilon}{2})$, an $ij$-trajectory $\gamma$ with $\gamma(0)=u$ maps
        $$\gamma((-\infty,-T))=(a,a+{\epsilon}),\quad \gamma((T,\infty))=(b-\epsilon,b),$$
        if its orientation coincides with that of $S^1$. Else, the same holds exchanging the two target intervals above. Here the constants $\epsilon,T\in\R_+$ do not depend on the choice of the lift of the trajectory.\hfill$\Box$
    \end{enumerate}
\end{lemma}
%{\YJN Reeb positivity only appeared because we wanted to extend our discussions to exact Betti cobordisms/singular meromorphic curves. If you want to introduce these notions later in the paper, I have moved it to the betiLagcobordism.tex file}
\noindent {\bf For case (2)} above, studying flowlines near a $D_4^{-}$-singularity, we proceed as follows. The base surface is $S=\C$ with coordinate $z\in\C$ and the Betti Lagrangian is $\scurve=\Re\{w^2-z=0\}\sse T^*\C$. Note that this is the real part of the Lagrangian projection of the holomorphic Legendrian disk
$$\La=\{(z,w;v)\in T^*\C\times\C: z=w^2, v=2w^3/3\}\sse(J^1\C,\xi_{hol}).$$
Here $\La$ is the Legendrian lift of the holomorphic simple cusp front $\{(w^2,2w^3/3):w\in\C\}\in \C_z\times\C_v$, where $\xi_{hol}=\ker\{dv-wdz\}$. The projection $L\lr\C_z$ is a 2-fold branched cover, branched at $z=0$, and thus there are only two sheets and 12 or 21 flowlines. For such $\scurve$, there are three gradient flowlines converging to $z=0$, given by the three rays $\mathbb{R}_{\geq 0}, e^{\frac{2\pi}{3}}\mathbb{R}_{\geq 0}$ and $e^{-\frac{2\pi}{3}}\mathbb{R}_{\geq 0}$. For $|z|$ large, these are respectively (rays over) the three Reeb chords at the Legendrian boundary, which is given by $\beta=\sigma_1^3$, cyclically understood. These three rays near a $D^4_-$ singularity are said to be the \textit{initial trivalent rays}. In general, given be a Betti Lagrangian $L\sse T^*S$ and $b\in S$ a branched point of the projection $L\lr S$, $L$ decomposes into its germ near the ramification point over $b$ and the smooth sheets. Since the former germ comes from the real part of the holomorphic cusp, we refer to that connected component as the \textit{cusp} component. We also refer to the gradient flowlines between the two sheets of the cusp component as cusp-cusp flowlines, and if the gradient flowlines are between the cusp component and a smooth sheet, we call them smooth-cusp flowlines. Thus, near a ramification point, only a cusp-cusp flowline in $S$ is allowed to terminate at the associated branched point.

%%%%%%%%%%%%%%%%%%%%%%%%%%%%%%%%%%%%%%%%%%%%%%%%%%%%%%%%%%%%%%%%%%%%%%%%
%%%%%%%%%%%%%%%%%%%%%%%%%%%%%%%%%%%%%%%%%%%%%%%%%%%%%%%%%%%%%%%%%%%%%%%%
%%%%%%%%%%%%%%%%%%%%%%%%%%%%%%%%%%%%%%%%%%%%%%%%%%%%%%%%%%%%%%%%%%%%%%%%

\subsection{Asymptotics for conical ends (trapping lemma)}\label{subsection:trappinglemma} Let us now explain how conical ends, as introduced in Definition \ref{def:conicalend}, allow us to control the asymptotic behavior of flowlines as the marked points $\mkpts\sse S$. Specifically, given the germ of a Lagrangian conical end, we prove a trapping property in Lemma \ref{lemma:morsetrappinglemma} below: there exists a neighborhood of marked points such that the gradient trajectories in $S$ cannot leave once inside. Technically, we first derive the trapping lemma for the conical metric $g=d\theta^2+d\radial^2$ near the marked points, instead of the polar metric $d\radial^2+\radial^2 d\theta^2$ used for the exact Betti Lagrangians in Definition \ref{definition:topologicalexactspectralcurve}.(ii). We then argue via Lemma \ref{lem:flowlineidentification} that there is a bijective correspondence between flowlines for the conical metric and those for the polar metric.\\ 

Consider the germ of a Lagrangian conical end $L=S^1\times[1,\infty)\sse T^*D^2$ as in Definition \ref{def:conicalend}, where we use the same notation $(\theta,r)\in S^1\times[1,\infty)$ and $(x(\theta),y(\theta),z(\theta))\in J^1S^1_\theta$ for its conical type; here the origin of $D^2$ is given by $r\to\infty$. Then the gradient flowline equation (\ref{eq:ij-gradientflow}) takes the form
\begin{equation}\label{eq:LCendflowline}
	\gamma'(s)+\left[ (y_i(\theta)-y_j(\theta)) \radial\frac{d}{d\theta}+ (z_{i}(\theta)-z_{j}(\theta))\frac{d}{dr}\right]_{|_{\gamma(s)}}=0,\quad (\theta,r)\in S^1\times[1,\infty).
\end{equation}

To analyze (\ref{eq:LCendflowline}), we first observe that the projection of equation (\ref{eq:LCendflowline}) to the $S^1_\theta$ component becomes
$$\gamma'(s)+\left[ (y_i(\theta)-y_j(\theta)) \radial\frac{d}{d\theta}\right]_{|_{\gamma(s)}}=0,\quad (\theta,r)\in S^1\times[1,\infty),$$
which coincides with equation (\ref{eq:braidflowline}), up to the boosting $r$-coefficient: for a fixed $r\in\R$, its solutions behave as the gradient flowline equation for the Legendrian link $\Lambda=(s,y(s),z(s))\sse J^1S^1$ studied above. At an angle $\theta=\theta(s)$ with a Reeb chord, i.e.~ $y_i(\theta)=y_j(\theta)$, the corresponding solution of \eqref{eq:LCendflowline}  becomes stationary in this $S_\theta^1$-component, and thus gives us a Reeb chord ray, cf.~\cref{subsection:Betti_surface}. This determines the behavior in terms of the $y_i$-variables.\\

In particular, using \cref{eq:braidflowline} we can characterize the Reeb-positive cusp-free Legendrian links as follows.

\begin{lemma}\label{lemma:Reebpositivity}
Let $\lknot\sse (T^\infty S,\xi_\std)$ be a Legendrian braid link. Then $\lknot$ is Reeb-positive if and only if given any $ij$-trajectory, the difference function $z_i-z_j$ is positive near $s=-\infty$, and negative near $s=\infty$.
\end{lemma}
\begin{proof} By definition, the Lagrangian projection of $\lknot$ has a positive crossing at an $ij$-Reeb chord if and only if  the height difference function $z_i-z_j$ has Morse index $1$. For $(\Rightarrow)$, suppose $\lknot$ is Reeb positive. If an $ij$-trajectory begins at $c$, since the neighborhood around $c$ must be the unstable manifold of $z_i-z_j$, we must have $z_i>z_j$. Similarly, if an $ij$-trajectory terminates at a Reeb chord $c$, as $\lknot$ is Reeb positive, the stable manifold of $c$ must be a point, and so we have $z_i<z_j$ at $c$. %Therefore, it terminates at a $ji$-Reeb chord, and by positivity, the height difference $z_j-z_i$ is necessarily positive near $s=+\infty$.
For $(\Leftarrow)$, by contradiction: assume that there exists a negative crossing at $c$ with $z_i>z_j$. In that case, the stable manifold of $z_j-z_i$ near $c$ is necessarily the point $c$. Therefore, there exists an $ji$-trajectory terminating at $c$ and with $z_j>z_i$.
\end{proof}

\noindent To understand the effect of the difference factor $(z^{i}-z^{j})$ in equation \eqref{eq:LCendflowline}, let us write $\gamma(s)=(\theta(s),\radial(s))\in S^1\times[1,\infty)$ and observe that the projection onto $S^1_\theta$ can be maximally extended to a solution of the $ij$-gradient flow-equation on $S^1_{\theta}$, after orientation-preserving reparametrization. Denote by $P_\gamma(\theta)$ this projected $ij$-flowline on $S^1_\theta$, traveling from a Reeb chord $c'$ to a Reeb chord $c$, both Reeb chords of $\La$. The sign of the derivative $\radial'(s)$ is equal to the sign of the difference $-(z^{i}-z^{j})(\gamma(s))$. In particular, since the difference function $(z^{i}-z^{j})$ decreases along the projected trajectory $P_{\gamma(s)}$, $(z^{i}-z^{j})$ decreases along the trajectory $\gamma(s)$ as well. Figure \ref{fig:flowasymptotics} provides a schematic picture on the behavior of the flowlines $\gamma(s)$.
\begin{center}
	\begin{figure}[h!]
		\centering
		\includegraphics[scale=1.4]{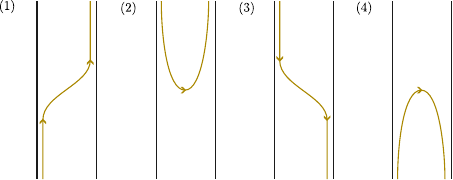}
		\caption{The asymptotics of an $ij$-flowline $\gamma(s)=(r(s),\theta(s))$ based on the sign of $(z^{i}-z^{j})$ at the Reeb chords. The case depicted in (1) occurs if $(z^{i}-z^{j})$ goes from negative to negative. Case (2) occurs if $(z^{i}-z^{j})$ goes from positive to negative. Case (3) if $(z^{i}-z^{j})$ goes from positive to positive. Case (4) never occurs, which would correspond to $(z^{i}-z^{j})$ going from negative to positive. The fact that the fourth case is excluded is the key ingredient behind the proof of \cref{lemma:trappinglemma} (the Trapping Lemma).} 
		\label{fig:flowasymptotics}
	\end{figure}
\end{center}
%{\YJN Diagram needed: four cases (i) if $(z^{i}-z^{j})(c')<0$ and $(z^{i}-z^{j})(c)<0$, then $\gamma$ travels upwards, (ii) if $(z^{i}-z^{j})(c')>0$ and $(z^{i}-z^{j})(c)<0$ then $\gamma$ travels down from $s=-\infty$, then travels back up again, (iii) if $(z^{i}-z^{j})(c')>0$ and $(z^{i}-z^{j})(c)<0$, then $z$ only travels downward, the case (iv) $(z^{i}-z^{j})(c')<0$ and $(z^{i}-z^{j})(c)>0$ will never happen. }

This leads to the following useful observation:
\begin{lemma}\label{lem:first_noescape}
Let $\gamma(s)=(\radial(s),\theta(s)),s\in [0,\infty)$ be an $ij$-flowline on $S^1\times [1,\infty)$ solving equation \eqref{eq:LCendflowline} and suppose that $\radial'(0)>0$. Then $\radial(s)$ strictly increases and it asymptotes to a Reeb chord ray. 
\end{lemma}
\begin{proof}
Since $\gamma(s)$ solves $\eqref{eq:LCendflowline}$, the component $\theta(s)$ solves the $\radial(s)$-boosted gradient flowline equation for the Legendrian link $\lknot$. Therefore, the difference $(z_i-z_j)$ strictly decreases along $\gamma(s)$. By equation $\eqref{eq:LCendflowline}$, we have $r'(0)>0$ if and only if $(z_i-z_j)(0)<0$. Therefore, the hypothesis $\radial'(0)>0$ implies that $\radial'(s)>0$ for all $s\in\R_+$, since $(z_i-z_j)(s)$ is negative at $s=0$ and strictly decreases as a function of $s\in\R_+$. 
\end{proof}
%{\RC EDITED TILL HERE}
\noindent Visually, Lemma \ref{lem:first_noescape} implies that, near a small enough neighborhood of a marked point, once a flowline start traveling up towards the marked point then it will not reverse course. The following trapping lemma builds on \cref{lem:first_noescape} to provide the more complete description that we need:

\begin{lemma}[Trapping Lemma]\label{lemma:morsetrappinglemma}
	Let $\Laggerm\sse T^*D^2$ be the germ of a Lagrangian conical end over a generic cusp-free Legendrian link $\lknot\sse(J^1S^1,\xi_\std)$, where the zero-section of the 1-jet space is identified with $T_0^\infty S^1$. Then there exist positive constants $T,R\in\R_+$ such that:
	\begin{enumerate}

		\item For the annular neighborhood $S^1\times (R, 4R)\sse D^2$ of the origin, any trajectory $\gamma(s)$ passing through $S^1\times [2R,3R]$ at $s=0$ must stay inside $S^1\times [R, 4R]$ after time $\pm T$.\\
		\item The domain of definition of either $\gamma(s)$ or $\gamma(-s)$ can be extended to $[0,\infty)$.\\
		\item In the former case, $\radial(s)$ strictly increases on $[T,\infty)$ and asymptotes to a Reeb chord ray. In the latter case, $\radial(s)$ strictly decreases on $[-\infty,T)$.
                %the limiting Reeb chord at $s=+\infty$ is necessarily positive, and 
        %the limiting Reeb chord is necessarily positive, and
	\end{enumerate}
    That is, any trajectory $\gamma$ passing through $S^1\times [2R,3R]$ must stay inside $S^1\times [R,\infty)$ in at least one direction.
\end{lemma}
\begin{proof}
Consider the constants $\epsilon(c,c')\in\R_+$ for all the pairs $c,c'\in S^1$ associated to Reeb chords of the end Legendrian link $\lknot$, as in Lemma \ref{lemma:minimalescapetime}. Let $\epsilon\in\R_+$ be their minimum and we choose $T$ to be the corresponding minimal escape time. Choose some radii $R_0,R_1\in\R_+$, $R_1>R_0\gg 1$ such that for any gradient trajectory passing through $S^1\times [R_1,\infty)$, the time $\pm T$ image of the trajectory remains inside $S^1\times [R_0,\infty)$. Such radii exist as there is a finite uniform upper bound on $\abs{z_{i}-z_{j}}$ and trajectories satisfy equation \eqref{eq:LCendflowline}. Now given an $ij$-trajectory $\gamma$ passing through $S^1\times [R_1,\infty)$ at time $s=0$, we consider $P_\gamma(\theta)$ the trajectory on $S^1_\theta$ obtained by reparametrizing the $S^1_\theta$-component of ${\gamma}$ and consider $c\in S^1$ be the (angle associated to the) limiting Reeb chord at $s=-\infty$, while $c'$ is the limiting Reeb chord at $s=+\infty$. This establishes (1), as the required $R$ can be chosen from $R_1$, and also shows (2).

For (3), since $z_{i}-z_{j}$ has to decrease along $\gamma$, it follows that both $(z_{i}-z_{j})(c)<0$ and $(z_{i}-z_{j})(c')>0$ cannot happen. By the choice of $R_0$, the gradient flow $\gamma$ still remains in $S^1\times [R_0,\infty)$ after time $\pm T$. If $(z^{i}-z^{j})(\theta(T))<0$, \cref{lem:first_noescape} implies that $\radial(s)$ strictly increases, i.e.~that $\gamma$ never goes back. If else $(z^{i}-z^{j})(\theta(T))>0$, it must be that $(z^{i}-z^{j})>0$ along $P_\gamma:(-\infty,T]\to S_\theta^1$ and then the backward $ji$-gradient flow will have $\radial'(\gamma(-s))>0$. To analyze the cases in the other intervals, we proceed similarly: e.g.~ if $P_\gamma(0)\in (a+\epsilon,b-\epsilon)$ then \cref{lemma:minimalescapetime} implies that $(z^{i}-z^{j})(\theta(T))\neq 0$ and the same argument above applies. If $P_\gamma(0)\in (a,a+\epsilon]$, then there are two cases. Either $(z_i-z_j)(c')<0$ and then $(z_i-z_j)(c')$ must stay negative, so $\gamma(s)$ will travel up as $s\to +\infty$. Or $(z_i-z_j)(c')>0$, in which case the backward flow satisfies $(z_j-z_i)(c)<0$ along it and so the backward flow will always travel up. The cases with $P_\gamma(0)\in(b-\epsilon,b]$ are deduced in the same way and (3) follows.
\end{proof}

The characterization of asymptotics of trajectories for exact Betti Lagrangians is thus:
\begin{proposition}\label{prop:asymptoticoftrajectories}
Let $L\sse (T^*S,\la_\std)$ be an exact Betti Lagrangian and $\gamma:(a,b)\to \obis$ a gradient trajectory for $L$, $(a,b)\sse\R$ the maximal domain of definition of $\gamma$. Then:
	\begin{enumerate}
        \item The trajectory $\gamma$ cannot be periodic.
		\item If $a$ is finite, then $\gamma(a)\in S$ must be at a $D_4^{-}$-branch point of the projection $L\lr S$. If $b$ is finite, the same applies to $\gamma(b)\in S$.
		\item If $a=-\infty$, then $\gamma(s)$ asymptotes to a Reeb chord ray as $s\to-\infty$. Similarly, if $b=\infty$, then $\gamma(s)$ asymptotes to a Reeb chord ray as $s\to\infty$.
	\end{enumerate}
\end{proposition}

\begin{proof}
For (1), \cref{lemma:flow-energy} implies that a closed trajectory would yield a smooth closed curve $\tau\sse L$ such that integral of the Liouville form along $\tau$ does not vanish. Since $L$ is exact, this is a contradiction and thus there are no closed trajectories. For (2), showing that $\gamma(a)$ is a $D_4^{-}$-branch point if $a$ is finite, we fix a large radius $R\in\R_+$ as in \cref{lemma:morsetrappinglemma} and consider the complement $K\sse S$ of the union of all the cylinders ${S^1\times (R,\infty)}$ near the marked points of $S$. We claim that a trajectory must either terminate at a finite end or leave the compact set $K$.

Indeed, suppose that image of the trajectory $\gamma$ stays contained in $K$ as $s\to a$. If $a$ is finite, the flowline equation \eqref{eq:ij-gradientflow} and since $L$ is embedded, it follows that $\lim \gamma(a)$ must be branch point of $D_4^{-}$-singularity. The flowtime needed to approach the $D_4^{-}$-singularity is finite, and thus we have a finite end, as required. If else $a=-\infty$, we need to argue that the trajectory escapes $K$, and \cref{lemma:morsetrappinglemma} then determines its asymptotics. By the flowline energy formula of \cref{lemma:flow-energy}, it must be that $\gamma(s)$ lies in a small neighborhood of a $D_4^{-}$-branched point for $s\in\R_+$ large enough. Indeed, by energy quantization: the difference of the sheets have norm uniformly bounded below outside such neighborhood, and there is a minimal flow-time needed to travel from one boundary to another, but the flow-energy must be finite and so eventually the flowline cannot leave the neighborhood of the branch points. This implies that $\gamma(s)$ is a cusp-smooth flowline that either ends at the branch point or coincides with one of the trivalent rays coming out of the $D_4^{-}$-singularity. The flowtimes for either of these events are finite, and thus this is a contradiction with $a=-\infty$ and staying within $K$. Thus the trajectory must leave the compact set $K$ if $a=-\infty$, from which (3) follows after applying \cref{lemma:morsetrappinglemma}.
\end{proof}

Finally, let us compare the behavior of trajectories for the polar metric $g=r^2d\theta^2+d\radial^2$ and the conical metric $g=d\theta^2+d\radial^2$ near the marked points $\mkpts\sse S$. By the flowline equation \eqref{eq:LCendflowline}, trajectories are directed by the vector field of the form $\nabla f \dd_\theta+(f/r)\dd_r$ where $f=z^i-z^j$. We then have:

\begin{lemma}\label{lem:flowlineidentification}
The diffeomorphism $\phi: S^1\times [0,\infty)\to S^1\times [1,\infty)$ given by $(\theta,\tau)=\phi(\theta,\radial)=(\theta,e^{\radial^2/{2}})$ identifies the gradient flowlines for the metric $d\radial^2+d\theta^2$ and the gradient flowlines for the metric $\tau^2d\theta^2+d\tau^2$.
\end{lemma}
\begin{proof}
The flowlines for the metric $\tau^2 d\theta^2+d\tau$ are directed by $(\nabla f,\tau f)=\nabla f\dd_\theta+\tau f \dd_\tau$. Since $\phi_*(\dd_\theta)=\dd_\theta$ and $\phi_*(\dd_r)=re^{r^2/2}\dd_\tau$, it follows that $\phi_{\ast}(\nabla f, f/r)=(\nabla f, e^{\radial^2/{2}} f)=(\nabla f, \tau f)$, as claimed. 
\end{proof}

\noindent For context, the conical metric $g=d\theta^2+d\radial^2$ is better suited when cylindrizing the Betti Lagrangian $\scurve$ to a cylindrical Lagrangian filling, cf.~\cref{definition:topologicalexactspectralcurve}. This is, arguably, more natural from a geometric viewpoint and it fits better with the study of augmentations of the Legendrian contact dg-algebra, but the analysis is a bit more elaborate. If $L\sse T^{\ast}\obis$ is an exact Betti Lagrangian, with conical ends, then the polar metric is better suited as it makes $L$ uniformly bounded nears its asymptotic ends. Lemma \ref{lem:flowlineidentification} implies that we can switch between these two metrics when it comes to describing the qualitative behavior of of asymptotics of flowlines.

%%%%%%%%%%%%%%%%%%%%%%%%%%%%%%%%%%%%%%%%%%%%%%%%%%%%%%%%%%%%%%%%%%%%%%%%
%%%%%%%%%%%%%%%%%%%%%%%%%%%%%%%%%%%%%%%%%%%%%%%%%%%%%%%%%%%%%%%%%%%%%%%%
%%%%%%%%%%%%%%%%%%%%%%%%%%%%%%%%%%%%%%%%%%%%%%%%%%%%%%%%%%%%%%%%%%%%%%%%

\subsection{Holomorphic Morse flowlines for meromorphic spectral curves}\label{subsection:holMorse}
The analysis in \cref{subsection:trappinglemma} applies to conical ends, and thus to the exact Betti Lagrangians in \cref{subsubsection:bLagandebLag}. For the meromorphic spectral curves $\scurve\sse(\T^*\sS,\omega_\C)$ in \cref{subsubsection:iregcurvesandmeroscurves}, with $O(-1)$-ends, a modification of the arguments in \cref{subsection:trappinglemma} is needed to study the Betti Lagrangians $e^{i\theta}\scurve\sse (T^*S,\la_\std)$ given by their $\theta$-parts. Consider the following holomorphic version of the flowline equation:

\begin{definition}
Let $(\sS,g)$ be Riemann surface endowed with a K\"ahler metric, $\theta\in S^1$, and write $\lambda_1,\ldots,\lambda_n$ for the holomorphic 1-forms on $\sS$ whose graphs give the multigraph $\scurve\sse (T^*\sS,\omega_\C)$. By definition, a smooth map $\gamma:(-\epsilon,\epsilon)\lr \sS$ is said to be a holomorphic $ij$-gradient flowline for $\scurve$ if it satisfies
\begin{equation}\label{eq:complexgradientflow}
	\gamma'(s)+\overline{e^{-i\theta}g^{-1}(\lambda_i-\lambda_j)}=0,\quad i,j\in [1,n],
\end{equation}
which is said to be the $ij$-\textit{holomorphic} gradient flow equation at phase $\theta$.\hfill$\Box$
\end{definition}

The holomorphic flowline equation \cref{eq:complexgradientflow} is the real gradient flowline equation for the real parts of holomorphic functions $e^{i\theta}q_1,\ldots,e^{i\theta}q_n$, for some choices of local holomorphic primitives $q_i$ of the holomorphic 1-forms $\lambda_{i}$, $i\in[1,n]$. In relation to the flow-energy, as in \cref{lemma:flow-energy}, note that $e^{i\theta}$ times the flow-energy equals the integral of the holomorphic Liouville 1-form $\la_\C\in\Omega^{1,0}(\sS)$ over the relative cycle $[\bar{\gamma}^{j}-\bar{\gamma}^{i}]$ in $\scurve\sse T^*\sS$.\\

In contrast to trajectories solving the gradient flowline equation in \cref{def:flowline}, solutions of the $ij$-holomorphic flowline equations only depend on the conformal class of the metric $g$, not $g$ itself. Indeed, given $\gamma:(-\epsilon,\epsilon)\lr \sS$, consider the following alternative differential equation for $\gamma$:
\begin{equation}\label{eq:1formgradientflow}
	\Im(e^{-i\theta}(\lambda_{i}-\lambda_{j})(\gamma'(s)))=0,\qquad \Re(e^{-i\theta}(\lambda_i-\lambda_j)(\gamma'(s)))\leq 0
\end{equation}
\cref{eq:1formgradientflow} is said to be the $ij$-phase $\theta$ WKB equation at phase $\theta$. Solutions to its first part, the equality, give rise to a line field on $\sS$ whose integral curves are known as WKB lines in the literature. The second condition, the inequality, specifies an orientation for its parametrization, so the line field acquires a direction.
%{\RC The first one gives you a line field, with the second one being a choice of parametrization ``direction''}\\
%called the $ij$-phase $\theta$ \textit{WKB line equation}. Observe that
\begin{remark}
By replacing $\scurve$ with $e^{i\theta}\scurve$, the WKB lines at phase $\theta$ become WKB phase at phase $0$; phase $\theta=0$ lines are known as horizontal WKB lines. Whenever we replace $\iregc$ with $e^{i\theta}\iregc$, the scalar factor $e^{i\theta}$ will rotate the anti-Stokes rays, for small enough $\theta$. We implicitly assume onwards that whenever there is a rotation of the phase $\theta$, we also rotate the anti-Stokes rays of the irregular class.\hfill$\Box$
\end{remark}

\noindent Since the solutions of \eqref{eq:complexgradientflow} solve \eqref{eq:1formgradientflow}, it follows that the image of the trajectories for the holomorphic gradient flow equation do not depend on the choice of the K\"{a}hler metric but only on the conformal structure, i.e.~ conformally changing the K\"ahler metric only changes the reparametrization of the solutions, not the integral curves themselves.

\begin{exmp}\label{ex:d4singularity} (1) Solutions of \cref{eq:1formgradientflow} admit rather explicit forms in the local conformal coordinates $q_{ij}=q_i-q_j$ in a sectorial neighborhood of $\gamma$: in the $q_{ij}$-coordinate, $\lambda_i-\lambda_j=dq_{ij}$, and the equation transforms to $\Im(e^{-i\theta}\gamma'(s))=0$, which is the straight line equation in the $q_{ij}$-plane.\\

(2) ($D_4^-$-singularity) Consider a neighborhood of a ramification point $\scurve\lr\sS$, which we model as $\scurve=\{(z,w)\in T^*\C:w^2-z=0\}$ and locally there are only two indices $i,j\in[1,2]$, with $\la_1=\sqrt{z}dz$ and $\la_2=-\sqrt{z}dz$. Near the branch point $z=0$ in $\C$ at phase $\theta=0$, the line field solving \cref{eq:1formgradientflow} must be in $\ker \Im(\sqrt{z}dz)$. This line field coincides with that given by the real gradient $ \nabla(\Re{z^{3/2}})$, and thus the WKB line equation of phase $0$ defines the same singular foliation as in the case of real gradient flowlines, which we studied in \cref{subsection:Morselines}. The leaves of the foliation are spanned by the maximal geodesics of $\abs{z}{dz}^2$, and the critical locus of the foliation is given as before.\\

\noindent Varying the phase $\theta$ does not qualitatively affect this picture: the WKB singular line field is given by $\ker I(e^{i\theta}\sqrt{z}dz)$, which coincides with the line field directed by $\nabla(e^{-i\theta}\Re{z^{3/2}})$. Therefore, changing the phase $\theta$ rotates the initial trivalent rays in the anti-clockwise direction until it goes back to itself at $\theta=\pi$. By convention, we order the indices along the rays so that the flowlines are directed \textit{outward}. Away from the branch point, as in (1), the analysis above can just be done by considering the locally defined coordinate $W=\int \sqrt{z}dz$, so that the 1-forms $\la_i=\pm\sqrt{z}{dz}$ reduce to $\pm dW$ and the WKB lines are geodesics in the $W$-plane.\hfill$\Box$
\end{exmp}

%%%%%%%%%%%%%%%%%%%%%%%%%%%%%%%%%%%%%%%%%%%%%%%%%%%%%%%%%%%%%%%%%%%%%%%%
%%%%%%%%%%%%%%%%%%%%%%%%%%%%%%%%%%%%%%%%%%%%%%%%%%%%%%%%%%%%%%%%%%%%%%%%
%%%%%%%%%%%%%%%%%%%%%%%%%%%%%%%%%%%%%%%%%%%%%%%%%%%%%%%%%%%%%%%%%%%%%%%%

\subsection{Asymptotics for WKB trajectories}\label{subsection:asymptoteWKB}
Let us study the asymptotics of the WKB trajectories, whose behavior is need to understand spectral networks for Betti Lagrangians associated to meromorphic spectral curves. In the case of conical ends (\cref{subsection:trappinglemma}), we had the Legendrian link at infinity parametrized by $(x(\theta),y(\theta),z(\theta))$. In the case of the $\theta$-part of a meromorphic spectral curve $\scurve\sse T^*\sS$, the Legendrian link near the $O(-1)$-ends is given by the Stokes diagram. We now need to compare the holomorphic flowline equation near the poles to the Morse flowline equation associated to the cone over the Legendrian given by the Stokes diagram.\\

Consider a small neighborhood of a marked point $\mkpts\sse\sS$ parametrized by $z\in U\sse\C$, $U$ an open disk near the origin, with the marked point given by $z=0$. Let the associated irregular data $\Theta$  (cf.~\cref{example:BettiSurfaces_StokesData}) at that marked point be given by% {\RC Why a $log(z)$ here?}{\YJN You need some log terms if you want to allow multiplicity}
\[\irreg=\sum_{k_j\in \mathbb{Q}_{>0}} \frac{c_{k_j}}{k_j}z^{-k_j}\in \mathbb{C}((z^{1/\infty}))/z^{-1}\mathbb{C}[[z^{1/\infty}]],\quad c_{k_j}\in\mathfrak{t}_{reg}\sse\mathfrak{g}\mathfrak{l}_n.\] In general, given a Puiseux series $\irreg$ with finite polar part and convergent holomorphic part, its \emph{meromorphic multigraph} $\Gamma_{d\irreg}\sse T^*\C$ is the union of the Lagrangian graphs of ${d\irregsheet_1},\ldots,{d\irregsheet_n}$, where $\irregsheet_1,\ldots,\irregsheet_n$ are branches of $\irreg$. Near a pole $z=0$ for $Q$ as above, there are sectors $\theta\in(\theta_0,\theta_1)$ for the phase of $z=re^{i\theta}$, such that $\Gamma_{\irreg}$ equals:
\begin{align}\label{eq:normalLCform}
    \Gamma_{d\irreg}=\big(\theta,-\sum_{k_j\in \mathbb{Q}}  {\radial}^{-k_j}\Im(c_{k_j}e^{i\theta k_j}),\radial, -\sum_{k_j\in \mathbb{Q}} {\radial}^{-k_j-1}\Re(c_{k_j}e^{i\theta k_j})\big)\sse T^*\C,\quad (r,\theta)\in\R\times (\theta_0,\theta_1).
\end{align}
\noindent Note that \eqref{eq:normalLCform} is a generalization of the conical ends condition \eqref{eq:LCend}, arising naturally in the setting of asymptotics given by irregular classes.\footnote{Note that in order to represent irregular data with multiplicity, we would need to introduce additional terms of the form $d \log(z)$, for $t\in \mathfrak{t}_{reg}$, and some holomorphic terms.}

\begin{exmp}[$D^4_-$-multigraph]\label{ex:D4polarform} This example continues \cref{ex:d4singularity} and aligns with \cref{ex:irregularclass_Legendrian} for the irregular class $\langle x^{3/2}\rangle$, which is the Stokes data for the $D^4_-$-singularity. For the Puiseux series $Q$ associated to $\Sigma=\{w^2-z=0\}\sse T^*\C$, the meromorphic multi-graph is parametrized as
	\begin{align}\label{eq:D_4form}
			(r,\theta)\lr\left(\theta,\pm \frac{3}{2}f(\radial)\cos(\frac{3\theta}{2}),\radial,\pm f'(\radial)\sin(\frac{3\theta}{2})\right)\in T^*\R^2_{r,\theta},\quad \mbox{where }f(\radial)=\frac{2}{3}\radial^{3/2},\quad (r,\theta)\in\R\times S^1.
		\end{align}
In fact, the Hamiltonian isotopy class of this Lagrangian is independent of the specific choice of function $f(r)$, as long as it is positive and strictly increasing. Deforming $f(r)$ to be linear at infinity recovers the conical end form.\hfill$\Box$
\end{exmp}
Let us analyze the gradient equations near the $O(-1)$ ends, i.e.~a pole at $z=0$. Consider the inversion $\abs{z}^{-1}$ from $\mathbb{C}^{\ast}\simeq S^1\times (0,\infty)$ to send $+\infty$ to $0$ and pushforward the conical metric $\abs{dz}^2$ to $\abs{z}^{-2}\abs{dz}^2$; the latter metric is $d\theta^2+\radial^{-2} d\radial^2$ in polar coordinates. For the difference of the sheets to have growth at least $z^{-1}$, we need the following condition:
\begin{align}\label{eq:logdifference}
(d\irregsheet_i-d\irregsheet_j)-\frac{C}{z}=\frac{c_{ij}}{z^{k_{ij}}}+ o(z^{-k_{ij}})+\mbox{(holomorphic terms)},
\end{align}
for some constants $C,c_{ij}$ where $C\neq 0$ only when $c_{ij}=0$. In the case $C=0$, we require that $c_{ij}\neq 0$ and $k_{ij}>0$. By using the parametrization \cref{eq:normalLCform}, the $ij$-gradient flow equations for the $O(-1)$-ends given by the irregular data $Q$ read as follows:
\begin{align}\label{eq:flowline_irregclass}
&\frac{dr}{ds}+\radial \Re{\left(\frac{C}{z}\right)}-\Re(c_{ij}e^{-ik_{ij}\theta})\radial^{-k_{ij}+1}+o(\radial^{-k_{ij}+1})+o({\radial})=0\\
&\frac{d\theta}{ds}-\Im(C/z)-\Im(c_{ij}e^{-ik_{ij}\theta})r^{-k_{ij}}+o(r^{-k_{ij}})=0.
\end{align}
For instance, in the case $d\irreg(z)=z^{-1}dz$, we obtain straight rays going into the origin for $\Im(c)=0$, periodic circles for $\Re(c)=0$, and logarithmic spirals otherwise. In general, for $c_{ij}\neq 0$, the dominating contribution as $r\to 0$ comes from $c_{ij}e^{-ik_{ij}\theta}$ only and $\theta$ varies slowly while $r$ increases fast as $r$ is sufficiently small and the angle $\theta$ is close to the anti-Stokes rays. These are the qualitative features that we have established for the conical ends. Thus, for $O(-1)$-ends, we can effectively repeat the previous arguments in \cref{lemma:morsetrappinglemma} on minimal escape time and conclude that the $ij$-flowlines satisfy the trapping property (cf.~\cref{lemma:trappinglemma}), and trajectories also asymptote to the corresponding anti-Stokes rays, which coincide with the angles of Reeb chords.

Furthermore, the Stokes Legendrians without multiplicities further enjoy the following Reeb-positivity property.

\begin{proposition}[Stokes Legendrians are Reeb-positive]\label{prop:reebpositivestokes}
Let $\iregc$ be an irregular class without multiplicity, and $\lknot_{\iregc}$ be the corresponding Stokes Legendrian. Then for $\radial$ small enough, $\lknot_{\iregc}$ is Reeb-positive.  
    \end{proposition}
\begin{proof}
Let $\iregc$ be an irregular class and let $\irregsheet_i$ denote the branches of $\iregc$. The difference function is given by the real part of $\irregsheet_i-\irregsheet_j=\frac{1}{k_{ij}}c_{ij}z^{-k_{ij}}+o(z^{-k_{ij}})$. For $\radial$ small enough, the dominating term comes from $\frac{c_{ij}}{k_{ij}}z^{-k_{ij}}$ and so the local gradient flow equation \cref{eq:LCendflowline} reads
\begin{align}
\frac{d\theta}{ds}-r^{-k_{ij}}\Im(c_{ij}e^{-i\theta k_{ij}})+o(r^{-k_{ij}+1})=0.
\end{align}
%Assume that $\theta(s)$ has the same orientation as the real blow-up boundary. 

Near $s=-\infty$, $\partial^2_{\theta}\Re{(\irregsheet_i-\irregsheet_j)}$ over the limiting anti-Stokes direction has to be \textit{negative}, because the limiting anti-Stokes direction has to be the \textit{unstable} manifold of the difference function.  However, the second derivative is computed to be equal to $k_{ij} r^{-k_{ij}}\Re{(c_{ij}e^{-i\theta}k_{ij})}$. So we see that the second derivative is negative if and only if the difference function is \textit{positive}. Near $s=+\infty$, we can argue the same by replacing $\theta(s)$ with $\theta(-s)$. So by \cref{lemma:Reebpositivity}, $\lknot_{\iregc}$ is Reeb positive. 
\end{proof}

In this meromorphic case, there is also the following alternative approach to prove the necessary trapping lemma, using the cameral cover construction from \cite{ionita2021spectralnetworksnonabelianization,kuwagaki2024genericexistencewkbspectral}. Given  $\scurve\to \obic$ be a meromorphic spectral curve, its \textit{cameral cover} is defined to be the Riemann surface
\begin{align}\label{def:cameralcover}
\scurve^{cam}:=\{(\holsheet_{\sigma(1)}(z)dz,...,\holsheet_{\sigma(n)}(z)dz):z\in \bic,\sigma\in S_n\}.
\end{align}
For $1\leq i<j\leq n$, its \textit{intermediate cover} $\scurve^{ij}$ is the quotient of $\scurve^{cam}$ by the relation
\begin{align}\label{def:intermediatecover}
(\holsheet_{\sigma(1)},\ldots,\holsheet_{\sigma(i)},\ldots,\holsheet_{\sigma(j)},\ldots,\holsheet_{\sigma(n)})\sim (\holsheet_{\sigma(1)},\ldots,\holsheet_{\sigma(j)},\ldots,\holsheet_{\sigma(i)},\ldots,\holsheet_{\sigma(n)})
\end{align}
over all permutations $\sigma$ satisfying $\sigma(i)<\sigma(j)$. Note that the quotient map $\scurve^{cam}\to \scurve^{ij}$ can be realized as the spectral cover of the quadratic differential defined by $(\holsheet_{\sigma(i)}-\holsheet_{\sigma(j)})^{\otimes 2}$, and that $ij$-trajectories of phase $\theta$ lift to trajectories of phase $\theta$ for the rank $2$ spectral curve $\scurve^{cam}\to \scurve^{ij}$. Employing a cameral cover, we readily obtain the following:

\begin{lemma}\label{lemma:trappinglemma}
Let $\scurve\sse T^*\sS$ be a meromorphic spectral curve with $O(-1)$ ends. Then, for generic $\theta$, there exists a neighborhood of the marked points $\mkpts\sse\sS$ such that any WKB $\theta$-trajectory entering this neighborhood cannot leave, and such trajectory asymptotes to an anti-Stokes ray.
\end{lemma}
\begin{proof}
From the $O(-1)$ ends condition, it follows that the cover $\scurve^{cam}\to \scurve^{ij}$ has poles of order greater equal than $2$. For quadratic differentials, the claimed statement follows from the literature, cf.~\cite{MR523212,Strebel84_QuadDiffBook}. Alternatively, the arguments in the proof of \cref{lemma:morsetrappinglemma} work for poles of order $\ell\geq 3$ if one uses the Legendrian link defined by the function $\pm \cos((\ell-2)\pi \theta/2)$. The case of poles of order $2$ follows from the standard classification of trajectories near it, as discussed after \cref{eq:flowline_irregclass} or see \cite[Chapter 7]{Strebel84_QuadDiffBook}.
\end{proof}

\noindent In line with \cref{prop:asymptoticoftrajectories}, the asymptotics of WKB trajectories are summarized by:
\begin{proposition}\label{prop:asymptoticofWKBtrajectories}
Let $L\sse (T^*S,\la_\std)$ be the $\theta$-part of a meromorphic spectral curve $e^{i\theta}\scurve\sse T^*\sS$ with $O(-1)$ ends at the marked point $\mkpts\sse S$, and $\gamma:(a,b)\lr \obis$ a WKB trajectory for $e^{i\theta}\scurve$. Then, for generic $\theta$, we have that:
	\begin{enumerate}
 \item The trajectory $\gamma$ cannot be periodic.
		\item If $a$ is finite, then $\gamma(a)\in S$ must be at a $D_4^{-}$-branch point of the projection $L\lr S$.\\
        If $b$ is finite, the same applies to $\gamma(b)\in S$.
		\item If $a=-\infty$, then $\gamma(s)$ asymptotes to an anti-Stokes ray of $e^{i\theta}\scurve$ as $s\to-\infty$. (Similarly for $b=\infty$.)
	\end{enumerate}
\end{proposition}
\begin{proof}
Observe that any trajectory will lift to trajectory of some trajectory $\scurve^{cam}\to \scurve^{ij}$. For (2) and (3), the argument is the same as for \cref{prop:asymptoticoftrajectories}, applied to the cameral cover lift, given that there are no recurrent trajectories (\cite[Chapter 10.2]{Strebel84_QuadDiffBook}). To show that there are no periodic trajectories, note that such a periodic trajectories would lift to periodic trajectories of some $\scurve^{cam}\to \scurve^{ij}$. That said, for generic $\theta$, complete quadratic differentials that are saddle-free do not admit periodic or recurrent trajectories. This completes the statement.
\end{proof}
 %    \newpage
%%%%%%%%%%%%%%%%%%%%%%%%%%%%%%%%%%%%%%%%%%%%%%%%%%%%%%%%%%%%%%%%%%%%%%%%
%%%%%%%%%%%%%%%%%%%%%%%%%%%%%%%%%%%%%%%%%%%%%%%%%%%%%%%%%%%%%%%%%%%%%%%%
%%%%%%%%%%%%%%%%%%%%%%%%%%%%%%%%%%%%%%%%%%%%%%%%%%%%%%%%%%%%%%%%%%%%%%%%
\section{Morse spectral networks}\label{section_mnetwork}
This section introduces and develops first results on Morse spectral networks. The ingredients needed on flowtrees are discussed in \cref{subsec:Morseflowtreeandsolitons} and the definition of Morse spectral networks is presented in \cref{subsubsection:mnetwork}. The proof of \cref{thm:existence} is then established in \cref{subsection:constructionofexactnetwork}, for the exact case, and in \cref{subsection:constructionofWKBnetwork} for the meromorphic case. The section concludes with a discussion on BPS indices, in \cref{subsec:snetworkandpathdetours2_BPS_index}, and the local study of rigid flowtrees near a $\dfs$-singularity, in \cref{ssec:flowtrees_specnet}.

%%%%%%%%%%%%%%%%%%%%%%%%%%%%%%%%%%%%%%%%%%%%%%%%%%%%%%%%%%%%%%%%%%%%%%%%
%%%%%%%%%%%%%%%%%%%%%%%%%%%%%%%%%%%%%%%%%%%%%%%%%%%%%%%%%%%%%%%%%%%%%%%%
%%%%%%%%%%%%%%%%%%%%%%%%%%%%%%%%%%%%%%%%%%%%%%%%%%%%%%%%%%%%%%%%%%%%%%%%

\subsection{Morse flowtrees and soliton classes}\label{subsec:Morseflowtreeandsolitons}

Flow trees arise in the study of gradient trajectories of tuples of Morse functions, see e.g.~\cite{Fukaya97_Morseflow,FukayaOh97_MorseHomotopy}, \cite[Section 6]{ruan2006fukaya} and \cite[Section 2]{Abouzaid11_MorseCategory}. In particular, \cite{ruan2006fukaya} shows that certain counts of rigid holomorphic discs agree with counts of gradient flow trees in the context of Lagrangian Floer theory. In the framework of Legendrian submanifolds, \cite{Morseflowtree} established a correspondence between the counts of rigid flow trees and boundary-punctured rigid pseudo-holomorphic disks governed by a Legendrian submanifold whose front has cusp-edge singularities. In the context of Betti Lagrangians, we use \cref{section_setup} to define and study Morse spectral networks in terms of flow trees. We also establish their existence, proving \cref{thm:existence}. Note that the fronts for the Legendrian lifts of Betti Lagrangians have $D_-^4$-singularities, which are non-generic, and the techniques of \cite{Morseflowtree} need to be modified appropriately. From the viewpoint of studying the BPS states in \cite{GNMSN,GMN13_Framed,GMN14_Snakes}, it will also be important to study relative homology classes associated to flow trees.\\

\noindent A tree will be a properly embedded planar graph $\treegraph\sse\R^2$ with possibly semi-infinite edges such that any two vertices are connected by exactly one finite path. Vertices of valence one are allowed and referred to as finite leaves. We consider trees endowed with a cyclic ordering of the edges at each vertex, as surfaces will be naturally oriented, and an orientation for each edge. By definition, a rooted tree will be an oriented tree with a choice of either a univalent vertex or a semi-infinite edge (but not both) such that the all edges of the tree point outward from the root. Following \cite[Definition 2.10]{Morseflowtree}, we consider the following:
%In part following \cite[Chapter 9d]{SeidelZurich}, a tree is said to be stable, resp.~semistable, if all its vertices have valence at least three, resp.~at least two.

\begin{definition}\label{def:Morseflowtree}
	Let $(M,g)$ be a Riemannian manifold, $\multigraph\sse T^*M$ a Lagrangian multigraph over rank $n$ over $(M,g)$, and $\treegraph\sse\R^2$ a connected tree. By definition, for a Morse flowtree for $L$ is a continuous map $\treemap:\treegraph\lr M$ such that:
	\begin{enumerate}

		\item {\it Flow-edge condition}. The restriction $\treemap|_e$ to each oriented edge $e\in E(\treegraph)$ is an injectively parametrized $\ell_1\ell_2$-flowline for $L$, for some $\ell_1,\ell_2\in[1,n]$. If the domain of definition of such trajectory is infinite, then we require $e$ to be a semi-infinite edge.\\
        
        \noindent We denote by $\overline{\gamma}_e^{\ell_i}$ its $\ell_i$-cotangent lift, $i=1,2$, whose image lies in $L$, and write $\overline{\gamma}_e^{\ell_i}(v)$ for the image of a vertex $v\in\dd e$ as a limit point for $\overline{\gamma}_e^i(e)$.\\

		\item {\it Balancing condition}. Consider a vertex $v\in\treegraph$ with cyclically ordered edges $(e_1,\ldots,e_s)$ and denote $\overline{\gamma}_\ell=\overline{\gamma}_{e_\ell}$. Then we require
        $$\overline{\gamma}_\ell^{\ell_2}(v)=\overline{\gamma}_{\ell+1}^{\ell_1}(v),$$
        cyclically in $\ell\in[1,s]$, and $\overline{\gamma}_\ell^2$ is oriented towards the point $\overline{\gamma}_\ell^{\ell_2}(v)$ whereas $\overline{\gamma}_{\ell+1}^{\ell_1}(v)$ is directed away from it.\\
        
		\item {\it Relative cycle condition}. The relative cycles $\overline{\gamma}_e^{\ell_1}-\overline{\gamma}_e^{\ell_2}$, where $e\in E(\treegraph)$ ranges over all edges, piece together to give an oriented relative cycle with endpoints above the univalent vertices.\\
        
        \item {\it Minimality}. There are no 2-valent vertices $v\in\treegraph$ with edges $e_{1},e_2$ such that $F(v)\not\in\pi(\caustL)$ and both are $ij$-flowlines with same set $\{i,j\}$. Similarly, there are no 2-valent vertices $v\in\treegraph$ such that $F(v)\in\pi(\caustL)$ and both the $i$th and $j$th sheets extend smoothly over $v$.\footnote{2-valent vertices $v\in\treegraph$ such that $F(v)\in\pi(\caustL)$ are allowed, but we force them to be switch-vertices in the terminology of \cite{Morseflowtree}.}\\
        
        \item {\it Special vertices}. There is a chosen collection $S(\treegraph)$ of univalent vertices.\hfill$\Box$
        
	\end{enumerate}
\end{definition}
\noindent\cref{def:Morseflowtree}.(5) is in line with partial flowtrees: the univalent vertices above are called special punctured in \cite[Section 2.2.C]{Morseflowtree}. Special punctures are used in the proofs of our arguments, but the type of trees relevant for the study of spectral networks, namely $\dfs$-trees as in \cref{def:d4tree}, do not have any special punctures. The integral of $\lambda$ over the oriented relative homology class in \cref{def:Morseflowtree}.(3) is called the \textit{flow-energy} of $\treemap$. In the case of exact Betti Lagrangians, it is possible to discuss {\it positive} and {\it negative} punctures, as in \cite{Morseflowtree}. That is, a vertex $v\in\treegraph$ is a positive puncture, resp.~negative, if the height difference between the two Legendrian lifts of
$$\overline{\gamma}_{\ell+1}^{(\ell+1)_1}(v)=\overline{\gamma}_{\ell}^{(\ell)_2}(v)$$ is positive, resp.~negative. By Stokes' theorem, the flow-energy of a tree is given by adding such difference of heights for all punctures of the tree. Lemma \ref{lemma:flow-energy} implies that such quantity is positive, thus $\treemap$ contains at least one positive puncture, and the $\treemap$ travels \textit{from} the positive puncture to the negative punctures.

%%%%%%%%%%%%%%%%%%%%%%%%%%%%%%%%%%%%%%%%%%%%%%%%%%%%%%%%%%%%%%%%%%%%%%%%
%%%%%%%%%%%%%%%%%%%%%%%%%%%%%%%%%%%%%%%%%%%%%%%%%%%%%%%%%%%%%%%%%%%%%%%%
%%%%%%%%%%%%%%%%%%%%%%%%%%%%%%%%%%%%%%%%%%%%%%%%%%%%%%%%%%%%%%%%%%%%%%%%

\subsubsection{$D_4^{-}$ flowtrees}{\label{subsubsection:D_4flowtrees}} The type of Morse flowtrees that appear in the study of spectral networks, as related to Betti Lagrangians, are rather particular Morse flow trees. We now introduce this class of flowtrees in \cref{def:d4tree} below.\\

\noindent The branch points of the projection $\blag\lr \obis$ from the Betti Lagrangian to the base will also sometimes referred to as \textit{$D_4^{-}$-singularities} in $S$. Technically, the $D_4^-$-singularity is a front singularity for the projection of the Legendrian lift of a neighborhood of the ramification point in $L$: so this is a slight abuse of notation, but it is sometimes convenient. In the base surface $(S,g)$, we assume that in a (small enough) neighborhood of the branch points in $S$, $(S,g)$ is endowed with a local holomorphic coordinate, the metric $g$ is locally K\"{a}hler, and the multigraph $\blag$ above it is locally holomorphic. In the case of exact Betti Lagrangians, we also assume that the metric is of the polar form $\radial^2d\theta^2+d\radial^2$ in a neighborhood of the marked points $\mkpts$.

 \begin{definition}[$\dfs$-trees]\label{def:d4tree}
 Let $(\bis,\mkpts;{\bf \lknot})$ be a Betti surface, $\blag$ a Betti Lagrangian over $(\bis,g)$ for some Riemannian metric $g$ and $z\in\bis$ a point. By definition, a {\it $D_4^{-}$ flowtree} for $\blag$ is a rooted Morse flowtree $\treemap:\treegraph\lr\bis$ for $\blag$ such that:

 \begin{enumerate}
     \item If the root of $\treegraph$ is a vertex, then it must map to either a branch point, corresponding to a $D_4^-$-singularity, or to the fiber $T^*_z\bis$ above $z$. In the former case, the unique edge adjacent to the root vertex is mapped out of one of the three initial ray flow-lines associated to the $D_4^-$-singularity in the outward direction.\\
     %goes to positive puncture {\RC e.g.~either at a fiber or semiinfinite edge} or the root vertex gets mapped into the $D_4^{-}$ singularities. IN the latter case, the adjacent edge gets mapped into one of the 3 rays from $D_4$ in the \emph{outward} direction.\\
     \item If the root edge is a semi-infinite edge, then it must be asymptotic to a Reeb chord of a Legendrian link in ${\bf \lknot}$ (See \cref{prop:asymptoticoftrajectories} and \cref{prop:asymptoticofWKBtrajectories})\\
     \item If an internal vertex $v$ of $\treegraph$ maps to a branch point, then there are no adjacent edges whose image near $v$ is a cusp-cusp flowlines for that $D_4^-$ singularity.\\
     \item Non-root univalent vertices of $\treegraph$ must map to branch points (at the $D_4^{-}$-singularities), and each corresponding adjacent edge maps to one of the three initial ray flow-lines associated to the $D_4^-$-singularity in the inward direction.
     %\item all finite leaves map to $D_4$ inwards.\\
     %\item at most one semi-infinite edge, and if it exists it must be the root. (if it does not exist, this is a 4d BPS state.)\\
 \end{enumerate}
 Given a $D_4^{-}$-flowtree, the associated relative cycle in $L$ is said to be its \textit{soliton class}. A $D_4^{-}$-flowtree is \textit{rigid} if $\treegraph$ if all non-univalent vertices are trivalent and none of them map to a $D_4^{-}$-singularity.% A $\dfs$ flow-tree is said to be {\it closed} if the root vertex maps outside the $\dfs$-singularities.\hfill$\Box$
 \end{definition}
 %{\RC ``closed does not seem to be used much''}.
\noindent We often refer to $\dfs$-flowtrees as $\dfs$-trees, implicitly understanding that there are gradient flowline conditions as in \cref{def:d4tree} and it is not just a combinatorial graph object.
\begin{remark}\label{rmk:homotopy_homology}
By construction, a $\dfs$-tree in fact defines a regular homotopy class, representing its soliton class. It is possible to develop the theory for relative homotopy classes on $\blag$, and not just relative homology classes.\hfill$\Box$%{\RC ADD REMARK SAYING WE CAN SPECIFY AN IMMERSED PATH.}\\
\end{remark}

\noindent In the physics literature \cite{GNMSN,GMN13_Framed,GMN14_Snakes}, soliton classes of $\dfs$-trees have the following nomenclature:\\
\begin{itemize}
    \item[(i)] For a $\dfs$-tree with a root vertex mapping to a branch point in $S$, its soliton class is said to be a \color{purple}4d BPS state\color{black}.\\
    
    \item[(ii)] For a $\dfs$-tree with a root vertex mapping to a fiber $T^*_z\bis$, its soliton class is said to be a \color{purple}vanilla 2d-4d BPS state\color{black}. The point $z\in\bis$ above which the $\dfs$ starts is said to be the ultra-violet parameter for a certain 1/2-BPS surface defect. This surface defect is assumed to be massive in the infrared and have finitely many vacua: in our setting, this is equivalent to the puncture being positive and there existing finitely many intersections $L\cap T^*_z\bis$, so-called vacua, above $z\in\bis$.\\
    
    \item[(iii)] For a $\dfs$-tree with a root edge, which must be semi-infinite and asymptote to a Reeb chord, we are not aware of any specific terminology in the physics context. They are certainly fundamental objects in the study of spectral networks, and produce Hamiltonian invariants for exact Betti Lagrangians: we refer to them as {\it augmented} $\dfs$-trees.  This notation is chosen due to its relation to the augmentation of the Legendrian contact dg-algebra, cf.~\cref{section_weave}.\\

    \item[(iv)] The \color{purple}framed 2d-4d BPS states \color{black} in \cite[Section 3.3]{GNMSN} do not correspond to rigid $\dfs$-trees. These are associated to two points $z_1,z_2\in\bis$ and a path $\rho$ between them, the latter defining a supersymmetric interface between the surface defects associated to $z_1,z_2\in\bis$. In our context, this is captured by the continuation strips from $L\cap T^*_z\bis$ to $L\cap T^*_z\bis$, where the path $\rho$ is interpreted as a morphism between the fibers, cf.~\cref{section_Floer}. Note that in \cite[Section 3.3]{GNMSN} it is written ``It is generally believed that the defect [$\ldots$] does not depend on the precise path $\rho$ but only on its homotopy class. In the present paper we will take this as an assumption.''. In our formalization, the properties of partially wrapped Fukaya categories readily imply that the (cobordism classes of) moduli of continuation strips are indeed only depending on the homotopy class of $\rho$, not the representative $\rho$ itself, for any compactly supported deformation keeping the Floer datum and $L$ fixed at ends.
    
   % The Floer datam $(\blag,\fibre{z})$ also needs to be regular, and between two Floer datum, there might be interesting bifurcations. 
    %}
    %{\RC This is only for compactly supp deformation of $\rho$, keeping Floer datum and $L$ fixed at ends.}\\
\end{itemize}

\noindent As will be proven in Sections \ref{section_adg} and \ref{section_Floer}, a $\dfs$-tree will correspond to the adiabatic limit of a punctured pseudoholomorphic disk bounded by the Betti Lagrangian $L\sse(T^*S,\la_\std)$. In the case that the root vertex maps to a branch point, there are no punctures and the absolute cycle in $L$ defined by the boundary of such pseudoholomorphic disk coincides the soliton class of the tree. This case can occur  if $L\sse(T^*S,\la_\std)$ is not exact and {\it does} occur for certain meromorphic Betti Lagrangians. The case that the root of a $\dfs$-tree is a semi-infinite edge does occur for exact Betti Lagrangians: the $\dfs$-tree is the adiabatic limit of a pseudoholomorphic disk with {\it one} positive puncture at the Reeb chord associated to the semi-infinite edge. The boundary of such pseudoholomorphic disk, understood as a relative cycle in $(L,\dd L)$, coincides with the soliton class of such $\dfs$-tree. See Sections \ref{section_adg} and \ref{section_Floer} for more details.

\begin{remark}
In the physics context of theories of class $S$, general 4d BPS states also include graphs which are not necessarily $\dfs$-trees, cf.~\cite[Section 3.1]{GNMSN}. There is a natural generalization of \cref{def:d4tree} to the notion of $\dfs$-graph, but it will not be used in this manuscript. In the framework of Floer theory, such general $\dfs$-graph would correspond to a configuration of several pseudoholomorphic disks (with no punctured) bounded by the Betti Lagrangian, whose boundaries intersect each other.\hfill$\Box$
\end{remark}

%%%%%%%%%%%%%%%%%%%%%%%%%%%%%%%%%%%%%%%%%%%%%%%%%%%%%%%%%%%%%%%%%%%%%%%%
%%%%%%%%%%%%%%%%%%%%%%%%%%%%%%%%%%%%%%%%%%%%%%%%%%%%%%%%%%%%%%%%%%%%%%%%
%%%%%%%%%%%%%%%%%%%%%%%%%%%%%%%%%%%%%%%%%%%%%%%%%%%%%%%%%%%%%%%%%%%%%%%%

\subsection{Definitions on spectral networks} \label{subsubsection:mnetwork} Let us introduce the definition of a spectral network associated to a Betti Lagrangian $\blag\sse(T^*\obis,\la_\std)$, where a Riemannian metric $(\obis,g)$ has been fixed. First we introduce {\it pre}-spectral networks in \cref{def:flownetwork}, discuss coherent extensions and the associated soliton classes. Then we define spectral networks in \cref{def:spectralnetwork}.

%%%%%%%%%%%%%%%%%%%%%%%%%%%%%%%%%%%%%%%%%%%%%%%%attempt to simplify the definition

\begin{definition}[Pre-spectral networks]\label{def:flownetwork}
Let $(S,g)$ be a Betti surface endowed with a Riemannian metric, and $\blag\sse(T^*\obis,\la_\std)$ a compatible Betti Lagrangian of rank $n$. By definition, a {\it pre-spectral network} $\fnetwork$ compatible with $(\blag,\obis,g)$ is a properly embedded finite directed graph $\fnetwork\sse S$ such that:

%smooth oriented stratified $1$-manifold with finitely many strata, satisfying the following conditions.
\begin{enumerate}
    \item Each edge of $\fnetwork$ lies outside $\caustL$ and is decorated with an ordered pair $(ij)$, $i,j\in[1,n]$ distinct.
    \item Each directed edge decorated with $(ij)$ is an $(ij)$-gradient flowline for the multigraph $L\sse (T^*S,\la_\std)$.
    \item There are four types of vertices allowed in $\fnetwork$: initial, interaction, non-interaction and inconsistent. These are depicted in Figure \ref{fig:vertices_spec_net2}; note that this classification requires the decoration on the adjacent edges.
%{\RC What about 6-valent vertices? Also, is there a condition that says that the intersection of flowlines must be a vertex?}\\
   % \noindent 
    %and the adjacent walls are of the following four types. {\RC I will define this and add a picture for each of them.}(Initial; \textit{Interaction joints}; \textit{Non-Interaction joints}; \textit{Correctable} $ij$ and $jk$ walls go in, \textit{no} $ik$ wall out {\YJN Better name needed.})
    \item All branch points of the projection $\blag\to \obis$ are initial vertices of $\fnetwork$. At an initial vertex, all adjacent edges are decorated with $(ij)$, where $i,j$ index the sheets of $L$ for the ramification point above the branch point.
    \item The network $\fnetwork$ is flow-acyclic, in the sense that it does not admit a directed cycle of edges $(\wall_1,\ldots,\wall_n)$  in $\fnetwork$, satisfying the following two conditions. 
    \begin{enumerate}
    \item For $1\leq m\leq n-1$, if $\wall_m$ and $\wall_{m+1}$ meet at an inconsistent or a non-interaction vertex, then the two decorations agree.
    \item If $\wall_m$ and $\wall_{m+1}$ meet at an interaction vertex, then either the two decorations agree, or $\wall_{m+1}$ is the outgoing edge decorated with $(ik)$ in  \cref{fig:vertices_spec_net2}.
\end{enumerate}
        \end{enumerate}
    
   %{\YJN Acylicity is needed to be incorporated for the following reason. Given such a cyclie, we can think of adjoining the part of the network "below" each of the vertices. This will create an absolute homology cycle (which is not a tree) with non-zero area.}
\noindent Edges of $\fnetwork$ are also said to be {\it walls} for the pre-spectral network $\fnetwork$. An edge decorated with $(ij)$ is often said to be an $(ij)$-wall. A {\it wall} which (asymptotically) ends at a marked point $\mkpts\sse\bis$ is said to be semi-infinite. Finally, in the meromorphic case, where $(\sS,\scurve,g)$ has a Riemann structure, a WKB pre-spectral network $\fnetwork_{\theta}$ of phase $\theta$ is defined to be a pre-spectral network compatible with $(\Re(e^{i\theta}\scurve),\obis,g)$.\hfill$\Box$
\end{definition}

\begin{center}
	\begin{figure}[h!]
		\centering
		\includegraphics[scale=1.1]{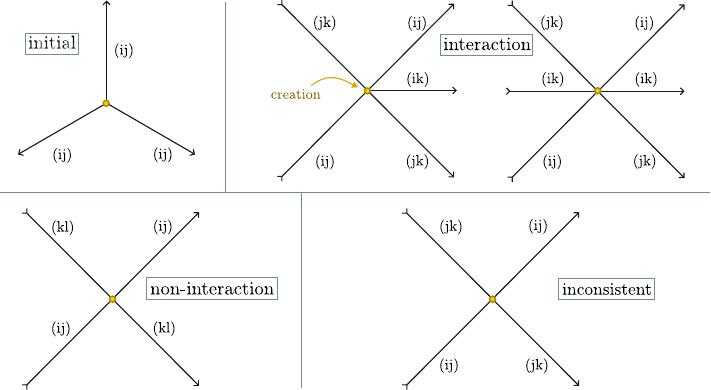}
		\caption{The four local models for vertices in pre-spectral networks in \cref{def:flownetwork}: initial, interaction, non-interaction and inconsistent. We refer to the two types of interaction vertices as creation and 6-valent. In all these models $i,j,k,l\in[1,n]$ are all different.%{\RC Annihilation must be changed to hexavalent. In adiabatic degeneration no anihilation because there are flowtrees that terminate there, so the counts are zero (so no out edge in $\mu$-count) but there are disks there.}
        }\label{fig:vertices_spec_net2}
    \end{figure}
\end{center}

\noindent Vertices of $\fnetwork$ are also referred to as joints, and a pre-spectral network is said to be {\it creative} if all the interaction joints are creation joints. Observe that creative pre-spectral networks are automatically flow-acyclic. For notational ease, we will always assume that a given WKB pre-spectral network has phase $0$.\footnote{Because a WKB pre-spectral network of phase $\theta$ for $\scurve$ is the same object as the WKB pre-spectral network of phase $0$ for $e^{i\theta}\scurve$.} Also, we will often omit explicitly writing the data of the Riemannian metric $g$ in $(S,g)$ and a Betti Lagrangian $\blag\sse(T^*\obis,\la_\std)$ will always implicitly have a compatible choice of such $g$.\\

\begin{remark}\label{rmk:trees_vs_specnet_orientations} (1) Note that there are two different orientations on $\snetwork$, either seen as a spectral network or as a flowline graph. The orientation as a spectral network is outgoing to infinity: e.g.~the three rays emerge out of a $\dfs$-singularity. The orientation when seen as a flowtree is opposite, incoming from infinity: the three flowlines near a $\dfs$-singularity convergence towards it.\\

\noindent (2) In the physics literature, annihilation type joints, opposite to creation joints, are also allowed as interaction vertices. This is because the BPS index vanishes and non-active walls are not drawn. These will not appear in our context, as flowlines canceling each other are accounted by themselves, e.g.~cf.~\cref{ssec:explicit_comp} and Figure \ref{fig:weaves_example_cancelation_pair2} therein.
\hfill$\Box$
\end{remark}

\noindent The following result allows us to describe the pre-spectral networks in \cref{def:flownetwork} as unions of $D_4^{-}$ flow-trees, as in \cref{def:d4tree}.
%Let $\wall$ be an $(ij)$-wall. Then there exists a $D_4^{-}$ flow-tree whose root edge flowline coincides with that of $\wall$, and whose image is contained in that of $\fnetwork$. The set of such flow-trees is finite and it is unique, if $\fnetwork$ is (creative). 
\begin{proposition}[$\dfs$-trees from pre-spectral networks]\label{prop:soliton}
Let $\blag\sse (T^*S,\la_\std)$ be a Betti Lagrangian and $\fnetwork$ a compatible pre-spectral network. Consider a point $z\in\wall$ on a wall $\wall\in\fnetwork$. Then:
\begin{enumerate}
    \item There exists a $D_4^{-}$-tree whose root vertex maps to $z$ and the flowline for its adjacent edge is mapped to $\wall$.
    \item The entire image of such $\dfs$-tree is contained in that of $\fnetwork$.
    \item The set of such flow-trees is finite and, if $\fnetwork$ is creative, it contains a unique element\footnote{We can drop flow-acyclicity at the cost of changing the statement to the set of flow-trees being finite \textit{up} to some energy cut-off.}. 
\end{enumerate}
\end{proposition}

\begin{proof}
Consider a rooted (partial) flowtree $\treegraph$ whose root edge maps to a wall of $\fnetwork$ and the non-root univalent vertices map to the vertices of $\fnetwork$. Let us now explain an iterative procedure that extends $\treegraph$ to a larger partial flow-tree $\treegraph'$ by adding walls of $\fnetwork$ to each of the non-root univalent vertices of $\treegraph$. This algorithm produces finitely many such extensions and the $\dfs$-flowtree in the statement of \cref{prop:soliton} is then declared to be the set of maximal extensions of the partial flow-tree given by the flowline connecting $z$ to the starting point of $\wall$.\\

The procedure is described as follows. Let $v$ be a univalent vertex and $\wall$ a wall of $\fnetwork$ having $v$ as the starting vertex:
\begin{enumerate}
    \item If $v$ maps to an initial vertex, then the procedure ends without doing anything.
    \item If $v$ maps to a non-interaction joint, add the ingoing wall labeled with the same pair of sheets as $\wall$. 
    \item If $v$ maps to a hexavalent interaction joint. Let $\wall_{ij},\wall_{jk},\wall_{ik}$ be the ingoing edges and $\wall_{ij}',\wall_{jk}',\wall_{ik}'$ the outgoing edges. In the cases $\wall=\wall_{ij}'$ or $\wall=\wall_{jk}'$, we add the ingoing wall labeled with the same pair of sheets as $\wall$. If $\wall=\wall_{ik}'$, then we can either add the two edges $\wall_{ij}$ and $\wall_{jk}$, or we add the single wall $\wall_{ik}$. 
    \item If $v$ maps to a creation joint, and $\wall=\wall_{ij}'$ or $\wall=\wall_{jk}'$, do the same as above. In the case $\wall=\wall_{ik}'$, we add the two walls $\wall_{ij}$ and $\wall_{jk}$. 
    \item Finally, if $v$ maps to an inconsistent joint, either $\wall=\wall_{ij}'$ or $\wall=\wall_{jk}'$, and so we can proceed on as above.
\end{enumerate}
Since the spectral network $\fnetwork$ is finite, and a vertex cannot reappear with the same flowline edge by the flow-acyclicity of $\fnetwork$, given any $\treegraph$ the process will stop after finitely many steps.  Since each step can produce more than one trees if and only if there are some hexavalent vertices involved, $\fnetwork$ being creative implies that it can produce exactly one extension.
\end{proof}
\color{black}

\noindent An example of a $\dfs$-tree as \cref{prop:soliton} is depicted in Figure \ref{fig:SpecNet_D4TreeWalls} (Center). We can use \cref{prop:soliton} to associate a $\dfs$-tree to a wall $\wall\in\fnetwork$ without specifying a point $z\in\wall$. Indeed, since $\wall$ is directed, we choose $z$ to be the unique boundary point of $\wall$ reached by flowing backwards along $\wall$. This would be the ``negative end'' of $\wall$, and note that it might be a marked point $\mkpts\sse\bis$, in which case the $\dfs$-tree is interpreted to have a root semi-infinite edge. In either case, the argument for \cref{prop:soliton} thus shows that there is a well-defined $\dfs$-tree for $\wall$. By definition, a $D_4^{-}$-tree obtained from \cref{prop:soliton} is said to be a $D_4^{-}$-tree\textit{ on the pre-spectral network} $\fnetwork$.\\
\color{black}

\begin{center}
	\begin{figure}[h!]
		\centering
		\includegraphics[scale=1.1]{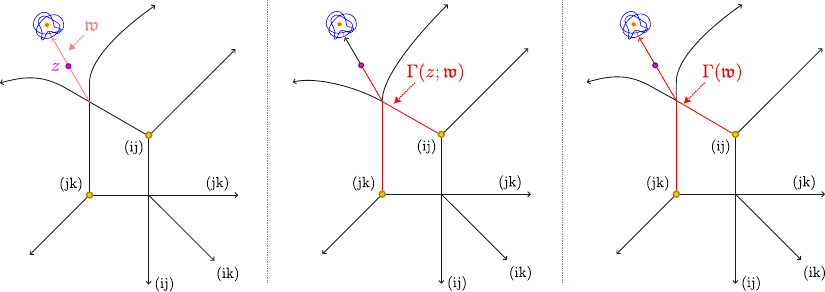}
		\caption{(Left) A spectral network $\fnetwork$ with a wall $\wall$, in clear orange, and a point $z\in\wall$, in pink. This wall $\wall$ is a semi-infinite edge. (Center) The $\dfs$-tree $\Gamma(z;\wall)$ with a root vertex at $z$, in red. (Right) The $\dfs$-tree $\Gamma(\wall)$ in solid red, with a semi-infinite root edge asymptotic to a Reeb chord (in dashed red). In this example, one might take $i=1,j=2$ and $k=3$. As with the above figures, yellow vertices are initial vertices of $\fnetwork$, the marked point in $\mkpts\sse\bis$ is drawn in orange and its associated Legendrian front in blue.}\label{fig:SpecNet_D4TreeWalls}
    \end{figure}
\end{center}

\begin{definition}\label{def:D4tree_from_walls}
Let $\fnetwork$ be a pre-spectral network, $\wall$ a wall and $z\in \wall$ a point. Each $\dfs$-tree obtained as in \cref{prop:soliton} is denoted by $\Gamma(z;\wall,\fnetwork)$ and its soliton class is denoted by $\soliton{z;\wall,\fnetwork}$. 
\hfill$\Box$
\end{definition}
%{\YJN $\Gamma(\wall)$ doesn't appear anywhere else, except for this page. Maybe redundant?}

\begin{remark} If $L\sse(T^*\obis,\la_\std)$ is an exact Betti Lagrangian, then the set of $\dfs$-trees associated to a wall in \cref{def:D4tree_from_walls} contains a unique element for each wall $\wall$.%{\RC The only issue is the hexavalent vertex, where you get a $Y$ and and $I$ from the middle wall in hexavalent.}
\hfill$\Box$
\end{remark}

\noindent Note that creative pre-networks are rather robust objects, i.e.~invariant under smooth isotopies of $\bis$ and $C^\infty$-small perturbations of the base metric $(S,g)$. This is the content of the following:\\

The existence of {\it new Stokes' lines}, in the nomenclature of \cite{BerkNevinsRoberts82_NewStokes}, a.k.a.~higher-order Stokes phenomenon in \cite[Section 2]{HowlsLangmanOlde04_HigherOrderStokes} (see also \cite[Section 4]{AokiKawaiTakei01_StokesCurves}), implies that the pre-spectral networks in \cref{def:flownetwork} do not contain sufficient information to study spectral curves with irregular singularities, or Stokes local systems on a Betti surface, for rank three or higher. Newer walls for the pre-spectral networks must be introduced to account for all the possible $\dfs$-trees dictated by a Betti Lagrangian. From this perspective, the role of inconsistent vertices is to give birth to the new Stokes' lines, producing interaction vertices. Inspired by the terminology in \cite[Definition 2.26]{GrossHackingKeel15_MSLogCYSurfaceI}, we introduce the following notion to capture the appearance of new Stokes' lines:

\begin{definition}\label{def:simpleextensionflownetwork}
Let $\fnetwork,\fnetwork'$ be two pre-spectral networks compatible with $(L,S,g)$ such that $\fnetwork\sse\fnetwork'$ is an inclusion\footnote{Note that this is different than a graph inclusion, as there might be more vertices in $\fnetwork'$ than in $\fnetwork$.} of stratified $1$-manifolds. By definition, $\fnetwork'$ is said to be a consistent extension of $\fnetwork$ if:
\begin{enumerate}
\item The vertex set of $\fnetwork$ is contained in the vertex set of $\fnetwork'$,
\item the inconsistent vertices in $\fnetwork$ are interaction joints in $\fnetwork'$.
\end{enumerate}
In such a situation, we denote $\fnetwork\to \fnetwork'$ for a consistent extension of $\fnetwork$.\hfill$\Box$
\end{definition}

\noindent \cref{def:simpleextensionflownetwork} is one step in an iterative process, as a consistent extension of a pre-spectral network might itself have inconsistent vertices. A first idea is to iterate this process ad infinitum. The only issue is that the flow-energy of the newly created $\dfs$-trees might remain bounded, potentially leading to infinitely many $\dfs$-trees below a finite flow-energy cut-off. For that, we introduce the following notion of a {\it gapped} sequence of consistent extensions, a terminology inspired by the notion of gapped $A_\infty$-algebras from \cite{FOOO1}. 

\begin{definition}\label{def:energystable}
Let $\fnetwork_1\to \fnetwork_2\to \ldots$ be a sequence of consistent extensions of pre-spectral networks. Given a point $z\in \fnetwork_{m+1}-\fnetwork_{m}$, consider the minimum $\min_{m+1}(\energy(\soliton{z}))$ of the flow-energies for all $\dfs$-trees $\Gamma(z;\wall,\fnetwork_{m+1})$, for $\wall$ a wall in the pre-spectral network $\fnetwork_{m+1}$ containing $z$. By definition, the sequence is said to be \textit{gapped} if there exist positive constants $\hbar,M\in\R_+$ such that
$$\min_{m+1} \big(\energy{(\soliton{z})}\big)>\floor{m/M}\hbar$$
for all $z\in \fnetwork_{m+1}-\fnetwork_{m}$ and $m\geq M$, $m\in\N$.\hfill$\Box$
%{\YJN Slightly edited the $\wall$-dependence}
%{\YJN changed to the more precise notion}
%Then such a sequence is energy-finite, if for any $E>0$, there exists some $N(E)$ such that $\fnetwork_i(E)=\fnetwork_{i+1}(E)$ for all $i\geq N$.   
\end{definition}
Definitions \ref{def:flownetwork}, \ref{def:simpleextensionflownetwork} and \ref{def:energystable} allow us to introduce the definition of a spectral network:
%{\YJN This notion of "gapped" is directly borrowed from FOOO, for gapped $A_{\infty}$-algebras. This really is an energy quantization statement. Suggest a better name if needed.}
%{\RC ADD LEMMA FOR PERTURBATION OF PRENETWORKS STAYS STABLE}
\begin{definition}[Spectral networks]\label{def:spectralnetwork}
Let $\blag\sse(T^*\bis,\la_\std)$ be a Betti Lagrangian with $(S,g)$ a chosen Riemannian metric. By definition, the Betti Lagrangian $\blag$ is said to admit a \textit{Morse spectral network} $\snetwork$ if there exists a gapped sequence
$\snetwork_1\to \snetwork_2\to \ldots$
of consistent extensions of pre-spectral networks $\snetwork_i,i\in \N$. By definition, a Morse spectral network is \textit{finite} if $\snetwork_i=\snetwork_{i+1}$ for sufficiently large $i\in\N$. In the meromorphic case, $\scurve$ is said to admit a WKB spectral network $\snetwork_{\theta}$ of phase $\theta$ if it admits a Morse spectral network consisting of WKB pre-spectral networks of phase $\theta$. \hfill$\Box$
    %\to (\Gamma_2,\snetwork_2)\to\ldots$. 
    %there exists a directed acyclic graph $\snetworkgraph$, a decomposition $V(\snetworkgraph)=V^{0}(\snetworkgraph)\sqcup V^{nonint}(\snetworkgraph) \sqcup V^{int}(\snetworkgraph)$, and a continuous embedding $\snetwork:\snetworkgraph\into \obis$ such that the following holds.
\end{definition}
\noindent \cref{def:spectralnetwork} is more general than the notion of spectral networks in the physics literature \cite{GNMSN,GMN12_WallCrossCoupled,GMN13_Framed,GMN14_Snakes}, thus the terminology Morse spectral network. That said, for notational ease, we often refer to Morse spectral networks in \cref{def:spectralnetwork} as {\it spectral networks} and refer to the spectral networks in the physics literature as WKB spectral networks, as they are always associated to meromorphic spectral curves.\\
%{\RC Maybe Morse for general, but we will call them spectral networks. And holomorphic ones we call WKB.}

In the introduction, as stated in \cref{thm:existence}, we claimed the existence of spectral networks compatible with a Betti Lagrangian $\blag\sse(T^*\bis,\la_\std)$. The next two sections are devoted to the proof of \cref{thm:existence}. Specifically, \cref{thm:existence}.(i), for exact Betti Lagrangians, is established in \cref{subsection:constructionofexactnetwork}. \cref{thm:existence}.(ii), for WKB spectral networks, is proven in \cref{subsection:constructionofWKBnetwork}.\\

\noindent Note that, thanks to the gapped condition, we can ensure that for a given energy level $E\in\R_+$, there exists $m\in\N$ such that any $z\in \snetwork_{m+1}-\snetwork_m$, satisfies $\min_{m+1}(\energy(\soliton{z}))>E$. This allows us to define the energy-filtration on spectral network.%{\RC Say something about where this is used.}

\begin{definition}[Energy filtration]\label{def:energyfiltration}
Let $\fnetwork$ be a spectral pre-network and $E\in\R_+$. By definition, $\fnetwork(E)$ is the set of points in $\fnetwork$ such that $\min_{\fnetwork}(\energy(\soliton{z}))\leq E$. Similarly, for a spectral network $\snetwork=(\snetwork_1\to \snetwork_2\to \ldots)$ we define $\snetwork(E)$ to be $\snetwork_m(E)$, where $m\in\N$ is the smallest number such that any $z\in \snetwork_{m+1}-\snetwork_m$ satisfies $\min_{m+1}(\energy(\soliton{z}))>E$.\hfill$\Box$ 
 \end{definition}

\begin{remark}
By construction, for any $z\in \snetwork(E)$, any $\dfs$ flowtrees that end at $z$ with energy less than $E$ must be contained in $\snetwork(E)$. By scanning according to the energy filtration in \cref{def:energyfiltration}, it follows that a Morse spectral network is the union of all possible $\dfs$-trees.\hfill$\Box$
\end{remark}

%%%%%%%%%%%%%%%%%%%%%%%%%%%%%%%%%%%%%%%%%%%%%%%%%%%%%%%%%%%%%%%%%%%%%%%%
%%%%%%%%%%%%%%%%%%%%%%%%%%%%%%%%%%%%%%%%%%%%%%%%%%%%%%%%%%%%%%%%%%%%%%%%
%%%%%%%%%%%%%%%%%%%%%%%%%%%%%%%%%%%%%%%%%%%%%%%%%%%%%%%%%%%%%%%%%%%%%%%%

\subsection{Proof of \cref{thm:existence}.(i)}\label{subsection:constructionofexactnetwork} Let $L\sse(T^*S,\la_\std)$ be an exact Betti Lagrangian and $(S,g)$ a metric adapted to $L$. The goal is to show that there exists a finite Morse spectral network $\snetwork$ compatible with $L$ and $(S,g)$, as in \cref{def:spectralnetwork}. In the local model for Betti Lagrangians near the $\dfs$-ramification points of the projection $\pi:L\lr S$, the Lagrangian multigraph $L$ can be (and is) assumed to be holomorphic over a neighborhood of the $\dfs$-branch point in $S$, and the Riemannian metric $(S,g)$ is taken to be K\"ahler near such neighborhood. Thus we can apply the local model for the three initial flowline rays near a $\dfs$-singularity studied in \cref{subsection:Morselines}, and consider the associated trajectory, maximally extended, for each of these three rays out of each $\dfs$-branch point in $S$, using the asymptotics of trajectories (\cref{prop:asymptoticoftrajectories} and \cref{prop:asymptoticofWKBtrajectories}).

Let $\snetwork_1$ be the union of all these trajectories, i.e.~for each $\dfs$-branch point, consider the union of all the trajectories in $S$ obtained by flowing out of the branch point from each of the three initial rays. This graph $\snetwork_1$ might nevertheless not define a pre-spectral network due to the following phenomena:
\begin{enumerate}
    \item A trajectory from a $\dfs$-branch point might pass through another $\dfs$-branch point. Note that a trajectory might do that in a way that it flows towards a $\dfs$-branch point entering it via its cusp-cusp sheets, or it might merely pass through it while having two smooth sheets, or a cusp-smooth sheet, above it. By exactness, such trajectories cannot enter a $\dfs$-branch point via a cusp-cusp edge, so that case we have discarded, but the other two cases are still possible.
    \item The trajectories forming $\snetwork_1$ might intersect each other in a non-transverse manner, e.g.~tangentially.
\end{enumerate}
\noindent In order to construct a pre-spectral network from $\snetwork_1$, we must perturb the initial Riemannian metric $(S,g)$ so that these two issues above do not occur. In this case of an exact Betti Lagrangian, such perturbation exists by \cite[Section 3.2]{Morseflowtree}, which itself uses \cite{Smale61_GradientDynamical}; the possible perturbations form an open dense set within the corresponding space of Riemannian metrics. We thus apply such perturbation to $(S,g)$ and obtain a graph in $S$, which we still denote by $\snetwork_1$. This graph $\snetwork_1$, after perturbing $g$, does define a pre-spectral network which satisfies:
\begin{enumerate}
    \item $\snetwork_1$ has only initial, non-interaction and inconsistent vertices. That is, $\snetwork_1$ has no interaction vertices. The initial vertices of $\snetwork_1$ are in bijection with the $\dfs$-branch points in $S$.
    \item the edges of $\snetwork_1$ adjacent to an initial vertices are given by the three possible trajectories out of the $\dfs$-branch point.
\end{enumerate}
\noindent The edges of $\snetwork_1$ are decorated automatically by the condition on the initial vertex in \cref{def:flownetwork}: the three trajectories out of a $\dfs$-branch point are decorated with $(ij)$, where $i,j\in[1,n]$ are the two sheets joined at the ramification point above it. This decoration extends uniquely to all edges of $\snetwork_1$ because the incoming decorations at any of the vertices determine the outgoing ones, as in Figure \ref{fig:vertices_spec_net2}. This $\snetwork_1$ is the first of the (to be) sequence of pre-spectral networks leading to the spectral network $\snetwork$.\\

Let us construct a first consistent extension $\snetwork_1\to\snetwork_2$, as follows. For each inconsistent vertex of $\snetwork_1$, consider the unique flowline that starts at the inconsistent vertex: this exists and it is well-defined due to the local model of the three sheets for the multigraph $L\sse(T^*S,\la_\std)$ above such an inconsistent vertex. Let $\snetwork_2$ be the graph in $S$ defined by the union of all such flowlines; the decoration on the new edges of this graph are uniquely determined by the decorations in $\snetwork_1$ and Figure \ref{fig:vertices_spec_net2}. Similar to the case of $\snetwork_1$, the resulting graph in $S$ might have non-transverse intersections and trajectories passing through a $\dfs$-branch point which are not initial rays for that branch point. As before, we further perturb the Riemannian metric data so that such behavior is not present: the space of such perturbation is the intersection of two dense open sets, open and non-empty. Let us still denote by $\snetwork_2$ the result of such perturbation. As a graph, $\snetwork_2$ might appear to be a pre-spectral network: it is nevertheless possible that some of the new trajectories being introduced are periodic or, in general, violate the acyclicity condition \cref{def:flownetwork}.(4). That is, there might be a trajectory in $\snetwork_2$ that starts at an inconsistent vertex of $\snetwork_1$ and ends at the same interaction vertex; e.g.~entering it via one of the two existing incoming flowlines from $\snetwork_1$. By \cref{prop:asymptoticoftrajectories}, exactness of $L\sse(T^*S,\la_\std)$ implies that this cannot occur and thus $\snetwork_2$ defines a pre-spectral network with no periodic trajectories.\\

We iterate the construction in the previous paragraph, thus constructing a consistent extension $\snetwork_2\to \snetwork_3$ whose new trajectories start or end at the inconsistent vertices of $\snetwork_2$. In order to ensure that $\snetwork_3$ is a pre-spectral network, we perturb the metric data again and use \cref{prop:asymptoticoftrajectories}. In particular, $\snetwork_3$ is flow-acyclic, because by our construction, $\snetwork_3$ network is creative. This argument proves flow-acyclicity of $\fnetwork_3$. as above. Iterating this construction for $i\in\N$, we obtain a sequence $\snetwork_1\to \snetwork_2\to \ldots$
of consistent extensions of pre-spectral networks. Our task is now to ensure that this sequence is gapped. In the exact Betti Lagrangian case, we will establish an a priori bounds on the energy to prove finiteness.%, and thus finiteness will follow from gapped. The existence of an a priori estimate is key in this case, and;
It is in line with the a priori inequalities used in Gromov-compactness results for pseudo-holomorphic strips. The following subsection is devoted to proving gapped and finite, thus establishing the necessary results to conclude the proof of \cref{thm:existence}.(i).

\begin{center}
	\begin{figure}[h!]
		\centering
		\includegraphics[scale=1.2]{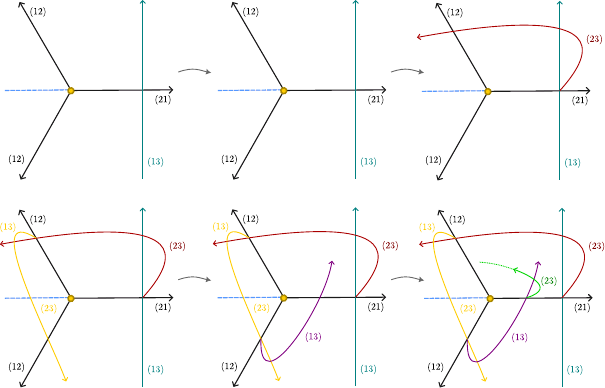}
		\caption{(Upper left) An initial vertex in a pre-spectral network of rank 3 with its adjacent three trajectories and a $(13)$-trajectory (from another vertex) interaction with a wall of this initial vertex. Notice that we have denoted the sheets as in \cite{GNMSN,GMN13_Framed,GMN14_Snakes} and added a branch cut (in dashed blue) for the first and second sheets. (Upper center) A consistent extension of the upper left piece obtained by adding the $(23)$-wall, in red: its behavior is that of spiraling around the initial $\dfs$ vertex, approaching it. (Upper right onward) A sequence of consistent extensions where the newly created walls from inconsistent vertices spiral around the initial $\dfs$-vertex, infinitely approaching it. As with the above figures, the yellow vertex is the initial vertex of $\fnetwork$ and each new wall for a consistent extension is drawn in a different color: red, yellow, purple and green, in order of appearance.}\label{fig:SpecNet_Spirals2}
    \end{figure}
\end{center}

\begin{remark}\label{rmk:spiraling_behavior} A type of behavior that might a priori occur is that of a spiraling sequence of creation vertices approaching a $\dfs$-vertex, violating the finiteness condition for $\snetwork$. Figure \ref{fig:SpecNet_Spirals2} is an example of such a sequence.\hfill$\Box$
\end{remark}

%%%%%%%%%%%%%%%%%%%%%%%%%%%%%%%%%%%%%%%%%%%%%%%%%%%%%%%%%%%%%%%%%%%%%%%%
%%%%%%%%%%%%%%%%%%%%%%%%%%%%%%%%%%%%%%%%%%%%%%%%%%%%%%%%%%%%%%%%%%%%%%%%
%%%%%%%%%%%%%%%%%%%%%%%%%%%%%%%%%%%%%%%%%%%%%%%%%%%%%%%%%%%%%%%%%%%%%%%%

\subsubsection{Energy gaps in consistent extensions}\label{subsection:preliminarytransversalitycondition}
Near each $\dfs$-singularity $b\in L$, we have a holomorphic neighborhood where both the sheets of the Betti Lagrangian and a neighborhood $U_b\sse S$ of $\pi(b)$ are endowed with holomorphic coordinates. Let $z\in U_b$ be a local conformal coordinate and $\holsheet_1,\ldots,\holsheet_n$ define the smooth sheets of $\blag$ over $U_b$. As in \cref{example:BettiSurfaces_StokesData}, consider the associated front given by the set of functions
$$\{\pm \Re(z^{3/2}),\Re(\abs{z} \holsheet_1(0)e^{i\theta}),\ldots,\Re(\abs{z}\holsheet_n(0)e^{i\theta}),\},\quad z\in U_b.$$
\noindent By a $C^\infty$-small perturbation $L\sse (T^*S,\la_\std)$, compactly supported near the ramification point, we can assume that this front is generic. That is, all the associated Reeb chords occur at different values of $\theta$ and it contains only double transverse points. We always choose a local model near a $\dfs$-singularity satisfying these conditions and, inspired by \cite[Section 3.1]{Morseflowtree}, we refer to this as satisfying preliminary transversality conditions at the $\dfs$-singularity.\\
%{\YJN Nitpick: is it obvious that this is always possible? It is obvious to us, but would the referee find it obvious?}{\RC I'll add a sentence.}

\noindent The spiraling behavior from \cref{rmk:spiraling_behavior}, which we want to ensure does not occur in an infinite manner, can be formalized as follows:
\begin{definition}[Chain interactions]\label{def:chainofinterractions}
Locally in a neighborhood $U_b\sse S$ of a $\dfs$-branch point, a \textit{chain interaction tree} is a flow-tree $\treemap:\treegraph\to U_b$ such that:
%{\RC Ask for picture of this}
\begin{itemize}
    \item[-] The distinguished out-going edge is a smooth-cusp flowline.
    \item[-] All the internal vertices are trivalent and lie outside the origin. Furthermore, given an internal vertex, there is a unique out-going edge which is a smooth-cusp flowline, and there is exactly one in-going edge which is given by an initial ray at $\pi(b)$.\hfill$\Box$
%{\YJN The internal vertex must be: one initial edge in, one smooth-cusp flowline out (and therefore one smooth-cusp flowline out by the balancing condition) }
\end{itemize}
\end{definition}
Let us show that chain interaction trees have uniformly bounded number of internal vertices.
\begin{proposition}\label{prop:positivesingularityinterractionchain}
In the notation above, let $z\in U_b\sse S$ be a local coordinate centered the $\dfs$-branch point $\pi(b)$. Then, there exists two positive finite constants $\delta\in\R_+$ and $N\in\N$ such that any chain interaction trees with image contained in $\{z:\abs{z}\leq \delta\}\sse U_b$ has its number of internal vertices bounded above by $N$.
%{\RC Need to define length.}
\end{proposition}
\begin{proof} Consider a chain interaction tree in $U_b$ and a sufficiently small circle $S^1_\theta$ giving the base for the front of the Legendrian at the $\dfs$-singularity. We project the chain interaction tree to a flowtree on this circle $S_\theta^1$, as follows. Given the tree, we totally order the vertices of $\treegraph$ by the distance from its unique root vertex: then the in-going external edge coming from a smooth-singular flowline projects to a flowline of the same index on $S_\theta^1$. In addition, each external edge coming from the Reeb chord ray meeting the smooth-singular flowline gives a \textit{two}-valent internal vertex with a \textit{negative puncture}. (For instance, if the incoming smooth-singular flowline $ij$ interacts with a singular-singular flowline $jk$ to create a smooth-singular flowline $ik$, then the projected flowtree is \textit{two}-valent, with the in-going flowline $ij$ and the out-going flowline $ik$.) Therefore, each such external edge gives rises to an internal \textit{negative puncture}. The unique out-going external edge and the remaining unique in-going singular-smooth flowline edge may be either positive or negative. By \cite[Section 3]{Morseflowtree} and the preliminary transversality condition, a cusp-free generic Legendrian link has an upper bound on the number of internal vertices and edges, given a finite bound on the number of positive punctures, which in our case is at most $2$. In conclusion, there is a uniform upper bound on the length of the chain interaction tree, as required.
%We use that the Legendrian link $\lknot_{\Laggerm}$ is well-defined up to smooth isotopies in the front-projection, whose resulting isotopies in the Lagrangian projections send Reeb chords to Reeb chords. So for small enough $\radial=\abs{z}$, we can project the chain interaction tree to a flowtree on $S^1$ as follows.
\end{proof}

\begin{proposition}\label{thm:gromovcompactness}
Let $\snetwork=(\snetwork_1\to \snetwork_2\to \ldots)$ be a spectral network constructed as in \cref{subsection:constructionofexactnetwork}. Then the spectral network $\snetwork$ is gapped and  finite.
\end{proposition}

\begin{proof}
Let us show $\snetwork$ is gapped. First, apply \cref{prop:positivesingularityinterractionchain} to obtain two positive constants $\delta\in\R^+$ and $N\in\N$ such that chain interaction trees contained in $\nbh_{2\delta}$ neighborhoods of the $\dfs$-singularities have lengths bounded above by $N$. By choosing $\delta\in\R^+$ small enough, we can and do assume that the neighborhoods $\nbh_{2\delta}$ are a disjoint union of small balls containing all the $\dfs$ singularities. Second, choose a positive constant $\smallcst\in\R_+$ smaller than the minimum of the flow-energies for all flowlines traveling from the boundary $\partial \nbh_{\delta}$ to the boundary $\partial \nbh_{2\delta}$. By \cref{eq:weaklyboundedmultigraph}, Betti Lagrangians are weakly bounded and thus we can choose this minimal energy $\smallcst$ to be smaller than the minimal flow-energy needed to travel between the distinct neighborhoods $\partial \nbh_{2\delta}$ of the $\dfs$-singularities in $S$.\\ 

Let $\treegraph$ be a rigid $\dfs$ tree with flow-energy $\leq E$. Suppose the non-univalent internal vertices avoid the $\dfs$-singularities. Then we will show that $E>\floor{\left(|V(\treegraph)|/N\right)}\smallcst,$ where $N$ is the upper bound from \cref{prop:positivesingularityinterractionchain}. From this, it follows that $\snetwork$ must be gapped, since $\dfs$-trees in in $\snetwork_i-\snetwork_{i-1}$ are all rigid, and none of the non-univalent internal vertices get mapped to the branch points, and these trees have at least $i$ internal vertices. To prove the claim, we make the obesrvation that there are at most $\floor{\left(|V(\treegraph)|/N\right)}$  \textit{maximal} interaction chain trees contained in $\treegraph$, since by \cref{def:chainofinterractions}, such trees have at most $N$ internal vertices. Furthermore, by our choice of $\smallcst\in\R_+$, the unique external in-going smooth-cusp edge in the interaction chain tree must carry flow-energy at least $\smallcst$ and thus a maximal chain interaction tree must also have flow-energy at least $\smallcst$. Therefore, $E>\floor{\left(|V(\treegraph)|/N\right)}\smallcst$, as claimed. 

To argue that $\snetwork$ is finite, note that the primitive of $\la_\std$ restricted to the exact Lagrangian $L\sse(T^*S,\la_\std)$ is a smooth function on the compact domain given by the complement of the $U_{2\delta}$-neighborhoods in $S$. It is therefore a bounded primitive and, by \cref{lemma:flow-energy}, the total energy of any $\dfs$-tree must be uniformly bounded as well. Since $\snetwork$ is gapped, only finitely many $\dfs$-trees can appear and thus $\snetwork$ is finite, on any precompact subset of $\obis$. Finally, in order to control finiteness of $\snetwork$ at infinity, we use \cref{lem:first_noescape} implies that the walls that intersect the boundary transversely, in the inward direction, cannot leave, and this condition holds for any other wall that is created at interaction joints. Therefore, we are reduced to the finiteness of flow-trees in $J^1S^1$ with a fixed number of negative punctures. Again, such flow-trees have finitely many internal vertices. Therefore, $\snetwork$ is finite everywhere. 
%\color{black}
\end{proof}

%%%%%%%%%%%%%%%%%%%%%%%%%%%%%%%%%%%%%%%%%%%%%%%%%%%%%%%%%%%%%%%%%%%%%%%%
%%%%%%%%%%%%%%%%%%%%%%%%%%%%%%%%%%%%%%%%%%%%%%%%%%%%%%%%%%%%%%%%%%%%%%%%
%%%%%%%%%%%%%%%%%%%%%%%%%%%%%%%%%%%%%%%%%%%%%%%%%%%%%%%%%%%%%%%%%%%%%%%%

\subsection{Proof of \cref{thm:existence}.(ii)}\label{subsection:constructionofWKBnetwork} The structure of the argument is similar to the exact Betti case, as presented in \cref{subsection:constructionofexactnetwork}. The only significant difference is that the only perturbations that we are allowed are those performed in the holomorphic setting and varying the phase $\theta$. In particular, we can no longer guarantee that $\snetwork$ is creative, since three flowlines may intersect at a single point.  The sequence $\snetwork=(\snetwork_1\to \snetwork_2\to \ldots)$ of consistent extensions of pre-spectral networks is built in the same manner. As before, the issues to be addressed are non-transverse intersections and flow-acyclicity. For removing periodic trajectories,  we use \cref{prop:asymptoticofWKBtrajectories} in order to argue that the sequence of consistent extensions of pre-spectral networks produces no periodic trajectories (instead of \cref{prop:asymptoticoftrajectories}). For the former, and the more general flow-acyclicity, we need a different argument, as \cite[Section 3.2]{Morseflowtree} is not a holomorphic perturbation. \cref{lem:tangenciesareremovable} below proves that tangencies can indeed be removed. Assuming that result, we apply \cref{thm:gromovcompactness}, which also works in this setting, and conclude \cref{thm:existence}.(ii).\\

%{\RC perturbing away tangency is the only thing that changes. The preliminary transversality works the same.}
%We prove \cref{thm:finitewkbnetwork} by studying the deformations of $ij$-flowlines under changing $\theta$. Then the rest is applying \cref{prop:asymptoticofWKBtrajectories}, which tell us that for generic $\theta$, we can remove periodic trajectories. 

To argue that we can remove non-tranverse intersections, consider $z\in\sS$ a point in the base and let $\gamma(s)$ be an $ij$-flowline of phase $0$ starting at $\gamma(0)=z$ and let $C\in \mathbb{R}_{+}$ be a constant. 
%{\YJN 
%i.e.~$C$ is the height difference of the points in $\gamma$ above $z$.}{\RC What is initial charge?} 
%{\YJN The initial charge is more to do with the particular choice of the local primitive of the holomorphic Liouville form. For instance, if you shoot off initial trajectories and these initial trajectories meet at a point, then you want the initial charge of the newly created trajectory to be $C=Z_{ij}+Z_{jk}$}
Define the charge map of $\gamma$ at time $s$ with initial charge $C$ to be the constant $C+\int_0^s(\lambda^i-\lambda^j)(\gamma'(t))dt$. By definition, the flat coordinate\footnote{The map $W$ is a flat conformal coordinate map in the sense that the pull-back of $\holsheet^i-\holsheet^j$ becomes $dW$ in the $W$-coordinate.} $W$ on a neighborhood of $\gamma$ in $\sS$ is the analytic continuation of the charge map over a small sectorial neighborhood of $\gamma$ avoiding the $ij$-branch points, so we have $W=C+\int(\lambda^i-\lambda^j)$ in such a neighborhood. The constant term $C$ matters in the case $\gamma$ is a new-born edge at an interaction joint. Consider the $\theta$-deformation of $\gamma$ given by the curve 
\begin{equation}\label{eq:theta_defn}
s\to W^{-1}((C+s)e^{i\theta}),
\end{equation}
so the $\theta$-deformation has the same energy (the absolute value of charge) as $\gamma$ at time $s$. By differentiating \cref{eq:theta_defn} with respect to $\theta$, we deduce that the normal deformation vector at time $s$ is given by $(C+s)/(\lambda^i-\lambda^j)$. Here the appearance of $(C+s)$ in the expression in \cref{eq:theta_defn} is because charge-preserving deformations necessarily have their normal deformation affected by the energy. Let us now show that via $\theta$-deformations we can indeed ensure transversality:

%We then define the charge of $\gamma$ at time $t$ to be $C+\int_0^t(\lambda^i-\lambda^j)(\gamma')dt.$ The charge map can be extended to a sectorial neighbourhood of $\gamma$ by integrating $\lambda^i-\lambda^j$ along radial line segments.  We write  for the map taking the sectorial neighbourhood to the 

\begin{lemma}\label{lem:tangenciesareremovable}
Let $\gamma_1,\gamma_2$ be a $(ij)$-flowline and $(kl)$-flowline of phase $\theta=0$, $i,j,k,l\in[1,n]$ distinct, and suppose that $\gamma_1,\gamma_2$ have an isolated tangential intersection at a point $w\in\sS$. Then, for small enough generic $\theta\neq 0$, their $\theta$-deformations intersect transversely.
\end{lemma}
\begin{proof}
%{\YJN This proof feels a bit ad-hoc and so I've added some "intro" paragraph. Made some slight changes and so feel free to revert it back}

Let $W_{ij}$ and $W_{kl}$ denote the charge maps with initial charges $l_1$ and $l_2$, associated to $\gamma_1$ and $\gamma_2$ near $z$. The goal is to show that the variation vectors of $\gamma_1$ and $\gamma_2$ are $W_{ij}/(\holsheet_i-\holsheet_j)$ and $W_{kl}/(\holsheet_k-\holsheet_l)$ and that tangencies are stable if and only if the variation vectors agree.% However, since all the functions involved are holomorphic, there are only finitely many such points, which can be avoided by taking perturbation.

%Let $l_1,l_2$ denote the energies of $\gamma_1,\gamma_2$ at the point $w\in\sS$. Consider the flat conformal coordinate $z$ on a neighborhood of $\gamma_1$, so that $z(w)=l_1$, and we can write
%$z=l_1,\quad

We begin the proof. In the $z=W_{ij}$-coordinate, we can write $$ \holsheet_i-\holsheet_j=dz,\quad \holsheet_k-\holsheet_l=\phi(z) dz,$$ for some holomorphic function $\phi$ defined in this $z$-domain. Then $\phi(z)$ is not constant because $\holsheet_k-\holsheet_l$ cannot be a constant multiple of $\holsheet_i-\holsheet_j$ everywhere, as $w$ is an isolated tangential intersection point. As before, $W^{-1}((l_2+s)e^{i\theta})$ gives rise to the $\theta$-deformation of $\gamma_2$. As we change $\theta$, the initial points of $\gamma_1$ and $\gamma_2$ typically change too (unless they are initial flowlines themselves) and so the quantities $l_1=l_1(\theta),l_2=l_2(\theta)$ also depend on $\theta$. Thus, our aim is to show that the map
\begin{equation}
F(s_1,s_2,\theta)=(s_1,s_2,\theta)\to ((l_1(\theta)+s_1)e^{i\theta},W^{-1}((l_2(\theta)+s_2)e^{i\theta}))\in \sS\times\sS
\end{equation}
is transverse to the diagonal for generic $\theta$. This is achieved in two steps, as follows.\\

\paragraph{{\it Step 1}}%{\RC Might not be necessary this step 1 on tangency locus}
Let us analyze the locus of points where the $\theta$-deformations intersect tangentially. We claim that the $\theta$-deformations can be only tangential over the tangency locus, where the \textit{tangency locus} is defined to be the set where the imaginary part $\Im(\phi)=0$ vanishes. Indeed, suppose the $\theta$-deformations are tangential, with the common tangency given by the vector $v$, and note that
$$v\in \ker \Im(e^{i\theta})\Leftrightarrow v\equiv_{\mathbb{R}}e^{-i\theta},\quad \mbox{ and so }\quad v\in \ker(\Im(e^{i\theta}\phi))\Leftrightarrow v\equiv_{\mathbb{R}}e^{-i\theta}/\phi,$$
where $v\equiv_{\mathbb{R}}w$ denotes $v,w$ being real colinear. Therefore, in this case, $1$ and $1/\phi$ should be real colinear at the point of the tangency which is if and only if $\Im(\phi)=0$, as claimed. Therefore tangencies can only appear along $\Im(\phi)=0$.

Now, the tangency locus is a locally finite stratified manifold, see e.g.~\cite{Wall75_RegularStratifications}. Let us show that the tangency locus is smooth. For that, choose a smooth parametrized $1$-dimensional stratum $\eta(t)$ of the tangency locus and note that differentiating the condition $\Im(\phi)=0$ leads to $\Im(\phi' \eta'(t))=0$. This latter question is that of the flowline of $\phi'$ passing through $l_1$. Then Picard–Lindel\"of's uniqueness of solutions of differential equations implies that the tangency locus must be the entire flowline of $\phi'$ passing through $l_1$ and thus it is smooth.\\

\paragraph{{\it Step 2}} Let us compute the set of points in the tangency locus where the transversality of the map $F$ also fails. By construction, the vector space given by $\text{im}(dF)+T\cdot\mbox{diag}$ is generated by the rows of the $4\times 4$ matrix:
$$\begin{pmatrix}
e^{i\theta} & 0 \\
0 & e^{i\theta}/\phi \\
1 & 1 \\
\dd_\theta l_1e^{i\theta}+i(l_1+s_1)e^{i\theta} & (e^{i\theta}/\phi)\dd_\theta l_2+i(l_2+s_2)e^{i\theta}/{\phi}
\end{pmatrix}.$$

\noindent After an elementary transformation eliminate the terms containing $\dd_\theta l_1$ and $\dd_\theta l_2$, the determinant of the matrix is equal to $$(1/\phi)((l_2+s_2)/\phi-(l_1+s_1)).$$
Therefore the matrix is not invertible, and so transversality fails, if and only if 
\begin{align}\label{eq:transversalityfailure}
l_1+s_1= (l_2+s_2)/\phi 
\end{align}
\noindent This set is discrete. Indeed, since $z=(l_1+s_1)e^{i\theta}$ and $W(z)=(l_2+s_2)e^{i\theta}$, this set belongs to the zero set of the holomorphic function
\[z=W(z)/\phi(z)\]
which is discrete unless $W(z)=z\phi(z)$. This latter equality cannot occur though: differentiating both sides gives $\phi(z)dz=dW=(\phi(z)+z\phi'(z))dz$ and so $\phi'(z)=0$ along the 1-stratum $\eta(t)$. This is a contradiction since $\phi(z)$ is holomorphic and not a constant function. Therefore, since the set of points where \eqref{eq:transversalityfailure} holds (and so transversality fails) is finite, standard parametric transversality applies and we conclude that, for generic $\theta$, the $\theta$-deformations do not intersect tangentially.
\end{proof}

\noindent As argued above, at the beginning of \cref{subsection:constructionofWKBnetwork}, \cref{thm:existence}.(ii) then follows now that \cref{lem:tangenciesareremovable} is proven, after we sort out flow-acyclicity. To see that flow-acyclicity holds, we impose the genericity condition that for each inductive step, $\snetwork_i$ consists only of  vertices that are stable under generic infinitesimal perturbations of $\theta$, and that there does not exist an interaction vertex such that the out-going $ik$-wall is an ancestor to one of the in-going walls (in other words, some of the in-going walls have multiplicities). Such vertices are not generic, since the normal deformation vector for the new WKB flowline at $v$ has its norm greater than that of the out-going wall at $v$. We now proceed on by induction. Suppose the acyclicity condition holds for $\snetwork_k$ but fails for $\snetwork_{k+1}$ and let $(\wall_1,\ldots,\wall_n)$ be the directed cycle that satisfies conditions (a) and (b) in \cref{def:flownetwork}, for which the decorations of $\wall_1$ and $\wall_n$ agree. Since there are no periodic trajectories, this is possible only if $\wall_1$ is a wall contained in $\snetwork_k$, and so we see that there is some $i$ such that $\wall_i\subset \snetwork_k$ is an ancestor to $\wall_1$ in $\snetwork_k$. However, then the terminal vertex of $\wall_i$ fails to satisfy the genericity condition. \cref{thm:existence} is thus established, as \cref{thm:existence}.(i) was concluded in \cref{subsection:constructionofexactnetwork}.

%%%%%%%%%%%%%%%%%%%%%%%%%%%%%%%%%%%%%%%%%%%%%%%%%%%%%%%%%%%%%%%%%%%%%%%%
%%%%%%%%%%%%%%%%%%%%%%%%%%%%%%%%%%%%%%%%%%%%%%%%%%%%%%%%%%%%%%%%%%%%%%%%
%%%%%%%%%%%%%%%%%%%%%%%%%%%%%%%%%%%%%%%%%%%%%%%%%%%%%%%%%%%%%%%%%%%%%%%%

\subsection{Spectral networks and 2d-4d BPS indices}\label{subsec:snetworkandpathdetours2_BPS_index} In the context of the supersymmetric field theories in \cite{GNMSN,GMN12_WallCrossCoupled,GMN13_Framed,GMN14_Snakes}, BPS states are certain irreducible unitary representations of a fixed super Lie algebra acting on a given Hilbert space. In this framework, the counts of BPS states are achieved by considering super traces of certain operators on this Hilbert space. In Floer theory, counts of pseudo-holomorphic disks are given by considering the (oriented) cobordism type of their moduli spaces. For $\dfs$-trees, we have the choice of counting them either way: by counting the pseudo-holomorphic strips that limit to them, or by treating them combinatorially via the spectral network $\snetwork$ and using supertraces; the results of this article show that these counts coincide, cf.~\cref{ssec:detour_are_continuation_strips}. The Floer theoretic count is discussed in \cref{section_Floer}, and we now discuss the combinatorial count.\\

Let $\sphb,\sphsc$ be the unit tangent bundles of $\obis$ and $\blag$, respectively. Let $H$ be the homotopy class representing the sphere bundle fiber. Given $z\in \sphb$, we write ${z_i}, {i\in [1,n]}$ for the lifts of $z$ to $\sphsc$. We denote by $\mathbb{C}[\pi_1\big(\sphsc)]$ the group algebra of the fundamental group of (relative) homotopy classes in $\sphsc$: as further discussed in \cref{section_Floer}, we do not explicitly include the set of endpoints in the notation, as it is implicitly understood by context. For instance, if we consider a soliton class $\soliton{z;\wall}$, it is implicitly understood that it is a relative homotopy class with endpoints in the lifts above $z\in S$.\\

\noindent Let $\fnetwork$ be a pre-spectral network and $\wall$ a wall. We denote by $t(\wall)$, resp.~$s(\wall)$, the sphere bundle lift of the terminal point of $\wall$, resp.~initial, given by the unit velocity vector of $\fnetwork(\wall)$ at the terminal point, resp.~initial. Regard $\mathbb{Z}$ as a category with a single object, with $\hom$ identified with $\mathbb{Z}$ (as a set, not as an additive group), and the set of walls $W(F)$ as a category with trivial morphisms.

\begin{definition}\label{def:bpsindex}
Let $\fnetwork$ be a finite pre-spectral network and $W(\fnetwork)$ its set of walls. By definition, the vanilla 2d4d BPS index on $\fnetwork$ is the functor
\begin{align}
\mu:\mathbb{C}[\pi_1\big(\sphsc)]\times W(\fnetwork)\to  \mathbb{Z}
\end{align}
uniquely defined by the following constraints:

\begin{enumerate}
\item $\mu([\ccc(z;\wall)];\wall)=1$ for $\wall$ an initial wall.
\item $\mu(\cdot ;\wall)$ vanishes on all classes, except for the $H$-orbits of $\ccc{(z;\wall)}\in \{\soliton{z;\wall}\}$ for $z\in \wall$.
\item %The components of $\mu$ are $H$-anti-invariant: 
$\mu(H\rho;\wall)=-\mu(\rho;\wall).$

\item $\mu$ is independent of the basepoint $z\in \wall$ for the classes in $\{\soliton{z;\wall}\}$.%{\RC $\mu$ evaluated at $\soliton_i{z;\wall}$ and $\soliton_i{z';\wall}$ is the same if $z,z'$ same wall.}
\item Given a non-interaction vertex $v$ with ingoing walls $\wall_{ij}, \wall_{kl}$ and outgoing walls $\wall_{ij}', \wall_{kl}'$, let $v(\wall_{ij})$ denote the unit velocity lift of $v$ at $\wall_{ij}$. Then 
$$\mu(\ppp;\wall_{ij})=\mu(\ppp;\wall'_{ij}),$$
for all $\ppp\in \hom(v(\wall_{ij})_j,v(\wall_{ij})_i)$, and similarly for $\mu(\cdot;\wall_{kl})$ and $\mu(\cdot;\wall_{kl}')$.\\

\item Let $v$ be an interaction vertex with ingoing walls $\wall_{ij},\wall_{jk},\wall_{kl}$, outgoing walls $\wall_{ij}',\wall_{jk}',\wall_{kl}'$. and $v(\wall_{ij}), v(\wall_{jk}), v(\wall_{ik})$ the corresponding unit velocity lifts. Then 
\begin{align}\label{eq:wcf}
&\mu(\ppp;\wall'_{ij})=\mu(\ppp;\wall_{ij}),\\%\forall \ppp\in \hom(v(\wall_{ij})_j,v(\wall_{ij})_i)\\
&\mu(\ppp;\wall'_{jk})=\mu(\ppp;\wall_{jk}),\\ %\forall \ppp\in \hom(v(\wall_{jk})_k,v(\wall_{jk})_j)\\
&\mu(\ppp;\wall'_{ik})=\mu(\ppp;\wall_{ik})+\sum_{[\ppp_1]\circ [\ppp_2]=[\ppp]} \mu(\ppp_1;\wall_{ij})\mu(\ppp_2;\wall_{jk}), %\forall \ppp\in \hom(v(\wall_{ik})_k,v(\wall_{ik})_i) 
\label{eq:horivafa}
%\bpsindex{\sphsoliton{i'j'}}\\
%&\bpsindex{\sphsoliton{jk}}=\bpsindex{\sphsoliton{j'k'}}\\
%&\bpsindex{\sphsoliton{i'k'}}=\bpsindex{\sphsoliton{ik}}+(-1)^{\inner{\sphsoliton{ij}+\sphsoliton{jk}}{\sphsoliton{ik}}}\bpsindex{\sphsoliton{ij}}\bpsindex{\sphsoliton{jk}}
\end{align}
%Let $e_{ij},e_{kl}$ denote the in-going edges and let $e_{ij}',e_{kl}'$ denote the out-going edges. In this case, wthe sets $\{\ccc_{\ast}(e_{ij})=\ccc_{\ast}(s(e'_{ij});e'_{ij})$ $,\ldots,\ccc_k(e_{ij})\}$ and $\{,\ldots,\ccc_k(s(e'_{ij});e'_{ij})\}$ agree. 
\end{enumerate}
where $\rho$ is a path such that the end points of the projection all lie at $v$, and the sum is over the mod $2H$-orbits.\hfill$\Box$
%sum in \cref{eq:horivafa} is taken over
\end{definition}

\noindent By definition, a wall of $\snetwork$ is said to be active if its BPS index is non-zero.\\

\begin{remark}
(1) \cref{eq:horivafa} is known as the 2d Hori-Vafa wall-crossing formula. In the framework of spectral networks, for generic phase, the counts as captured by \cref{def:bpsindex} are given by a vanilla 2d-4d BPS index which is truly a 4d BPS index modified to accout for the 2d interactions. In the above we are able to restrict to pure 2d BPS indices because in our context they do coincide with the vanilla 2d4d BPS indices.\\

\noindent (2) In \cref{section_Floer} it will be proven that \cref{def:bpsindex} is actually a count of pseudo-holomorphic strips for a fiber near that wall: it counts continuation strips oriented by the wall along a $\snetwork$-adapted path. \hfill$\Box$
\end{remark}

%%%%%%%%%%%%%%%%%%%%%%%%%%%%%%%%%%%%%%%%%%%%%%%%%%%%%%%%%%%%%%%%%%%%%%%%
%%%%%%%%%%%%%%%%%%%%%%%%%%%%%%%%%%%%%%%%%%%%%%%%%%%%%%%%%%%%%%%%%%%%%%%%
%%%%%%%%%%%%%%%%%%%%%%%%%%%%%%%%%%%%%%%%%%%%%%%%%%%%%%%%%%%%%%%%%%%%%%%%

\subsection{Spectral networks and rigid flowtrees}\label{ssec:flowtrees_specnet} Flowtrees are often used in the context of Floer theory as a limiting model for configurations of pseudoholomorphic curves. To wit, \cite[Chapter 1]{Fukaya93_MorseHomotopyAinfty} studies such trees as a first model of the Fukaya $A_\infty$-category, see also \cite[Theorem 2.3]{FukayaOh97_MorseHomotopy} and \cite[Chapter 1]{Fukaya97_Morseflow}. In the framework of Legendrian submanifolds, rigid flowtrees are studied in \cite{Morseflowtree}, compared to pseudo-holomorphic disks and used to provide certain computational models for the Legendrian contact dg-algebra, see also \cite{RutherfordSullivan20_CellularDGA1}. A first aspect in the Floer-theoretic understanding of spectral networks is its relation to such flowtrees, which we analyze in this subsection. Note that this is far from sufficient in order to prove our main results or compare to \cite{GMN12_WallCrossCoupled,GMN13_Framed,GMN14_Snakes}, and Sections \ref{section_adg}, \ref{section_Floer} and \ref{section_wfc} proceed beyond these initial techniques. For now, we use the notation and concepts from \cite{Morseflowtree} and establish the following result:

\begin{proposition}[Augmented $\dfs$-trees and rigid flowtrees]\label{theorem:spectralnetworkandrigidflowtrees} Let $(S,\mkpts,{\bf\lknot})$ be a Betti surface, $\blag\sse(T^*S,\la_\std)$ an exact Betti Lagrangian, $(\obis,g)$ an adapted metric and $\snetwork\sse S$ a compatible creative Morse spectral network. Then, there exists a front-generic exact Lagrangian $\pneck\sse(T^*S,\la_\std)$ Hamiltonian isotopic to $\blag$, and a perturbed metric $(S,\neckg)$ such that:
\begin{enumerate}
\item $\pneck=\blag$ and $\neckg=g$ outside a compact subset, and otherwise $\pneck\stackrel{C^\infty}{\approx}\blag$ and $\neckg\stackrel{C^\infty}{\approx}g$.
\item There is a bijective correspondence between augmented $\dfs$-trees on $\snetwork$ and the rigid flowtrees of $(S,\neckg)$ with a single positive puncture on $\dd_\infty\pneck$. 
\end{enumerate}
\end{proposition}

\noindent The relation to spectral networks is \cref{theorem:spectralnetworkandrigidflowtrees}.(2), whereas \cref{theorem:spectralnetworkandrigidflowtrees}.(1) describe the type of perturbation that it is used, i.e.~the Betti Lagrangian and the metric essentially only need to be modified near the branch points, equiv.~the vertices of $\snetwork$.

\noindent Note that the results of \cite{Morseflowtree} only apply to Legendrians with generic front singularities. The Legendrian lift of an exact Betti Lagrangian is {\it not} in that class because $\dfs$-singularities are not generic real Legendrian singularities. Thus, to establish \cref{theorem:spectralnetworkandrigidflowtrees}, we first need to understand rigid flowtrees near a $\dfs$-singularity.

\subsubsection{Trees near $\dfs$-singularities}\label{subsubsection:localstudy} Consider the $\dfs$-front singularity, as discussed in Case (2) of \cref{subsection:Morselines}, cf.~also \cite[Section 2.2.3]{legendrianweaves}. Its germ can be parametrized as
$$\delta_4^-:\R^2_{u,v}\lr\R^3_{x,y,z},\quad \delta_4^-(u,v)=\left(u^2-v^2,2uv,\frac{2}{3}(u^3-3uv^2)\right).$$
The singularity itself is at the origin of $\R^3$ and the front is invariant under $(2\pi/3)$-rotation along the $z$-axis. Since $\delta_4^-(0,v)=\delta_4^-(0,-v)$, the three edges of $A_1^2$-singularities are located at
$$E_0:=\{(x,y,z)\in\R^3:x<0,y=z=0\}=\delta_4^-(0,v)=\R_{<0},\quad v\in\R\setminus\{0\}$$
and its two additional images under this rotation, which are
$$E_1:=\{(x,y,z)\in\R^3:\sqrt{3}x=2y,z=0\}=\delta_4^-(\sqrt{3}v,v)=e^{-2\pi i/3}\R_{<0}$$
$$E_2:=\{(x,y,z)\in\R^3:\sqrt{3}x=-2y,z=0\}=\delta_4^-(-\sqrt{3}v,v)=e^{2\pi i/3}\R_{<0},$$
also with $v\in\R\setminus\{0\}$. The three initial rays discussed in Case (2) of \cref{subsection:Morselines} are
$$R_0:=\{(x,y,z)\in\R^3:x>0,y=z=0\}=\R_{>0},\quad R_1:=e^{-2\pi i/3}\R_{>0},\quad R_2:=e^{2\pi i/3}\R_{>0},$$
which lie exactly in between these three edges $E_0,E_1,E_1$, cyclically alternating as $R_0,E_1,R_2,E_0,R_1,E_2$. The local lifts for these three non-generic flowtrees are depicted in Figure \ref{fig:FlowTreesNearD4}.

\begin{center}
	\begin{figure}[h!]
		\centering
		\includegraphics[scale=0.6]{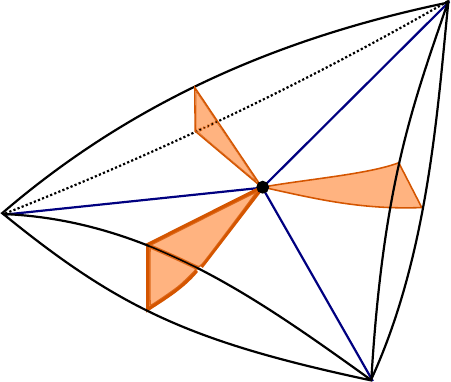}
		\caption{A non-generic model for the flowtrees near the $D_4^-$-singularity.}
		\label{fig:FlowTreesNearD4}
	\end{figure}
\end{center}

In a nutshell, the claim is that a generic $C^\infty$-small compactly supported Hamiltonian perturbation of such non-generic $\dfs$-front gives rise to a generic Legendrian front whose only rigid flowtrees are $C^\infty$-close to the rays $R_0,R_1,R_2$ in the support of such perturbation, and otherwise coincide with them. This is the content of the following:

\begin{lemma}[Sneaky trees]\label{lemma:sneaky_trees}
Let $\Pi\sse\R^3$ be the germ of a $\dfs$-singularity, modeled as above, and $\La(\Pi)$ its Legendrian lift. There exists a compactly supported $C^\infty$-small Legendrian isotopy of $\La(\Pi)$ such that the resulting Legendrian $\La^\dagger$ has a front $\Pi^\dagger$ such that:
\begin{enumerate}
    \item The Lagrangian projection of $\La^\dagger$ is embedded.
    \item there are exactly three rigid flowtrees $T_0,T_1,T_2$ in $(\R^2_{x,y},g_\std)$ for the generic Legendrian $\La^\dagger$,
    \item Each $T_i$ is $C^\infty$-close to $R_i$ and, outside a compact set, $T_i=R_i$, $i=1,2,3$.
    \item The compact support of the Legendrian isotopy can be taken to be arbitrarily small near the domain of the $\dfs$-singularity.
\end{enumerate}
\end{lemma}

\begin{center}
	\begin{figure}[h!]
		\centering
		\includegraphics[scale=0.5]{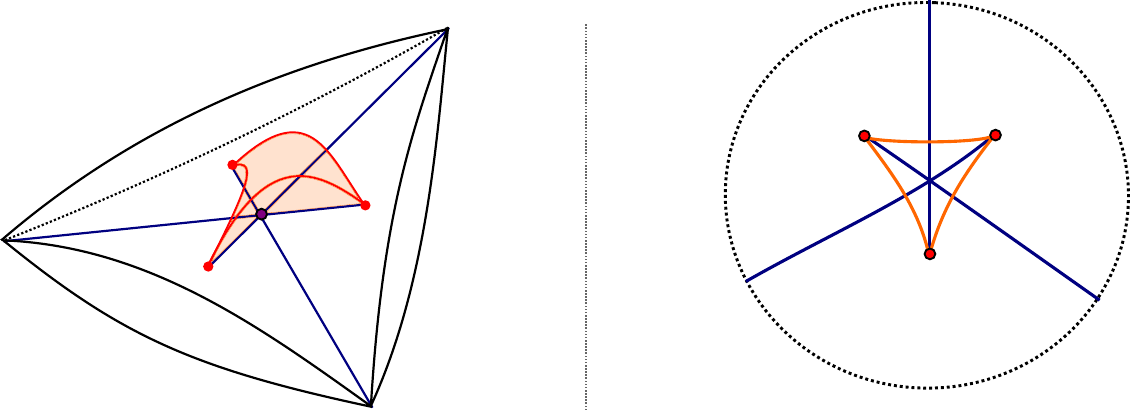}
		\caption{(Left) The spatial front for a generic $C^\infty$-small perturbation of the $\dfs$-front singularity. (Right) The view for this front from above, with orange encoding $A_2$-cusp edges and dark blue edges indicating $A_1^2$-crossing edges. Each of the three orange dots indicates a $A_3$-swallowtail singularity.}
		\label{fig:D4Generic}
	\end{figure}
\end{center}

\begin{proof}
Choose a $C^\infty$-small Legendrian isotopy of $\La(\Pi)$ which preserves the $(2\pi/3)$-rotation symmetry of the front and let $\Pi^\dagger$ be its front. Due to this symmetric choice and the fact that a flowtree between two sheets cannot pass through a crossing of those sheets, it suffices to analyze rigid flow trees in the sectorial region above $E_1$ and $E_2$. A $C^\infty$-small perturbation of $\dfs$-singularity gives a front $\Pi$ with three swallowtails as depicted in Figure \ref{fig:D4Generic}. In Figure \ref{fig:D4Generic} (right) we choose this region betwee $E_1,E_2$ to be the bottom sector, so that $R_0$ is a vertical semiray from the center downwards. Since the perturbation is compactly supported and the gradient condition is local, any rigid flowtrees for $\Pi$ coincide with those of $\Pi^\dagger$ away from a compact set, and thus must coincide with $R_0,R_1,R_2$. The argument now contains two parts:

\begin{enumerate}
    \item the construction of a rigid Morse flow tree $\Gamma(D_4^-)$ coinciding with $R_0$ at the boundary,
    
    \item showing that no other rigid trees exist in this local model.
\end{enumerate}

\begin{center}
	\begin{figure}[h!]
		\centering
		\includegraphics[scale=0.5]{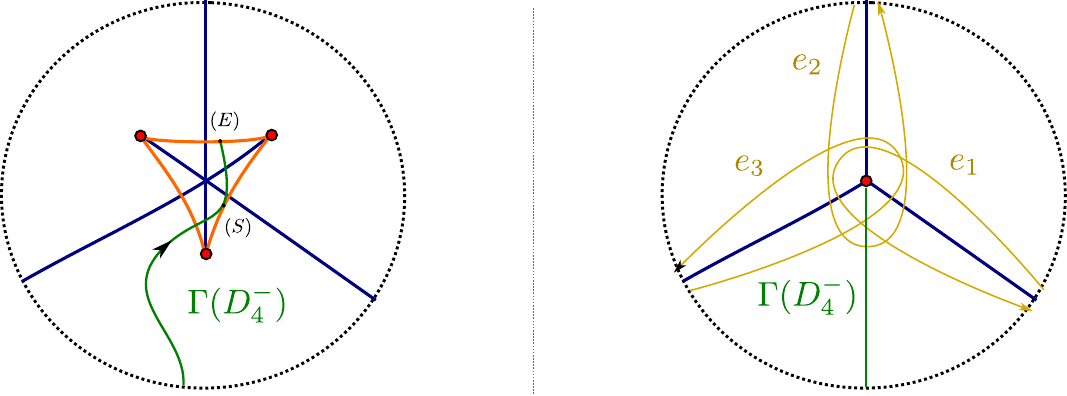}
		\caption{(Left) The flowtree $\Gamma(D_4^-)$ as seen from above, with two black dots indicating its unique switch (S) and end (E). (Right) Three homology classes $e_1,e_2,e_3$ in the (Legendrian lift of the) front and the flowtree $\Gamma(D_4^-)$ seen from far above, so that the three swallowtails are contained in the central red dot.}
		\label{fig:D4TreeHomology}
	\end{figure}
\end{center}

For Part (1), a flowtree $\Gamma(D_4^-)$ is constructed following Figures \ref{fig:D4TreeHomology} and \ref{fig:D4SneakyTreeMovie}. In Figure \ref{fig:D4TreeHomology} (left) the flowtree is depicted in green, as seen from above, and Figure \ref{fig:D4TreeHomology} (right) determines the relative homology class of its boundary by declaring that it geometrically intersects each relative cycles $e_1,e_2,e_3$ once, $e_1,e_3$ positively and $e_2$ negatively. We have also described $\Gamma(D_4^-)$ in Figure \ref{fig:D4SneakyTreeMovie} as a movie, slicing the front in Figure \ref{fig:D4TreeHomology} (left) with horizontal slices, from bottom to top. This flowtree $\Gamma(D_4^-)$ has an asymptotic positive puncture $\rho$ at the boundary condition given by $R_0$, no negative punctures, one {\it switch} (S) and one {\it end} (E). This is a legitimate flowtree by \cite[Section 2]{Morseflowtree}. We choose the perturbation of $\La(\Pi)$ and perturb the metric $g$ to $g^\dagger$ such that the tree is transversely cut out, which can be done by \cite[Prop. 3.14]{Morseflowtree}. The dimension of the moduli space for a general tree $\Gamma$ is given in \cite[Definition 3.4]{Morseflowtree} as:
		$$\dim(\Gamma)=-2+I_u(\rho^+)-\left(\sum_{k=1}^{|\Gamma^-|}(I_s(\gamma_k^-)-n+1)\right)+e(\Gamma)-s(\Gamma)-y_1(\Gamma),$$
where $I_u(\rho^+)=2$ is the index at the positive puncture $\rho^+$, i.e. the dimension $\dim(W^u(\rho^+))$ of the unstable manifold, $I_s(\gamma^-)=2$ is the coindex at the negative puncture $\gamma^-$, i.e. the dimension $\dim(W^s(\gamma^-)$) of the stable manifold, $e(\Gamma)$ is the number of ends in $\Gamma$, $s(\Gamma)$ is the number of switches and $y_1(\Gamma)$ the number of $Y_1$-type vertices. In the case of $\Gamma(D_4^-)$, the formula for the dimension yields
	$$\dim(\Gamma)=-2+I_u(\rho)+e(\Gamma)-s(\Gamma)=-2+2+1-1=0,$$
as there are no negative punctures nor $Y_1$-vertices. This proves $\Gamma(D_4^-)$ is a {\it rigid} flowtree.
	
	\begin{center}
		\begin{figure}[h!]
			\centering
			\includegraphics[scale=0.8]{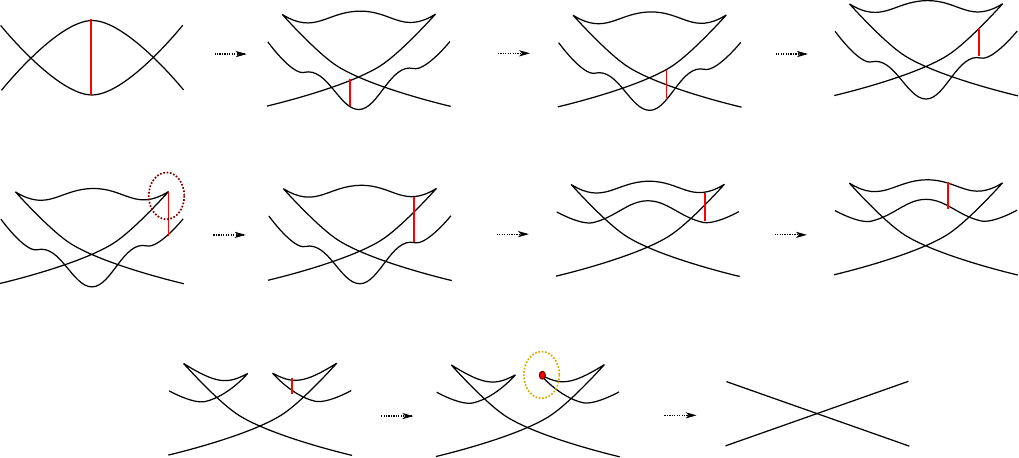}
			\caption{\color{black} Slices of the rigid flowtree $\Gamma(\dfs)$, drawn in red, in the germ of a $D_4^-$-front singularity. The front is sliced horizontally, bottom to top. A {\it switch} (S) occurs at the leftmost diagram in the second row and an {\it end} (E) occurs at the center diagram in the third row.\color{black}}
			\label{fig:D4SneakyTreeMovie}
		\end{figure}
	\end{center}

For Part (2), let us argue that no additional rigid flowtrees exist, which we deduce using the dimension formula for $\dim(\Gamma)$ above. First, for any germ of a Legendrian singularity, there exists a front with no Reeb chords. Hence, given that $\rho$ is the only allowed boundary condition away from the singularity and in the sectorial region between $E_1$ and $E_2$, there can be no negative punctures. By the classification of vertices in \cite[Section 3]{Morseflowtree} and the combinatorics from Figure \ref{fig:D4SneakyTreeMovie}, the rigid flowtrees must have one end, and therefore one switch. Since the switch can only occur as in $\Gamma(D_4^-)$, any rigid flowtree that matches $R_0$ actually coincides with $\Gamma(D_4^-)$. Finally, the combinatorics of the sheets imply that no two of the three rigid flowtrees, given by $\Gamma(D_4^-)$ and its images under the $(2\pi/3)$-rotation, can interact with a $Y_0$-vertex. Thus these are the only possible rigid flowtrees.
\end{proof}

\begin{remark}
These trees in the proof of \cref{lemma:sneaky_trees} were first introduced by the first author in his work \cite{CasalsDGAcubic}, and referred to as sneaky trees, due to their shape.\footnote{A coarse chalkboard explanation of these trees can be found in the talk ``Differential Algebra of Cubic Graphs'' at the Harvard Center of Mathematical Sciences for ``A Celebration of Symplectic Geometry 2017: 15 Years of JSG.'', around minute 53.}\hfill$\Box$
\end{remark}

\subsubsection{Proof of \cref{theorem:spectralnetworkandrigidflowtrees}} This is now a consequence of \cref{lemma:sneaky_trees} and the bijection established in \cite{Morseflowtree}. The only point left is appropriately choosing the neighborhoods of the $\dfs$-singularities, i.e.~the initial vertices of $\snetwork$, cut off the part of the rigid flow-trees that begin at such neighbourhoods, and arguing that the resulting partial flowtrees can be extended. For each $\dfs$-singularity $b\in S$, we choose a neighborhood $U_b(\delta)\sse S$ whose only initial vertex is $b$, there are no non-initial vertices of $\snetwork$, and $\dd U_b(\delta)$ transversely intersects $\snetwork$ at the three initial rays associated to $b$. The perturbation needed is now obtained by locally perturbing according to \cref{lemma:sneaky_trees} at each such $U_b(\delta)$. Since the rigid flow-trees must begin at ends, which are located at the cusp-edges, we see that we can cut-off rigid flow-trees along edges that first leave the $U_b(\delta)$. Note that the resulting partial flow tree in $U_b(\delta)$ can be extended to a global flowtree using \cref{lem:first_noescape}.
Such extension is rigid, as rigidity depends only on the vertices of the flowtree, and so it must have been the sneaky trees constructed in \cref{lemma:sneaky_trees}. Then the property in Part (1) holds by construction, and Part (2) holds by \cite[Theorem 1.1.(b)]{Morseflowtree}.\hfill$\Box$

\begin{remark}
(i) Technically, the perturbation in \cite[Theorem 1.1]{Morseflowtree} is global, and thus the spectral network $\snetwork$ is perturbed by a $C^\infty$-small smooth isotopy because of transversality. We implicitly absorb such smooth isotopy in our notation, as spectral networks are invariant under such isotopies.

(ii) We are implicitly using that the exact Betti Lagrangian $\blag$ is Maslov $0$, which will be shown in \cref{lem:Maslov0}, which tells us that the topological index of the boundary must be zero. The topological dimension of such a disk (with single positive puncture) is zero only if the limiting Reeb chord has index $1$ (and so the index of the limiting puncture is $2$). This is if and only if the Reeb chord corresponds to a positive crossing, by the argument used in \cref{subsection:trappinglemma} (see also \cref{fig:flowasymptotics}  for a heuristic explanation). 
\hfill$\Box$
\end{remark}
%    \newpage
%%%%%%%%%%%%%%%%%%%%%%%%%%%%%%%%%%%%%%%%%%%%%%%%%%%%%%%%%%%%%%%%%%%%%%%%
%%%%%%%%%%%%%%%%%%%%%%%%%%%%%%%%%%%%%%%%%%%%%%%%%%%%%%%%%%%%%%%%%%%%%%%%
%%%%%%%%%%%%%%%%%%%%%%%%%%%%%%%%%%%%%%%%%%%%%%%%%%%%%%%%%%%%%%%%%%%%%%%%
\section{Adiabatic convergence for $\dfs$-trees}\label{section_adg}
The main object of this section is to prove the direction $(\Longleftarrow)$ of the characterization in \cref{thm:characterization}. Namely, we show the contrapositive statement: points in the complement of the spectral network $\snetwork$ do {\it not} have such pseudo-holomorphic strips through them and, for \cref{thm:characterization}.(i), that for a point in the spectral network there are no such pseudo-holomorphic strips in any relative homology class except for possibly the soliton class. The precise statement is \cref{thm:snetworkadg}. The direction $(\Longrightarrow)$ is proven in \cref{section_Floer}.\\

\noindent The image of a Betti Lagrangian $L\sse T^*S$ under the $\e$-scaling action on $(T^*S,\omega_\std)$ is denoted $\e L\sse T^*S$. Throughout the manuscript, the expression {\it adiabatic limit} refers to the limit $\e\to0$ and for this section we assume that $L$ has either conical ends or is meromorphic with $O(-1)$-ends. The following notions formalize the scenarios with non-existence of pseudo-holomorphic strips in the adiabatic limit:
\begin{definition}\label{def:diskterm}
Let $(\bis,\mkpts,\lknot)$ be a Betti surface, $\blag\sse T^*\obis$ a Betti Lagrangian adapted to $(S,g)$ and $J_g$ the almost complex structure associated to $g$. Consider a spectral network $\snetwork$ compatible with $L$, as in \cref{def:spectralnetwork}. For a point $z\in\obis$ and the cotangent fiber $F_z=T^*_z\obis$:
\begin{enumerate}
\item The pair of Lagrangians $(F_z,\blag)$ is said to be {\it uniformly disk-free} in the adiabatic limit if for any given energy $E\in\R_+$, there is some neighborhood $U_z(E)\sse S$ of $z$ and a positive constant $\e_0(E,z)\in\R_+$ such that there is no non-constant $J_g$-holomorphic disk of energy below $E$ between the Lagrangians $\e \blag$ and $F_w$, for all $w\in U_z(E)$ and $\e\in\R_+$ with $\e\in (0, \e_0(E,z)]$.\\

\item If $z\in \snetwork$ belongs to a wall, the pair of Lagrangians $(F_z,\blag)$ is said to be {\it uniformly Stokes} in the adiabatic limit if for any given energy $E\in\R_+$, there is some neighborhood $U_z(E)\sse S$ of $z$ and a positive constant $\e_0(E,z)\in\R_+$ such that there is no non-constant $J_g$-holomorphic disk of energy below $E$ bounded between the Lagrangians $\e \blag$ and $F_w$ whose relative homology class is not of the parallel transport of the soliton class associated to $z$, for all $w\in U_z(E)$ and $\e\in\R_+$ with $\e\in (0,\e_0(E,z)]$.\hfill$\Box$
\end{enumerate}
\end{definition}

\noindent That is, informally, $(L,F_z)$ is uniformly disk-free if there are no pseudo-holomorphic strips bounded by $(\e L,F_w)$ for $\e\ll 1$ small enough and $w$ close to $z$, for a given energy cut-off $E\in\R_+$. Similarly for uniformly Stokes. The use of the word {\it uniformly} in \cref{def:diskterm} is to emphasize that the neighborhood $U_z(E)$ is independent of the adiabatic parameter $\e$. Since every statement in this section will be in the adiabatic limit as $\e\to0$ we often drop the cue {\it in the adiabatic limit}, e.g.~we simply refer to being uniformly disk-free or uniformly Stokes. The main result, implying the first part of \cref{thm:characterization}, reads as follows:

\begin{theorem}[Pseudo-holomorphic strips in spectral networks]\label{thm:snetworkadg}
Let $(\bis,\mkpts,\lknot)$ be a Betti surface, $\blag\sse T^*\obis$ a Betti Lagrangian adapted to $(S,g)$ and $\snetwork$ a compatible spectral network. Then
    \begin{enumerate}
        \item If $z\in \snetwork^c$, then $(F_z,\blag)$ is uniformly disk-free.
        \item If $z\in \snetwork$, then $(F_z,\blag)$ is uniformly Stokes.
    \end{enumerate}
\end{theorem}    

\noindent The purpose of this section is to prove \cref{thm:snetworkadg}. The intuition for Part (1) is summarized as follows. If $z\in\snetwork^c$ was such that $(\e L,F_{z_\e})$ bounded a pseudo-holomorphic strip in $(T^*S,J_g)$, then we want to have such sequences of strips converge, as $\e\to0$, to a $\dfs$-tree in $S$ passing through $z$. If such convergence statement held, then the results in \cref{section_mnetwork} would force the $\dfs$-tree to be contained in $\snetwork$ and thus $z\in\snetwork$, reaching a contradiction. We remark that all the results we prove in this section rely on the relevant disks having uniformly bounded diameter (c.f. \cref{lem:diametercontrol}), which follows from the geometric boundedness of $\blag$ (\cref{prop:geombounded}).  

Of course the technical part is establishing convergence results for such pseudo-holomorphic strips with boundary conditions on $(\e L,F_z)$ in a manner that the limit is indeed a $\dfs$-tree, or at least it is a $\dfs$-tree away from a fixed small neighborhood of the branch points. Note that this is a much stronger result than just stating that the strips would converge to a flowtree, as such a weaker statement would provide no control on the trajectories associated to leaves and semi-infinite edges of the tree, and thus the resulting tree would typically {\it not} be a $\dfs$-tree nor there would be any understanding of their behavior with regards to the initial rays at the branch points. In particular, a new aspect in the proof of \cref{thm:snetworkadg} is that we must have control on the behavior of the degeneration of such pseudo-holomorphic strips as it approaches the boundary of a neighborhood of the branch points, even if $z\in\snetwork^c$ is typically far from any branch points. The precise convergence result that we establish reads as follows:

\begin{proposition}[$\e$-strips converge to trimmed $\dfs$-trees]\label{prop:convergence_strips_D4tree}
Let $(\bis,\mkpts,\lknot)$ be a Betti surface and $\blag\sse T^*\obis$ a Betti Lagrangian adapted to $(S,g)$ with a compatible spectral network $\snetwork$. Suppose that $(u_\e,z_\e,\Delta_m(\e),E)$ is an $\e$-strip sequence $u_\e:\Delta_2\lr S$, with boundary of $(\e L,T^*_{z_\e}S)$. Then, after possibly considering a subsequence and a reparametrization of the domain $\Delta_2$, its limit
$$\displaystyle\lim_{\e\to0} (u_\e,z_\e,\Delta_2,E):\Gamma\lr S$$
exists and it is a (broken) trimmed $\dfs$-tree. The tree is unbroken if $z$ lies in the interior of $\snetwork(E)$.
\end{proposition}
%{\YJN This proposition needs the hypothesis that $\blag$ admits a spectral network}
\noindent The concepts needed to state \cref{prop:convergence_strips_D4tree}, including the notion of $\e$-strip sequences, their limits and broken trimmed $\dfs$-trees, are introduced in \cref{subsection:adg}  and \cref{prop:convergence_strips_D4tree} is then proven in \cref{ssec:convergence_trimmed_trees}. \cref{prop:convergence_strips_D4tree} is used to conclude \cref{thm:characterization} in \cref{ssec:convergence_trimmed_trees2}.\footnote{Note that it is nevertheless possible to upgrade \cref{prop:convergence_strips_D4tree} to show convergence to an actual (broken) $\dfs$-tree. This would add even more technicality and pages to the manuscript and, for the sake of balance, we have not pursued this technical enhancement.}

%%%%%%%%%%%%%%%%%%%%%%%%%%%%%%%%%%%%%%%%%%%%%%%%%%%%%%%%%%%%%%%%%%%%%%%%
%%%%%%%%%%%%%%%%%%%%%%%%%%%%%%%%%%%%%%%%%%%%%%%%%%%%%%%%%%%%%%%%%%%%%%%%
%%%%%%%%%%%%%%%%%%%%%%%%%%%%%%%%%%%%%%%%%%%%%%%%%%%%%%%%%%%%%%%%%%%%%%%%

\subsection{$\e$-strips and their limits}\label{subsection:adg} Let us first set up the geometric ingredients to discuss $\e$-strips and their limits. We endow its cotangent bundle $(T^*S,\omega)$ of $(S,g)$ with the compatible almost-complex structure $J_g$ associated to $\omega$ and $g$ via the Levi-Civita (Ehresmann) connection of the latter; $J_g$ is also known as the Sasaki almost-complex structure.\\

\noindent We use the following type of domains for pseudo-holomorphic strips: fix $\slit\in(0,1)$, $\slit\ll1$, and consider a point $c=(c_1,...,c_{d})\in \mathbb{R}^{d}$. By definition, $\triangle_{d+1}(c_1,\ldots,c_d)$ is the  subdomain of $(-\infty,\infty)\times[0,d+1]\sse\C$ given by removing $d$ horizontal slits of vertical width $\slit$, each starting at the point $(c_j,j)$ in the direction of $+\infty$, for $j\in[1,d]$. Each such horizontal slit is thus centered around the semi-infinite ray $[c_j,\infty)\times\{j\}$, see e.g.~ Figure \ref{fig:domains}. If the input position $c\in\R^d$ depends on a parameter $\varepsilon\in\R^+$ but $c$ is implicitly understood (or arbitrary and not particularly relevant), then we denote such subdomain by $\triangle_{d+1}(\e)$. Such domains $\triangle_{d+1}(\e)$ are given the conformal structure inherited as subdomains of $\C$. \\

\begin{definition}\label{def:eps_strip_sequence}
Let $(\bis,\mkpts,{\bf \lknot})$ be a Betti surface and $\blag\sse T^*\obis$ a Betti Lagrangian adapted to $(S,g)$. By definition, an $\e$-strip sequence $(u_\e,z_\e,\triangle_m(\e),E)$ is a collection of $J_g$-holomorphic maps
$$u_\e:\triangle_m(\e)\lr T^*\obis,$$
defined for all $\e\in\R_+$ small enough and such that:
\begin{enumerate}
    \item the image of $u_\e$ is bounded by $L$ and $F_{z_\e}$,
    \item $\mbox{area}(u_\e)\leq\e E$.\hfill$\Box$
\end{enumerate}
\end{definition}

\noindent An additional piece of notation on such domains $\triangle_{d+1}(c)$: given a boundary component of $\dd\triangle_{d+1}(c)$ with both of its ends at $+\infty$, its \textit{boundary minimum} of $I$ the unique point with the smallest real part. Also, a subdomain in $\triangle_{d+1}$ is said to be \textit{horizontal} if it is of the form $[x_1,x_2]\times [y_1,y_2]$ with $[x_1,x_2]\times \{y_1\}$ and $[x_1,x_2]\times \{y_2\}$ lying on the boundary of $\triangle_{d+1}$. See Figure \ref{fig:domains}.

%We call $\triangle_{d+1}$ a \textit{conformal representation}.
\begin{center}
	\begin{figure}[h!]
		\centering
		\includegraphics[scale=1.8]{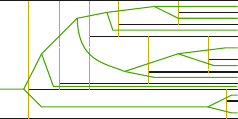}
		\caption{A depiction of the type of domain $\triangle_{d+1}(c)$ being used to describe the conformal structure. The domain is drawn in the strip with the slits highlighted in black, and the conformal structure in inherited from $\C$. The $x$-coordinate of the boundary minima is highlighted with a yellow vertical dashed line and the tree associated to the domain is depicted in green.}
        \label{fig:domains}
	\end{figure}
\end{center}

\color{black}

%%%%%%%%%%%%%%%%%%%%%%%%%%%%%%%%%%%%%%%%%%%%%%%%%%%%%%%%%%%%%%%%%%%%%%%%
%%%%%%%%%%%%%%%%%%%%%%%%%%%%%%%%%%%%%%%%%%%%%%%%%%%%%%%%%%%%%%%%%%%%%%%%
%%%%%%%%%%%%%%%%%%%%%%%%%%%%%%%%%%%%%%%%%%%%%%%%%%%%%%%%%%%%%%%%%%%%%%%%

\subsubsection{Combinatorial trees associated to domains $\triangle_{d+1}$}\label{subsubsection:conformalstructures}
There is a (combinatorial) tree associated to each of the domains $\triangle_{d+1}(c_1,\ldots,c_d)$, constructed as follows.

\begin{definition}[Trees associated to $\triangle_{d+1}$]\label{def:conformaltree}
Let $\triangle_{d+1}(c_1,\ldots,c_d)\sse\C$ be a domain, its associated tree $T:=T(\triangle(c))$ is the rooted $d$-leaved tree whose interior vertices have valence at least three and satisfies:
\begin{enumerate}
    \item The vertices of $T$ are in bijection with the values of $c\in\R^{d+1}$. These can be visually described by vertical rays passing through the boundary minima of $\triangle_{d+1}(c_1,\ldots,c_d)$.
\item The edges of $T$ are in (natural) bijection with the connected components of the complement of such vertical rays in $\triangle_{d+1}(c_1,\ldots,c_d)$.
%that pass through the boundary minima
\item By definition, two edges of $T$ share a vertex if and only if the corresponding components of those two edges are adjacent along the vertical ray corresponding to that vertex.\hfill$\Box$
\end{enumerate}
\end{definition}
%{\RC ADD FIGURE}\\
\noindent See Figure \ref{fig:domains} for an example of \cref{def:conformaltree}. By definition, an open subdomain $D\sse\triangle_{d+1}$ is a \textit{subdisk} if the part $\partial D\setminus(\partial \triangle_{d+1})$ of its boundary consist of vertical rays avoiding the boundary minima.\footnote{These subdisks are often disconnected, each connected component being topologically a disk.} As in \cref{def:conformaltree}, subdisks $D\sse\triangle_{d+1}$ yield rooted $(d+1)$-leaved subtrees $T_{D}\subset T$ by introducing vertices at each of the edges whose corresponding components contain the vertical boundaries of $\partial D$. We denote by $T/T_D$ the rooted $(d+1)$-leaved tree obtained by contracting the subtree $T_D\subset T$ associated to a subdisk $D\sse\triangle$.\\

\noindent We shall cover these domains $\triangle$ by subdisk domains. We always implicitly assume that the vertical boundary segments of the components of such covers avoid the boundary minima, are all disjoint, and their intersections are also given by subdisk domains. In the same manner that a subdisk gave a subtree, an open cover $\triangle_m=D_0\cup D_1$ by two subdisk domains $D_0,D_1$ gives an open cover of the associated tree: $T(\triangle)=T(D_0)\cup T(D_1)$. Note that $D_0,D_1$ are typically each disconnected.

%%%%%%%%%%%%%%%%%%%%%%%%%%%%%%%%%%%%%%%%%%%%%%%%%%%%%%%%%%%%%%%%%%%%%%%%
%%%%%%%%%%%%%%%%%%%%%%%%%%%%%%%%%%%%%%%%%%%%%%%%%%%%%%%%%%%%%%%%%%%%%%%%
%%%%%%%%%%%%%%%%%%%%%%%%%%%%%%%%%%%%%%%%%%%%%%%%%%%%%%%%%%%%%%%%%%%%%%%%

\subsubsection{Limits of $\e$-strip sequences}\label{subsubsection:weakadg}
The goal of this subsection is to rigorously explain what it means for an $\e$-strip sequence, as in \cref{def:eps_strip_sequence}, to converge to a trimmed $\dfs$-tree. In particular, we refine the type of domain subdivision introduced in \cite[Section 5.2]{Morseflowtree}, introduce adapted covers in \cref{def:reparametrizationsequence} and limits of $\e$-strip sequences in \cref{def:flowlineconvergence}. \\

For each branch point $b\in\obis$, we have fixed a neighborhood $U_b\sse S$ as in \cref{subsection:preliminarytransversalitycondition}. We can and do choose such $U_b$ to be disks of a certain radius $r$, denoting that by $U_b=U_b(r)$. Specifically, we fix two truncation parameters $\delta,\eta\in\R_+$ such that the disks $U_b(2\delta+\eta\delta)$ are disjoint and over each subdisk $U_b(2\delta)$, the piece of the Betti Lagrangian $\blag\cap T^*U_b(2\delta)\sse T^*S$ is holomorphic and the metric is the standard flat metric restricted from $\mathbb{C}$. We denote the union of such neighborhoods by
$$B_{\eta,\trunparam}:=\bigcup_{b\in \caustL} U_{b}(2\delta+\eta\delta).$$
By definition, a broken trimmed $\dfs$-tree is the intersection of a disjoint union of $\dfs$-trees in $S$ with the domain $S\setminus B_{\eta,\trunparam}$ such that the flowlines entering each disk $U_b(2\delta+\eta\delta)$, if such flowlines exist, coincide with initial rays at $b\in S$. The tree is unbroken if there is a single $\dfs$-tree component.\\

%{\YJN Definition of trimmed tree is problematic. }
%{\YJN slight fix}\\
From now on, whenever we discuss the metric and the covariant derivative, we will assume that we are using the Sasaki metric, and the associated Levi-Civita connection. Given an $\e$-strip sequence $(u_{\e})_\e$, as in \cref{def:eps_strip_sequence}, we want to understand its behavior in the adiabatic limit, i.e.~as $\e\to0$. As observed in \cite[Lemma 5.6]{Morseflowtree}, $\abs{\nabla u_{\e}}=O(\e)$ when restricted to a half-disk subdomain of fixed radius, that maps outside some fixed neighborhood of the caustics. However, near the caustics, we will not be guaranteed such a control. The idea is to subdivide the domains of each $u_{\e}$ into open sets where we have some control of either $\abs{\nabla u_{\e}}$ or the image of $u_\e$. As per usual application of gradient estimates, convergence for points where $\abs{\nabla u_{\e}}$ remains uniformly $O(\epsilon)$-bounded follows (c.f. \cite[Theorem 2]{FloerWittenInfinite}, \cite[Proposition 9.7, Proposition 9.8]{FukayaOh97_MorseHomotopy}). The challenge is to find such subdivisions and control the behavior of the image when $\abs{\nabla u_{\e}}$ blows up.

%The main cause of divergence is the unboundedness of the derivatives $\abs{\nabla u_{\e}}$. Here the covariant derivative and the norms are taken by using the Levi-Civita connection and norm associated to $g$.{\RC Tighten this.} \\

\noindent The initial $J_g$-holomorphic strip, before performing any $\e$-scaling, has domain $\triangle_2$ and maps to $(T^*S,J_g)$, with two punctures, as the boundary lies in $L$ and a cotangent fiber $F_z$, for some $z\in\obis$. In order to obtain the subdivision, we will reparametrize that domain $\triangle_2$, as we $\e$-scale, by a sequence of $\e$-depending domains $\triangle_m(\e):=\triangle_m(c(\e))$, $m\geq3$, where the parameters $c(\e)$ depend on the adiabatic parameter $\e$ but the number $m\in\N$ does not. A first useful open cover for these domains is described as follows:

\begin{definition}[Adapted covers for $\e$-strip sequences]\label{def:reparametrizationsequence}
Let $(u_\e,z_\e,\triangle_m(\e),E)$ be an $\e$-strip sequence with $z_{\e}$ converging to a point in $\obis\setminus B({2\trunparam+\eta\trunparam})$ as $\e\to0$. By definition, an $\e$-sequence of open covers $\stdm(\e)=D_0(\e)\cup D_1(\e)$ is said to be adapted to $(u_\e,z_\e,\triangle_m(\e),E)$ if the following holds for $\e\in\R_+$ small enough:
\begin{enumerate}
    \item Each $D_i(\e)$ are subdisks, $i=1,2$.
    \item The topology associated to the domains $D_i(\e)$ is independent of $\e$, $i=1,2$. That is, the order of the parameters $c(\e)$, the topology of the components of $D_i(\e)$, and the topology of their intersections (eventually) remains constant as $\e\to0$.
    \item $u_\e(D_0(\e))$ lies outside $B(\trunparam)$ and $\abs{\nabla u_{\e}}=O(\e)$ for points in $D_0(\e)$.
    \item $u_\e(D_1(\e))$ lies inside $B(2\trunparam)$.
\end{enumerate}
In addition, a subdisk $W_0(\e)\sse D_0(\e)$ is said to give an adapted horizontal decomposition of $D_0(\e)$ if
\begin{itemize}
    \item[(i)] $W_0$ has width $O(\log(\e^{-1}))$ and contains all the boundary minima contained in $D_0$. 
    \item[(ii)] The vertical boundary part of $\dd W_0$ is disjoint from the vertical boundaries of $D_i(\e)$, $i=1,2$, and its topology remains constant as $\e\to0$.
\end{itemize}
A triple $(D_0(\e),D_1(\e),W_0(\e))$ is said to be adapted to $(u_\e,z_\e,\triangle_m(\e),E)$ if $\Delta_m(\e)=D_0(\e)\cup D_1(\e)$ is an adapted cover and $W_0(\e)$ is an adapted horizontal decomposition of $D_0(\e)$.\hfill$\Box$
\end{definition}

\noindent The geometric conditions on $u_\e$ are \cref{def:reparametrizationsequence}.(3),(4) and (i): on $D_0(\e)$ the norm of the derivative $\abs{\nabla u_{\e}}=O(\e)$ remains controlled (constant with respect to the scaling measure), and the width condition $O(\log(\e^{-1}))$ on its subdomains $W_0(\e)$ ensures that the images $u_\e(W_0)$ converge to a point. In particular, the associated trees $T(\e)/T_{W_0(\e)}$ give models for the domain of a map that contracts the connected subdisks in $W_0$ to points.

\begin{remark}\label{rmk:tree_welldefined_adaptedcover}
The conditions \cref{def:reparametrizationsequence}.(2) and (ii) imply that there is a fixed number of components for $D_0(\e),D_1(\e)$ and $W_0(\e)$ in an adapted cover, independent of $\e$. In fact, the combinatorial trees associated to each such domains, as in \cref{subsubsection:conformalstructures}, are well-defined and independent of $\e$. In particular, the combinatorial type of the tree $\Gamma$ associated to $\triangle_m(\e)$ is independent of $\e$ and so are the corresponding induced subtrees.\hfill$\Box$
\end{remark}

The adapted covers of \cref{def:reparametrizationsequence} are still not sufficient to conclude \cref{prop:convergence_strips_D4tree}, i.e.~that the $\e$-strip sequence converges to a trimmed $\dfs$-tree. The remaining additional condition that we must require (and we will prove can be achieved) reads as follows:

\begin{definition}[Limits of $\e$-strip sequences]\label{def:flowlineconvergence}
The limit of an $\e$-strip sequence $(u_\e,z_\e,\triangle_m(\e),E)$ is said to exist if there exists an adapted cover $(D_0(\e),D_1(\e),W_0(\e))$ such that the following two types of conditions hold:
\begin{enumerate}
    \item There are three types of components of $D_0\setminus W_0$: fiber, ghost or flowline. They are classified depending on their possible images under $u_\e$, as follows:\\
            \begin{enumerate}[label=(\roman*)]
                \item (Fiber components) Some horizontal boundary components map to the cotangent fiber, with projections $C^{\infty}$-converging to the point $\lim z_{\e}$ as $\e\to0$.

                \item (Ghost components) Horizontal boundaries map to the same sheets of $\blag$, with projections $C^{\infty}$-converging to points.
                
                \item (Flowline components) Horizontal boundaries map to different sheets of $\blag$ if the domain has widths of size $O(\e^{-1})$ and the (projection of the) rescaling $u\circ(\e^{-1})$ $C^{\infty}$-converges to a flowline. The flowlines are said to have zero length if the resulting limiting flowline is a point. \\
%{\YJN We didn't specify what it means for a flowline to have non-zero length}{\RC Right, we'll add that.}
            \end{enumerate}
    \item Let $T$ be the tree associated to $\triangle_m(\e)$ and consider the (reduced) tree $\treegraph^{red}$ obtained by removing the subtrees of $T/T_{W_0(\e)}$ corresponding to fiber components of $D_0\setminus W_0$, contracting the subtrees corresponding to the ghost components of $D_0\setminus W_0$, removing the subtrees corresponding to components in $D_1$, and removing any 2-valent vertices that connect the two flow-edges of the same index. Then, we require that $\treegraph^{red}$ admits a set of edges such that:\\

            \begin{itemize}
                \item[(i)] (Edges map to initial $\dfs$-rays) Each edge is such that there exists a vertical cut of the associated flowline component subdomain of $\triangle_m(\e)$ with $u_\e$ mapping its right adjacent component inside $T^*B_{\eta,\trunparam}$, and the edge converges to one of the initial rays associated to a branch point, with the \textit{inward} orientation.\footnote{Note that the spectral network $\snetwork$ is oriented opposite to the flowline direction.}% {\YJN The orientation of the flow-edge is an important point}\\

                \item[(ii)] (Maximality) For any flowline component $\Theta\sse D_0\setminus W_0$ with non-zero flow, there exists an edge as in ($i$) that belongs to the right adjacent component of $\Theta$.\hfill$\Box$
            \end{itemize}
\end{enumerate}
\end{definition}

In essence, proving \cref{prop:convergence_strips_D4tree} is tantamount to showing that, given the $\e$-strip sequence, there exists an adapted cover satisfying the conditions in \cref{def:flowlineconvergence}. The limiting object for the $\e$-strip sequence is technically not a flow-tree $F:\treegraph\lr S$ due to the presence of the ghost components and flowlines components with zero lengths in $\triangle_m(\e)$, cf.~\cref{def:flowlineconvergence}.($1.ii$) and ($1.iii$). For instance, an entire ghost component might map to a point in the adiabatic limit; these are said to be ghost edges of $\treegraph$. This matter is resolved combinatorially by contracting the edges of the tree $\treegraph$ associated to $\triangle_m(\e)$ that correspond to these two types of components. The resulting tree is often referred to as the {\it reduced} tree, and one can effectively work with the reduced tree instead of the original $\treegraph$; this is what is happening in \cref{def:flowlineconvergence}.(2).\\

If the limit of an $\e$-strip sequence exists, as in \cref{def:flowlineconvergence}, it is denoted by
$$\displaystyle\lim_{\e\to0} (u_\e,z_\e,\Delta_m(\e),E):\Gamma\lr S.$$
It is implicitly understood as a flow-tree with ghost components, so the map factors through its reduction $\treegraph^{red}$, as explicitly described in \cref{def:flowlineconvergence}.(2). In particular, the flowline edges of this limit are given by the adiabatic limit of the rescaling $u\circ(\e^{-1})$ along each flowline component of non-zero length.

\begin{remark}
The limit of an $\e$-strip sequence depends on the truncation parameters $(\delta,\eta)$ fixed initially and the corresponding trees $\treegraph=\treegraph_{\delta,\eta}$ might not be nested if one decreases the parameters $\eta,\delta\to0$. That said, the associated reduced trees $\treegraph^{red}$ can (and we will show will) be nested, in that $\treegraph^{red}_{(\delta,\eta)}\subset \treegraph^{red}_{(\delta',\eta')}$ for smaller truncation parameters $(\delta',\eta')$. This is another technical reason for reduction.\hfill$\Box$
\end{remark}

%%%%%%%%%%%%%%%%%%%%%%%%%%%%%%%%%%%%%%%%%%%%%%%%%%%%%%%%%%%%%%%%%%%%%%%%
%%%%%%%%%%%%%%%%%%%%%%%%%%%%%%%%%%%%%%%%%%%%%%%%%%%%%%%%%%%%%%%%%%%%%%%%
%%%%%%%%%%%%%%%%%%%%%%%%%%%%%%%%%%%%%%%%%%%%%%%%%%%%%%%%%%%%%%%%%%%%%%%%

\subsection{Proof of \cref{prop:convergence_strips_D4tree}}\label{ssec:convergence_trimmed_trees} In this subsection we prove the convergence of $\e$-strip sequences to trimmed $\dfs$-trees. In the initial setup, there are a number of technical steps which are mirroring \cite[Section 5]{Morseflowtree} and \cite[Section 5]{nho2024familyfloertheorynonabelianization}; we thus simply refer the necessary parts of those manuscripts and continue focusing on the new part of the argument. Specifically, we must build an adapted cover as in \cref{def:flowlineconvergence}. 

First, in order to build such cover satisfying the component constraints from \cref{def:flowlineconvergence}.(1), we proceed as follows. By Lemmata 5.7, 5.13, 5.16 and 5.17 in \cite{Morseflowtree}, the domain $\triangle_m(\e)$ itself admits an adapted cover $(D_0(\e),D_1(\e),W_0(\e))$. Here the horizontal decomposition subdisk $W_0(\e)$ is obtained by considering subdisks which are neighborhoods of the boundary minima of the required width $O(\log(\e^{-1}))$.

\begin{remark}
(i) As far as we know, the construction of the domain-adapted cover in \cite[Section 5]{Morseflowtree} has a gap in the proof of \cite[Lemma 5.4]{Morseflowtree}, in the claim that such balls of radius $\delta\e$ exist (his $\lambda$ is our $\e$). It must be justified that there is a uniform upper bound on the boundary length of the images of $u_\e$ outside of the union of the domains $T^*U_b(2\delta)$. This upper bound nevertheless exists, as it follows from the truncated reverse isoperimetric inequality that we prove in the appendix, cf.~\cref{thm:truncatedreverseiso}.\\
(ii) Lemmatas 5.6 and 5.7 in \cite{Morseflowtree} are stated only for exact Lagrangians. However, the only place the proof uses exactnesss is to establish the $O(\e)$-estimate on energy. Since this condition is part of the definition of an $\e$-strip sequence, the proof goes through verbatim.
\hfill$\Box$
\end{remark}

The components of $D_0(\e)\setminus W_0(\e)$ consist of horizontal subdomains of $\stdm$ of the three types required in \cref{def:flowlineconvergence}.(1). By \cite[Lemma 5.15]{nho2024familyfloertheorynonabelianization}, the fiber components $C^{\infty}$-converge to the point $\lim z_{\e}$. By \cite[Lemma 5.17]{nho2024familyfloertheorynonabelianization}, the ghost components also $C^{\infty}$-converge to a point. By \cite[Lemma 5.18]{Morseflowtree} (see also \cite[Lemma 5.17]{nho2024familyfloertheorynonabelianization}), the flowline components $C^{\infty}$-converges to flowlines, after taking the $\e$-dependent reparametrization $u_{\e}(\e^{-1}s,\e^{-1}t)$. Therefore, the new contribution at this point is to construct the adapted cover such that \cref{def:flowlineconvergence}.(2) holds: controlling the behavior of the limit flowtree near the branch points.\\

%%%%%%%%%%%%%%%%%%%%%%%%%%%%%%%%%%%%%%%%%%%%%%%%%%%%%%%%%%%%%%%%%%%%%%%%
%%%%%%%%%%%%%%%%%%%%%%%%%%%%%%%%%%%%%%%%%%%%%%%%%%%%%%%%%%%%%%%%%%%%%%%%
%%%%%%%%%%%%%%%%%%%%%%%%%%%%%%%%%%%%%%%%%%%%%%%%%%%%%%%%%%%%%%%%%%%%%%%%

\noindent Given the $\e$-strip sequence $(u_\e,z_\e,\triangle_m(\e),E)$, denote by $\treegraph$  the tree associated to the domain $\triangle_m(\e)$, which is well-defined by the adapted cover constructed above, cf.~\cref{rmk:tree_welldefined_adaptedcover}. Note that, as in \cite[Lemma 5.18]{nho2024familyfloertheorynonabelianization}, the $\e$-finiteness of energy implies that there cannot be flowline edges of infinite length. Thus, given the bound on the covariant derivative, any leaf in (the subtree associated to) the $D_0(\e)$-region of $\treegraph$ corresponds to either a ghost component or a fiber component. In particular, there can be no leaves that are flowline edges.Furthermore, the following lemma tells us that there exists at least one external $D_1$-edge of $\treegraph$:
\begin{lemma}\label{lem:externaledge}
There exists at least one external edge of $\treegraph$ that corresponds to a horizontal component in $D_1(\e)$.
%of $\treegraph$ correspond to the components that belong to $D_0(\e)$.
\end{lemma}
\begin{proof}
Suppose by contradiction that there is no such an external edge. By construction, this is only if the image of $u_{\e}$ lies outside some fixed neighbourhood of the branch points, for small enough $\e$. It follows that for the given subsequence, the puncture removal step was unnecessary and $\abs{\nabla u_{\e}}=O(\e)$. However, there is then a single horizontal strip component which necessarily must be a fiber component. Therefore, the image of $u_{\e}$ converges to the point $\lim z_{\e}$ and that implies that the geometric energy must have been zero, which a contradiction.
\end{proof}

Let us cut the domain $\stdm=\stdm(\e)$ along the vertical rays through the boundary minima of $\stdm$ that connect to a horizontal boundary component mapping to the fiber. We denote by $\bar{\triangle}_{m_1},...,\bar{\triangle}_{m_n}$ the resulting right-adjacent components and by $\dcgraph_{m_i}$ the induced tree in $\treegraph$ associated to each such components. By the above, the $D_0(\e)$-leaves of $\treegraph_{m_i}$ are all ghost components.\\

Consider the set of vertices contained in $\treegraph_{m_i}$ that are terminal points of a flowline edge of non-zero length; such a vertex is not $1$-valent. These vertices are ordered by their $x$-coordinate, and a \textit{maximal vertex} (for a given $m_i$) refers to such a vertex with highest $x$-coordinate within $\treegraph_{m_i}$. In particular, given two maximal vertices $v,w$, their right-adjacent components are necessarily disjoint. Let $v$ be a maximal vertex, and let $\Theta_v=\Theta_v(\e)$ be the corresponding horizontal domain in $D_0(\e)\setminus W_0(\e)$. Then the component $\Theta_v$ associated to a maximal vertex satisfies the following:

\begin{lemma}\label{lem:nogetout}
Let $v\in\treegraph$ be a maximal vertex as above. Then there exists a branch point $b(v)\in S$ such that the right-adjacent components of $\Theta_v(\e)$ are mapped into the neighborhood $T^{\ast}U_{b(v)}((2+\eta)\delta)$. 
 \end{lemma}
 \begin{proof}
Since $\stdm=D_0(\e)\cup D_1(\e)$ is adapted, the components in $D_1(\e)$ map into the region $T^{\ast}B_{\eta,\trunparam}\sse T^*S$. By maximality, it also follows that the horizontal domains of $D_0(\e)\setminus W_0(\e)$ right-adjacent to $\Theta_v$ are either ghosts, fiber components or flowline components of zero length. Since there are finitely many such components, their images $C^{\infty}$-converge to points, and by construction the subdisk  $W_0$ also converges to points. Since the neighborhoods $U_b(2\delta+\eta\delta)$ are all disjoint and the right-adjacent component to $\Theta_v$ is connected, it must map into one connected component $T^{\ast}U_b(2\delta+\eta\delta)$ of $T^{\ast}B_{\eta,\trunparam}$, which determines the branch point $b(v)=b\in\obis$.
\end{proof}

Now we arrive at the crux of the argument: we must show that $\Theta_v(\e)$, possibly after refining with vertical cuts, is such that its right-adjacent component not only maps to $T^{\ast}U_{b(v)}(2\delta+\eta\delta)$ but does so converging near $T^{\ast}\dd (U_{b(v)}(2\delta+\eta\delta))$ to an initial flowline ray of the $\dfs$-branch point $b(v)\in\obis$ in the adiabatic limit. This is the content of the following:

\begin{proposition}[Behavior near $\dfs$-points]\label{lem:initialedgelemma}
Let $(u_\e,z_\e,\triangle_m(\e),E)$ be an $\e$-strip sequence, $v\in\treegraph$ a maximal vertex and $b(v)\in\obis$ its associated branch point. Then, after possibly taking a subsequence, there exists a sequence of vertical cuts $v_{\e}$ for the domain $\Theta_v(\e)$ such that the restriction $u_{\e}\vert_{\Theta^{\rightarrow}_v}$ of $u_\e$ converges to one of the initial flowline rays of $b(v)$, where $\Theta^{\rightarrow}_v$ denotes the right adjacent component of $\Theta_v(\e)$ after the cuts. 
\end{proposition}
 \begin{proof}
By construction, the open set $U_{b(v)}({2\delta+\eta\delta})$ is such that the metric $g$ is flat, the associated Sasaki metric coincides with the standard flat metric on $\mathbb{C}^2\cong T^*\C$, and the Betti Lagrangian $L\sse T^*S$ above this open set splits into the holomorphic cusp component and the smooth sheet components. Consider the limiting flowline $\gamma_v$ associated to the component $\Theta_v$. Since it is not a ghost edge, fiber component or zero-length flowline, it must (eventually) enter $B_{\eta,\trunparam}$, specifically the connected component containing $b(v)\in\obis$. Choose $\eta'\in(\eta/2,\eta)$ such that $\gamma_v$ is transverse to the boundary $\partial B_{\eta',\delta}$ and let $v(z)$ be the final intersection point at which $\gamma_v$ is oriented inward. After passing to a subsequence, there is a sequence of vertical rays $v_{\e}$ for $\Theta_v$ converging to the point $v(z)$; fix such a sequence of rays. We claim that the energy $\int \la_{\C}^{\ast}(\e^{-1}(u_{\e}\vert_{v_{\e}}))$ of the $\e$-scaled $u_\e$ converges to zero along these vertical cuts.\\

\noindent Indeed, write the restriction $u_{\e}\vert_{v_{\e}}=(z_{\e},w_{\e})$ in coordinates and apply the gradient estimate $\abs{\nabla z_{\e}}=O(\e)$ to obtain
\begin{align}\label{eq:verticalrayintegral}
\abs{\int_{v_{\e}}\la_{\C}^{\ast}({\e}^{-1}u_{\e})\vert_{v_{\e}}}\leq m\cdot \sup\abs{\e^{-1}w_{\e}}\cdot \abs{\nabla z_{\e}}.
\end{align}
Here the quantity $m$ appears in the upper bound because the height of $\Theta_v$ is at most $m$. Since $\abs{\e^{-1}w_{\e}}$ is uniformly bounded and $\abs{\nabla z_{\e}}$ is of order $O(\e)$, the integral on the left converges to zero in the adiabatic limit, as claimed.\\

Let $v_{\e}(\text{top})$ and $v_{\e}(\text{bottom})$ denote the top and the bottom end-points of each vertical cut $v_{\e}$ in $\Theta_v$, $\triangle_{v_{\e}}$ the right-adjacent components to $v_{\e}$ and $\partial^h(\triangle_{v_\e})$ the horizontal boundary of $\triangle_{v_\e}$. Since the restriction $u_{\e}\vert_{\triangle_{v_{\e}}}$ is $J_{g}$-holomorphic, with $g=g_{\std}$, the pointwise symplectic area $\omega_{\C}^{\ast}u_{\e}$ over ${\triangle_{v_{\e}}}$ is real and positive, except possibly at finitely many points. Therefore, the symplectic area
\begin{equation}\label{eq:sympl_area_positive}
\int_{\triangle_{v_{\e}}} \omega_{\C}^{\ast}u_{\e}\in\R_+
\end{equation}
is real and positive, which will be used momentarily in \cref{eq:energycomputation}. Independently, the restriction $u\vert_{\partial^h(\triangle_v)}$ extends to a connected curve in $L$ and, since $\Theta_v$ is a flowline component, $\partial^h(\Theta_v)$ maps to distinct sheets of the Betti Lagrangian $\blag$. Given that the smooth sheets of $\blag$ are disjoint over $B_{\eta,\delta}$, \cref{lem:nogetout} implies that $u\vert_{\partial^h(\triangle_v)}$ must map to the holomorphic cusp component. Thus the flowline $\gamma_v$ is a cusp-cusp flowline.

%The Betti Lagrangian $\blag$ splits into the holomorphic cusp component and the smooth sheets. The smooth sheets are, however, all disjoint. Therefore, it follows  

Let us finally deduce that $\gamma_v$ coincides with an initial trivalent ray at $b$, as follows. By applying Stokes' theorem, the $\e$-scaled energy reads
\footnotesize
\begin{align}\label{eq:energycomputation}
\e^{-1}\int_{\triangle_{v_{\e}}} \omega_{\C}^{\ast}u_{\e}=\e^{-1}\int_{\partial^h \triangle_{\e}}\la_{\C}^{\ast}\partial u_{\e}-\e^{-1}\int_{v_{\e}}\la_{\C}^{\ast}u_{\e}\vert_{v_{\e}}=\pm \Big((z(v_{\e}(\text{top}))^{3/2}-(z(v_{\e}(\text{bottom}))^{3/2}\Big)-\e^{-1}\int_{v_{\e}}\la_{\C}^{\ast}u\vert_{v_{\e}}.
\end{align}
\normalsize
By \cref{eq:verticalrayintegral}, the second term on the right hand side of \cref{eq:energycomputation} converges to zero in the adiabatic limit $\e\to 0$. Since $(z(v_{\e}(\text{top}))$ and $(z(v_{\e}(\text{bottom}))$ converge to $z(v)$, the right hand side of \cref{eq:energycomputation} converges to $\pm z(v)^{3/2}$. By \cref{eq:sympl_area_positive}, the left hand side of \cref{eq:energycomputation} is real and positive and thus $\pm z(v)^{3/2}$ is also real and positive. This can only occur if the (image of the) point $v$ lied on one of the initial rays for its associated branch point $b(v)$. Since the cusp-cusp flowlines form a singular foliation, the leaves of which lie outside $b(v)$, except for the three trivalent rays, we see that the restriction $u_{\e}\vert \Theta_v$ must converge to one of the initial rays coming out of $b(v)$ in the adiabatic limit. 
 \end{proof}

\noindent Let us apply \cref{lem:initialedgelemma} to each maximal vertex $v^i$ associated to the subdomains $\triangle_{m_i}$ for the initial adapted cover $(D_0(\e),D_1(\e),W_0(\e)$ at the beginning of the argument. \cref{lem:initialedgelemma} implies that these adapted covers can be refined to satisfy the condition in \cref{def:flowlineconvergence}.(2.i). The maximality condition \cref{def:flowlineconvergence}.(2.ii) follows from our choice of domains $\triangle_{m_i}$ and maximal vertices. In consequence, there exists an adapted cover for the given $\e$-sequence such that its limit exists. 

Let $\{\treegraph\}$ be the set of all $\dfs$-trees in $\snetwork$ with the univalent edges given by the initial edges of $(u_{\e},z_{\e},\triangle(\e),E)$, with energy less than $E$. The energy and the endpoint constraint imply that the set of such trees must be finite. By construction, each component $\lim_{\e\to 0}(u_{\e},z_{\e},\triangle(\e),E)$ must be a trimmed $\dfs$-tree in $\{\treegraph\}$, since the initial edges in \cref{lem:initialedgelemma} are directed inward. Furthermore, the internal vertices of such trees can only appear at the joints of $\snetwork(E)$. It follows that the $\dfs$-tree is contained entirely in $\snetwork(E)$. Furthermore, in the case $z$ lies in the interior of $\snetwork(E)$, the tree cannot be broken because their soliton classes cannot concatenate. Thus we have established \cref{prop:convergence_strips_D4tree}. \hfill$\Box$
%{\YJN We need to emphasize that the flow-trees that are born from the initial edges cannot leave the neighbourhood of $\snetwork$ and so the sequence degenerates to a trimmed $\dfs$-tree. Did a small fix.}
%%%%%%%%%%%%%%%%%%%%%%%%%%%%%%%%%%%%%%%%%%%%%%%%%%%%%%%%%%%%%%%%%%%%%%%%
%%%%%%%%%%%%%%%%%%%%%%%%%%%%%%%%%%%%%%%%%%%%%%%%%%%%%%%%%%%%%%%%%%%%%%%%
%%%%%%%%%%%%%%%%%%%%%%%%%%%%%%%%%%%%%%%%%%%%%%%%%%%%%%%%%%%%%%%%%%%%%%%%

\subsection{Proof of \cref{thm:snetworkadg}}\label{ssec:convergence_trimmed_trees2}
For Part (1), suppose that the point $z\in \snetwork^c$ is not uniformly disk-free. By definition, there exists an $\e$-strip sequence $(u_\e,z_\e,\triangle_m(\e),E)$ with $z=\lim z_\e$. Consider the spectral subnetwork $\snetwork(E)\sse\snetwork$ with energy at most $E$ and fix the truncation parameters $\delta,\eta\in\R_+$ as in the proof of \cref{prop:convergence_strips_D4tree}, so that the only vertices of $\snetwork(E)$ that lie in $B_{\eta,\delta}$ are the initial trivalent vertices. \cref{prop:convergence_strips_D4tree} implies that, after possibly considering a subsequence, the limit $\lim (u_\e,z_\e,\triangle_m(\e),E)$ exists and it is a trimmed $\dfs$-tree through $z$ with energy below $E$. Furthermore, such a $\dfs$ tree must be entirely contained in $\snetwork(E)$. It thus follows that $z\in\snetwork(E)$, which contradicts the initial hypothesis $z\in \snetwork^c$. Part (2) follows analogously.\hfill$\Box$
%%%%%%%%%%%%%%%%%%%%%%%%%%%%%%%%%%%%%%%%%%%%%%%%%%%%%%%%%%%%%%%%%%%%%%%%
%%%%%%%%%%%%%%%%%%%%%%%%%%%%%%%%%%%%%%%%%%%%%%%%%%%%%%%%%%%%%%%%%%%%%%%%
%%%%%%%%%%%%%%%%%%%%%%%%%%%%%%%%%%%%%%%%%%%%%%%%%%%%%%%%%%%%%%%%%%%%%%%%
\section{Family Floer and Non-abelianized Local Systems}\label{section_Floer}
%%%%%%%%%%%%%%%%%%%%%%%%%%%%%%%%%%%%%%%%%%%%%%%%%%%%%
%%%%%%%%%%%%%%%%%%%%%%%%%%%%%%%%%%%%%%%%%%%%%%%%%%%%%
%%%%%%%%%%%%%%%%%%%%%%%%%%%%%%%%%%%%%%%%%%%%%%%%%%%%%

Let $(S,\mkpts,{\bf \lknot})$ be a Betti surface and $L\sse(T^*S,\omega_\std)$ a Betti Lagrangian of rank $n$. The object of this section is two-fold. First, in the exact case, to construct two functors $\ff$ and $\Phi_\snetwork$ that carry $\GL_1$-local systems on the Betti Lagrangian to $\GL_n$-local systems on the base $S$ and show that they are equivalent. Second, finish the proof of \cref{thm:characterization}, which concerns both exact and meromorphic Betti Lagrangians. The construction of the functors depends on the existence of a compatible spectral network $\snetwork\sse S$. In a nutshell:

\begin{enumerate}
    \item The functor $\Phi_\snetwork:\loct_1(L)\lr\loct_n(S)$
    is constructed by considering a variant of $\dfs$-trees, with the parallel transport of image local system being modified according to the interaction of a given path in $S$ with the spectral network $\snetwork$. It is constructed in \cref{subsubsection:pathdetourclasses}.\footnote{Technically, $\Phi_\snetwork(V)$ has only parallel transport defined for $\snetwork$-paths. In contrast, $\ff(V)$ has parallel transport for all paths.} \\

    \item The functor $\ff:\loct_1(L)\lr\loct_n(S)$ is constructed by considering continuation maps between the Family Floer cohomology groups of the Lagrangian $L$ with the cotangent fibres of $T^*S$. It is constructed in \cref{subsection:familyfloer}.\\ %In general, methods from Family Floer fail at caustics: we use the spectral curve $\snetwork\sse S$ to correct such matters.{\YJN The last sentence is a bit unclear}\\

    \item A natural transformation between these two functors is established in \cref{ssec:detour_are_continuation_strips}. The proof of \cref{thm:characterization} is then completed in \cref{ssec:characterization_endproof}.
\end{enumerate}

\noindent Note that the projection $L\lr S$, given by restricting to $L$ the projection $T^*S\lr S$ to the zero section, is a branch cover and thus the push-foward of a local system on $L$ does {\it not} give a local system on $S$, just a constructible sheaf on $S$. The point of these two functors above is that they conceptually explain how to correct such constructible sheaf on $S$ (coming from the push-forward of a local system on $L$) back to a local system on $S$. Each approach presents a different type of correction: $\Phi_\snetwork$ uses the detour paths coming from $\snetwork$ to correct the lack of flatness, in line with \cite{GNMSN,GMN13_Framed,GMN14_Snakes}, whereas $\ff$ uses pseudo-holomorphic strip counts. As one of the main themes of this article is the comparison between spectral networks and Floer theory, we shall explain how the former is an expression of the adiabatic limit of the latter.\\ 

In the above, $\Loc_k(M)$ denotes the dg-category of rank-$k$ local systems of $\C$-modules on a smooth manifold $M$, and $\loct_k(M)$ denotes the dg-category of rank-$k$ twisted local systems of $\C$-modules on a smooth manifold $M$. We recall that a twisted local system on $M$ is a local system on the unit cotangent bundle $T^\infty M$ of $M$ whose monodromy around the spherical cotangent fiber is given by $-\mbox{Id}$. %{\YJN I know this for $\dim M=2$. Is this also true for dimension of $M$ greater than $2$?}
The unit cotangent bundle will also be denoted $\mathring{M}$ and its projection to the base by $\pi:\mathring{M}\lr M$. There are equivalent descriptions of local systems of $\C$-modules, e.g.~as locally constant sheaves, as modules $\mbox{Fun}(\Pi(M),\C\mbox{-mod})$ over the fundamental $\infty$-groupoid $\Pi(M)$ of $M$, or as $\C$-modules over chains $C_{-*}(\Omega M;\C)$ on the base loop space; e.g.~cf.~\cite{higher_algebra}. For an open surface $M=K(\pi_1(M),1)$ and $\mathring{M}\cong M\times S^1$, so (twisted) local systems effectively translates to representations $\pi_1(M,m)\lr\mbox{GL}(V)$ of the fundamental group of $M$, where $V$ is a $\C$-vector space and $m\in M$ a base point, which will be implicitly understood. Such representations are equivalent to modules over the group ring $\C[\pi_1(M)]$, and the functors $\ff$ and $\Phi_\snetwork$ will essentially be described as
$$\ff,\Phi_{\snetwork}:\C[\pi_1(\mathring{L})]\mbox{-mod}\lr\C[\pi_1(\mathring{S})]\mbox{-mod}.$$
%{\YJN Question: We know this at the level of classical fundamental groupoid. Is our enhancement to the $\infty$-fundamental groupoid more or less automatic? Does Floer theory give the full homotopy data? Another question: is homotopy theory of $\infty$-groupoid rich enough so that we get an automatic enhancement to derived local systems? Worth checking}
In fact, both $\ff,\Phi_\snetwork$ will be described locally (for short paths) and thus we shall be considering paths with varying endpoints. So our precise choice for describing $\Loc_k(M)$ will be that of the fundamental $\infty$-groupoid. That is, for this manuscript, a local system $\mathcal{L}\in\Loc_k(M)$ consists of the data of a stalk $\mathcal{L}_m$ at every point $m\in M$, where $\mathcal{L}_m$ is a $k$-dimensional $\C$-vector space, and an isomorphism $\mathcal{L}(\gamma):\mathcal{L}_x\lr \mathcal{L}_y$ for any continuous path $\gamma:[0,1]\lr M$, $\gamma(0)=x,\gamma(1)=y$, which only depends on the homotopy class of $\gamma$. The structure group of the associated bundles for these local systems will always be $\GL_n(\C)$, if we are in a given rank $n$.

\begin{remark}\label{rmk:loc_vs_twloc} (1) The branch cover $\pi:L\lr S$ lifts to a branch cover $\mathring{\pi}:\mathring{L}\lr\mathring{S}$. In this case, there are identifications $\mathring{L}\cong L\times S^1$, $\mathring{S}\cong S\times S^1$, and the latter map reads $\mathring{\pi}=(\pi,\mbox{id})$.\\

(2) Intuitively, the functors $\ff,\Phi_{\snetwork}:\C[\pi_1(\mathring{L})]\mbox{-mod}\lr\C[\pi_1(\mathring{S})]\mbox{-mod}$ will be described by giving a group algebra morphism $\C[\pi_1(\mathring{S})]\lr\C[\pi_1(\mathring{L})]$. That is, they are described by explaining how to lift paths in $\mathring{S}$ to paths in $\mathring{L}$. Rigorously, a path in $\mathring{S}$ with two endpoints $z,z'\in \mathring{S}$ will lift to a collection of paths in $\mathring{L}$ whose endpoints are allowed to be in a set of $2n$ points (the pre-images of $z,z'$). Such paths are not morphisms in the fundamental groupoid of $\mathring{L}$: they would roughly be morphisms between direct sums of objects, which only make sense once we consider $\C$-modules. This is a technical reason to describe the functors between the module categories.\\

(3) For an open surface $M$, there is a dg-equivalence $\Loc_k(M)\simeq \loct_k(M)$ once a spin structure is chosen; cf.~\cref{ssec:lsys_vs_tlsys} or \cite[Appendix B.2]{casals2022conjugate}. Since protected spin characters are not the scope of this article, this effectively implies that $\ff,\Phi_\snetwork$ can be understood as functors carrying a $\GL_1$-local system on the Betti Lagrangian $L\sse T^*S$ to a rank $\GL_n$-local system on the base $S$.\hfill$\Box$
\end{remark}

\noindent Due to \cref{rmk:loc_vs_twloc}.(2), we adopt the following notational conventions. First, $\pi_1(\mathring{S})$ will stand for the relative homotopy classes $\pi_1(\mathring{S},\mathfrak{t}_S)$, where $\mathfrak{t}_S$ is a collection of points in $S$, often understood implicitly by context. Second, consider the functor $f^*:\loct(L)\lr\loct(S)$ that is the pull-back of a group algebra morphism $f:\C[\pi_1(\mathring{S})]\lr\C[\pi_1(\mathring{L})]$ and $V\in\loct(L)$. Given a path $\gamma\sse S$ with
$$f(\gamma)=\sum_{i=1}^ga_i\tau_i,\quad \tau_i\in\pi_1(\mathring{L},\mathfrak{t}_L), a_i\in\C$$
a linear combination of relative homotopy classes in $\pi_1(\mathring{L},\mathfrak{t}_L)$, with the points in $\mathfrak{t}_L$ projecting to those in $\mathfrak{t}_S$, we denote
$$f^*(V)(\gamma)=\sum_{i=1}^g a_iV(\tau_i)[\tau_i]$$
to indicate that the parallel transport $f^*(V)(\gamma)$ of $f^*(V)$ along $\gamma$ is given by
$$\displaystyle\sum_{i=1}^g a_iV(\tau_i)$$
and zero otherwise, and that the map $f^*$ in fact admitted the homotopy refinement $f$ which (further) keeps track of the relative homotopy classes $[\tau_i]$, a linear combination of which gives the image of $\gamma$. This notation is used in the following subsections, c.f.~e.g.~\cref{def:detourpath} and \cref{cor:homotopy_refinement_continuation_map}.

\begin{remark}
Alternatively, we can add the dependence on base points $\tt_S$ and $\tt_L$ explicitly in the notation, as in \cref{subsec:snetworkandpathdetours2_BPS_index}. Such notation becomes rather excessive and we have leaned towards having such endpoints of relative homotopy classes implicitly understood. This is not a particular challenge: the sets are always given a point in $S$, or $\sphb$, where we will consider the cotangent fiber, and its lifts in the Betti Lagrangian $L\sse T^*S$, or $\sphsc$.\hfill$\Box$
\end{remark}

%%%%%%%%%%%%%%%%%%%%%%%%%%%%%%%%%%%%%%%%%%%%%%%%%%%%%
%%%%%%%%%%%%%%%%%%%%%%%%%%%%%%%%%%%%%%%%%%%%%%%%%%%%%
%%%%%%%%%%%%%%%%%%%%%%%%%%%%%%%%%%%%%%%%%%%%%%%%%%%%%

\subsection{The functor $\Phi_\snetwork$: non-abelianization of local systems} \label{subsubsection:pathdetourclasses} The goal of this section is to define the functor
$$\Phi_\snetwork:\loct_1(L)\lr\loct_n(S),$$
using a finite spectral network $\snetwork$ compatible with a Betti Lagrangian $L\sse T^*S$. It formalizes the constructions in \cite{GNMSN,GMN13_Framed,GMN14_Snakes}, where such types of local systems were studied in the context of providing coordinates for the moduli of flat connections on Riemann surfaces, see \cite[Section 10.1]{GNMSN}. In particular, the images $\Phi_\snetwork(V)$ of (twisted) local systems $V\in\loct_1(L)$ encode the formal generating functions of framed $2d-4d$ states from \cite{GNMSN} after evaluating the spin character at $y=1$. The local systems $\Phi_\snetwork(V)$ will be described by giving their stalks $\Phi_\snetwork(V)_z$ at $z\in\mathring{S}$ and the associated parallel transport isomorphisms $\Phi_\snetwork(V)(\rho):\Phi_\snetwork(V)_z\lr\Phi_\snetwork(V)_{z'}$ along certain paths $\rho:[0,1]\lr\sphb$ with $\rho(0)=z,\rho(1)=z'$.

\begin{remark}\label{rmk:notation_points_in_sphb}
For notational ease, we often denote by $z\in\sphb$ points in $\sphb$, even if $z$ also denoted points in $S$. Given a regularly parametrized path in $S$ that starts at $\rho(0)=z\in S$, the unit cotangent lift via the velocity vector $\rho'(0)$ at $z$ gives a unique point in $\sphb$ projecting to $z$. That is the type of point we also denote $z\in\sphb$, as it is uniquely determined by and determines $z\in S$ once the path is chosen.\hfill$\Box$
\end{remark}

The stalks $\Phi_\snetwork(V)_z$ at $z\in\sphb$ are simple to describe for points $z\in\sphb$ with $\pi(z)\not\in K_L$ not a branch point. By definition, we choose
$$\Phi_\snetwork(V)_z:=V_{z_1}\oplus\cdots\oplus V_{z_n},$$
where $z_1,\ldots,z_n\in\sphsc$ are the lifts to $\sphsc$ of the $n$-preimages of $\pi(z)\in\obis$ via $\pi:\blag\lr\obis$.\footnote{Since $V$ is a twisted local system, all choices of lifts are isomorphic.} The requirement that $\Phi_\snetwork(V)\in\loct(S)$ determines the stalks at all $z\in\sphb$. If $V\in\loct_1(S)$, $V_{z_i}$ is a rank-$1$ $\C$-module and $\Phi_\snetwork(V)_z$ is therefore a rank-$n$ $\C$-module.\\

The parallel transport isomorphisms $\Phi_\snetwork(V)(\rho):\Phi_\snetwork(V)_z\lr\Phi_\snetwork(V)_{z'}$ are more interesting and depend on the spectral network $\snetwork$. In order to describe them, we shall define parallel transport for an open dense set of paths $\rho$ that interact generically with $\snetwork$, as follows.

\begin{definition}[$\snetwork$-adapted paths]\label{def:networkadapated}  
Let $I\sse\R$ be a finite closed interval. An immersed path $\ppp:I\to \sphb$ is:
\begin{itemize}
    \item[(i)] $\snetwork$-adapted, if its projection $\pi\big(\ppp):I\lr\obis$ is immersed and transverse to $\snetwork\sse\obis$.
    \item[(ii)] free, if the image of $\pi\big(\ppp)$ is embedded and disjoint from $\snetwork$.
    \item[(iii)] short, if its the unit-velocity lift of an embedded path in $S$ intersecting $\snetwork$ transversely and exactly once.\hfill$\Box$
\end{itemize}
\end{definition}

\noindent Any relative homotopy class with endpoints in $\snetwork^c$ is represented by a $\snetwork$-adapted path, as in \cref{def:networkadapated}. Given a $\snetwork$-adapted path, we can trivialize the sheets of $\blag$ along the path $\pi(\ppp)$: we denote by $\ppp^i$ be the lifts of $\ppp$ to $\sphsc$, $i\in [1,n]$; they are given by lifting $\pi(\rho)$ to $n$ lifts in $L$ and then considering the unit-velocity lifts to $\sphsc$.\\

Let $\ppp:[-\e,\e]\to \sphb$ be a short $\snetwork$-adapted path (with its projection) intersecting $\snetwork$ at a point $z=\rho(0)$ in an $(ij)$-wall $\wall$. Let $\mathfrak{S}(z;\wall):=\{\soliton{z;\wall}\}$ be the set of soliton classes at $z$: each $\soliton{z;\wall}$ defines a regular homotopy class with relative endpoints $z^{j}$ and $z^{i}$, oriented from the former to the latter.% {\YJN This notation should have been used more. I will implement the changes}
We denote by $\ccc{(z)}$ be the unit-velocity lift of $\soliton{z;\wall}$ to $\sphsc$ and set the signs $\sgn(\ppp)(z)=1$ if the frame $\left<\snetwork'(z),\ppp'(0)\right>$ is positively oriented, and $-1$ otherwise. We need the following deformations for short paths. If $\sgn(\ppp)(z)=1$, then we deform $\pi(\ppp)\vert_{[-\e,0]}$ near $0$ so that it becomes tangential to $\snetwork$ at $0$ and directed against the orientation on $\snetwork$: denote by $\ppp_{-}$ the lift to $\sphb$ of such resulting path. Similarly, let $\ppp_{+}$ be the lift to $\sphb$ of a deformation of $\pi(\ppp)\vert_{[0,\e]}$ near $0$ so that it becomes tangential to $\snetwork$ at $0$, now in the same direction as that of $\snetwork$. The following definition formalizes the {\it detour paths} in the physics literature, cf.~e.g.~\cite[Section 3.2]{MR3395151} or \cite{MR3769247}.
%Let $\ppp_{+}$ and $\ppp_{-}$ be their sphere bundle lifts to $\sphb$.  %we can concatenate $\ddd^{i}\big(\ppp_{-})$ with $\ccc{(z)}$ and $\ccc{(z)}$ with $\ddd^{j}\big(\ppp_{-}\big)$. 

\begin{definition}[Detour paths]\label{def:solitondetour}
Let $\ppp:[-\e,\e]$ be a short $\snetwork$-adapted path with $\sgn(z)=+1$ and $\ppp_{\pm}$ its two deformations. By definition, the regular homotopy class $[\ccc{(\ppp)}]$, with relative endpoints $\ppp^i(-\e)$ to $\ppp^j(\e)$, is given by
\[[\ccc{(\ppp)}]:=\big[\big(\ppp_{+}\big)^i\circ \ccc{(z)}\circ \big(\ppp_{-})^j\big].\]
For $\sgn(z)=-1$, the regular homotopy class $[\ccc{(\ppp)}]$ is defined to be $[\ccc{(\ppp)}]:=\big[\overline{\ppp}_{-}^i\circ \ccc{(z)}\circ \overline{\ppp}_{+}^i\big]$, where $\bar{\rho}$ denotes the inverse path. \hfill$\Box$
\end{definition}
Intuitively, the representative $\big(\ppp_{+}\big)^i\circ \ccc{(z)}\circ \big(\ppp_{-})^j$ of $[\ccc{(\ppp)}]$ in \cref{def:solitondetour} is (the lift of) a path that starts as $\rho$, when it hits $z$ it takes a detour going around the soliton $\soliton{z;\wall}$, and then goes back to $z$ and continues the path $\rho$. These detour paths are the crucial correction so as to define $\Phi_\snetwork$ in a manner that $\Phi_\snetwork(V)$ is indeed a local system on $\sphb$, and not merely a constructible sheaf. For instance, at a $\dfs$-singularity, these detour paths correct the non-trivial monodromy around a branch point to be the identity. Specifically, the parallel transport isomorphisms for $\Phi_\snetwork(V)$ are defined as follows:\\
%{\RC EDITED TILL HERE 250228 11:30am}\\
\begin{definition}[Parallel transport for $\Phi_\snetwork$]\label{def:detourpath}
Let $V\in\loct_1(L)$, $(\snetwork,\mu)$ and indexed spectral network, and $\ppp:I\lr\sphb$ a $\snetwork$-adapted path. By definition, the parallel transport isomorphisms $\tdfd(\ppp)$ are defined as follows:\\
%{\RC THINK ABOUT PARALLEL TRANSPORTS VERSUS A FORMULA FOR ISOMORPHISMS. First one is about $\mathbb{Z}[\pi(L)]$-modules the second one is about $\pi[L]$ maps, which is more information}\\
\begin{enumerate}[label=(\roman*)]
\item If $\ppp$ is free, then
 \[\tdfd(\ppp):=\sum_{i=1}^n V(\ppp^i)[\ppp^i],\]
 where the sum runs over the $n$ lifts $\ppp^i$ of $\ppp$ to $\sphsc$, $i\in [1,n]$.\\
 
 \item If $\ppp$ is short, with its projection intersecting $\snetwork$ at $z$, then
 \[\tdfd(\ppp):=\sum_{i=1}^n V(\ppp^i)[\ppp^i]+\sum_{\ccc{(z)}\in \mathfrak{S}(z;\wall)} \mu(\ccc{(z)})[\ccc{(\ppp)}]V\big(\ccc{(\ppp)}\big),\]
 where the second sum runs over all the soliton classes $\soliton{z;\wall}$ associated to $z\in\snetwork$.\\

\item If $\ppp$ is given as a composition $\ppp=\ppp_1\circ \ppp_2 \circ\ldots \circ \ppp_n$, each $\ppp_i$ either short or free, then
 \[\tdfd(\ppp)=\tdfd(\ppp_1)\circ \tdfd(\ppp_2)\circ \ldots \circ \tdfd(\ppp_k).\]
\end{enumerate}
\hfill$\Box$
%A component of $\Phi^V_W(\ppp)$ is called a \textit{detour path} of $\ppp$. 
\end{definition}
\noindent \cref{def:detourpath} allows us to define parallel transport for all $\snetwork$-adapted paths, by declaring parallel transport along a homotopy class to be independent of a representative. For instance, if a class is represented by a free path, then \cref{def:detourpath}.(i) applies. The contributions from the second summand of \cref{def:detourpath}.(ii) are said to come from {\it detour paths}, in line with \cref{def:solitondetour}. It is possible to verify combinatorially by hand that \cref{def:detourpath} yields a well-defined local system, independent of the representative of a given (twisted) homotopy class of a path, see e.g.~\cite[Section 5.6]{GNMSN}. Given the Floer-theoretic framework we develop, this property will be rather a consequence of interpreting $\snetwork$ in terms of pseudo-holomorphic strips. %A stronger invariance property will hold: the functor $\Phi_\snetwork$ will be an invariant of the Betti Lagrangian $L$, up compactly supported Hamiltonian isotopy.

\begin{remark} A brief explanation of the homotopy invariance from a Floer-theoretic viewpoint is as follows. Given a $\snetwork$-adapted path $\ppp$, we can consider the exact Lagrangian isotopy $\fibre{\pi(\ppp(t))}$ given by moving the cotangent fibres along the path $\ppp$. After an appropriate choice of (a family of) almost complex structures, we consider continuation strips with boundary on $\blag$ and $\fibre{\pi(\ppp(t))}$ and, using the data of $\ppp(t)$ and the (twisted) local system $V$, we orient the moduli space of such strips. The cobordism type of such moduli space depends only on the twisted isotopy class of $\ppp$, and thus invariance follows once one shows that the formal sum of the boundaries of such continuation strips coincides with $\tdfd(\ppp)$.\hfill$\Box$
\end{remark}

\noindent Note that \cref{def:detourpath} suffices to define parallel transport for paths with endpoints in the dense complement $\snetwork^c$. This is in line with the current understand of wall-crossing phenomena and related constructions in the literature. In contrast to see, the Floer theoretic parallel transports are well-defined for paths with endpoints on $\snetwork$.\\ %whereby it remains unknown how to define invariants at the walls themselves, and similarly for Family Floer at the caustics. This is an interesting topic to eventually explore.\\
%{\RC Is ``infinite'' path a good notation? Instead of paths at infinity, say? Is it needed really?}\\
Due to the presence of Reeb chords at infinity, given by the data ${\bf \lknot}$ of the Betti surface, we also have an {\it infinite} version of $\snetwork$-adapted paths. By definition, an infinite path $\ppp:(-\infty,\infty)\to \sphb$ is said to be $\snetwork$-adapted if $\ppp$ is $\snetwork$-adapted over any interval $[-N,N]$, $N\in \mathbb{N}$, with a uniformly finite number of intersection points with $\snetwork$, and its projection $\pi(\ppp(t)),t\to \pm \infty$ asymptotes to radial Reeb chord rays near a subset of marked points. The discussion above all apply to these types of paths as well.

\noindent Given a Reeb chord $\mathfrak{r}$ of a Legendrian link in $\bf{\lknot}$ at a marked point $m\in\mkpts$. By \cref{prop:asymptoticoftrajectories}, there exists a small sectorial neighborhood
$$S_{\e}(\mathfrak{r})\cong\{(e^{i\theta},r)\in S^1\times\R:\theta\in [\theta_0-\e,\theta_0+\e],r>R\},$$
with $R\in\R_+$ large enough, such that the Reeb chord $\mathfrak{r}$ corresponds to $\theta_0$ and $r\to\infty$ and $S_{\e}(\mathfrak{r})$ contains only the edges of $\snetwork$ that are asymptotic to $\mathfrak{r}$. Let $\ppp(\mathfrak{r}):(-\infty,\infty)\lr S$ be the infinite path given by smooth the boundary $\dd S_{\e}(\mathfrak{r})$, with $\ppp(\infty)=\ppp(-\infty)=m$, which we can assume intersects transversely $\snetwork$. %The relative homotopy class of $\ppp(\mathfrak{r})$ is independent of $R$ if chosen large enough, and up to the natural identification of endpoints.
Such infinite paths $\ppp(\mathfrak{r})$ are said to be the Reeb paths of $\mathfrak{r}$. The following result establishes that the morphism of the functor $\Phi_\snetwork$ over a Reeb path, even with its homotopy refinement, is an invariant of the Hamiltonian isotopy class of $L\sse(T^*S,\la_\std)$, up to compactly supported Hamiltonian isotopy.

\begin{theorem}[Invariance of soliton classes for Reeb paths] \label{thm:pathdetourclassesareHaminvariant}
Let $(S,\mkpts,{\bf\lknot})$ be a Betti surface, $\reeb$ a Reeb chord of a component of ${\bf\lknot}$ near a marked point with Reeb path $\ppp(\reeb)$, and $\blag\sse(T^*S,\la_\std)$ an exact Betti Lagrangian endowed with a compatible spectral network $\snetwork\sse S$. Then $\Phi_\snetwork(V)(\ppp(\reeb))$ is an invariant of $\blag$, for any $V\in\loct(L)$ and up to compactly supported Hamiltonian isotopy.
\end{theorem}
%%%%%%%%%%%%%%%%%%%%%%%%%%%%%%%%%%%%%%%%%%%%%%%%
\cref{thm:pathdetourclassesareHaminvariant} will be a consequence of the results we establish in \cref{subsection:familyfloer} and \cref{ssec:detour_are_continuation_strips}, c.f.~\cref{thm:familyFloer}.

%%%%%%%%%%%%%%%%%%%%%%%%%%%%%%%%%%%%%%%%%%%%%%%%%%%%%%%%%%%%%%%%%%%%%%%%
%%%%%%%%%%%%%%%%%%%%%%%%%%%%%%%%%%%%%%%%%%%%%%%%%%%%%%%%%%%%%%%%%%%%%%%%
%%%%%%%%%%%%%%%%%%%%%%%%%%%%%%%%%%%%%%%%%%%%%%%%%%%%%%%%%%%%%%%%%%%%%%%%

\subsection{The functor $\ff$: Family Floer and local systems}\label{subsection:familyfloer} The goal of this section is to define a functor
$$\ff:\loct_1(L)\lr\loct_n(S)$$
using Floer-theoretic methods, and independently of a choice of a spectral network in $S$. For that, we use a version of Family Floer cohomology. Family Floer cohomology was initially pioneered by K.~Fukaya, see e.g.~\cite{Fukaya_FamilyFloer}, along with \cite{Abouzaid_ICM2014} and references therein. In brief, adapting this framework to exact Betti Lagrangians allows us to define a type of Floer cochain complexes $\CF(\fibre{z},L)$, recording Lagrangian intersections between $L\sse T^*S$ and the cotangent fiber $\fibre{z}\sse T^*S$, and continuation maps between these as $z\in S$ varies. Such maps are defined by studying certain (cobordism type of) moduli spaces of pseudo-holomorphic strips. In particular, given a path between $z,z'\in S$, it is possible to construct a parallel transport from $\CF(\fibre{z},L)$ to $\CF(\fibre{z'},L)$, which is what ultimately yields the functor $\ff$. This section presents the necessary details to construct Family Floer cohomology for Betti Lagrangians and use it to define $\ff$. For the rest of the section, we will assume that  $\blag$ is an exact Betti Lagrangian.

%%%%%%%%%%%%%%%%%%%%%%%%%%%%%%%%%%%%%%%%%%%%%%%%%%%%%%%%%%%%%%%%%%%%%%%%
%%%%%%%%%%%%%%%%%%%%%%%%%%%%%%%%%%%%%%%%%%%%%%%%%%%%%%%%%%%%%%%%%%%%%%%%
%%%%%%%%%%%%%%%%%%%%%%%%%%%%%%%%%%%%%%%%%%%%%%%%%%%%%%%%%%%%%%%%%%%%%%%%

\subsubsection{Framework for Floer cohomology}\label{sssec:framework_floer}
Let $\blag\sse(T^*S,\omega_\std)$ be a Betti Lagrangian. The difference between Betti Lagrangian and other types of Lagrangians previously studied in the literature (e.g.~closed cases, or partial wrappings in Liouville sectors) is the existence of horizontal ends for the former, c.f.~the conical and $O(-1)$-ends in ~\cref{section_setup}. These ends naturally appear when studying spectral curves, and thus we now develop the Floer-theoretic techniques for them. In this context, the type of Hamiltonian isotopies that we employ is described as follows:

\begin{definition}\label{def:Hamiltonian}
A Hamiltonian $H:T^{\ast}\obis\to \mathbb{R}$ is said to be \textit{cylindrical} if it is of the form $H=h\cdot \radial$ outside the radius-$R$ disk bundle $D_R^{\ast}\obis$, for some sufficiently large $R\in\R_+$, where $h:\mathring{\obis}\to \mathbb{R}$ is a smooth function on the unit sphere bundle and $r:T^*S\to S$ is the (radial) distance coordinate to the zero section. Such a Hamiltonian $H$ is said to have compact \textit{horizontal support} if the projection $\pi(\text{supp}(dH))$ lies in a pre-compact subset of $\obis$. Families of Hamiltonians parametrized over a compact manifold are said to be \textit{uniformly horizontally supported} if these Hamiltonians are themselves horizontally supported over a compact set and are uniformly cylindrical.\hfill$\Box$\end{definition}

\noindent Let $Z$ be the Liouville vector field on $T^{\ast}\obis$ uniquely determined by $\iota_Z \omega_{st}=\la_{st}$. By definition, an $\omega_{st}$-compatible almost complex structure $J$ is \textit{cylindrical} if $L_Z J=J$ for large enough $r$. As in \cite{knotcontact}, we deform the Sasaki almost complex structure to make it cylindrical via
\begin{align}\label{eq:defac}\defac:=
		\begin{bmatrix}
			0 & \beta(\radial)^{-1}g^{-1} \\ 
			-\beta(\radial) g & 0
		\end{bmatrix}
	\end{align}  
where $\beta:[1,\infty)\to \mathbb{R}_+$ is a smooth increasing positive function such that $\rho(r)\equiv1$ near $\radial=1$ and $\rho\equiv\radial$ near infinity. We onwards implicitly assume that the chosen (families of) cylindrical almost complex structure $J$ agree with $\defac$ outside the cotangent bundle over a pre-compact subset of $\obis$.\\
%The $\rho$-deformed Sasaki almost complex structure is the following. 

\begin{remark}
For simplicity, we provide the construction of $\ff$ for exact Betti Lagrangians. The key result obtained (e.g. \cref{thm:characterization}) for the WKB case is similar: the presence of $O(-1)$-ends, rather than conical or cylindrical ends, requires modifying some of the arguments but the qualitative descriptions remains the same, as illustrated by several of the results proven above in both settings. However, in the current paper, we will not construct the full functor $\ff$ for the non-exact case, though we essentially compute it for infinitesimally short line segments. \hfill$\Box$
\end{remark}
%Is $\ff$ only for exact Betti, or also WKB? Can't use conical one for WKB, we add remark that we focus on exact case.}\\

Let $\mathcal{J}(T^{\ast}\obis)$ denote the space of compatible almost complex structures on $(T^{\ast}\obis,\omega_{\std})$. To introduce the precise Floer data needed to construct Floer cochain complexes, we use the following interplay between the certain Lagrangians $P\sse T^*S$ and almost complex structures:

\begin{definition}\label{def:geombounddef}
Let $J_s\in \mathcal{J}(T^{\ast}\obis)$ be a 1-dimensional family of almost complex structures, with real parameter $s\in\R\sse\C$. By definition, a Lagrangian $\tlag\sse T^*S$ is said to be \textit{uniformly geometrically bounded} with respect to $J_s$ if it is geometrically bounded with respect to the family of metrics induced by $J_s$ and $\omega_\std$. A 1-dimensional family of Lagrangians $\tlag_s$ is said to be uniformly geometrically bounded if given the associated movie totally real submanifold $\{(\tlag_s,s)\}\sse (T^{\ast}\obis\times \mathbb{C},\omega_\std\oplus\omega_\std)$, there exists a constant $C\in\R_+$ and a $J_s$-compatible symplectic form $\omega_{\tlag_s}$ such that $(\tlag_s,s)$ is Lagrangian and the triple $(\omega_{\tlag_s},J_s,(\tlag_s,s))$ is geometrically bounded.\hfill$\Box$
\end{definition}

A summary result of the needed outcome from the literature on Lagrangian Floer theory reads as follows:

\begin{proposition}[Floer Complexes]\label{prop:standardfloertheory}
Let $P,K\sse (T^*S,\omega_\std)$ be a transverse pair of spin Lagrangians, each endowed local systems $\locsys_{P},\locsys_{K}$ of $\C$-modules and with chosen spin structures. The following holds:
\begin{enumerate}
\item Let $J_{P,K}:[0,1]\to \mathcal{J}(T^{\ast}\obis)$ be a generic time-dependent family of cylindrical almost complex structures\footnote{This is sometimes known as the Floer data.} such that $P$ and $K$ are uniformly geometrically bounded. Then, there exists a cochain complex
$$(CF^{\ast}((P,\locsys_{P}),(K,\locsys_{K});J_{P,K}),\dd),$$
known as the Lagrangian Floer complex, generated by the intersection points in $P\cap K$ and whose differential is obtained by counting the $0$-dimensional part of the moduli space of $J_{P,K}$-holomorphic strips bounded by $P$ and $K$.\\

\item The chain-homotopy type of the complex in (1) is independent on the choice of $J_{P,K}$: in particular, a homotopy from $J_{P,K}$ to $J'_{P,K}$, for which $P$ and $K$ are uniformly geometrically bounded, induces a quasi-isomorphism
$$(CF^{\ast}((P,\locsys_{P}),(K,\locsys_{K});J_{P,K}),\dd)\lr(CF^{\ast}((P,\locsys_{P}),(K,\locsys_{K});J'_{P,K}),\dd).$$
%This quasi-isomorphism depends only on the homotopy, up to chain homotopy.\\
%\item Given a family of Floer data $J^s_{P,K}(t)$, there is an induced quasi-isomorphism  which only depends 
\item Let $K_s:K\times [0,1]\to T^{\ast}\obis$ be a Lagrangian isotopy, with associated local systems $\locsys_{K_s}$, such that either its movie is uniformly geometrically bounded, or the isotopy generating $K_s$ is a composition of a compactly supported Hamiltonian isotopy with a global Hamiltonian isotopy fixing $P$. Then $\{K_s\}_s$ induces a continuation quasi-isomorphism
$$c:(CF^{\ast}((P,\locsys_{P}),(K_0,\locsys_{K_0});J_{P,K}),\dd)\lr(CF^{\ast}((P,\locsys_{P}),(K_1,\locsys_{K_1});J_{P,K}),\dd),$$
which is compatible with concatenation, and depends on $\{K_s,\locsys_{K_s}\}$ only up to chain homotopy.\\

\item If $P,K$ are both Maslov zero Lagrangians, then the differential of (1) vanishes and thus
$$(CF^{\ast}((P,\locsys_{P}),(K,\locsys_{K});J_{P,K}),\dd)\cong \displaystyle\bigoplus_{x\in P\pitchfork K}(\hom(\locsys_{P}(x), \locsys_{K}(x))).$$
is the $\C$-module generated by the intersection points, endowed with the zero differential.\hfill$\Box$
\end{enumerate}
\end{proposition}
\begin{proof}
For $P,K$ closed Lagrangians and $\locsys_P,\locsys_K$, the Lagrangian Floer complex and its invariance were established in \cite{Floer88_MorseLagrangianIntersections}. A more contemporary treatment which includes the construction and invariance of such complexes for exact and monotone Lagrangians is \cite[Chapter 16]{Ohsymplectictopology}. The case were the local systems $\locsys_P,\locsys_K$ are non-trivial is addressed in \cite{kontsevichhms}. This addresses Parts (1) and (2) in these cases. The method of moving Lagrangians, for Part (3), was introduced in \cite{ohfloercohomology}, with \cite[Chapter 14.4]{Ohsymplectictopology} providing the necessary energy estimates. For $P,K$ (and their movies) not necessarily compact Lagrangians, but geometrically bounded, the statement follows by using the monotonicty inequalities introduced in \cite[Section 5]{Audin1994HolomorphicCI}, cf.~also the arguments surrounding the proof of \cite[Proposition 3.19]{GPSCV}. For Part (4), with $P,K$ Maslov $0$, the vanishing of the differential is a consequence of the virtual dimension of the moduli space of strips between $P$ and $K$: it necessarily has to be $-1$ and so, for regular Floer data, it must be empty. 
%{\RC Nick will add references: it can be a laundry list for each item, that's fine.}
\end{proof}

\color{black}
\noindent The case of interest in this manuscript is $P=\fibre{z}$ a cotangent fiber transverse to $\blag$, and $K=L$ a Betti Lagrangian.  For \cref{prop:standardfloertheory}.(4) to apply, we must ensure the hypothesis of Maslov $0$ for $\blag$. It is verified as follows:
 \begin{lemma}\label{lem:Maslov0}
A Betti Lagrangian $\blag\sse T^*S$ is Maslov $0$.
\end{lemma}
\begin{proof}
By definition, the metric $g$ induces a complex structure on $\obis$, and $\blag$ is holomorphic over some neighborhood of the $\dfs$-singularities. Now, over any pre-compact subset, the smooth sheets of the rescaled $\e \blag\sse T^*S$ uniformly $C^{\infty}$-converge to the zero section, outside of some fixed neighborhood of the $\dfs$-singularities. Since the zero section is holomorphic and $\blag$ is holomorphic near the branch points, $\e\blag$ is arbitrarily close to being holomorphic over any precompact subset of $\obis$ as $\e\to0$. Therefore the phase function of $\e\blag$ is arbitrarily close to zero over any pre-compact subset of $\obis$ and thus $\blag$ is Maslov $0$ (cf.~ also \cite[Section 6.2.1]{nho2024familyfloertheorynonabelianization}).
\end{proof}
%{\RC where is it said that \cref{prop:standardfloertheory} works with varying local system coefficients?}\\

\noindent The computation of the Floer complexes in \cref{prop:standardfloertheory}.(1) and the continuation maps in \cref{prop:standardfloertheory}.(3) is typically challenging: for reasonably generic Lagrangians $P,K$, there is no explicit description of the continuation maps beyond their definition. This in part is caused by a difficulty in explicitly computing the (cobordism type of the moduli spaces of) pseudo-holomorphic strips involved in these maps. An important conceptual point is that in our setting we are able to use a spectral network $\snetwork$ adapted to a Betti Lagrangian $L\sse(T^*S,\omega_\std)$ in order to succeed at explicitly computing such pseudo-holomorphic strips, at least in the adiabatic limit. For instance, the arguments in \cref{section_adg} ensure that, adiabatically, there are no pseudo-holomorphic strips bounded by $\e\blag$ and $\fibre{z}$ for $z\in\snetwork^c$, cf.~\cref{thm:snetworkadg}.(1), which is the key ingredient behind showing that the continuation strips for $(CF^\ast(\fibre{z},\blag),\dd)$ with $z\in\rho$ moving along a $\snetwork$-free path $\rho$ are ``trivial" (we will make this notion of ``trivial strips" precise using \cref{lem:gaugesmallconstant} and the discussion afterward).

\begin{remark}(1) The dependence on spin structures in \cref{prop:standardfloertheory} can be more compactly repackaged by passing to sphere bundles, after choosing the Floer data. This is related to the map $\Phi_\snetwork$ in \cref{subsubsection:pathdetourclasses} being defined on homotopy classes of paths on the unit sphere bundle; see also \cref{rmk:loc_vs_twloc}.(3).

(2) We often write $CF^{\ast}(P,K)$ for the Lagrangian Floer complex, due to the invariance \cref{prop:standardfloertheory}.(2), and implicitly understanding $\dd$ and the local systems if the context allows for that. For instance, if $\tlag=\fibre{z}$ is a cotangent fibre, the local system $\locsys_{K}$ will always be trivial.\hfill$\Box$
\end{remark}

%{\RC We are using the spec net $\snetwork$. First, we adiabatically rescale to ensure that there are no disks between $\e\blag$ and $F_z$ for all $z\in\snetwork^c$. Second, this allows to Floer datum regular and gauarantees that the continuation strips in a path outside the spec net are trivial. In general, we would get a map, but no hope to compute it.}

%%%%%%%%%%%%%%%%%%%%%%%%%%%%%%%%%%%%%%%%%%%%%%%%%%%%%%%%%%%%%%%%%%%%%%%%
%%%%%%%%%%%%%%%%%%%%%%%%%%%%%%%%%%%%%%%%%%%%%%%%%%%%%%%%%%%%%%%%%%%%%%%%
%%%%%%%%%%%%%%%%%%%%%%%%%%%%%%%%%%%%%%%%%%%%%%%%%%%%%%%%%%%%%%%%%%%%%%%%

\subsubsection{The stalks of $\ff(V)$}\label{sssec:stalks_floer_locsys} A naive application of \cref{prop:standardfloertheory} leads to the following ansatz for constructing $\ff:\loct_1(L)\lr\loct_n(S)$. Given a local system $V\in \loct_1(L)$ of rank-1 $\C$-modules, we define the stalk $\ff(V)_z$ of $\ff(V)$ at a point $z\in\obis$ to be $\C$-module underlying $CF(\fibre{z},\blag;V)$. The following first considerations then appear:

\begin{enumerate}
    \item $\ff(V)_z$ is well-defined as a cochain complex. Thus considering just its underlying $\C$-module is not invariant unless the differential vanishes.\footnote{In which case the $\C$-module coincides with its cohomology and is invariant.}

    \item $\ff(V)_z$ is a rank-$n$ $\C$-module if $L\sse T^*S$ is a Betti Lagrangian only if $z\not\in K_L$ is not a branch point.

    \item The stalk $\ff(V)_z$ and the resulting local system $\ff(V)$ are independent of the adiabatic parameter $\e$, up to isomorphism. However, $\Phi_\snetwork$ and the comparison to spectral network certainly requires passing to the adiabatic limit $\e\to0$.
%    {\YJN Slightly changed how you said the stalk is independent of the adiabatic parameter $\epsilon$}
\end{enumerate}

%\color{orange}
%\color{black}

\noindent The three consideration are intertwined, and they are all solved at once by considering $z\in\snetwork^c$ and introducing an adiabatic parameter $\e\to0$. This is achieved by considering the stalk $\ff(V)_z$ to be $\C$-module underlying $CF(\fibre{z},\e L;V)$ for $\e$ small enough. This is well-defined and addresses the above thanks to the following fact, which allows us to choose the Floer datum to be $\defac$.

\begin{proposition}\label{prop:regulardata}
Let $z\in \snetwork^c$ and $U_z\sse S$ be the disk-free neighborhood. Then, there exists $\e(z)\in\R_+$ the moduli space of non-constant $\defac$-holomorphic strips between $\e \blag$ and $\fibre{w}$ is empty for any $w\in U_z$.
\end{proposition}
%{\RC Where is the dependence on $z$?}
\begin{proof}
Choose $\e\in\R_+$ small enough such that $\e\blag$ lies in the unit disk bundle. Then, the integrated maximum principle argument (see e.g.~\cite[Lemma 3.28]{nho2024familyfloertheorynonabelianization}) implies that there exists a bijective correspondence between $\defac$-holomorphic strips bounded by $\e\blag$ and $\fibre{z}$, and $J_g$-holomorphic strips bounded $\e\blag$ and $\fibre{z}$. Indeed, the integrated maximum principle guarantees that the images of the $\defac$ and $J_g$-holomorphic strips lie within the unit disk bundle, and $\defac\equiv J_g$ coincide inside this disk bundle. Then \cref{thm:snetworkadg} implies the statement.
\end{proof} 

\begin{remark}\label{eq:salamon-weber equvialence}
The dependence of such putative stalk $\ff(V)_z$ (more precisely, the Lagrangian $\blag$) on the adiabatic parameter $\e$ can also effectively be removed, at the cost of introducing $\e$-dependence on the Floer datum instance. Indeed, it follows from \cite[Remark 1.3]{MR2276534} that the rescaling $v(s,t)=\e^{-1}\cdot u(s, t)$ gives a bijective correspondence between $J_g$-holomorphic discs bounded by $\e\blag$ and $\fibre{z}$, and $J_{\e^{-1}g}$-holomorphic discs bounded by $\blag$ and $\fibre{z}$.\hfill$\Box$
\end{remark}

\noindent Let $V\in\loct(L)$, then the above discussion, including \cref{prop:regulardata} and \cref{eq:salamon-weber equvialence}, leads to the following definition for the stalks of $\ff(V)\in\loct(S)$:

\begin{definition}[Stalks for $\ff(V)$]\label{def:sbchaincomplex}
Let $V\in\loct(L)$ and $z\in\sphb$ be such that $\pi(z)\in \snetwork^c$ and $\e\in\R_+$. The stalk $\ff(V)_z$ is defined to be
\begin{align}
\ff(V)_z:=CF(\fibre{z},\e\blag)=\bigoplus_{i=1}^n V_{z^i.}
\end{align}
where $z^i\in\sphsc$ are the lifts of $z$. That is, the stalk at $z$ is the Lagrangian Floer complex defined by the rescaled Betti Lagrangian $\e\blag$ and the cotangent fiber $\fibre{\pi(z)}\sse T^*\obis$, where all the elements are concentrated in degree $0$ and the differential is trivial.\hfill$\Box$
%{\RC this greek z, etc.}
%{\YJN I have changed the deformation function from $\rho$ to $\beta$, since $\beta$ doesn't really appear elsewhere in the paper (except for the reverse isoperimetric ienquality)}
\end{definition}
Two comments on \cref{def:sbchaincomplex}:\\
\begin{enumerate}
    \item We use of the sphere bundles $\sphb$ and $\sphsc$ in order to have signs over $\mathbb{Z}$ (and thus $\C$) well-defined. As said above, spin structures could be used instead when working with the Betti Lagrangian $L$ and cotangent fibres in $T^*S$. An argument in favor of using sphere bundles is that one can refine this framework to quantum local systems, where the $q$-variable, keeping track of the protected spin character $y$ in \cite[Section 3.2]{GMN13_Framed} (see also \cite[Section 4.2]{GMN12_WallCrossCoupled}), is now the value of the monodromy of the local system along the sphere fiber. This restricts to twisted local systems when $q\to-1$, which is our setup, but it is also valuable to develop the theory without specializing the $q$-variable and thus we setup the groundwork for such generalization to be rather natural.\\

    \item The dependence on $\e$ is emphasized as a reminder that computations take place in the adiabatic limit $\e\to0$, where $\snetwork$ can be characterized Floer-theoretically. It is nevertheless not strictly necessary, which can be argued as follows. First, $\blag$ is a uniformly bounded multi-graph and thus there exists a radius $\heightR\in\R_+$ such that the Betti Lagrangian $\blag\subset D^{\ast}_{\heightR} \obis$ lies within the disk bundle of $T^*S$ of radius $R$. Second, consider a smooth family $\beta_s(\radial), s\in[\e,1]$, of deformation functions such that\\
    \begin{itemize}
        \item[(i)] Each $\beta_s(r)$ is constant for $\radial\leq \heightR$ and are uniformly linear at infinity $r\to\infty$,
        \item[(ii)] $\beta_{\e}=\e$ for $r\in[0,2\heightR)$ and $\beta_\e=\radial$ for sufficiently large $\radial$,
        \item[(iii)] $\beta_1=1$ for $\radial\in[0,3\heightR)$.\\
    \end{itemize}

\noindent If $\defac^s=J_g^{\beta_s}$ denotes the associated family of deformed Sasaki almost complex structures, then $\defac^s$ is such that $\defac^{\e}=J_{\e^{-1}g}$ for $\radial\leq R$ and $\defac^{1}=J_{g}$ for $\radial \leq R$. Also, the family is uniformly geometrically bounded and of generalized contact type, in the sense that $\radial^2 \beta^{-1}_s$ are all $\defac^s$-plurisubharmonic. Therefore, we can apply \cref{prop:regulardata} and  \cref{eq:salamon-weber equvialence} and identify $CF(\fibre{z},(\e\blag,\locsys_{\e\blag}),\defac)\cong CF(\fibre{z},(\blag,\locsys_{\blag}),\defac^{\e})$, the latter being defined using $\defac^{\e}$. Note that interpolating $CF(\fibre{z},\blag;\defac^s)$ from $s=\e$ to $s=1$ can be done using standard techniques because there are no moving boundaries involved. \\

\item The stalks for $\ff(V)_z$ are defined for points in $z\in\snetwork$. Similarly, the parallel transport in \cref{sssec:holonomy_floer_locsys} is also defined for paths with endpoints on $\snetwork$. The choices of Floer data, which are done homotopy-coherently, allow for such well-definedness on the entirety of $S$.\\
\end{enumerate}

The stalks for $\ff(V)$ given as in \cref{def:sbchaincomplex}, we now move to the parallel transport for the twisted local systems $\ff(V)$. As with \cref{subsubsection:pathdetourclasses}, the parallel transport isomorphisms contain the more subtle and interesting information, in this case that of continuation strips.\\

%%%%%%%%%%%%%%%%%%%%%%%%%%%%%%%%%%%%%%%%%%%%%%%%%%%%%%%%%%%%%%%%%%%%%%%%
%%%%%%%%%%%%%%%%%%%%%%%%%%%%%%%%%%%%%%%%%%%%%%%%%%%%%%%%%%%%%%%%%%%%%%%%
%%%%%%%%%%%%%%%%%%%%%%%%%%%%%%%%%%%%%%%%%%%%%%%%%%%%%%%%%%%%%%%%%%%%%%%%

\subsubsection{The parallel transport of $\ff(V)$}\label{sssec:holonomy_floer_locsys} Let $V\in\loct_1(L)$ be a twisted $\GL_1$-local system. In \cref{sssec:stalks_floer_locsys}, we explained the choice of stalks for $\ff(V)$, which consisted of $\ff(V)_z=CF(\fibre{z},\e\blag)=\bigoplus_{i=1}^n V_{z^i}$ for a point $z\in \sphb$ whose projection to $\obis$ is not a branch point of the projection $L\lr\obis$. Given a path $\rho:[0,1]\lr \mathring{S}$, with $z=\rho(0)$ and $z'=\rho(1)$, we now want to construct a parallel transport isomorphism $\ff(V)(\rho):\ff(V)_z\lr\ff(V)_{z'}$. This will uniquely determine the image twisted local system $\ff(V)\in\loct_n(S)$. We will first construct a local system $\ff_{\epsilon}(V)$ for each adiabatic parameter $\epsilon>0$ with the stalks as in \cref{def:sbchaincomplex}, for points contained outside some $\e$-dependent nested neighbourhood of $\snetwork$, that converges to $\snetwork^c$ as $\e\to 0$. These local systems will come with the property that for all $0<\e'<\e$, there is a unique natural isomorphism $\ff_{\epsilon}(V)$ and $\ff_{\epsilon'}(V)$, which stabilizes to the identity map over morphisms with end-points on $\snetwork^c$. We then define $\ff(V)$ as the direct limit local system. Intuitively, the parallel transport for $\ff_{\epsilon}(V)$ is a generalization of the continuation maps in \cref{prop:standardfloertheory}.(3). The continuations maps in \cref{prop:standardfloertheory}.(3) cannot be applied in our case because the (moving) fibers $\fibre{z}$ are not geometrically bounded. We will consider instead deformed continuations strips, which introduce correction terms to the standard continuation maps that account for the failure of geometric boundedness of the moving fiber. Such parallel transport is constructed as follows.\\

%{\RC Somewhat the discussion that follows is on $\obis$, not in $\sphb$? All in $S$, $\sphb$ only appears because we must count things.}
Let $H_s:T^*S\times[0,1]\lr\R$ be a uniformly horizontally supported time-dependent Hamiltonian isotopy generating the fiber transport along the \textit{inverse} path of $\inva$ of $\alpha$, $X_s$ its Hamiltonian vector field and $\psi_s$ its flow, and consider the infinite strip $\strip:=\mathbb{R}_s\times [0,1]_t\sse\C$. Choose a smooth, increasing function $\elon:(-\infty,\infty)\to [0,\infty)$ such that $\elon=0$ near $-\infty$ and $\elon=1$ near $+\infty$. We call such functions elongation functions. Let $J_0(t)$ and $J_1(t)$ be a family of almost complex structures obtained by compactly supported perturbations of $\defac$, and let $J(s,t)$ be a family of almost complex structures such that
\begin{itemize}
    \item[(i)] $J(s,t)=J_0(t)$ for $s\ll 1$ and $J(s,t)=J_1(t)$ for $s\gg 1$,
    \item[(ii)] $J(s,t)$ is also given as a compactly supported perturbation of $\defac$.
\end{itemize}
\noindent Such family of almost complex structures is said to be \textit{admissible}. 

\begin{definition}\label{def:continuation_strip}
Let $\alpha:[0,1]\lr S$ be a path, $L\sse T^*S$ a Betti Lagrangian, and $J(s,t)$ be an admissible family of almost complex structures. A map  $u:\mathcal{Z}\to T^{\ast}\obis$ is said to be a \textit{deformed continuation strip} for $\alpha$ and $L$ if it solves the following partial differential equation:
\begin{align}\label{eq:perturbedpassive}
	\begin{cases}
		\dd_s u+{J}(s,t)\dd_t u-l'(s)X_{l(s)}(u)=0 &  \\ 
		u(s,1)\subset \blag&\\
		u(s,0)\subset \fibre{\inva (l(s))}&\\
		\displaystyle\lim_{s\to \infty} u(s,t)\in \fibre{\alpha(1)}\cap \blag\\
		\displaystyle\lim_{s\to -\infty} u(s,t)\in \fibre{\alpha(0)}\cap \blag.
	\end{cases}
\end{align}
For paths $\alpha:[a,b]\to \obis$, the deformed continuation strip for $\alpha$ will mean \cref{{eq:perturbedpassive}} using the path $\alpha(bt+(1-t)a)$ instead.
\hfill$\Box$
\end{definition}

In \cref{eq:perturbedpassive}, the condition $u(s,0)\subset \fibre{\inva (l(s)))}$ is a moving boundary condition and, in a sense, the term $l'(s)X_{l(s)}(u)$ accounts for that. Indeed, up to composing with the Hamiltonian flow $\psi$ of $H_s$, the deformed continuation strips in \cref{def:continuation_strip} are essentially pseudo-holomorphic strips: taking the gauge transform $\tilde{u}:=\psi^{-1}_{l(s)}u$ and $\tilde{J}:=\psi^{\ast}J(s,t)$, we recover the homogeneous pseudo-holomorphic curve equation $\partial_{\tilde{J}}\tilde{u}=0$ with moving boundary condition on $L$ (instead of the fibers). By the exponential estimate \cite[Proposition 14.1.5]{Ohsymplectictopology}, the solutions of \eqref{eq:perturbedpassive} are asymptotic to the intersection points of $\blag$ and $\fibre{\alpha(1)}$, as $s\to -\infty$, and of of $\blag$ and $\fibre{\alpha(0)}$, for $s\to\infty$. By \cite[Proposition 3.37]{nho2024familyfloertheorynonabelianization}, the solutions of the gauge transforms have bounded diameter, and this diameter can be chosen uniformly for a uniform family of Hamiltonian isotopies. Thus the moduli space of solutions of \eqref{eq:perturbedpassive} admits a Gromov compactification. By the transversality arguments in \cite[Chapter 2]{SeidelZurich}, the solutions of \eqref{eq:perturbedpassive} are transversely cut-out for a generic choice of $J$ and the 0-dimensional part of the moduli space of solutions of \eqref{eq:perturbedpassive} gives a finite set. Given a generic family $J(s,t)$ of admissible almost complex structures, we denote by $M(\alpha,J(s,t))$ the moduli space of solutions of \cref{eq:perturbedpassive}. The cobordism type of $M(\alpha,J(s,t))$ is invariant under deformations of $J(s,t)$, $l(s,t)$ and $\alpha$. Observe that since we are in cohomological grading, the induced continuation map sends a generator in $L\pitchfork \fibre{\alpha(0)}$ at $s=+\infty$ to $L\pitchfork \fibre{\alpha(1)}$ at $s=-\infty$. This explains why we are using $\inva$, and the inverse of the Hamiltonian isotopy instead. \\
%We can run the same discussion using $\e \scurve$ instead of $\scurve$. To relate the different scaling parameters, we can use the following trick. For $\e>\e'>0$, we can consider the following family of 

The parallel transports for $\wcflm$ are defined by counting the solutions of \eqref{eq:perturbedpassive}: intuitively, it is the number of points in $M(\ppp,J(s,t))$. Since only the cobordism type of $M(\ppp,J(s,t))$ is invariant, the count must be such that a canceling pair of points contributes zero to the count. Specifically, to count solutions of \eqref{eq:perturbedpassive} we must invariantly orient $M(\ppp,J(s,t))$.\\

%{\RC Spin structure is used directly in $H_1(\sphsc)$}\\

For that, we first use the following fact, often used in Floer theory: given a bordered\footnote{For our purposes, it is sufficient to consider $C$ being the disk, the strip $\mathcal{Z}$, or the half plane.} Riemann surface $C\sse\C$ and a Lagrangian sub-bundle $E\lr\dd C$ of a bundle $\C^2\lr C$, a boundary condition on the Lagrangian sub-bundle $E$ induces the Cauchy-Riemann operator
\[\overline{\partial}:C^{\infty}(C,\mathbb{C}^2,E)\to C^{\infty}(C,\Omega_C^{0,1}\otimes \mathbb{C}^2),\] 
where $C^{\infty}(C,\mathbb{C}^2,E)$ denotes the space of $C^{\infty}$-sections of $\mathbb{C}^2$ whose restriction to the boundary lies in $E$.

\begin{remark}[On orientations]\label{rmk:orientations}
%{\RC Needs reference}
1. Let $U_2/O_2$ be the Lagrangian Grassmannian of $\mathbb{C}^2$, $\rho:S^1\to U_2/O_2$ a loop of index $0$, and $E_{\rho}$ be the Lagrangian sub-bundle of $\mathbb{C}^2$ over $S^1$ given by $\rho$. If the index of the induced Cauchy-Riemann operator $\bar{\partial}$ on $(D^2,S^1)$ operator is $2$, then $\ker(\bar{\partial})=\rho(1)$ and the homotopy class of the trivialization of $(\mathbb{C}^2,E)\simeq (\mathbb{C}^2,\mathbb{R}^2)$ determines the orientation on $\det \bar{\partial}$; cf.~e.g.~\cite[Chapter 6]{FOOO2}.\\

2. The homotopy class of the trivialization is detected by $\pi_2(U_2/O_2)$: the second Whitney class defines a map $w_2:\pi_2(U_2/O_2)\to \mathbb{Z}_2$ and we can identify $\pi_2(U_2/O_2)\simeq \mathbb{Z}\langle H\rangle$. Under the natural homomorphism $\pi_2(U_2/O_2)\to \pi_1(O_2)$, $H$ is the generator that gets mapped to the family of orthonormal frames on $\mathbb{R}^2$, whose lift to the sphere bundle $\mathbb{R}^2\times S^1$ winds around the the unit sphere fiber once. The orientation of $\bar{\partial}$ depends only on the mod $2H$ class; declaring the orientation induced by the trivial loop to be $+1$, the orientation induced by $H$ is the opposite one \cite[Chapter 6.2]{FOOO2}. \hfill$\Box$
\end{remark}

\noindent Technically, the stalks as defined in \cref{def:sbchaincomplex} implicitly use an orientation of the determinant line bundle of the associated linearized Cauchy-Riemann operator, which would themselves be tensoring the stalks $V_{z^i}$ in the Floer complex. Let us precise on the choice of these orientations, which thus trivialize these line bundles and legitimize our choice of stalk in \cref{def:sbchaincomplex}. Near each intersection point in $\blag\pitchfork\fibre{z}$, we use a specific linearized Cauchy-Riemann operator of the constant half-disk at each $z^{i}\in \blag\pitchfork \fibre{z}$, described as follows. The configuration tuple $(T^{\ast}\obis;\blag,\fibre{z},\omega_{st})$ is locally isomorphic to 
$(\mathbb{C}^2;\left< u_1,u_2\right>,i\mathbb{R}_{(x,y)}^2, \omega_{std})$, where $u_1,u_2$ is a unitary frame. Consider $U$ the unitary matrix with columns $u_1,u_2$ and the path of unitary matrices
$$\phi:[0,1]_t\lr U_2,\quad \phi(t):=i\exp(-t\log(U)),$$
each of which acts on $\mathbb{C}^2$ as an isometry by left multiplication. Consider the path of linear Lagrangians
$$L_{t}:=\phi_t(\left< u_1,u_2\right>)\sse T^{\ast}\mathbb{C}^2$$
from $\langle u_1,u_2\rangle$ to $i\mathbb{R}^2$. The path $\phi_t$ takes an orthonormal basis on $\left< u_1,u_2\right>$ to an orthonormal basis on $i\mathbb{R}^2$ and the square phase of this as a path in the Lagrangian Grassmannian $U_2/O_2$ is $e^{-2\pi t}$ and so the grading of the path becomes $-2t, t\in [0,1].$ The path $L_t$ defines a linearized Cauchy-Riemann operator of the constant half-disk at each $z^{i}\in \blag\pitchfork \fibre{z}$ by the moving boundary condition on $L_t$. 

We now discuss the specific trivialization to choose on the moving boundary part $\fibre{\alpha(s)}$.  Given a $\snetwork$-adapted path $\ppp$, the metric $g$ on $\obis$ defines a complex structure $I$ and we can consider the frame $F=\left<\ppp(a),I\ppp(a)\right>$ at $T_{\pi(\ppp(a))}\obis$. Choose the frame on the tangent space of $\fibre{z}\pitchfork \blag$ to be the one given by the image of $\exp(\log(U))$ on the dual frame $F^{\ast}=\left<\ppp^{\ast}(a),I^{\ast}\ppp^{\ast}(a)\right>$. The trivial spin cover over the frames $(\phi_t\circ\phi_1^{-1})(F)$ gives an orientation on the linearized Cauchy Riemann operator, which is then identified with the trivial line bundle. This is our trivialization, and thus we conclude the discussion on orienting the moduli spaces $M(\ppp,J(s,t))$ of deformed continuation strips, solving \cref{eq:perturbedpassive}.

We can now construct the continuation maps without the assumption that the Lagrangian $K$ in \cref{prop:standardfloertheory}.(3) is geometrically bounded, which is needed in this framework of spectral networks as we must apply it to $P=\fibre{z}$ a cotangent fiber. The result reads as follows:

\begin{proposition}[Continuation of $CF^\ast$ under moving fiber]\label{prop:twistedcontinuation}
%{\RC Where is the proof? Isn't this defined afterwards? Also, there is no $\snetwork$ here at all, this should be moved way after.}
Let $(S,\mkpts,{\bf \lknot})$ be a Betti surface, $L\sse(T^*S,\la_\std)$ a Betti Lagrangian, $\snetwork$ a compatible spectral network and $V\in\loct(L)$. Then:
\begin{enumerate}
    \item(Well-defined) Given $\ppp:[0,1]\lr\sphb$ be a $\snetwork$-adapted path from $z_0=\ppp(0)$ and $z_1=\ppp(1)$, the signed count of points in the (oriented cobordism type of the) moduli spaces $M(\ppp,J(s,t))$ of solutions to \cref{eq:perturbedpassive} determines a continuation map
    \begin{align}\label{eq:spcontmap}
    \wcfle{V}(\rho):CF^\ast(\fibre{\pi(z_0)},(\e\scurve,V))\to CF^\ast(\fibre{\pi(z_1)},(\e\scurve,V))
    \end{align} 
    which is independent of the representative of the smooth homotopy class $\ppp$.
    
    \item(Twisted equivariance) For any $\ppp:[0,1]\lr\sphb$ be a $\snetwork$-adapted path, the continuation map in $(1)$ satisfies
    $$\wcfle{V}(H{\ppp})=-\wcfle{V}(\ppp),$$
    where $H$ is the homotopy class of the cotangent circle.\\
    
    \item(Concatenation morphism) Given two $\snetwork-$adapted path $\ppp_1,\ppp_2$, the continuation map in $(1)$ satisfies
    $$\wcfle{V}(\ppp_1\circ\ppp_2)=\wcfle{V}(\ppp_1)\circ \wcfle{V}(\ppp_2).$$
\end{enumerate}
\end{proposition}
\color{black}

\begin{proof}(of \cref{prop:twistedcontinuation})
It suffices to orient the moduli $M(\ppp,J(s,t))$ using the twisted local system $V\in\loct(L)$ such that it becomes cobordism invariant with respect to $\ppp$. Indeed, once that is established, the cobordism argument in Floer theory, cf.~e.g.~\cite{SeidelZurich}, concludes invariance (1) and the concatenation property (3).
%Since the proof will show that the space is well-defined up to oriented cobordism, we will drop $J(s,t)$ and write $M^V(\ppp)$ instead. 
To orient, consider a deformed continuation strip $u$ traveling between two lifts $z^i_0$ to $z^j_1$ in $\sphsc$, i.e.~with positive punctures on these two points, the lower boundary mapped to $\blag$ and the upper boundary to the moving fibers along $\rho$. Choose a sphere bundle lift of the path $u(s,1)$, $s\in (-\infty,\infty)$, so that it agrees with $\phi_1^{-1}(\ppp(a))$ at $z^i_a$, for all $a=0$ and $a=1$. The induced metric on $\obis$ gives a framing of $\obis$ along $u(s,1)$. Considering the two points $z^i_0,z^j_1\in\sphsc$ as constant half-disks, we can glue them to the ends of the strip $u$ and obtain the glued disk $z^i_0 \sharp u \sharp z^j_1$. As discussed above, this data defines a linearized Cauchy Riemann operator over the disk $z^i_0 \sharp u \sharp z^j_1$ which has index $2$. Its orientation depends on the choice of the trivialization of the pull-back of the tangent bundles of the Lagrangian spaces: our choice for orientation is to concatenate the sphere bundle lift of $u(s,1)$ traveling through $\blag$, the paths $(\phi_t\circ\phi_1^{-1})(F)$ and then (the dual frame given by) the sphere bundle path $\ppp$ traveling through the moving fibers. This choice is independent of the sphere bundle lift of $u(s,1)$ and, by construction, twisting the path $\ppp$ by $H$ reverses the sign of the orientation, which established (2).% Compatibility with concatenation of paths, and invariance under homotopy of $\ppp$ follows from standard cobordism argument, cf. \cite{SeidelZurich}. 
\end{proof}

We now construct the local systems $\ff_{\e}$ and $\ff$. Firstly, make a choice of the regular Floer datum\footnote{This datum must also contain a time-dependent Hamiltonian term in the case that $z$ is a branch point, c.f.~ \cite[Section 8(e)]{SeidelZurich}.} for each pair $(\fibre{z},\blag)$, $z\in \obis$, so that we get the stalk given by the (abstract) Floer cohomology groups $HF(\fibre{z},\blag)$. Furthermore, the parallel transport is well-defined once the Floer datum has been chosen at each intersection point, since the moduli space of solutions of \cref{def:continuation_strip} is well-defined, up to cobordism. Thus, regardless of the choice of the regular Floer datum, the family Floer twisted local system, regarded as a functor $\Pi_1(\mathring{\obis})\to \text{Vect}_n$, is well-defined, up to unique natural isomorphism.\\

\noindent The dependence on the adiabatic parameter $\e$ is incorporated as follows. By \cref{thm:snetworkadg}, for any precompact subset $U \subset \obis-\snetwork$ there exists $\e_U>0$ such that for all $0<\e<\e_{U}$, the Floer datum $\defac^{\e}$ for $(\fibre{z},\blag,\defac), z\in \overline{U}$, is regular. Choosing $\defac^{\e}$ to be the regular Floer datum for $(\fibre{z},\blag)$ for $z$ outside $\overline{U}$, we obtain a local system, which we consider of as a functor $\wcfle{V}:\Pi_1(\mathring{S})\to Vect_k$. Furthermore, by our description of $\defac^{\e}$, given any $U\subset U'$, and scaling stability (\cref{lem:scalestability}) the induced natural isomorphism between $\wcfle{V}$ and $\ff_{\e'}(V)$ acts as the identity on morphisms that represent paths that are contained fully in $U$, for $\e'<\min(\e_{U'},\e)$. 

\begin{definition}\label{def:fflocsys}
The family Floer cohomology local system $\ff$ is the direct limit local system $\varinjlim\ff_{\e}$.\hfill$\Box$
\end{definition}

\noindent The discussion above implies that $\ff$ is well-defined up to unique isomorphism. By scaling stability, whenever we compute parallel transport maps of $\ff$, we may compute it for $\e$ small enough. The results of the current section then can be regarded as showing that the local system $\ff$ coincides with $\tdfd$, over objects that lie outside $\snetwork^c$. \\

\color{black}
We conclude this section with a discussion on how the continuation maps in \cref{prop:twistedcontinuation} can be enhanced to record (twisted) homotopy classes, while remaining invariant. Specifically, let $M(\ppp;J(s,t),V)$ denote the moduli space $M(\ppp,J(s,t))$ equipped with the orientation (from $V$) as in the proof of \cref{prop:twistedcontinuation}. For each deformed continuation strip $u\in M(\ppp;J(s,t),V)$, we denote by $\ccc{u}$ the sphere bundle lift  of $u(s,0)$ and by $\sigma(\ccc{u},\ppp)$ the sign of ${\det (D \bar{\partial}_{J} u)}_{\ccc{u},\ppp}$, the oriented determinant line of the linearized Cauchy-Riemann operator obtained using $\ccc{u}$ and $\ppp$.

\begin{corollary}[Homotopy refinement]\label{cor:homotopy_refinement_continuation_map} In the notation above, the homotopy refinement of the deformed continuation map in \cref{prop:twistedcontinuation} is
    %Specifically, \item to take coefficients in $\mathbb{C}[\pi_1^{V,tw}(\sphsc;\ppp(0),\ppp(1))]$. Each $u\in \mathcal{M}^V(\ppp;J(s,t))$ comes with a choice of the sphere bundle lift $\ccc{u}$ of $u(s,0)$. Let ${\det (D \bar{\partial}_{J} u)}_{\ccc{u},\ppp}$ denote the oriented determinant line of the linearized Cauchy-Riemann operator obtained using $\ccc{u}$ and $\ppp$. Let $\mu(\ccc{u},\ppp)$ {\RC Isn't this notation particularly misguiding?} denote the sign of ${\det (D \bar{\partial}_{J} u)}_{\ccc{u},\ppp}$. Specifically, in coordinates this map is
    \begin{align}\label{def:formalsumpseudostrips}
        \wcfl{V}(\ppp)= \sum_{u\in M(\ppp;J(s,t),V)} \sigma(\ccc{u},\ppp) V(\ccc{u})[u(-s,1)],
    \end{align}\\
    and it only depends on the homotopy class of $\rho$, and is compatible with concatenation.
\end{corollary}

\noindent An advantage of \cref{def:formalsumpseudostrips} over \cref{eq:spcontmap} is that the former is invariant under the $H$-action, as an $H$-twist in the sphere bundle lift $\ccc{u}$ is canceled out by the parallel transport $V(\ccc{u})$. We shall momentarily show in \cref{ssec:detour_are_continuation_strips} that the parallel transport $\ff(V)$ as in \cref{def:formalsumpseudostrips} yields, in the adiabatic limit, the homotopy refinement of $\Phi_\snetwork(V)$. 

\subsubsection{Ancillary discussion on spin structures}\label{sssec:spin_structures} The (twisted) Floer complexes constructed above used the sphere bundles $\sphb,\sphsc$ instead of working directly with $S,L$, each endowed with spin structures. Two reasons are that, with this groundwork, it is reasonably direct to improve the above discussion to include the protected spin characters $y$-variable in \cite{MR3395151} and, more technically, it allows us to package the discussion on the dependence on spin structures directly in the sphere bundle classes. The former is achieved by keeping track of the homotopy class on the sphere bundle, without forcing the twisted evaluation $y\to -1$. Indeed, it is possible to show that the solutions of \cref{eq:perturbedpassive} are immersed and thus have canonical lifts to the unit sphere bundles, where their (relative) homotopy classes can be recorded.

To compare the Floer complexes constructed above using the sphere bundles, cf.~\cref{prop:twistedcontinuation}, with the Floer complexes one would define using $S,L$ and their spin structures, we proceed as follows. First, suppose that $\spins,\spinb,\beta$ are all trivial near $\blag\pitchfork \fibre{z}$, and that the frame bundle of $\blag$ has its torsor at $\overline{z}\in \fibre{z}\pitchfork \blag$ so that it is equal to the one given by the $\phi_t^{-1}$ image of some chosen frame $\left< e_1,e_2\right>$ on $i\mathbb{R}^2$. Taking the trivial spin cover over the framing $(\phi_t\circ \phi^{-1}_{1})\left< e_1,e_2\right>$ gives a path of spin structures on $L_t$ and the induced orientation is the trivial one. Then consider the sphere bundle lift of the path $\alpha$ and the sphere bundle lift of the path $\blag_t$, as prescribed by $\sphb$ and $\spins$. The orientation on the induced linearized Cauchy-Riemann operator then determines the sign and it coincides with that of the holonomy of $\ff(V)$ along $u(s,0)$. As a consequence of \cref{thm:familyFloer} and this discussion, we can summarize the outcome as follows:
\begin{corollary}\label{cor:spin}
Let $(\obis,\spinb)$ and $(\blag,\spins)$ be endowed with spin structures, and consider $\locsys\in\Loc(L)$ and $V\in\loct(L)$ such that $\spins^{\ast}V=\locsys$. Then the family Floer local system obtained using $\locsys$ is the $\spinb^{\ast}$-pullback of $\ff(V)$.\hfill$\Box$
\end{corollary}

%%%%%%%%%%%%%%%%%%%%%%%%%%%%%%%%%%%%%%%%%%%%%%%%%%%%%
%%%%%%%%%%%%%%%%%%%%%%%%%%%%%%%%%%%%%%%%%%%%%%%%%%%%%
%%%%%%%%%%%%%%%%%%%%%%%%%%%%%%%%%%%%%%%%%%%%%%%%%%%%%

\subsection{Equivalence between $\ff$ and $\Phi_\snetwork$}\label{ssec:detour_are_continuation_strips} The aim of this section is to prove the following result, comparing the functors $\ff$ and $\Phi_\snetwork$, i.e.~the twisted local systems obtained from Family Floer and those obtained via the combinatorics of spectral networks:

\begin{theorem}[$\ff$ and $\Phi_\snetwork$ are equivalent]\label{thm:familyFloer}
Let $(\obis,\mkpts,{\bf \lknot})$ be a Betti surface, $\blag\sse(T^*S,\la_\std)$ an exact Betti Lagrangian, $(S,g)$ a compatible metric and $\snetwork$ an adapted Morse spectral network. Then the two functors
$$\ff,\Phi_\snetwork:\loct_1(L)\lr\loct_n(S)$$
are equivalent. In particular, the combinatorial corrections from detour paths on $\snetwork$ coincide with the Floer-theoretic corrections from deformed continuations strips.
\end{theorem}

%{\RC Add that Hom commutes with $\Phi_\snetwork$, so we focus on objects.}\\

\noindent To compare these two functors it suffices to focus on the objects, as morphisms naturally commute with the construction. The statement of \cref{thm:familyFloer} also holds for the enhancements of those functors that keep track of specific (relative) homotopy classes in the unit sphere bundles. \cref{thm:familyFloer} is proven by analyzing the behavior of $\ff$ in the adiabatic limit $\e\to0$. Namely, understanding the deformed continuation strips from \cref{eq:perturbedpassive} with one boundary condition on $\e\blag$ and the other (moving) boundary condition on cotangent fibers as $\e\to0$. Since $\snetwork$ controls the adiabatic behavior of pseudo-holomorphic strips, there are two regions to consider:

\begin{enumerate}
    \item In the complement $\snetwork^c$ of the spectral network, the parallel transport of $\Phi_\snetwork$ is trivial. Therefore, we need to show that the deformed continuation strips defining the parallel transport for $\ff$ are also trivial if the path lies in $\snetwork^c$.\\

    \item For $\snetwork$-adapted short paths, we must show that the adiabatic limit of deformed continuation strips across a wall of $\snetwork$ is exactly given by the trivial strips {\it and} strips whose flowtree limit coincides with the detour paths introduced in \cref{subsubsection:pathdetourclasses}.
\end{enumerate}

In order to prove these facts above, we start with a quantitative refinement of \cref{eq:perturbedpassive}, estimating the behaviors of deformed pseudo-holomorphic strips in terms of their length. Specifically, consider $z\in \obis$ such that $\fibre{z}$ is transverse to the Betti Lagrangian $\blag$ and fix a unit tangent vector $v\in\sphb$ base at $z$. Employing the Riemannian metric $(S,g)$, consider the path $\alpha_v(s)=\exp_z(sv)$ starting at $z$ in the direction of $v$. For $s$ sufficiently small, $\alpha_v$ is an embedding and $\fibre{\alpha_v(s)}$ remains transverse to $\blag$. This latter condition implies in particular that there are no $\dfs$-singularities in a small enough neighborhood of the image of $\alpha_v$. % We write $\ppp(\alpha_v)$ for the unit velocity lift to $\sphb$.\\
We model such a path by considering local coordinates $(x_1,x_2)$ near $z=x_1+ix_2$ so that $\alpha_v(s)=(s,0)$, and then the (local) Hamiltonian $H^v:T^*\R^2_{x_1,x_2}\lr\R$, $H^v(x_1,x_2;p_1,p_2)=-p_1$ generates the transport of $\fibre{z}$ along $\bar{\alpha}_v(s)$. To achieve compactly supported behavior, localizing this fiber transport along $\bar{\alpha}_v(s)$, let us cut-off $H^v$ to a Hamiltonian $b(x_1,x_2)H^v$ with $b(x_1,x_2)$ a bump function concentrated at $\alpha_v(s)$, i.e.~ $b(x_1,x_2)$ is constant equal to $1$ in a neighborhood of $\alpha_v(s)$ and vanishes outside a slightly larger neighborhood. To simplify, we still denote by $H^v$ the cut-off Hamiltonian, denote by $X_v$ its Hamiltonian vector field and by $\psi_v$ its flow. The system \cref{eq:perturbedpassive} defining deformed continuation strips can now be specialized to paths $\alpha_v(s)$, where dependence on the unit tangent vector $v$ is acquired and $H^v$ is fixed.
%{\RC Where does the cutoff go in the equation? In $X_v$?}
%{\YJN We redefine $H^v=H^vb(x_1,x_2)$}

%We want to study the following class of continuation strips.
\begin{definition}
Let $\e\in\R_+$ and $J(s,t)$ be an $\e$-admissible family of almost complex structures on $\strip$, and consider an elongation function $\elon(s)$ and a length $d\in\R_{+}$. By definition, a map $u_d:\strip\lr T^{\ast}\obis$ is said to be a continuation strip along $\alpha_v(s):[0,\dist]_s\to \bis$  if it solves the system
\begin{align}\label{eq:distperturbedpassive}
	\begin{cases}
		\dd_s u_d+{J}(s,t)\dd_t u_d-\dist\cdot l'(s)X_{v}(u_{\dist})=0 &  \\ 
		u_d(s,1)\subset \blag&\\
		u_d(s,0)\subset \fibre{\bar{\alpha}_v(\dist\cdot l(s))}&\\
		\displaystyle\lim_{s\to - \infty} u_d(s,t)\in \fibre{\alpha_v(\dist )}\cap \blag\\
		\displaystyle\lim_{s\to +\infty} u_d(s,t)\in \fibre{z}\cap \blag.
	\end{cases}
    \end{align}
We often denote $u(s,t)=u_d(s,t)$, understanding the dependence on $d$ implicitly.\hfill$\Box$
\end{definition}

\noindent To ease notation, given a value for the adiabatic parameter $\e\in\R_+$, we assume that we have fixed an $s$-dependent deformation function $\rho_s(\radial)$, as in \cref{sssec:framework_floer}, and an $\e$-admissible family of almost complex structures refers to an admissible family of almost complex structures obtained by deforming $\defac^{\rho_\e}$, instead of $\defac$.  Our first task is choosing the elongation function $l(s)$ such that solutions to \cref{eq:distperturbedpassive} of small enough length have controlled behavior. This is the content of the following:

\begin{proposition}\label{prop:fibreparalleltransport} In the notation for \cref{eq:distperturbedpassive}, there exists a choice of an elongation function $\elon(s)$ and a constant $\dist_v\in\R_+$ such that the following holds: 
\begin{enumerate}

\item The continuation strips $u_{\dist}(s,t)$ along $\alpha_v(s)$ of length $\dist\in(0,d_v)$ whose limiting endpoints lie in the same smooth sheet of $\blag$ cannot leave the $d_v$-neighborhood of their endpoints.\\
%{\RC What does convergence mean here? And the ``else''?}
%{\YJN Say, from $\alpha_v(0)^i$ to $\alpha_v^i(1)$}

\item For the continuation strips $u_{\dist}(s,t)$ along $\alpha_v(s)$ of length $\dist\in(0,d_v)$ whose limiting endpoints lie in different smooth sheets of $\blag$, there exists a sequence of lengths $d_n\to0$ such that the associated family $u_{\dist_n}$ of strips Gromov converges to a broken $J(s,t)$-pseudo-holomorphic strip bounded between $\blag$ and $\fibre{z}$. 
%gauged transformed continuation strips $(\psi^{-1}_v)_{l(s)}(u_{\dist}(s,t))$ are necessarily constant.
\end{enumerate}
\end{proposition}
\begin{proof}

%See the proof of \cite[Proposition 6.11]{nho2024familyfloertheorynonabelianization} and \cite[Eq.6.11]{nho2024familyfloertheorynonabelianization}. 

%that  which is why the choice of $\elon(s)$ becomes necessary. 
%The same computation as in \cite[Lemma 6.9]{nho2024familyfloertheorynonabelianization} tells us  
For Part (1), let $u=u_d(s,t)$ be a solution of \cref{eq:distperturbedpassive} and consider its gauge transform $\tilde{u}=\psi^{-1}_{l(s)}u$. Then the symplectic area of $\tilde{u}$ is
\begin{align}\label{eq:energy_transformed_strip}
\int_\mathcal{Z} \tilde{u}^{\ast}\omega=-\int_{u(s,1)}\la_{st}-\dist\cdot \int H^v(u(s,1))ds=W((\psi^{-1}_1 \circ u)(-\infty,0))-W(u(+\infty,0))-\dist\cdot\int H^v(u(s,1))ds,
\end{align}
where $dW=\la_\std|_L$ is a primitive of the Liouville form along $L$. Along $\alpha_v(s)$, let $h>0$ be such that over some small neighbourhood of $\alpha_v(s)$, the smooth sheets of $\epsilon \blag$ have distance at least $h>0$ apart, with respect to the Sasaki metric. For each $z\in \alpha_v(s)$, let $\cup B_{h/2}(z^i)\subset T^{\ast}\obis$ denote the small neighbourhood of the intersection points between $\fibre{z}$ and $\blag$. Let $\ell\in \R_{+}$ be a constant so that $\defac$-holomorphic half-disks with energy less than $\ell/2$ intersecting $\cup_{i} B_{h/4}(z^i)$ cannot leave the larger neighbourhood $\cup_{i} B_{h/2}(z^i)$, for all $z\in \alpha_v(s)$. 

Choose an elongation function $\elon(s)$ so that $\elon'(s)$ is supported on $[0,\ell]$. By the hypothesis that the endpoints lie on the same sheet, the limiting endpoints of the gauged transformed strips have to be the unique intersection point of $\fibre{z}$ and $\blag$ lying on the same sheet as the endpoints of $u_{\dist}$. By \cref{eq:energy_transformed_strip}, the symplectic area (a.k.a.~energy) of the gauge-transformed strip is of the form $O(d)$ for for small enough. Therefore, the graph of $\psi_{d\cdot l(s)}^{-1}(u(s,t))$ on the moving part has energy $O(d)+l$.

In addition, for $d$ small enough, the symplectic 2-form making the movie $(\psi_{d\cdot l(s)}^{-1}(\blag),s)$ Lagrangian can be made arbitrarily close to $\omega_{st}\oplus \omega_{st}$ and the metric induced by $\psi_{d\cdot l(s)}^{\ast}J(s,t)$ can be made arbitrarily close to the correspondingly standard metric. Thus, \cite[Proposition 3.37]{nho2024familyfloertheorynonabelianization} implies that the graph of $\psi_{l(s)}^{-1}(u(s,t))$ on the moving part cannot leave the $(3\ell)/2$-neighborhood of the intersection points between $\blag$ and $\fibre{z}$ for small enough $\dist$, which concludes Part (1). For Part (2), it suffices to verify that the symplectic area of the family of solutions \eqref{eq:distperturbedpassive} admits a uniform positive lower bound. This follows from the fact that $u_d(s,t)$ travels between distinct sheets of $\blag$  (see \cite[Eq.6.1.10]{nho2024familyfloertheorynonabelianization}). Therefore, the Gromov compactness argument developed in \cite[Section 6.1.5]{nho2024familyfloertheorynonabelianization} applies and $u_{\dist}(s,t)$ Gromov converges to a broken $J(s,t)$-pseudo-holomorphic strips bounded between $\blag$ and $\fibre{z}$ for a subsequence of lengths $d_n\to0$.
\end{proof}

%By definition, a continuation strip $u$ is said to be \textit{gauge-constant}, resp.~{\it gauge-small}, 

By definition, a continuation strip $u$ is said to be {\it small} if it is as in \cref{prop:fibreparalleltransport}.(1).
%s gauged transformed continuation strips $\tilde{u}=\psi^{-1}_{l(s)}u$
%constant, resp.~ 

\begin{remark}
Note that there is a difference between \cref{prop:fibreparalleltransport} and \cite[Proposition 6.11]{nho2024familyfloertheorynonabelianization}: in the result above, the Hamiltonian isotopy does not necessarily fix the neighborhood of the intersection points. In order to correct this matter, we need to be very careful with the choice of the  elongation function $l$.\hfill$\Box$
\end{remark}
%We introduced the term $l'(s)X_v(u_d)$ in the continuation strips, cf.~\cref{eq:distperturbedpassive}, to have the freedom to choose

After a small perturbation of $J(s,t)$, we can and do assume that the moduli space of same sheet continuation strips consist of small continuation strips, and that it is transversely cut-out. 
\begin{lemma}\label{lem:gaugesmallconstant}
In the notation above, for small enough $\dist\in\R_+$, the moduli space of small continuation strips of length $d$ is cobordant to a singleton moduli space, consisting of a single strip whose boundary on $\blag$ may be identified with a lift of the arc $\alpha_v(s),s\in (0,d_v)$.  
\end{lemma}
\begin{proof}
Near $z$, a smooth sheet of $\blag$ can be described a graph of a smooth function $f$. As observed in \cite{FloerWittenInfinite}, we may regard $f$ as a (locally-defined) Hamiltonian on $T^{\ast}\obis
$ and write $\e\cdot \text{gr}(f)$ as the image of the zero section under the time $\e$-flow $\Psi$ of $f$. Taking the inverse of the fibre-preserving \textit{global} Hamiltonian isotopy $\Psi^{-1}_{\e}$, the configuration $(\e\blag, \fibre{z},\defac)$ transforms into $(Z,\fibre{z},\Psi_{\e}^{-1}(\defac))$. Under this transformation, the perturbation term in \cref{eq:perturbedpassive} turns into  $(d\Psi^{-1}_{\e})_{\ast}(X_v(u_d))$. Since we are only interested in small continuation strips, whose diameter we can a priori control, using the monotonicity argument in the proof of \cite[Proposition 6.11]{nho2024familyfloertheorynonabelianization}, we are free to interpolate the perturbation term from $(d\Psi^{-1}_{\e})_{\ast}X_v$ to $l'(s)X_v$. We claim that the solutions of the new equation are all constant, up to gauge-transformation. This leads to a cobordism of the moduli space of small continuation strips, whose other end consists of a single strip.

To verify the claim, consider the solutions of \cref{eq:perturbedpassive} with $J(s,t)$ replaced by $\psi^{\ast}_{\e}J(s,t)$ and the boundary condition $\blag$ replaced by the zero section. Since the Hamiltonian fixes the zero section, we are back to the situation of \cite[Proposition 6.11]{nho2024familyfloertheorynonabelianization} where the Hamiltonian fixes the Lagrangians near the intersection points. Therefore, after taking the gauge-transform $\Psi_s^{-1}u(s,t)$, the strip becomes \textit{constant}. As claimed, we may identify its boundary on $\blag$ with a lift of the arc $\alpha_v(s),s\in (0,d_v)$ to the smooth sheet $\text{gr}(f)$.
\end{proof}
\noindent Strips as in \cref{lem:gaugesmallconstant} are said to be \textit{trivial strips}.\\

\noindent We introduce the following notions, analogous to \cref{def:diskterm}:

\begin{definition}\label{def:contdiskterm}
Let $\blag\sse(T^*S,\la_\std)$ be a Betti Lagrangian and $\snetwork$ an adapted finite creative Morse spectral network.

\begin{enumerate}[label=(\roman*)]
\item A point $z\in \obis$ is said to be {\it continuation-free} in the adiabatic limit if there exists $\e(z)\in\R_+$ and a neighbourhood $U_z$ such that for any $\e\in(0,\e(z))$, $(\fibre{w},\blag)$ is $\defac^{\e}$ disk-free for $w\in U_z$, and there exists a length $\dist(z,\e)\in\R_+$ and an elongation function $l(s,\epsilon)$ such that for all points $w\in U_z$ there are \textit{no} $\defac^{\e}$-holomorphic continuation strips that are not small along arcs of the form $\alpha_v:[0,d]_s\to \obis$, for $d\in (0,\dist(z,\e))$, $w\subset U_z$ and $v\in T_w\obis$ a unit tangent vector at $w$\\

\item A point $z\in \snetwork$ different from a vertex is said to have {\it Stokes transport in the adiabatic limit} if there exists $\e(z)\in\R_+$ and a length $d(z)\in \R_{+}$ such that for any $\e\in(0,\e(z))$, there exists an elongation function $l(s,\e)$ for which there are \textit{no} large $\defac^{\e}$-holomorphic strips for $\blag$ along the arc $\alpha_v(s):[-\dist(z),\dist(z)]_s\to \obis$, that do not belong to the path detour class along $\alpha_v(s)$, where $v$ denotes the outward normal vector to $\snetwork$ at $z$. 

%where denotes the outward normal vector to $\snetwork$ at $z$. \hfill$\Box$
\end{enumerate}
\end{definition}

\noindent Note that vertices of $\snetwork$ do not need to be considered in \cref{def:contdiskterm} because the behavior near them has been fixed by the local models. We must now establish the analogue to \cref{thm:snetworkadg} for continuation strips, which reads as follows:

%With \cref{prop:fibreparalleltransport}, we can apply \cref{thm:snetworkadg} to arrive at the following proposition. 

%(Comment: Is this a theorem?)
\begin{proposition}\label{prop:continuationstripclassification}
Let $\blag\sse(T^*S,\la_\std)$ be a Betti Lagrangian and $\snetwork$ an adapted finite creative Morse spectral network. Then the following holds:

\begin{enumerate}[label=(\roman*)]
    \item If $z\in \snetwork^c$, then $z$ is continuation-free in the adiabatic limit.\label{prop:continuationstripclassification1}
    \item If $z\in \snetwork$ and not a vertex, then $z$ has Stokes transport in the adiabatic limit. \label{prop:continuationstripclassification2}  
    %Let $z\in \snetwork$ be a non-vertex point with index $ij$. Let $v$ denote the outward normal vector to $\snetwork$ at $z$. Then there exists some $\e(z)>0$ and $\dist(z)>0$ such that for there are \textit{no} non-constant $J(s,t)=\defac^{\e}$ strips for $\blag$ along arcs $\alpha_v(s),s\in (-\dist(z),\dist(z))$ that do not travel \textit{from} the $j$th sheet to the $i$th sheet. The same statement holds for small enough perturbations of $J(s,t)$. Furthermore, if the continuation strip travels from the $j$th sheet to the $i$th sheet, then we can choose small perturbations of $\defac$ and $\dist(z)>0$ such that the homotopy class of the boundary of the continuation strip along $\alpha_v(s):s\in [-\dist(z),\dist(z)]$ is the unique non-trivial path detour class along the arc $\alpha_v(s)$. 
    \end{enumerate}
\end{proposition}
\begin{proof}
For Part (1), fix a precompact open subset of $\snetwork^c$ whose closure does not intersect $\snetwork$. We argue by contradiction: suppose that there is a sequence of points $z_n$, unit tangent vectors $v_n$ of $z_n$, with lengths $\dist(z_n)$ converging to zero but such that there are $\defac^{\e}$-holomorphic continuation strips which are not small; for $\blag$ and along arcs $\alpha_{v_n}(z_n)$. By \cref{prop:fibreparalleltransport} such a sequence has a subsequence that converges to a broken $\defac$-holomorphic strip. By the integrated maximum principle and the equality $\defac^{\e}=J_{\e^{-1}g}$ along the disk bundle containing $\blag$, $\defac^{\e}$-holomorphic strips have to necessarily be $J_{\e^{-1}g}$-holomorphic strips. Then \cref{eq:salamon-weber equvialence} and \cref{thm:snetworkadg} together imply that such broken $J_{\e^{-1}g}$-holomorphic strips cannot exist, thus a contradiction. This establishes Part (1).

For part (2), let $U_z$ be the Stokes neighbourhood of $z$. By \cref{thm:snetworkadg}, we see that there exists some $\e(z)>0$ such that for all $\e\in (0,\e(z))$, the only disks in $(\fibre{w},\e\blag)$ must be parallel transports of the soliton classes. We claim that the maximal arc $\alpha_v(s):[-d_z,d_z]_s\to \obis$ contained in such a neighbourhood is the desired one. To see this, shrink $\e(z)$ so that $(\fibre{\alpha_v(0)},\blag)$ and $(\fibre{\alpha_v(1)},\blag)$ are all $\defac^{\e}$ disk-free, for all $\e\in (0,\e(z))$. Then we can combine \cref{prop:fibreparalleltransport} and the previous argument to conclude that there exists an elongation function $l(s)$ such that for each $t\in [-d_v(s),d_v(s)]$, there exists $d_t>0$ such that along $\alpha_v(t+s):[0,d]_s\to \obis$, for $d\in (0,d_t)$, the large $\defac^{\e}$-continuation strips must have the homotopy class of the detour path across $z$ (deformed to have the end-points of $\alpha_v(t+s)$, in a straightforward way). Indeed, this follows from \cref{thm:snetworkadg} because the path homotopy class of its boundary must converge, as $d\to 0$, to broken chains of $\soliton{z}$, none of which can concatenate. Then, using the compactness of $\alpha_v(s)$, we can split $\alpha_v(s)$ into finitely such segments. We can now perturb slightly to make sure the continuation strips are all transversely cut-out. By Gromov compactness, having no large continuation strips of certain homotopy class is a condition invariant under small enough perturbations of $J(s,t)$. We then apply the "solitons cannot concatenate" argument again to conclude that the glued continuation strips must be of the path detour class. We now turn off the perturbation, and argue by Gromov compactness again, to conclude that for some elongation function, the homotopy class must also be the path detour class, concluding (2). 

\end{proof}
%\color{purple}
\noindent In addition, a simpler argument leads to the following type of stability under $\e$-scaling:
\begin{lemma}[Scaling stability]\label{lem:scalestability}
For $z\in \snetwork^c$, there exists $\e(z)\in\R_+$ such that for all $\e',\e\in(0,\e(z))$, $\e'<\e$, the continuation map $CF(\fibre{z},\blag;\defac^{e})\to CF(\fibre{z},\blag;\defac^{e'})$ is the identity map.
\end{lemma}
\begin{proof}
The proof is similar to Part (1) in \cref{prop:continuationstripclassification}. Take $\e(z)$ as in \cref{thm:snetworkadg} so that $(\epsilon\blag,\fibre{z})$ is disk-free for all $\e\in(0,\e(z))$. By the previous discussion, we obtain a family $\defac^{s}$ such that $\defac^{s}=J_{s^{-1}g}$ in the disk bundle for $\e'\leq s\leq \e$.
Fix an elongation function $l$, then for each $\e'<a'<a<e$, the continuation map from $CF(\fibre{z},\blag;\defac^{a})$ to $CF(\fibre{z},\blag;\defac^{a'})$ is given by counting pseudo-holomorphic strips with the boundary conditions
\begin{align}\label{eq:scalecont}
	\begin{cases}
		\dd_s u+{J}(s,t)\dd_t u=0 &  \\ 
		u(s,1)\subset \blag&\\
		u(s,0)\subset \fibre{z}&\\
		\displaystyle\lim_{s\to \pm\infty} u(s,t)\in \fibre{z}\cap \blag,\\
		%\displaystyle\lim_{s\to -\infty} u(s,t)\in \fibre{z}\cap \blag.
	\end{cases}
\end{align}
using the family almost complex structures $J(s,t)$ that are given by small perturbations of the family $\defac^{l(s)a+[1-l(s)]a'}$. Since the equation is homogeneous, the strips that travel between the same lifts of $\fibre{z}$ must all be constant, and the strips that travel between distinct lifts of $\fibre{z}$ must have energy a priori bounded below by some small $h\in\R_+$. The same argument as in \cref{prop:continuationstripclassification} implies that for $\abs{a'-a}$ small enough, the continuation strips are all constant, and so they induce the identity map. The rest follows from concatenation compatibility, from which we conclude that the continuation map from $s=\e$ to $s=\e'$ must be the identity map.
%We argue first that for the family $\defac^{l(s)a'+[1-l(s)]a}$, there cannot be no non-constant strips, for $\abs{a'-a}$ small enough. Indeed, the Gromov compactness argument developed in \cite[Section 6.1.5]{nho2024familyfloertheorynonabelianization} together with \cref{prop:regulardata} applies to show that as $\abs{a'-a}\to 0$, such strips must converge to a broken chain of $J_{\e^{-1}g}$-disks bounded between $\fibre{z}$ and $\blag$, but $(\blag,\fibre{z})$ are disk-free for all $0<\e<\e(z)$. Therefore, all the continuation strips are constant for $\abs{a-a'}$ small enough. We then use compactness of $[e',e]$ to split it into short segments, overwhich the continuation strips are all constant, and so they all induce the identity map.  
\end{proof}
\color{black}

\noindent Let us finally establish the equivalence between $\ff$ and $\Phi_\snetwork$:

\begin{proof}[Proof of \cref{thm:familyFloer}] Let us first compare the parallel transports for a path of the form $\rho(s)=\alpha_v(s)$ starting at $z\in\snetwork^c$ and $s$ small enough. Since any $\Phi_\snetwork(V)$ has trivial parallel transport for $V\in\loct(L)$, c.f.~\cref{def:detourpath}.(i), we must argue that the same holds for $\ff(V)$, for $\e$ small enough. Fix a small neighbourhood $U_{\rho}\subset \snetwork^c$ of $\rho$. By \cref{thm:snetworkadg}, there exists some $\e(\rho)>0$ such that for all $0<\e<\e(\rho)$, $(\fibre{z},\e\blag),z\in U_{\rho}$ is disk-free. Fix $0<\e<\e({\rho})$, then by \cref{sssec:holonomy_floer_locsys}, the parallel transport of the latter is given by the moduli spaces of deformed continuation strips, cf.~\cref{prop:twistedcontinuation}. By \cref{lem:gaugesmallconstant} and \cref{prop:continuationstripclassification} the corresponding moduli spaces of continuation strips for such a path away from $\snetwork$ is cobordant to the moduli space of trivial strips. For the orientations, it suffices to note that the sphere bundle counts of the trivial strips give the required orientation, see e.g.~\cite[Lemma 6.14]{nho2024familyfloertheorynonabelianization}. This implies that the count of deformed continuation strips is that of trivial strips, and therefore the parallel transport of $\ff(V)$ along such $\alpha_v(s)\sse\snetwork^c$ equals that of $\Phi_\snetwork(V)$. By concatenation compatibility, we conclude that the parallel transports agree over any free $\snetwork$-adapted path.\\

For each $z\in \snetwork$, we can consider a $\snetwork$-adapted short path of the form $\rho(s)=\alpha_v(s)$, intersecting $\snetwork$ at $z$, just as in \cref{def:contdiskterm}. By Gromov compactness, we can perturb by a small amount so that the moduli space of large continuation strips is transversely cut-out, and so that the boundary homotopy classes belong to the path detour class. 
Let $u(s,t)$ be the resulting continuation strip along such path: we must show that the assignment $\soliton{z;\wall}\to \sigma(\ccc{u},\alpha_v(s))$ (where the count is defined using the homotopy refinement \cref{cor:homotopy_refinement_continuation_map}) gives a functor with the properties as in \cref{def:bpsindex}. For each vertex $v\in\snetwork$, introduce nearby points $b\in\snetwork$ on each edge adjacent to $v$ and for each such $b$, consider the corresponding short path arcs $\alpha_{\nu}^{(b)}(s)$, intersecting $\snetwork$ in the positive direction. We now proceed on inductively, by starting with the points close to the initial walls, and then following the birth of new walls along $\snetwork$. First, for these points $b$ close enough to the initial rays of the $\dfs$-vertex $v$, direct computation gives the count $\sigma(\ccc{u},\alpha_v(s))=1$, e.g.~c.f.~\cite[Section 5.6, Figure 21]{GNMSN}. It thus suffices to show that $\sigma(\ccc{u},\alpha_v(s))$ satisfies the same properties as in \cref{def:bpsindex} near the creation joints of $\snetwork$. At the same time, that these properties are satisfied is a consequence of \cref{prop:twistedcontinuation} and the fact that the trivial loop encircling the vertex of $\snetwork$ is contractible, and thus the associated (cobordism class of) moduli space of continuation strips consists of constant strips (one per intersection point), up to cobordism. The cancellation of the concatenation of non-trivial continuation strips along the trivial loop is encoded in the 2d Hori-vafa formula \cref{def:bpsindex}. Therefore, since the counting functions are the same, the two local systems $\tdfd$ and $\wcflm(V)$ must also be isomorphic.
\end{proof}

\noindent \cref{thm:pathdetourclassesareHaminvariant} now follows from \cref{thm:familyFloer}:

\begin{proof}[Proof of \cref{thm:pathdetourclassesareHaminvariant}]
By \cref{thm:familyFloer}, the parallel transport $\Phi_\snetwork(V)(\ppp(\reeb))$ is equivalent to that of $\ff(V)(\ppp(\reeb))$. In the adiabatic limit, the deformed continuation strips contributing to this latter parallel transport converge to $\dfs$-flowtrees with semi-infinite root edges asymptotic to $\reeb$, and the counts are in fact equal. Indeed, by \cite[Lemma 5.12]{EHKexactLagrangiancobordisms} and \cite[Theorem 1.5]{EHKexactLagrangiancobordisms}, there exists a bijective correspondence between rigid flow-trees for (a perturbed and neck-stretched version of)\footnote{This modification of $\blag$ is denoted by $\pneck$ in \cite{Morseflowtree,EHKexactLagrangiancobordisms}.} $\blag$ and rigid index-$0$ pseudo-holomorphic disks (adjusted, in the sense of \cite[Section 3.2]{EHKexactLagrangiancobordisms}) with boundary on that exact Betti Lagrangian, all constrained to be asymptotic to $\reeb$. In general, the cobordism type of the latter moduli space of pseudo-holomorphic disks is not an invariant because of bifurcations given by index $-1$ disks, see e.g.~ \cite[Lemma 3.13]{EHKexactLagrangiancobordisms}. That said, \cref{lem:Maslov0} shows that Betti Lagrangians $\blag\sse(T^*S,\la_\std)$ have Maslov index $0$ and thus such index-$(-1)$ disks cannot appear. Therefore, the cobordism type of this moduli space is an invariant up to compactly supported Hamiltonian isotopy and, consequently, so is the count of $\dfs$-trees asymptotic to $\reeb$. 
\end{proof}

%%%%%%%%%%%%%%%%%%%%%%%%%%%%%%%%%%%%%%%%%%%%%%%%%%%%%
%%%%%%%%%%%%%%%%%%%%%%%%%%%%%%%%%%%%%%%%%%%%%%%%%%%%%
%%%%%%%%%%%%%%%%%%%%%%%%%%%%%%%%%%%%%%%%%%%%%%%%%%%%%

\subsection{Proof of \cref{thm:characterization}}\label{ssec:characterization_endproof} The object of this subsection is to conclude the proof of \cref{thm:characterization}. The direction $(\Longleftarrow)$ of the characterization in \cref{thm:characterization} was established in \cref{section_adg}, c.f.~\cref{thm:snetworkadg}. We must thus show the implication $(\Longrightarrow)$. This will be first done for exact Betti Lagrangians in \cref{sssec:proof_characterization_exact}, using the results of this section (c.f.~\cref{prop:continuationstripclassification}) and then \cref{sssec:proof_characterization_mero} is devoted to the meromorphic WKB case. To align with \cref{def:diskterm}, the existence of a pseudo-holomorphic disk in the adiabatic limit can be made precise as follows:%{\RC why did $\e$ depend on $U$ here?}
\begin{definition}
The pair of Lagrangians $(\fibre{z},\blag)$ bounds a pseudo-holomorphic disk in the adiabatic limit if there exists a constant $\e_0\in\R_+$ such that for all $\e\in(0,\e_0)$ there exists a neighborhood $U_{\e}(z)\sse S$ and a point $w\in U_{\e}(z)$ in that neighborhood for which $(\e \blag,\fibre{w})$ bounds a non-constant pseudo-holomorphic disk.\hfill$\Box$
\end{definition}

The results established thus far, specifically \cref{thm:snetworkadg} and \cref{prop:continuationstripclassification}, now suffice to conclude \cref{thm:characterization} for exact Betti Lagrangians.

\subsubsection{Proof of \cref{thm:characterization} for exact Betti Lagrangians}\label{sssec:proof_characterization_exact}
Let $z\in \snetwork$ be a point in a wall of the spectral network, we must argue that there exists a pseudo-holomorphic strip between $T_z^*S$ and $L$ in the adiabatic limit. By \cref{thm:snetworkadg}, given any neighborhood $U\sse S$ of $z$, we can find a nested neighborhood $W\subset U$ and a constant $\e_0\in\R_+$, depending on $U$, such that the pair of Lagrangians $(\fibre{w},\e\blag)$ is disk-free for all points $w\in U\setminus W$ and $\e\in(0,\e_0)$. Consider a $\snetwork$-adapted short path $\alpha_v$ with its endpoints in $U\setminus W$ and intersecting $\snetwork$ and $z$. By the same argument as in the proof of \cref{thm:familyFloer}, there exists some $\e_0'\in\R_+$ such that for all $\e\in(0,\e_0')$ the Floer-theoretic parallel transport $\ff(V)(\mathring{\alpha}_v(s))$ equals that of $\tdfd(\mathring{\alpha}_v(s))$, where $\mathring{\alpha}_v(s)$ denotes the lift to the unit cotangent bundle $\sphb$. The latter parallel transport is non-trivial, because the 2d-4d BPS counting indices from \cref{subsec:snetworkandpathdetours2_BPS_index} are non-zero at all walls of $\snetwork$, as $\snetwork$ is creative. Thus the former parallel transport $\ff(V)(\mathring{\alpha}_v(s))$ is non-trivial. This implies the existence of a point $w\in W$ such that $(\fibre{w},\e \blag)$ is not disk-free, as otherwise \cref{prop:continuationstripclassification} would imply $\tdfd{(\mathring{\alpha}_v(s)))}$ was trivial. In conclusion, as the adiabatic parameter $\e\to0$ tends to zero, this concludes the existence of a pseudo-holomorphic strip between $L$ and $T^*_z\obis$ in the adiabatic limit, as required.

\subsubsection{Proof of \cref{thm:characterization} for meromorphic spectral curves}\label{ssec:characterization_meromorphic_scurves}

%%%%%%%%%%%%%%%%%%%%%%%%%%%%%%%%%%%%%%%%%%%%%%%%%%%%%
%%%%%%%%%%%%%%%%%%%%%%%%%%%%%%%%%%%%%%%%%%%%%%%%%%%%%
%%%%%%%%%%%%%%%%%%%%%%%%%%%%%%%%%%%%%%%%%%%%%%%%%%%%%
%\color{orange}
There are a few differences between the case of a meromorphic spectral curve $\scurve\sse(T^*S,\omega_\C)$ and an exact Betti Lagrangian $L\sse(T^*S,\la_\std)$. That said, when studying pseudo-holomorphic curves with respect to a generic meromorphic spectral curve $\scurve\sse(T^*S,\omega_\C)$, we will now establish results showing that the behaviors that occur are {\it as if} we were under an exactness assumption. In that sense, the proof follows closely that of the exact case and we only highlight the key ingredients. First, we have the following strengthening of the Stokes condition in \cref{def:diskterm}. 

\begin{lemma}\label{lem:stokesonstokesline}
Let $z\in \snetwork$ be a non-initial point. Then, for every $E\in\R_+$, there exists a neighbourhood $U_z(E)\sse S$ of $z$ and $\e_0(E,z)\in\R_+$ such that there are no non-constant $J_g$-holomorphic strips bounded between $\e \scurve$ and $\fibre{w}$ of energy less than $E$, for $w\in \snetwork^c(E)$. 
\end{lemma}
\begin{proof}
From \cref{thm:snetworkadg}, such pseudo-holomophic strips must have the homology class of the parallel transport of the solitons with energy less than $E$. %However, \cref{lem:liouvilleintegral} implies that 
Since the integral of the holomorphic Liouville 1-form along the boundary of such a strip must be real and positive, $z$ must lie on $\snetwork(E)$.
\end{proof}
%{\YJN How do I address the situation where the fibre is very close to the branch point?}

The following notation is useful:

\begin{definition}
Let $\blag\sse T^*S$ be a Betti Lagrangian. By definition, $\blag$ is said to be disk-free in the adiabatic limit, if for every $E\in\R_+$ there exists $\e_0(E)\in\R_+$ such that for all $\e\in(0,\e_0(E))$, the scaled Lagrangian $\e\blag$ does not bound a pseudo-holomorphic disk of energy less than equal to $E$.\hfill$\Box$
\end{definition}

Then the same argument used for \cref{thm:snetworkadg} implies the following variation, which tells us that $\scurve$ is tautologically unobstructed, up to some energy. In other words, we can run the usual Floer theoretic machinery (see c.f \cite[Section 8(c)]{QuarticHMS2003}) up to some energy cut-off.  

\begin{proposition}\label{prop:nodisk}
Let $\scurve$ be a meromorphic spectral curve with a WKB spectral network $\snetwork$. Then $\scurve$ is disk-free in the adiabatic limit. 
\end{proposition}
\begin{proof}
Let $u_{\e}$ be a sequence of somewhere injective disks of energy less than $E$ with boundary on $\e\scurve$. Introduce three punctures on the boundary to fix the conformal structure and apply the same adiabatic degeneration argument as in \cref{thm:snetworkadg}. By \cref{subsection:asymptoteWKB}, the walls of $\snetwork$ never return to the branch points and thus the initial edge condition implies that the disks cannot close up. Thus $\e\scurve$ must be disk-free for small enough $\e$.
\end{proof}

\noindent The scaling relations is as follows: the adiabatic cut-off $\e_0(E)$ tends to zero as the energy $E$ is sent to infinity, i.e. $\e_0(E)\to0$ for $E\to \infty$. The necessary variant of \cref{prop:fibreparalleltransport} can be discussed as follows.

Consider an energy cut-off $E\in\R_+$. The gapped condition implies that there exists a small enough parameter $h\in\R_+$ such that $\snetwork(E)$ and $\snetwork(E+h)$ are isotopic as stratified 1-manifolds with boundary. For each vertex $v$ of the energy-truncated $\snetwork(E)$, choose a Stokes neighborhood and a bump function $\rho$. Fix a neighborhood of $v$ contained in that of \cref{prop:nodisk}, with diameter $C\delta$ for some constant $C\in\R_+$. From \cref{prop:nodisk}, the pair $(\fibre{z},\e\scurve)$ is disk-free for $z\in B_{C\delta}(v)\setminus\snetwork(E)$. Choose $C$ such that $C\delta$ is less than the injectivity radius of $v$ consider $\alpha(s)$ a geodesic line segment contained in $B_{C\delta}(v)$. \\

Let $\wall$ be an $(ij)$-wall of $\snetwork(E)$, $W:=\int \lambda^i-\lambda^j=x+iy$ be the conformal flat coordinate and $(p^x,p^y)$ dual coordinates to $(x,y)$; here $W$ is defined by a line integral from the given point to the end variable point. For $z\in \wall$, we write $\alpha_z(s)$ for the path given by $s\to z+is$ in the $W$-coordinate. Let $\sup_{\rho}\abs{\blag}$ be the supremum of the norm of the sheets of $\blag$, over a small neighbourhood of $\rho$ that contains the support of the generating Hamiltonian $H^v$. Then observe that the Hamiltonian term $H^v$ in \cref{eq:energy_transformed_strip} is less than $\epsilon \sup_{\rho}\abs{\blag}$. The analogue of \cref{prop:fibreparalleltransport} then reads:
\begin{proposition}\label{prop:merofibreparalleltransportvertex}
In the notation above, there exists $\e_0(z,E)\in\R_+$ such that for all $\e\in(0,\e_0(z,E))$ the following holds:
\begin{enumerate}
\item  There exists $\dist(z,\e)\in\R_+$ and an elongation function such that, given a free $\snetwork$-adapted line segment $\alpha(s)$ contained in $B_{C\delta}$ of length $\e$, the moduli of $J_g$-holomorphic continuation strips of energy less than $\e(E+\hbar)$ are all small.\\

\item For an interior point $z\in \snetwork(E)\cap B_{C\delta}$, there exists $\delta(z,\e)\in\R_+$ and an elongation function such that along the short path $\alpha_z(s)$, $s\in [-\delta(z,\e),\delta(z,\e)]$, any large $J_g$-holomorphic continuation strips of energy less than $\e(E+\hbar)$ must have the homotopy class of a soliton at $z$.
\end{enumerate}
\end{proposition}
%{\RC We do not know how to control continuation strips with high energy (length of the arc is too large), so we need to break it into smaller pieces and always near the spectral network (we can do that because no disks outside). Effectively only allowing infinitesimal displacement of the fiber.}\\
 \begin{proof}
 The argument follows that of \cref{prop:fibreparalleltransport}. For (1), choosing the elongation function as in \cref{prop:fibreparalleltransport}, since the Hamiltonian term contribution in the energy formula is of size $d\cdot  \abs{\sup_{\rho}{\blag}} \epsilon \cdot \int l'(s)ds $, for $d$ small enough any pseudo-holomorphic strip bounded by $(\fibre{z},\e \scurve)$ which comes from breaking off a continuation strip must have energy less than $\e E$. Therefore, for any $s_0\in [0,1]$, there exists $\eta(s_0)\in\R_+$ such that the moduli of $J_g$-continuation strips along $[s-\eta(s_0),s+\eta(s_0)]$ is cobordant to the moduli of trivial strips. For (2), a similar argument applies: the key point is that for $J_g$, we can fix the homotopy classes of continuation strips with energy bounded above by $\leq E+\delta$.  
 \end{proof}

Since there are finitely many vertices in $\snetwork(E)$, let us choose $\delta\in\R_+$ such that \cref{lem:stokesonstokesline} and \cref{prop:merofibreparalleltransportvertex} hold. Given an $(ij)$-wall $\wall$, which is compact, we choose a (small) neighborhood $U_{\wall}\sse S$ of $\wall$ so that \cref{lem:stokesonstokesline} holds for all $w\in U_{\wall}\setminus\snetwork(E)$. We also (slightly) truncate $\wall$ off near the starting ends if $\wall$ begins at the initial vertices of $\snetwork(E)$. The following proposition has exactly the same proof as \cref{prop:merofibreparalleltransportvertex}:

%{\YJN Be more careful with the end-point conditions}
\begin{proposition}\label{prop:merofibreparalleltransportwall}
In the notation above, there exists $\e_0(v)\in\R_+$ such that, for all $\e\in(0,\e_0(v,E))$, there exists $\delta(\wall,\e)>0$ such that the following conditions hold for strips along the path $s\to z+is$, $s\in [-\delta(\wall,\e),\delta(\wall,\e)]$:
\begin{enumerate}
\item For any $z\in U_{\wall}\setminus\snetwork$, the $J_g$-holomorphic continuation strips of energy less than $\e(E+\hbar)$ are all small.\\

\item For any $z\in \wall\cap U_{\wall}$, then the large $J_g$-holomorphic continuation strips of energy less than $\e(E+\hbar)$ all have the homotopy class of a soliton at $z$.
%\item Let $z\in \wall\cap U_{\wall}$. Consider the path given by the concatenation of paths $s\to z+i\delta(\wall)-s$, $s\to z-\delta(\wall)+i(\delta(\wall)-s)$ and $s\to z+(s-1)\delta(\wall)-i\delta(\wall)$  $s\in [-\delta(\wall),\delta(\wall)]$
\end{enumerate}
%Furthermore, these properties are invariant under $\e$-small perturbations of $J_g$.
\end{proposition}
In general, $J_g$-holomorphic continuation strips are not going to be regular in this content. Fortunately, given the energy bound, all the relevant moduli space of continuation strips are cobordant with the given energy cut-off (here the perturbation of the almost complex structure must be chosen in a way that it agrees with $J_g$ on the boundary mapping to $\blag$, to avoid boundary bubbling, c.f \cite[Section 8(c)]{QuarticHMS2003}). Indeed, the energy formula for short enough arc implies that the dominant term in the energy is $\int \la_{st}$ over the boundary mapping to $\e\blag$, the Hamiltonian term being proportionately small, which is invariant up to homotopy. Therefore, the energy of the perturbed continuation strip is determined by the boundary homotopy class.\footnote{See also the energy estimates in \cite[Section 5.3.5]{FOOO2}; in their terminology, the Hofer norm of the Hamiltonian isotopy as observed by the disk is small enough.} It follows that we can regularize the moduli space by choosing some unspecified perturbation of $J_g$ to make the strips regular. The following terminology is convenient:

\begin{definition}\label{def:disksupport}
Let $\wall$ be an active wall of $\snetwork(E)$ and $z\in \wall$. By definition, $z$ supports a disk in the adiabatic limit if for any $E\in\R_+$ there exists $\e(z)\in\R_+$ such that for all $\e\in(0,\e_0(z))$ there exists $\delta(z,\e)\in\R_+$ such that the moduli space of large $J_g$-holomorphic continuation strip along $\alpha_z(s)$, $\abs{s}<\delta(z,\e)$, of energy less than $\e E$ is non-empty and its regularization is not cobordant to the empty set.\\

\noindent By definition, a wall $\wall$ is said to support a non-constant disk in the adiabatic limit if the set of points $z\in \wall$ that support a disk in the adiabatic limit is non-empty.\hfill$\Box$
\end{definition}

\begin{proof}[Proof of \cref{thm:characterization} in meromorphic case]
First, consider a wall $\wall$ of $\snetwork(E)$. Let us show that the set of points on $\wall$ that support a disk in the adiabatic limit is open and closed. Indeed, openness is a consequence of Gromov compactness and regularity of strips being an open condition. For closedness, note that \cite[Proposition 4.4]{nho2024familyfloertheorynonabelianization} shows the set is non-empty. Consider then a sequence of points $z_n\to z$ such that each support a disk in the adiabatic limit. For large enough $n\in\N$, $z_n$ belongs to the Stokes neighborhood of $z$ and, by assumption, the moduli space of non-diagonal $J_g$-holomorphic continuation strips along $\alpha_{z_n}(s)$ is non-empty. Concatenate the horizontal arc coming from the top of $\alpha_z(s)$ to the arc $\alpha_{z_n}(s)$, and further concatenate it with the horizontal arc connecting to the bottom of $\alpha_z(s)$. The resulting path is isotopic to the arc $\alpha_z(s)$. Now, any two moduli space of $J_g$-holomorphic continuation strips with given energy less than $\e(E+\hbar)$ are cobordant, and since the moduli space of large $J_g$-holomorphic continuation strips along $\alpha_{z_n}(s)$ is not (virtually) null-cobordant, the same holds for $\alpha_z(s)$.\\

Second, we discuss vertices of the spectral network. Suppose that the active ingoing walls all support disks in the adiabatic limit. We claim that this is also the case for the active outgoing walls. Indeed, by \cref{prop:merofibreparalleltransportvertex}, for a small neighborhood of $v$ the moduli space of $J_g$-holomorphic continuation strips for the short paths across the outgoing walls is virtually cobordant to the gluing of the ingoing moduli spaces. The count of these gluings is controlled by the index in \cref{def:bpsindex}, and so we are reduced to the exact situation.
%The rest of Theorem 2 follows from \cref{{prop:wallsupportclosed}}, \cref{lem:initialwall}  and \cref{lem:joint}. 
\end{proof}
\label{sssec:proof_characterization_mero}
%By \cref{prop:signcomputation}, we see that $\wcfl{\alpha_v(s)}$ for $\alpha_v(s)$ as in \cref{prop:continuationstripclassification} all equal the path detour classes. Our formula \eqref{eq:nonab} satisfies $\wcfl(\alpha\circ\beta)=\wcna{\alpha}\circ \wcna{\beta}$ and so by \eqref{eq:formalsumconcatenation}, we see that the formal sums $\nonab$ and $\wcflm$ must agree. However, $\wcflm$ is manifestly Hamiltonian isotopy invariant under isotopies of $\alpha$, and up to conjugacy, under isotopies of $\scurve$, and therefore so is $\nonab$. Furthermore, since $\wcflm$ gives rises to a twisted local system on the base, so does $\nonab$. This finishes the proof. 
\color{black}
%   \newpage
%%%%%%%%%%%%%%%%%%%%%%%%%%%%%%%%%%%%%%%%%%%%%%%%%%%%%%%%%%%%%%%%%%%%%%%%
%%%%%%%%%%%%%%%%%%%%%%%%%%%%%%%%%%%%%%%%%%%%%%%%%%%%%%%%%%%%%%%%%%%%%%%%
%%%%%%%%%%%%%%%%%%%%%%%%%%%%%%%%%%%%%%%%%%%%%%%%%%%%%%%%%%%%%%%%%%%%%%%%
\section{Spectral networks and 4d Fukaya Categories}\label{section_wfc}
In this section we relate spectral networks and Fukaya categories in 4-dimensions, in particular proving \cref{thm:specnet_fuk}. Specifically, we show how exact Betti Lagrangians can be understood as objects of a Fukaya category in two related ways, cf.~Sections \ref{subsection:WFC} and \ref{subsection:PWFC}, and how the parallel transport rules dictated by a spectral network $\snetwork$ are in fact encoding $\mu_2$-type morphisms for the underlying $A_\infty$-structures naturally appearing in Floer theory, cf.~\cref{thm:Familyfloerwrapped}, proven in \cref{subsec:prooffamilyfloerwrapped}, and \cref{thm:pwfcyonneda}, proven in \cref{subsection:PWFC}. An explicit description of the associated Yoneda $A_\infty$-module over $\pwfc$ is also provided in \cref{subsection:partialyonneda}. See \cite{GPSCV} for references and context on (partially) wrapped Fukaya $A_\infty$-categories. In brief, the present section enhances the Floer-theoretic characterization of spectral networks and non-abelian parallel transport, as presented in Sections \ref{section_adg} and \ref{section_Floer}, to the $A_\infty$-categorical framework. In this section, we will focus our attention on exact Betti Lagrangians.

%%%%%%%%%%%%%%%%%%%%%%%%%%%%%%%%%%%%%%%%%%%%%%%%%%%%%
%%%%%%%%%%%%%%%%%%%%%%%%%%%%%%%%%%%%%%%%%%%%%%%%%%%%%
%%%%%%%%%%%%%%%%%%%%%%%%%%%%%%%%%%%%%%%%%%%%%%%%%%%%%

\subsection{The Yoneda module of an exact Betti Lagrangian}\label{subsection:WFC} Let $\wfc$ denote the wrapped Fukaya $A_\infty$-category of the open Liouville sector $(T^*S,\la_\std)$, and $\kmod$ the dg-category of chain complexes of $k$-modules, understood as an $A_\infty$-category with higher $\mu_k$ vanishing, $k\geq2$. A Betti Lagrangian $L\sse(T^*S,\omega_\std)$ does {\it not} define an object in $\wfc$, due to the nature of its asymptotic ends near the marked points, cf.~\cref{section_setup}. Thus, even if we endow a Betti Lagrangian $L$ with a local system\footnote{Since \cite{GPSCV} uses local systems, not twisted, we use spin structures for Betti Lagrangians to orient. The relation between twisted local systems and spin structures is discussed in \cref{ssec:lsys_vs_tlsys}. We can use \cref{cor:spin} to pass from spin structure count to twisted count. Alternatively, one could develop \cite{GPSCV} for twisted local systems.} $V\in\Loc(L)$, it is a priori unclear how Betti Lagrangians and their spectral networks fits within the context of Fukaya categories. There are at least two solutions to address this, depending on whether we want to preserve the geometry of $L$ at infinity or we allow geometric modifications of $L$ at infinity. The first approach is presented in this section and the second in \cref{subsection:PWFC}.\\

For the first approach, the algebraic idea is simple: if $C\in\mbox{Ob}(\EuScript{C})$ is an object in a category, then $\mbox{Hom}_{\EuScript{C}}(\cdot,C)\in \mbox{Fun}(\EuScript{C}^{op},\kmod)$ defines a module over $\EuScript{C}$. This module is often known as the Yoneda module. Since a Betti Lagrangian $L\notin\mbox{Ob}(\wfc)$ is not an object, as it does not have compact horizontal support, we cannot merely construct a Yoneda module in this manner. Nevertheless, given a Betti Lagrangian $(L,V)$ endowed with a local system $V\in\Loc(L)$, we now construct an $A_\infty$-module over $\wfc$ which behaves {\it as if} it was the Yoneda module of an object given by $(L,V)$. Intuitively, its value on objects $\tlag_1,\ldots,\tlag_k\in\mbox{Ob}(\wfc)$ is given by counting pseudo-holomorphic polygons with boundary on the Lagrangians $\tlag_1,\ldots,\tlag_k$ and the Betti Lagrangian $\blag$. Consider the Floer cohomology groups $HF^\ast(P,(L,V))$ as in \cref{section_Floer}, the result reads as follows:

\begin{proposition}\label{prop:A-inftymodule}
Let $L\sse(T^*S,\omega_\std)$ be a Betti Lagrangian endowed with a local system $V\in\Loc(L)$. Then there exists an $A_{\infty}$-module
$\ylagv:\wfc\lr\kmod$ with cohomologies
$$H^{\ast}(\ylagv(\tlag))\cong HF^\ast(\tlag,\blag),$$
where $\tlag\sse T^*S$ is any exact cylindrical Lagrangian with compact horizontal support.%{\RC Is this just for fibers $F$?}
\end{proposition}

\begin{proof} The intuition is to formally add the Betti Lagrangian $\blag\sse T^*S$ as the most negative object in $\wfc$. We provide the necessary details, as follows. By \cite[Prop. 3.43]{GPSCV}, the wrapped Fukaya $A_\infty$-category $\wfc$ can be obtained as the strictly unital directed $A_\infty$-category of a certain (terminal) decorated poset $\dset(S)_{ter}$ for $(T^*S,\omega_\std)$ localized at the set of continuation morphisms, as discussed in Section 3.8 in ibid. To begin, consider a decorated poset $\dset(S)$, not necessarily terminal, which satisfies the following conditions:
%$\blag$-adapted if the following holds.
\begin{itemize}
    \item[(a)] Each Lagrangian $\tlag_k$ in the poset is transverse to $\blag$, and each family of almost complex structure can be extended to an almost complex structure on $T^{\ast}\obis$, that agrees with $\defac$ outside some uniform compact base.\\
    
    \item[(b)] For each chain $p_0>p_1>\ldots>p_k$, the moduli space of pseudo-holomorphic polygons bounded by the (ordered) tuple of Lagrangians $(\tlag_{p_0},\tlag_{p_1},\ldots,\tlag_{p_k},\blag)$ is also transversely cut-out and compatible with gluing.
\end{itemize}
Such decorated poset exists after perturbing the almost complex structure. This requires the fact that the relevant pseudo-holomorphic polygons do not escape off to infinity. Namely, that given the exact Betti Lagrangian $\blag$, the (pairwise transverse) cylindrical Lagrangians $\tlag_1,\dots,\tlag_k\sse T^{\ast}S$ and $\{J_s\}$ a family of almost complex structures on $T^{\ast}\obis$, we need that any $J_s$-holomorphic polygon with boundary conditions given by $(\tlag_1,\dots, \tlag_k,\blag)$ has uniformly bounded diameter, and the same holds with moving Lagrangian boundary conditions on $\tlag_1,\ldots,\tlag_k$ given by positive isotopies. This geometric condition does hold in our case: $\blag$ is contained in a finite-radius disk bundle, it is uniformly geometrically bounded with respect to $J_s$ and, being exact, the energy of such polygons is a priori bounded by the action of the intersection points. Therefore, the monotonicity argument in \cite[Proposition 3.29]{nho2024familyfloertheorynonabelianization} applies to guarantee such condition.

We then construct an $A_\infty$-module $\ylagv^{pre}$ over the poset in $\dset(S)$ with its $A_{\infty}$-operations given by counting pseudo-holomorphic polygons between the Lagrangians $\{P_i\}$ in $\dset(S)$. For such chain of Lagrangians $\{P_i\}$, we define the map
$$\mu_{k-1\vert 1}:\mathcal{O}(\tlag_{1},\tlag_{2})\otimes\dots \mathcal{O}(\tlag_{k-1},\tlag_k) \otimes \ylagv^{pre}(\tlag_k)\lr \ylagv^{pre}(\tlag_1),\quad \tlag_1>\tlag_2...>\tlag_k,$$
%{\RC What is $\mathcal{O}$ here?}
by counting pseudo-holomorphic $(k+1)$-gons with boundary conditions given by $(\tlag_1,....,\tlag_k,\blag)$, with $\mathcal{O}$ as in Eq.~(3.52) in ibid. We set the $A_{\infty}$-module map to be equal to zero otherwise, except for the case $\mu_{1\vert1}:\ylagv^{pre}(\tlag)\lr \ylagv^{pre}(\tlag)$ where $\dset(\tlag,\tlag)=\mathbb{Z}$ acts as the identity. The above defines $\ylagv^{pre}$ as $A_\infty$-module over the poset in $\dset(S)$, we need to argue that it descends to the localization $\mathfrak{W}_{\dset(S)}$ by continuation morphisms. For that, note that given any positive isotopy of cylindrical Lagrangians $\{K_t\}$, $t\in[0,1]$, the continuation map $CF(K_0,\blag)\to CF(K_1,\blag)$ is invertible. Indeed, since $\blag$ lies in a finite disk bundle, $\blag$ and $K_t$ do not intersect at infinity, thus the conclusion follows from the argument in \cite[Lemma 3.21]{GPSCV} and the final assertion in \cite[Lemma 3.26]{GPSCV}.  It also follows that $C^\ast(\ylagv^{pre}(\tlag))\simeq CF^\ast(\tlag,\blag)$.

It now suffices to relate $\mathfrak{W}_{\dset(S)}$ to $\wfc$. For that, consider the category of decorated posets which satisfy $(a)$ and $(b)$ above, are downward-closed and duplicate-free. Denote its terminal object by $\dset(S,L)_{ter}$, cf.~\cite[Lemma 3.42]{GPSCV}. The inclusion $\dset(S,L)_{ter}\sse\dset(S)_{ter}$ of posets induces an inclusion functor between the associated localized categories, and thus $\mathfrak{W}_{\dset(S)}$ admits a tautological inclusion functor into $\wfc$, which is in fact an equivalence. By precomposing $\ylagv^{pre}$ with the inverse of this equivalence, we obtain the required $A_\infty$-module $\ylagv$.%{\RC What does this mean after inverting, probably just composing?}
\end{proof}

\begin{remark} We use the following convention for composition and wrapping: the $\mu_2$-maps are given as
$$\mu_2:CF(A,B)\otimes CF(B,C)\to CF(A,C)$$
and, when wrapping is involved, a morphism $A\to B$ has its domain $A$ more positive than its codomain $B$, i.e.~positive wrapping gives $CF(A^w,A)\otimes CF(A,B)\to CF(A^w,B)$. In particular, the Betti Lagrangian $L\sse T^*S$, e.g.~a spectral curve, is always taken on the right slot of $CF(\cdot,\cdot)$, as it is the most negative object. Similarly, in the partially wrapped setting, the positive push-off of the (cylindrized) Betti Lagrangian is the most negative object, in that small negative wrapping makes it fall into $\dd_\infty L$.\hfill$\Box$
\end{remark}

In the same line as the equivalence established in \cref{ssec:detour_are_continuation_strips}, cf.~\cref{thm:familyFloer}, the $A_\infty$-modules $\ylagv$ in \cref{prop:A-inftymodule} are closely related to the non-abelianized local systems, as follows. First, for technical reasons, we normalize the Riemannian metric $(S,g)$ so that $\mbox{inj}(g)>2$, which is achieved after a constant metric rescaling $g\to c^2g$; this is necessary for the proof of \cref{prop:fibrewrappingsgenerate} to estimate action. Now, by definition, two points $z,w\in(\bis,g)$ in a Betti surface with a spectral network $\snetwork\sse S$ are said to be $\snetwork$-adapted if $\mbox{dist}_g(z,w)<\mbox{inj}(g)/20$ and the minimal geodesic $\alpha_{zw}$ between $z,w$ is a $\snetwork$-adapted path (cf.~\cref{def:networkadapated}). Since the metric $(S,g)$ is complete, we can and do choose $w$ in terms of $z$ such that the energy functional on the path space $\Omega_{z,w}$ between $z$ and $w$ is a Morse functional, cf.~ \cite[Theorems 14.2, 18.1 \& 18.2]{MilnorMorse}. This ensures the transversality of intersections of cotangent fibres in $(T^*S,\omega_\std)$ after wrapping by the geodesic flow.\\

\noindent We implicitly identify $\mbox{Hom}_{\wfc}(\fibre{z},\fibre{w}):=CF^\ast(\fibre{z},\fibre{w})$ with chains on the path space $\Omega_{z,w}S$, see \cite{AbW}. In fact, $CF^\ast(\fibre{z},\fibre{z})$ is $A_\infty$-quasi-isomorphic to $C_{-\ast}(\Omega_{z}S)$ endowed with the $A_\infty$-refined Pontryagin product. The reason we focus on cotangent fibers is that \cite{AbC} shows that any cotangent fiber $\fibre{z}\sse(T^*S,\omega_\std)$ generates $\wfc$, and thus we have an $A_\infty$-quasi-equivalence $\mbox{Tw}(\wfc)\simeq \mbox{Tw}(CF^\ast(\fibre{z},\fibre{z}))$. We denote by
$$\mu_{1\vert 1}:\mbox{Hom}_{\wfc}(\fibre{z},\fibre{w})\otimes\ylagv\lr\ylagv$$
the structure map of the $A_\infty$-module $\ylagv$ in these degrees evaluated at two cotangent fibers. The categorical statement relating Floer theory to non-abelianized local system reads as follows:

\begin{theorem}\label{thm:Familyfloerwrapped}
Let $(S,g)$ be a Betti surface, $L\sse(T^*S,\omega_\std)$ a Betti Lagrangian, $V\in\Loc(L)$, $\snetwork\sse S$ an adapted spectral network and $z,w\in S$ an $\snetwork$-adapted pair. Then the minimal geodesic generator $[\alpha_{zw}]\in \mbox{Hom}_{\wfc}(\fibre{z},\fibre{w})$ satisfies
\begin{equation}\label{eq:mu11_paralleltrans}
\mu_{1\vert 1}([\alpha_{zw}],\cdot)=\tdfd(\alpha_{zw}).
\end{equation}
\hfill$\Box$
\end{theorem}

\noindent In brief, \cref{eq:mu11_paralleltrans} equates a map given by Floer-theoretically counting pseudo-holomorphic strips, on its left, to the topological parallel transport given by the spectral network, on the right. Note that the left hand side of the isomorphism \cref{eq:mu11_paralleltrans} can be identified with $\ylagv([\alpha_{zw}])$, understanding $\ylagv$ as being applied to a morphism in $\wfc$. Thus the structure map $\mu_{1\vert 1}$ appearing in \cref{thm:Familyfloerwrapped} categorically (and cohomologically) describes the parallel transport map from $z$ to $w$ when its domain is evaluated at $\mbox{Hom}_{\wfc}(\fibre{z},\fibre{w})\otimes\ylagv(\fibre{z})$. \cref{thm:Familyfloerwrapped} will be proven in a moment, in \cref{subsec:prooffamilyfloerwrapped}, after we discuss a useful corollary.\\
%$\wfc\mbox{-mod}\cong C_{-\ast}(\Omega_zS)$-mod
Given the $A_\infty$-quasi-equivalence $\mbox{Tw}(\wfc)\simeq \mbox{Tw}(CF^\ast(\fibre{z},\fibre{z}))$, the right hand side is a model for $\Loc(\obis)\simeq C_{-\ast}(\Omega_zS)$-mod. Thus, a local system in $\Loc(\obis)$ can be understood as a $\wfc$-module. It is in general challenging to explicitly present $A_\infty$-modules over $\wfc$, as it involves solving partial differential equations, significant preliminary Floer data to be specified and homotopy (co)limits need to be computed. An advantageous consequence of \cref{thm:Familyfloerwrapped} is that we can explicitly describe the class of $A_\infty$-modules over $\wfc$ arising as non-abelianized local system via the $A_\infty$-quasi-equivalence above. Specifically, under these identifications, we have the following

\begin{cor}\label{cor:familyfloerwrapped}
Let $(S,g)$ be a Betti surface, $L\sse(T^*S,\omega_\std)$ a Betti Lagrangian, $V\in\Loc(L)$, and $\snetwork\sse S$ a spectral network adapted to $L$. Then, under the $A_\infty$-quasi-equivalence $\wfc\mbox{-mod}\cong C_{-\ast}(\Omega_zS)$-mod,
$$\ylagv\simeq\Phi_\snetwork(V),$$
as $A_{\infty}$-modules over $\wfc$. That is, the $A_{\infty}$-module
$\ylagv:\wfc\lr\kmod$ is homotopic to the $A_\infty$-module corresponding to $\Phi_\snetwork(V):\Loc(S)\lr\kmod$.\hfill$\Box$ 
\end{cor}

%%%%%%%%%%%%%%%%%%%%%%%%%%%%%%%%%%%%%%%%%%%%%%%%%%%%%
%%%%%%%%%%%%%%%%%%%%%%%%%%%%%%%%%%%%%%%%%%%%%%%%%%%%%
%%%%%%%%%%%%%%%%%%%%%%%%%%%%%%%%%%%%%%%%%%%%%%%%%%%%%

\subsection{Proof of \cref{thm:Familyfloerwrapped} (Floer $\mu_{1|1}$ is spectral transport)}\label{subsec:prooffamilyfloerwrapped}
%{\RC $\mu_2$ is multiplication maps and $\mu_{1|1}$ part os structure map for the module given by Betti Lagrangian. They both count triangles: the latter is $Hom(F_z,F_w)\otimes Y_L(F_z)\lr Y_L(F_w)$. Same geometrically, just passing to modules.}\\
We need to compute the wrapped Floer complex using a model that is also suitable for computing the $\mu_{1\vert1}$-maps in the statement of \cref{thm:Familyfloerwrapped}. For that, we introduce such a model in \cref{sssec:good_model}, then compute the structure maps in \cref{sssec:computing_str_maps}, and conclude the argument in \cref{sssec:conclusion_familyfloerwrapped}.

%For this purpose, we will need the notion of  \textit{positive isotopy} from \cite[Section 3]{GPSCV}. 

%%%%%%%%%%%%%%%%%%%%%%%%%%%%%%%%%%%%%%%%%%%%%%%%%%%%%
%%%%%%%%%%%%%%%%%%%%%%%%%%%%%%%%%%%%%%%%%%%%%%%%%%%%%

\subsubsection{A model for computing the wrapped Floer complex}\label{sssec:good_model} Given two nearby points $z,w\in S$, the goal of this subsection is to establish that appropriately chosen positive isotopies applied to $\fibre{z}$ exactly lead to a degree-0 cycle in $\mbox{Hom}_{\wfc}(\fibre{z},\fibre{w})$ represented by the minimal geodesic $\alpha_{zw}$ between $z$ and $w$, with this morphism complex given by direct limits as in \cite{GPSCV}.\\

\noindent Recall that an exact isotopy $\{\tlag\}_{t\in[0,1]}$ of cylindrical Lagrangians is said to be positive (at infinity), denoted by $\tlag_0\wto\tlag_1$, if $\alpha(\partial_t(\partial_{\infty}\tlag_t))>0$, where $\partial_{\infty}\tlag_t$ denotes the boundary at infinity of $\tlag_t$ and $\alpha$ is a given contact form. Note that a positive isotopy $\tlag_0\wto\tlag_1$ induces an element $c\in CF(\tlag_1,\tlag_0)$, called the continuation element, cf.~\cite[Section 3.3]{GPSCV}. The relationship between continuation elements and continuation maps is as follows. If $\tlag_0\wto\tlag_1$ and $K\sse T^*S$ is a Lagrangian with $K,P_0,P_1$ pairwise transverse, then the continuation map
$$c(\tlag_t):CF(\tlag_0,K)\to CF(\tlag_1,K)$$
associated to the continuation element $c\in CF^0(\tlag_1,\tlag_0)$ is chain homotopic to the evaluation $\mu_2(c,\cdot)$ of the map
$$\mu_2:CF(\tlag_1,\tlag_0)\otimes CF(\tlag_0,K)\to CF(\tlag_1,K)$$
of the $A_\infty$-structure at the continuation element $c$. In fact, this even holds if we fix $\tlag$ and move $K$ by a negative isotopy. Indeed, \cite[Lemma 3.26]{GPSCV} proves this for $K,\tlag$ having compact horizontal support; for $K$ an exact Betti Lagrangian, the same argument works because of the existence of uniform diameter bounds, as discussed in the proof of \cref{prop:A-inftymodule}. Thus we can and do implicitly switch between continuation elements and continuation maps.\\

In general, the wrapped Floer cohomology groups defined in \cite[Section 3.4]{GPSCV} are given as the direct limit
$$\displaystyle HW^\ast(\tlag,K):=\varinjlim_{(P\wto P^w)^+} HF^\ast(P^w,K)$$
over (cofinal) sequences of positive isotopies. Our task is to choose a suitable such cofinal sequence, which we explicitly present as a sequence $\tlag_0\wto\tlag_1\wto\tlag_1\wto\cdots$ of Lagrangians. We often denote $HF^\ast(\tlag,K)=HW^\ast(\tlag,K)$ when it is clear by context that wrapping must be taken: we always consider a positive wrapping of $\tlag$, in the sense above, though one can equivalently negatively wrap $K$. To define our cofinal sequence, it is convenient to introduce the following notation:

\begin{definition}[Wrapped fibres and fibre wrappings]\label{def:wrappedfibre}
A cylindrical Lagrangian $F\sse(T^*S,\omega_\std)$ is said to be a wrapped fibre if there exists a point $z\in \obis$ and a positive isotopy $\fibre{z}\wto F$ supported on the complement of the unit disk bundle $D^{\ast}\obis$. In this case, the point $z\in\obis$ is said to be the base point of $F$. By definition, a positive isotopy $\{F_t\}_t$ is called a fibre wrapping if $F_t$ is a wrapped fibre for each $t$, and $\{F_t\}$ are uniformly cylindrical and have uniform horizontal support. In this case, the path of base points of $\{F_t\}$ in $S$ is said to be its support arc.\hfill$\Box$
\end{definition}

In order to prove \cref{thm:Familyfloerwrapped} we need a cofinal sequence that relates to minimal geodesics: intuitively, we will consider a cofinal sequence given by fibre wrapping such that the associated continuation elements map to the minimal geodesics. For that, we need to cover $T^{\ast}\obis$ by sectorial covers, as defined in \cite[Section 12.5]{GPSSD}, and we achieve that by modifying the given metric $(S,g)$, as follows. Given the spectral network $\snetwork\sse S$ adapted to $(S,g)$, we consider a finite good cover $\{U_{\alpha}\}$ of $\bis$ such that each $U_{\alpha}$ contains at most one vertex of $\snetwork$, and the open sets $U_{\alpha}$ have their closures contained in $\obis$, except for a distinguished open set for each puncture of $\bis$, which we choose to be contained in the region where the metric is conical. We call the open sets of the former type the  ``internal" open sets. Then, we choose a collar neighborhood of each $\partial U_{\alpha}\sse S$ with a diffeomorphism to $S^1\times (0,1)$ such that:
\begin{itemize}
    \item[-] $\partial U_{\alpha}$ is mapped to $S^1\times \{1/2\}$,
    \item[-] $\dd U_{\eta}$, $\eta\neq\alpha$, and the edges of $\snetwork$ intersect the neighborhood $S^1\times (0,1)$ of $\partial U_{\alpha}$ as an arc of the form $\{p\}\times (0,1)$, $p\in S^1$ a point.
\end{itemize}
With this cover chosen, we modify the metric $(S,g)$ near each such collar of $\partial U_{\alpha}$ so it is of the form $g_{S^1}+dt\otimes dt$, where $g_{S^1}$ is the standard round metric on $S^1$. This can be achieved by an arbitrarily $C^0$-small compactly supported perturbation of $g$ in a manner that $\snetwork\sse S$ stays invariant, up to smooth isotopies supported away from the boundaries of these disks $U_\alpha$. Note that for two $\snetwork$-adapted points $z,w\in \snetwork^c$, the length of the second shortest geodesic will be bounded below by the minimal injectivity radius $\injradius$ of this perturbed metric. For the rest of this section, we assume the metric $(S,g)$  has been modified as just described, and still denote this perturbed metric by $(S,g)$.\\

The sequence of fibre wrappings that we use to prove \cref{thm:Familyfloerwrapped}, whose continuation elements limit to the minimal geodesics, is constructed in (the proof of) the following result:

\begin{proposition}[Fibre wrappings and minimal geodesics]\label{prop:fibrewrappingsgenerate} 
Let $F_z,F_w\sse(T^*S,\omega_\std)$ be wrapped fibres based at $z,w\in S$, two close enough generic points.%{\RC $\snetwork$-adapted? Just close and generic}
Then, there exists a sequence of fibre wrappings
$$F_z\wto F_1\wto F_2\wto\ldots$$
with the base point of $F_1$ equal to $z$ such that:
\begin{enumerate}
    \item $\displaystyle HW^\ast(F_z,F_w)\cong\lim_{i\to \infty} HF(F_i,F_w),$
    \item the image in $HW^0(F_z,F_w)$ of the continuation elements in $HF^0(F_i,F_w)$ is the class $[\alpha_{zw}]\in C_0(\Omega_{z,w})$ represented by the minimal geodesic $\alpha_{zw}$ between $z,w$.
\end{enumerate}
\end{proposition}

\noindent \cref{prop:fibrewrappingsgenerate}.(2) is the part that requires a new type of argument, whereas \cref{prop:fibrewrappingsgenerate}.(1) can be reasonably obtained by modified existing arguments in the literature, as the proof will show, cf.~Lemma \ref{lem:cofinality} and \ref{lem:ind-cofinality} below.

%%%%%%%%%%%%%%%%%%%%%%%%%%%%%%%%%%%%%%%%%%%%%%%%%%%%%
%%%%%%%%%%%%%%%%%%%%%%%%%%%%%%%%%%%%%%%%%%%%%%%%%%%%%

\subsubsection{Proof of \cref{prop:fibrewrappingsgenerate}}
First, we need to introduce the following notion of cofinality:

\begin{definition}\label{def:cofinality}
Let $\tlag_t\sse(T^*S,\omega_\std)$ be a positive isotopy of cylindrical Lagrangians, $t\in[0,\infty)$, and $\{\tbis_t\}$ a nested family of truncations of $\obis$ such that $\tlag_t\subset \tbis_t$ and $\cup \tbis_t=\obis$. By definition, a such a pair $(\tlag_t,\tbis_t)$ is said to be ind-cofinal if:

\begin{enumerate}
\item For any isotopy $\tlag_0\to \tlag^{w}$, there exists some $t\in\R_+$ and a positive isotopy $\tlag^{w}\wto \tlag_t$ such that the composition of $\tlag_0\to \tlag^w$ with $\tlag^w\wto \tlag_t$ is homotopic to a positive isotopy $\tlag_0\wto \tlag_t$. Furthermore, both the isotopy and the homotopy can be taken to be supported in the interior of $T^{\ast}\tbis_t$.\\
%It is wrapping with geodesic flow at infinity, to introduce positivity.
\item Any two positive isotopies $\tlag_0\wto \tlag_t$  coincide, up to homotopy, after composing with $\tlag_t\to \tlag_s$, for some large $s>t$. Furthermore, the homotopy (of positive isotopies) can be taken to be supported in the interior of $T^{\ast}\tbis_s$.
\end{enumerate}
A positive isotopy $\tlag_t$ is said to be ind-cofinal if there exists $\{\tbis_t\}$ such that $(P_t,\tbis_t)$ is ind-cofinal.\hfill$\Box$
\end{definition}

\noindent As in the non-limit case, the ind-cofinality in \cref{def:cofinality} of a positive isotopy $\tlag_t$ can be guaranteed by forcing the boundaries $\partial_{\infty}\tlag_t$ to move fast enough in a Reeb direction in the ideal contact boundary of $(T^*S,\omega_{\std})$. In precise terms:
%in The following is a straightforward extension of \cite[Lemma 3.29]{GPSCV}. 
\begin{lemma}\label{lem:cofinality}
Let $\tlag_t\sse(T^*S,\omega_\std)$ be a positive isotopy. Then, $\tlag_t$ is ind-cofinal if there exists a contact form $\alpha$ on the unit sphere bundle, coinciding with the restriction of $\la_{st}$ outside a compact set, such that
\begin{align}\label{eq:integral_infty}
\displaystyle\int_{0}^{\infty} \inf_{\partial_{\infty}\tlag_t} \alpha(\partial_t(\partial_{\infty}\tlag_t))dt=\infty.
\end{align}
In particular, the time-$t$ the geodesic flow at infinity applied to a Lagrangian $P\sse(T^*S,\omega_\std)$ gives an ind-cofinal family.
\end{lemma}
\begin{proof}
It suffices to modify \cite[Lemma 3.29]{GPSCV}. Since the metric is conical at infinity, the associated geodesic flow is complete. By the hypothesis \cref{eq:integral_infty}, we can reparametrize the family $\tlag_t$ in the $t$-variable such that $\alpha_{st}(\partial_t(\partial_{\infty}P_t))\geq 2$, $\alpha_\std$ being the restriction of $\la_\std$ to the ideal contact boundary, and the horizontal support of $\tlag_t$ is contained in $\tbis{}_t=U_{\radial>f(t)}$\footnote{Here $r$ is the one over the distance to the puncture, so puncture corresponds to $r=\infty$.} for some function $f:\R_+\to\R_+$ such that $f(t)>3t$.\\

Since $\alpha_{st}(\partial_t(\partial_{\infty}P_t))\geq 2$, then
$$P^0\wto\phi_{-t}(\tlag^t)$$ is a positive isotopy because $\alpha_{st}(\partial_t(\partial_{\infty}(\phi_{-t}P_t)))\geq 1$, where $\phi_t$ is the geodesic flow. Independently, we can use the geodesic flow on the inverse isotopy $P^w\to P^0$ to also make it positive and, since the metric is conical at infinity, we can ensure $\phi_t(T^{\ast}\tbis_s)\subset \tbis_{s+t}$. Now that $P^w\to \phi_t(P^0)$ has been made positive, note that
$$\phi_t(\tlag^0)\wto(\phi_t\circ\phi_{-t})(\tlag^t)=P^t$$
is positive because $P^0\wto\phi_{-t}(\tlag^t)$ is positive. This establishes property \cref{def:cofinality}.(1). For \cref{def:cofinality}.(2) the argument is the same as GPS, with the truncation being preserved.
%The metric near the puncture is conical, and so we understand the geodesic flow. The metric is complete, so the geodesic flow is also complete. 
\end{proof}

\cref{lem:cofinality} gives an ind-cofinal sequence, as in \cref{def:cofinality}, and the first part \cref{prop:fibrewrappingsgenerate}.(1) is then a consequence of the following lemma:

\begin{lemma}\label{lem:ind-cofinality} Let $\tlag$ and $K$ be Lagrangians in $T^{\ast}\obis$ with compact horizontal support. Let $\{\tlag_j\}$ be an ind-cofinal sequence with $\tlag_0=\tlag$, then $$
HW(\tlag,K)=\lim_{j\to \infty} HF(\tlag_j,K).$$
\end{lemma}
\begin{proof} 
The sequence $\tlag_i$ regarded as a decorated poset is cofinal in the sub-poset of the final decorated poset $P_{final}$ given by those $\tlag_p$ that belong to the wrapping category of $\tlag_0$, c.f.~\cite[Section 3.4]{GPSCV}. Therefore, the natural map $HF(\tlag_j,K)\to HW(\tlag,K)$ is an isomorphism in the limit.
\end{proof}

Let us now focus on establishing \cref{prop:fibrewrappingsgenerate}.(2). The difficulty here is naming the necessary geometric objects in a homotopically coherent manner. To wit, we must be able to describe the geodesic generator $\alpha_{z,w}$ within the framework of sectorial coverings of \cite{GPSCV}, which is itself a highly non-explicit setting. For that, consider the finite good cover from \cref{sssec:good_model} and the correspondingly perturbed metric $(S,g)$. Then the Liouville sectors $\{T^{\ast}\overline{U}_{\alpha}\}_\alpha$ form a sectorial covering of $T^{\ast}\tbis$\footnote{Using the product decomposition near the boundary, we obtain the projection structure 
\begin{align}\label{eq:productdecomposition}
Nbh^Z(\partial T^{\ast}(U_{\alpha}))\simeq T^{\ast}(\partial U_{\alpha})\times T^{\ast}[0.5,1)
\end{align}
where we identify $T^{\ast}[0.5,1)\simeq \mathbb{C}_{0.5\leq Re(z)< 1}$. Unlike the setting in \cite{GPSCV}, here we equip $\mathbb{C}$ with the Liouville structure obtained from $\mathbb{C}\simeq T^{\ast}\mathbb{R}$. The argument \cite[Lemma 2.41]{GPSCV} that prevents pseudo-holomorphic curves from escaping off the boundary applies for almost complex structures on $\mathbb{C}$ obtained from the conically deformed Sasaki metric on $T^{\ast}{\mathbb{R}}$ coming from the Euclidean metic. Therefore, we may impose the constraint that the almost complex structure on $T^{\ast}U_{\alpha}$ coincides with $\defac$ near the boundary.} and we can consider the associated wrapped Fukaya categories $\mathfrak{W}(T^{\ast}\overline{U_{\alpha}})$ on each of these sectors. In this particular case, using the argument similar to \cite[Section 5.6]{GPSMS}, we can directly compute wrapped chain complexes as follows. For an internal open set $U_{\alpha}$, and each $z\in U_{\alpha}$, we consider a cofinal fibre wrapping $F_{z,t}$ of $\fibre{z}$ with the support arc given by $z$, such that for $t$ large enough, the ideal boundary of $F_{z,t}$ is given by some inward push-off of the conormal of $\partial U_{\alpha}$. These $F_{z,t}$ give an explicit description of $\mathfrak{W}(T^{\ast}\overline{U_{\alpha}})$ since
$$HF^\ast(\fibre{z},\fibre{z'})=HF^0(\fibre{z},\fibre{z'})\cong\mathbb{Z}[\gamma_{zz'}],\quad \forall z,z'\in U_{\alpha}.$$
In particular, we can identify $HF^\ast(\fibre{z},\fibre{z'})$ with the unique path homotopy class in $\pi_1(U_{\alpha};z,z')$. Furthermore, if $z,z',z''\in U_\alpha$, then the $A_\infty$-structure map is
$$\mu_2:HF^\ast(\fibre{z},\fibre{z'})\otimes H^\ast(\fibre{z'},\fibre{z''})\to H^\ast(\fibre{z},\fibre{z''}),\quad \mu_2([\gamma_{zz'}],[\gamma_{z'z''}])=[\gamma_{zz''}].$$

\noindent Note also that we have inclusion functors $(i_{\alpha})_{\ast}:\mathfrak{W}(T^{\ast}\overline{U_{\alpha}})\to \wfc $ and \cite[Theorem 1.35]{GPSSD} implies that the natural functor
\begin{equation}\label{eq:htypcolim}
\displaystyle\mbox{hocolim}(i_{\alpha})_{\ast}:\mbox{hocolim}_{\alpha_1,\dots,\alpha_k} \mathfrak{W}\big(T^{\ast}U_{\alpha_1\dots \alpha_k}\big)\to \wfc
\end{equation}
is a pre-triangulated equivalence, where $U_{\alpha_1,\dots,\alpha_k}:=U_{\alpha_1}\cap\ldots\cap U_{\alpha_k}$ and the homotopy category of the homotopy colimit is naturally identified with the fundamental groupoid of $\obis$. %Furthermore, fixing a finite collection of points on each $U_{\alpha}$.

%%%%%%%%%%%%%%%%%%%%%%%%%%%%%%%%%%%%%%%%%%%%%%%%%%%%%
%%%%%%%%%%%%%%%%%%%%%%%%%%%%%%%%%%%%%%%%%%%%%%%%%%%%%
\begin{proof}[Proof of \cref{prop:fibrewrappingsgenerate}]

As said above, Part (1) follows from \cref{lem:cofinality}. For Part (2), let $x,y\in U_{\alpha}$ such that $y$ is within the distance $\frac{1}{20} \injradius$ from $x$. By the prior discussion, we have a generator $[\gamma_{xy}]$ that represents the minimal geodesic in the homotopy colimit \ref{eq:htypcolim}. We must now compute the image of $[\gamma_{xy}]$ under the functor $\mbox{hocolim}(i_{\alpha})_{\ast}$ and show that, for careful choices, the minimal geodesic generator indeed gets preserved. First, we claim that there are fibre wrappings $\{F_k\}_k$ of $\fibre{x}$ such that
\begin{itemize}
    \item[(i)] the generators of $CF(F_k,\fibre{y})$ correspond to geodesics between $\fibre{x}$ and $\fibre{y}$,
    \item[(ii)] the wrapping continuation map $CF(F_k,\fibre{y})\to CF(F_{k+1},\fibre{y})$ preserves the minimal geodesic element.
\end{itemize}
Intuitively, these can be obtained by considering globally positive Hamiltonians $H_k$ that are linear at infinity but quadratic in a large compact part of the cotangent bundle and using them to flow $\fibre{x}$. Part (2) follows from the existence of such fibre wrappings, so let us prove the claim, as follows.\\

Let $\{\ell_i\}\to +\infty$ be an increasing sequence of positive real numbers, with $\ell_1>2$, such that no $\ell_i$ is in the length spectrum of geodesics between $x$ and $y$ and the minimal geodesic is the only geodesic with length less than $\ell_1$. (This is possible because we normalized so that $\injradius>2$.) For each $k\geq 0$, consider $\delta_k\in\R_+$ such that $[\ell_k- \delta_k,\ell_k+\delta_k]$ is disjoint from the length spectrum and, for each $\delta_k$, choose a smooth increasing positive function
$H_k:[1,\infty)\to [0,\infty)$ such that
\begin{itemize}
    \item[-] $H'_k(r)=0$ on $[1,1+\frac{1}{80}]$,
    \item[-] $H_k(r)=\frac{1}{2}r^2$ on $r\in [1+\frac{1}{40},\ell_k-\delta_k]$,
    \item[-] $H_k(r)=\ell_k\cdot r$ for $r\in [\ell_k,\infty)$.
    \item[-] With respect to $k$, $H_k$ is fixed on $[1,\ell_1]$ and the derivatives of $H_k$ form an increasing sequence of functions.
\end{itemize}
These Hamiltonians $H_k$ are linear at infinity and quadratic on large compact part. We then declare the sequence of fibre wrappings $\{F_k\}$ to be time-$1$ image of $\fibre{x}$ under the Hamiltonian flow of $H_k$. The family $\{F_k\}$ is cofinal because under a (large) compactly supported deformation, it becomes the linear Hamiltonian $\ell_k r$ and thus the cofinality criterion \cref{lem:cofinality} applies. Let us argue that it satisfies $(i)$ and $(ii)$.\\

By the observation in \cite[Section 3(c)]{biasedview}, the Hamiltonian chords of $H_k$ between $\fibre{x}$ and $\fibre{y}$ are of the form $(r,x(\ell s))$, where $x$ is the sphere bundle lift of a unit-speed geodesic, $\ell$ is its length, and $r>0$ is such that $H_k'(r)=\ell$. By our condition $\injradius>2$, given a non-minimal geodesic of length $\ell$ we must have $\ell>1$. Since the interval $[\ell_k-\delta_k,\ell_k+\delta_k]$ is disjoint to the length spectrum, the Hamiltonian chords lie on the interval where $H_k$ is quadratic and are thus in one-to-one correspondence with non-minimal geodesics of length bounded above by $\ell_k$. Furthermore, there is a unique Hamiltonian chord $[\gamma_{xy}]$ representing the minimal geodesic in the region $[1,\ell_1]$. This concludes $(i)$. For $(ii)$, consider the continuation map $CF(F_{k},\fibre{y})\to CF(F_{k+1},\fibre{y})$. Since $H_k$ is fixed on $[1,\ell_1]$, the minimal geodesic element gives a constant continuation strip: let us argue that this is the only continuation strip. Indeed, let $u$ be a homogeneous continuation strip for $CF(F_k,\fibre{y})\to CF(F_{k+1},\fibre{y})$ mapping a geodesic $x^{-}$ of length $\ell$ to the minimal geodesic $\gamma$. By the geometric energy formula and global positivity of the Hamiltonians, we must have $A(x^{-})\geq A(\gamma)$. In contrast, \cite[Section 3(c), Eq.(3.12)]{biasedview} implies that the action of the first non-minimal geodesic is $-\frac{1}{2}\ell^2$, so $A(x^{-})=-L^2/2$, and we had $\abs{A(\gamma)}<L^2/2$. This is a contradiction and thus the only strip must be constant, hence the minimal geodesic $[\gamma_{xy}]$ is preserved under wrappings.
\end{proof}

%%%%%%%%%%%%%%%%%%%%%%%%%%%%%%%%%%%%%%%%%%%%%%%%%%%%%
%%%%%%%%%%%%%%%%%%%%%%%%%%%%%%%%%%%%%%%%%%%%%%%%%%%%%

\subsubsection{Continuation maps for fibre wrappings}\label{sssec:computing_str_maps}
In order to continue our proof of \cref{thm:Familyfloerwrapped}, we need to compute $\mu_{1\vert 1}(c,\cdot)$, where $c$ is a continuation element. To employ the family Floer techniques from \cref{subsection:familyfloer}, we use the following facts to relate the maps $\mu_{1\vert 1}(c,\cdot)$ with continuation maps (c.f \cite[Lemma 3.26]{GPSCV}). First, an analogue of \cref{thm:snetworkadg} for fibre wrappings:%{\RC Where is this used?}

\begin{lemma}\label{prop:wrappingfibreadg}
Let $F$ be a wrapped fibre whose base-point is contained in the complement of $\snetwork$. Then there exists $\e_0(F)\in\R_+$ such that there are no non-constant $\defac$-holomorphic strips between $F$ and $\e \blag$ if $\e\in(0,\e_0(F))$.
\end{lemma}
\begin{proof}
Let $z$ be the base-point of $F$, then since $F=\fibre{z}$ on $D_1^{\ast}\obis$, the action of the intersection points are of size $O(\e)$. Therefore, monotonicity tells us that the $\defac$-holomorphic strips cannot escape $D_1^{\ast}\obis$. 
\end{proof}
\noindent The analogue of \cref{prop:continuationstripclassification} for fibre wrappings reads as follows. For brevity, we suppress the discussion on the necessity to make the right choice of the elongation functions.
 
\begin{proposition}\label{prop:wrappingfibrecont}
Let $\{F_t\}$, $t\in [0,1]$, be a fibre wrapping and $\snetwork\sse S$ a spectral network.

\begin{enumerate}
    \item If the support arc of $\{F_t\}$ lies outside $\snetwork$, then there exists a family $J(s,t)$ of almost complex structures such that the moduli space of $(F_t,\blag)$-continuation strips is cobordant to the moduli space of trivial strips.\\

    \item If the support arc of $\{F_t\}$ intersects $\snetwork$ at an $ij$-wall, then there exists a family $J(s,t)$ of almost complex structures such that the moduli space of $(F_t,\blag)$-continuation strips that do not travel from the $j$th sheet to the $i$th sheet is cobordant to that of trivial strips.
\end{enumerate}
 \end{proposition}
 
\begin{proof}
\cref{prop:continuationstripclassification} implies that there exists $\e_0(\{F_t\})\in\R_+$ such that for all $\e\in(0,\e_0(\{F_t\}))$ there are no non-constant $\defac$-holomorhic strips between $\e\blag$ and $F_t$. We can then apply the Gromov compactness argument as in the proof of \cref{prop:fibreparalleltransport} to conclude that the moduli space of continuation strips is cobordant to the moduli of trivial strips. 
\begin{comment}

By applying the Gromov compactness argument as in the proof of \cref{prop:fibreparalleltransport}, for each $t_0\in [0,1]$ we can find a small constant $\delta(t_0)\in\R_+$ such that the moduli of $\defac$-continuation strips for $F_{t_0+\delta(t_0)s\tau}$, $s\in [0,\tau],\tau\in [0,1]$ is cobordant to the moduli of trivial strips. 

By the compactness of $[0,1]$, there is a sequence $t_0<t_1<t_2<\ldots<t_k$, $t_i\in[0,1]$, $t_0=1,t_k=1$ such that $\abs{t_{i+1}-t_i}<\min({\delta(t_1),\ldots,\delta(t_k)})$. Gluing all the $(F_{t_i},F_{t_{i+1}})$-continuation strips, and perturbing if necessary, we obtain a family $J(s,t)$ of almost complex structures such that the $J(s,t)$-continuation strips are transversely cut-out and the moduli is cobordant to the moduli of trivial strips.
\end{comment}
\end{proof}

\noindent Since the scaling parameter $\e$ in \cref{prop:wrappingfibreadg} only depends on the base point, a consequence of \cref{prop:wrappingfibrecont} is that wrappings at infinity do not affect $\blag$. That is,
if $F\wto F'$ is a fibre wrapping with the same base point, then the continuation map $CF(F,\blag)\to CF(F',\blag)$ is the identity map. This is an analogue of the invariance statement in \cref{prop:standardfloertheory}.

%%%%%%%%%%%%%%%%%%%%%%%%%%%%%%%%%%%%%%%%%%%%%%%%%%%%%
%%%%%%%%%%%%%%%%%%%%%%%%%%%%%%%%%%%%%%%%%%%%%%%%%%%%%

\subsubsection{Concluding \cref{thm:Familyfloerwrapped} and \cref{cor:familyfloerwrapped}}\label{sssec:conclusion_familyfloerwrapped}
For \cref{thm:Familyfloerwrapped}, it suffices to show that $\ylagv$ is the the non-abelianization module, as a module over the cohomology subcategory generated by the wrapped fibres, which we identify with the fundamental groupoid. For that, note that given any $\snetwork$-adapted minimal geodesic, \cref{prop:fibrewrappingsgenerate} and \cref{prop:wrappingfibrecont} guarantee the upper-triangular form of $\ylagv([\alpha_{zw}])$. Then the same argument as in the proof of \cref{thm:familyFloer} and the wall-crossing formula imply that the parallel transport map must be equal to that of $\tdfd$.\\

For \cref{cor:familyfloerwrapped}, \cref{lem:Maslov0} implies that $\blag$ is Maslov $0$ and thus the chain complex $\ylagv(F)$ for a wrapping fibre $F$ must be cohomologically concentrated in degree zero. By homological perturbation, comparing to the non-abelianization modules, we can perturb $\ylagv$ to be defined over the subcategory of the wrapping fibres. Since the latter subcategory generate the wrapped Fukaya category, the statement follows from \cref{thm:Familyfloerwrapped}.\hfill$\Box$
%Using proposition \cref{prop:fibrewrappingsgenerate}, we will show that  Indeed, this is essentially the argument of  applied to the wrapping fibres rather than fibre parallel transports, since the relevant count of the moduli space of continuation strips is determined by the same argument.  then implies that   But 
\color{black}
%%%%%%%%%%%%%%%%%%%%%%%%%%%%%%%%%%%%%%%%%%%%%%%%%%%%%
%%%%%%%%%%%%%%%%%%%%%%%%%%%%%%%%%%%%%%%%%%%%%%%%%%%%%
%%%%%%%%%%%%%%%%%%%%%%%%%%%%%%%%%%%%%%%%%%%%%%%%%%%%%

\subsection{Cylindrized Betti Lagrangians and spectral networks}\label{subsection:PWFC} As discussed in \cref{subsection:WFC}, a Betti Lagrangian $L\sse(T^*S,\omega_\std)$ does {\it not} give an element of the wrapped Fukaya category $\wfc$, nor its partially wrapped modification $\pwfc$. The reason is the asymptotic behavior of $L$ at infinity, which is not compatible with the definitions of these categories. In \cref{subsection:WFC}, we showed that a Betti Lagrangian, endowed with $V\in\Loc(L)$, does nevertheless give an $A_\infty$-module $\ylagv$ over $\wfc$ which behaves as if it was the Yoneda module of an object in $\wfc$, and in turn we related the $\mu_{1|1}$-maps of such $A_\infty$-modules to the non-abelianization framework of spectral networks, cf.~\cref{thm:Familyfloerwrapped}. An alternative approach to comparing Betti Lagrangians and Fukaya categories to spectral networks is to allow ourselves to modify the Betti Lagrangian $L\sse(T^*S,\omega_\std)$ at infinity so that in fact it defines an object in a Fukaya category. The present section develops this approach, defining such cylindrizations and showing that the non-abelianized parallel transport from spectral networks is captured by $\mu_2$-map of the $A_\infty$-structure associated to such cylindrization.\\

Specifically, \cref{sssec:cylindrize_BettiLag} explains how to deform $\blag$ to a cylindrical Lagrangian $\blag_\circ\sse (T^*S,\omega_\std)$ with boundary on the Legendrian links associated to the Betti surface $(S,\mkpts,{\bf\lknot})$. Denote by $\pwfc$ the partially wrapped Fukaya category of $(T^*S,\omega_\std)$ partially stopped at $\La$, as constructed in \cite{GPSSD}. By construction, the cylindrization $\blag_{\circ}$ defines an object in $\pwfc$ once endowed with a local system $V\in\Loc(L)$. We denote by
$$\ycyl:\pwfc\to \kmod$$
its associated Yoneda $A_\infty$-module and consider its component
$$\mu_2:HW(\fibre{z},\fibre{w})\otimes HW(\fibre{z},\bcyl)\to HW(\fibre{w},\bcyl),$$
where $HW(\fibre{z},\fibre{w})=\mbox{Hom}_{\pwfc}(\fibre{z},\fibre{w})$ are now morphisms in $\pwfc$, not $\wfc$. That said, given a $\snetwork$-adapted pair $z,w\in S$, we had the minimal geodesic $\alpha_{zw}$ which defined a class in $\mbox{Hom}_{\wfc}(\fibre{z},\fibre{w})$. In order to transfer it to a class in $\mbox{Hom}_{\pwfc}(\fibre{z},\fibre{w})$, we consider a subset $\obis_{int}\sse\obis$ such that the front projections of the Legendrian links in ${\bf\lknot}$ are all contained in $\obis\setminus \obis_{int}$ and $S$ retracts to $S_{int}$: we obtain $\obis_{int}$ by removing a sufficiently large disk neighborhood of each marked point. Then there a natural inclusion functor $i_{\ast}:\mathfrak{W}(T^*\obis_{int})\to \pwfc$, and we can consider the image $i_{\ast}[\alpha_{zw}]$ of (the class defined by) $\alpha_{zw}$ under this functor.\\

For the cylindrized Betti Lagrangian $\blag_\circ$, now in the context of the partially wrapped Fukaya category $\pwfc$, the relation between the $A_\infty$-structure and non-abelianized parallel transport for spectral networks is as follows:

\begin{theorem}\label{thm:pwfcyonneda}
Let $(S,g)$ be a Betti surface, $L\sse(T^*S,\omega_\std)$ a Betti Lagrangian, $V\in\Loc(L)$, $\snetwork\sse S$ an adapted spectral network and $z,w\in S$ an $\snetwork$-adapted pair. Then the minimal geodesic $\alpha_{zw}\in\Omega_{z,w}$ satisfies
\begin{equation}\label{eq:mu2_paralleltrans}
\mu_2(i_{\ast}[\alpha_{zw}],\cdot )=\tdfd(\alpha_{zw})
\end{equation}
\end{theorem}

\noindent \cref{thm:pwfcyonneda} is momentarily proven in \cref{sssec:proof_blag_cyl_mu2}, after describing $\blag_\circ$ in \cref{sssec:cylindrize_BettiLag}. A consequence of \cref{cor:familyfloerwrapped} and \cref{thm:pwfcyonneda} is the relationship between $\ylagv$ and $\ycylv$, which reads as follows:%{\RC How does that follow? Is it statements or the proof? Because they coincide on generators, since we know it on generators $\alpha_{zw}$ and $\mu_2$ coincides with $\mu_{1\vert1}$ when taking Yoneda.}
\begin{cor}\label{cor:pwfcyonnedamodule}
There is a homotopy
$$\ycylv\circ i_{\ast}\simeq \ylagv$$
of $A_\infty$-modules over $\wfc$.
\end{cor}

%%%%%%%%%%%%%%%%%%%%%%%%%%%%%%%%%%%%%%%%%%%%%%%%%%%%%
%%%%%%%%%%%%%%%%%%%%%%%%%%%%%%%%%%%%%%%%%%%%%%%%%%%%%
%%%%%%%%%%%%%%%%%%%%%%%%%%%%%%%%%%%%%%%%%%%%%%%%%%%%%

\subsubsection{Cylindrization of Betti Lagrangians}\label{sssec:cylindrize_BettiLag}
Let $L\sse T^*S$ be a Betti Lagrangian on a Betti surface $(S,\mkpts,
{\bf\lknot})$, we want to modify $L$ near infinity (around the punctures of $S$) to obtain a Lagrangian $\lcyl\sse T^*S$ which is a Lagrangian filling of the Legendrian link in $(T^\infty S,\xi_\std)$ defined by $\bf\lknot$ and defines an object in $\pwfc$. The cylindrization procedure is in line with the modification for weave fillings in \cite{casals2022conjugate}.\\

From a symplectic topological viewpoint, without taking $\snetwork$ into account, the cylindrization process is as follows. Choose a collection $\EuScript{S}$ of circles in $S$, one circle per puncture chosen so that it encloses the puncture and is sufficiently close to it. The positive conormal lift of such collection $\EuScript{S}$ is a Lagrangian submanifold that cleanly intersects the zero section of $T^*S$. Let $C\sse T^*S$ be the Lagrangian surgery resolving that clean intersection, which is an exact Lagrangian that coincides with the zero section away from the punctures, and with the positive conormal lift of the circles near the punctures.  Since $C\sse T^*S$ is Lagrangian and cylindrical at infinity, there exists a Weinstein neighborhood $i_{cyl}:D^*C\to T^*S$ for a small enough disk bundle $D^*C$, which coincides with the disk bundle of $T^*S$ away from the punctures. Since $C$ is diffeomorphic to $S$, the Betti Lagrangian $L\sse T^*S$ can be identified with a homonymous Lagrangian $L\sse T^*C$ which can be assumed to belong to $D^*C$.

\begin{definition}\label{def:cylindrization}
The cylindrization $\lcyl$ of $L$ is defined to be $i_{cyl}(L)\sse (T^*S,\la_{\std})$.\hfill$\Box$
\end{definition}

\noindent To incorporate compatibility with a given spectral network $\snetwork$ for $L$, we require that the circles in $\EuScript{S}$ are chosen enclosing each puncture close enough such that:
\begin{enumerate}
    \item $\snetwork$ is transverse to the circles,
    \item each circle is contained in the trapping neighborhood, cf.~\cref{lemma:trappinglemma}, for the corresponding puncture,
    \item no vertex of $\snetwork$ gets mapped in the cylindrical region between each circle and its corresponding puncture. 
\end{enumerate}
This can be ensured by choosing the circles in $\EuScript{S}$ close enough to the punctures and generically. Note that \cref{lem:diametercontrol} implies that $\blag$ is uniformly bounded with respect to $g$, and so any $(F_z,\blag)$-disk for $z$ away from the circles must have its image contained in a similar region, truncated away from the punctures.

\begin{remark}
Note that we had to scale $L$ down so that $L\sse D^*C$ lied inside the Weinstein neighborhood. This is possible while being compatible with $\snetwork$ because the metric is of the form $r^2d\theta^2+dr^2$ near the punctures.\hfill$\Box$
\end{remark}

\noindent The cylindrization $\lcyl$ in \cref{def:cylindrization} defines an object in $\pwfc$ after applying a small positive push-off at infinity. We always implicitly understand a Lagrangian filling of a Legendrian $\La$ in $T^\infty S$ as giving an object in $\pwfc$ in this manner.

%%%%%%%%%%%%%%%%%%%%%%%%%%%%%%%%%%%%%%%%%%%%%%%%%%%%%
%%%%%%%%%%%%%%%%%%%%%%%%%%%%%%%%%%%%%%%%%%%%%%%%%%%%%

\subsubsection{Proof of \cref{thm:pwfcyonneda} $($Floer $\mu_2$ is spectral transport$)$}\label{sssec:proof_blag_cyl_mu2} The argument is in line with the family Floer techniques developed in \cref{section_Floer} and follows the same steps as \cref{subsec:prooffamilyfloerwrapped}, now suitably modified in the partially wrapped setting. Let $\tbis(\EuScript{S})$ denote the unique compact connected component of the complement in $S$ of the circles in $\EuScript{S}$. First, by geometric boundedness, the following locality lemma holds:
\begin{lemma}\label{lem:localitylemma}
Let $F_z$ be a wrapped fibre, with basepoint $z\in \tbis(\EuScript{S})$, whose boundary does not intersect $\La$. Then, any $(\blag_{\circ},F_z)$-disk belongs to a region of the form $T^{\ast}\tbis(\EuScript{S'})$, for $\EuScript{S'}$ another compatible collection of circles. Similarly, the same statement holds for any small enough negative push-offs of $\blag_{\circ}$ whose boundary stays disjoint to $\partial F_z$.\hfill$\Box$
\end{lemma}

\cref{lem:localitylemma} implies that the remainder of the proof of \cref{thm:pwfcyonneda} can be the same as the proof of \cref{thm:Familyfloerwrapped}, except for the fact that the cohomology module is now defined as the direct limit of $CF(\fibre{z},\bcyl^i)$ as $i\to \infty$. We thus need to take into consideration of wrapping morphisms
$$CF(\fibre{z},\bcyl^i)\to CF(\fibre{z}^{w},\bcyl^i),\quad CF(\fibre{z},\bcyl^i)\to CF(\fibre{z},\bcyl^{i+1}).$$
This does not affect the computations overall, as a combination of the argument in \cref{lem:localitylemma} with the Gromov compactness used in \cref{prop:wrappingfibrecont} implies the following fact:
%{\RC Have to explain the $CF_\Lambda$ vs $CF$ notation.}
\begin{lemma}\label{prop:cofinalsequence}
There exists a sequence of positive push-offs $\{\lcyl^i\}_i$, $\{F_z^j\}_j$ such that:
\begin{enumerate}
    \item For $z\in \tbis(\EuScript{S})$ and for $i-1>j$, the continuation maps
    $$CF_{\Lambda}(F^{j}_z,\lcyl^i)\to CF_{\Lambda}(F^{j+1}_z,\lcyl^i),\quad CF(F^{j}_z,\lcyl^i)\to CF_{\Lambda}(F^{j}_z,\,\lcyl^{i+1})$$
    are the identity map.\\

    \item $\displaystyle HF_{\Lambda}(F_z,\blag)=\lim_{j\to \infty}HF(F^j_z,\blag)=\lim_{i\to \infty}HF(F_z,\blag^i)$
\end{enumerate}
Here we denoted $CF_\La$ when morphisms are in $\pwfc$, i.e. partially stopped at $\La$.\hfill$\Box$
\end{lemma}

The remaining ingredient for the proof of \cref{thm:pwfcyonneda} is a description of the image morphisms $i_{\ast}(HW(\fibre{x},\fibre{y}))$ under the inclusion $i_*:\wfc\to\pwfc$. It is provided by the following proposition:

\begin{proposition}\label{prop:inclusionfullyfaithful}
Let $x,y\in \tbis(\EuScript{S})$ be two generic points at distance less than $\frac{1}{20}\injradius$. Then the inclusion map
$$i_{\ast}:HF^\ast(\fibre{x},\fibre{y})\to HF^\ast_{\Lambda}(\fibre{x},\fibre{y})$$
is an isomorphism. In fact, for $[\alpha_{xy}]$ the representative of the minimal geodesic, $i_{\ast}[\alpha_{xy}]=[\alpha_{xy}]$.
\end{proposition}

\begin{proof} 
In brief the core of the proof is using an action filtration argument combined with the fact that, since the metric is conical near the punctures, geodesics in the truncated region $\tbis(\EuScript{S})$ that enter the neighborhood of the punctures cannot leave.\footnote{Also, there is no conceptual importance to the factor $1/20$: it is just there for action estimates.} Let us provide the details of the argument.\\

First, we must construct a contact form $\alpha$ for $(T^\infty\obis,\xi_\std)$ whose Reeb chords correspond to geodesics between $x$ and $y$ while being stopped at $\La$, so as to be consistent with the partial wrapping in $\pwfc$. In precise terms, we claim that there exists a contact 1-form $\alpha\in\Omega^1(T^\infty\obis)$ for $\xi_\std=\ker \alpha$ such that:

\begin{itemize}
    \item[(i)] The Reeb flow of $\alpha$ is complete,
    \item[(ii)] $\alpha=\la_{st}|_{T^\infty S}$ outside a small standard neighborhood of $\lknot$,
    \item[(iii)] $\alpha$-Reeb chords between $\partial \fibre{x}$ and $\partial\fibre{y}$ are in bijection with $\la_{st}$-chords from $\partial \fibre{x}$ to $\partial \fibre{y}$.
\end{itemize}

The neighborhood in $(ii)$ is referred to as the stop region of $\alpha$. Let us argue that such $\alpha$ exists, as follows. By construction, $\lknot$ lies in a small neighborhood of the positive conormal of the circles in $\EuScript{S}$. In such a neighborhood, the geodesic flow for $(S,g)$ points outward, transverse to the boundary of $\tbis(\EuScript{S})$. By \cite[Lemma 3.29]{GPSCV}, there exists a contact form satisfying $(i)$,$(ii)$ with the stop region lying in this neighborhood. For this choice of neighborhood, $(iii)$ also holds. Indeed, an $\alpha$-Reeb chord from $x$ to $y$ is a geodesic chord away from the stop region. Now, on the one hand, if it entered such neighborhood of $\lknot$, the geodesic must enter the boundary in the transverse direction and then the conical metric structure implies that the original geodesic could not have returned to $y$. On the other hand, for the $\alpha$-Reeb chord to return to $y$, it must leave the stopping region but the only way the chord can leave it is by exiting the neighborhood of the positive conormal of $\EuScript{S}$. In that case the geodesic flow would further send the chord off to infinity. Therefore, geodesics between $x$ and $y$ could not have entered such neighborhood of the stop boundary and no $\alpha$-Reeb chord could have done the same. Thus $(iii)$ holds and we can choose a contact form $\alpha$ with the properties above.\\

Second, we proceed with an action filtration argument, in line with \cref{prop:fibrewrappingsgenerate}. By our choice of contact form above, the $\alpha$-Reeb chords are in bijection with geodesics between $x$ and $y$. By genericity of $x,y$, we can assume that the length of the geodesics are all distinct, the energy functional on the path space is Morse, and all the geodesics stay outside the stop collar region. We now consider two variations on the family of Hamiltonians from the proof of \cref{prop:fibrewrappingsgenerate}, one for $CF^\ast$ and one for $CF^\ast_\Lambda$.\\

(1) For $CF^\ast$, the wrapped case with no stop at $\La$, let $h_{\tbis}:T^\infty\tbis(\EuScript{S})\to\R$ be a smooth, non-negative function on the unit sphere bundle of $\tbis(\EuScript{S})$ with $\nabla h_{st}$ supported on the stop region, such that $h_{\tbis}\la_{\std}|_{T^\infty S}$ has complete Reeb flow. Let $H_n,\ell_n$ be as in the proof of \cref{prop:fibrewrappingsgenerate}, consider the linear-at-infinity Hamiltonians $h_{\tbis}\cdot H_n$ and let $F_n$ denote the time $1$-image of $\fibre{y}$ under the corresponding Hamiltonian flows. The Hamiltonian chords of $h_{\tbis}\cdot H_n$ all lie outside the stop region, and so the argument in \cref{prop:fibrewrappingsgenerate} implies that the Hamiltonian chords correspond to non-minimal geodesic of length $\ell$ less than $\ell_n$, and their actions are given by $-\ell^2/2$.\\

(2) For $CF^\ast_\La$, write the contact form $\alpha$ above as $\alpha=h_{stop}\la_{\std}$ and let $F^{stop}_n$ be the time-$1$ image of $\fibre{x}$ using the quadratic Hamiltonians $h^{stop}\cdot H_n$. By construction, the inequality $h_{\tbis}\leq h_{stop}$ holds and so there is a positive isotopy $F_n\wto F_n^{stop}$ and an induced continuation map $CF^\ast(F_n,\fibre{y})\to CF^\ast(F_n^{stop},\fibre{y})$. By Property $(iii)$ above, these two groups coincide as vector spaces. Since the actions are given by $-\ell^2/2$, and the length spectrum is discrete, we can consider a basis $[\gamma_1],[\gamma_2],\ldots, [\gamma_k]$, by ordering by length the geodesics of length bounded by $\ell_n$.\\

The last step is to argue that the continuation map is upper-triangular with respect to the basis $[\gamma_1],[\gamma_2],\ldots,[\gamma_k]$. Indeed, let $[\gamma_i]$ be a generator of $CF^\ast(F_n,\fibre{y})$ and consider the image $\sum n_j [\gamma_j]$ under the continuation map. By the geometric energy formula, we have
$A(\gamma_i)\leq A(\gamma_j)$ if $i\leq j$ and a continuation strip from $\gamma_i$ to itself must be constant. Therefore, the action formula implies that this image must be of the form $\displaystyle[\gamma_i]+\sum_{j<i} n_{ij}[\gamma_{j}]$. Thus the continuation map $CF(F_n,\fibre{y})\to CF(F_n^{stop},\fibre{y})$ is upper-triangular, with identity on the diagonal, and so it must be a quasi-isomorphism. Furthermore, the following diagram commutes, up to chain homotopy:
\[\begin{tikzcd}
	{CF(F_{n+1},\fibre{y})} & {CF(F^{stop}_{n+1},\fibre{y})} \\
	{CF(F_{n},\fibre{y})} & {CF(F^{stop}_n,\fibre{y})}.
	\arrow[from=1-1, to=1-2]
	\arrow[from=2-1, to=1-1]
	\arrow[from=2-1, to=2-2]
	\arrow[from=2-2, to=1-2]
\end{tikzcd}\]
In consequence, the induced map $WF(\fibre{x},\fibre{y})\to WF_{\lknot}(\fibre{x},\fibre{y})$ must also be a quasi-isomorphism. Finally, the minimal geodesic generator is preserved under these maps by arguing as in the proof of \cref{prop:fibrewrappingsgenerate}.
\end{proof}

%%%%%%%%%%%%%%%%%%%%%%%%%%%%%%%%%%%%%%%%%%%%%%%%%%%%%
%%%%%%%%%%%%%%%%%%%%%%%%%%%%%%%%%%%%%%%%%%%%%%%%%%%%%

\begin{proof}[Proof of \cref{thm:pwfcyonneda}]
By \cref{prop:inclusionfullyfaithful}, $i_{\ast}$ is cohomologically fully faithful, and the minimal geodesic elements are preserved under the inclusion functor. Consider the cohomology module of $\ycyl$ over $i_{\ast}$ restricted to the subcategory given by the cotangent fibres which is istelf identified with the fundamental groupoid. The cohomology module $\ycyl$ with the parallel transport morphisms is defined by the direct limit
$$\displaystyle\lim_{i\to \infty} CF(\fibre{z},\blag^i).$$
By \cref{prop:cofinalsequence}, such wrappings induce identity maps on the Floer chain complexes $CF_{\Lambda}(F_z^j,\blag^i)$ along arcs away from the spectral network $\snetwork$. So the local system given by (the cohomology module of) $\ycyl$ has trivial parallel transport over arcs that do not cross $\snetwork$. The same argument as in \cref{prop:wrappingfibrecont}.(2) implies that the associated parallel along an arc crossing $\snetwork$ must be as in the wall-crossing formula for non-abelianized local systems. The statement then follows by applying the argument used for \cref{thm:Familyfloerwrapped}.
\end{proof}

\begin{proof}[Proof of \cref{cor:pwfcyonnedamodule}]
By \cref{thm:pwfcyonneda}, $\ycyl\circ i_{\ast}\simeq \ylag$ over the cohomology subcategory of cotangent fibres. Since the inclusion functor $i_{\ast}$ is cohomologically fully faithful, and the cotangent fibres generate $\wfc$, the statement follows from the same argument as the proofs of \cref{thm:Familyfloerwrapped} and \cref{cor:familyfloerwrapped}.
\end{proof}

%%%%%%%%%%%%%%%%%%%%%%%%%%%%%%%%%%%%%%%%%%%%%%%%%%%%%
%%%%%%%%%%%%%%%%%%%%%%%%%%%%%%%%%%%%%%%%%%%%%%%%%%%%%
%%%%%%%%%%%%%%%%%%%%%%%%%%%%%%%%%%%%%%%%%%%%%%%%%%%%%
\color{black}
\subsection{Wrapping up the description of $\ycylv$}\label{subsection:partialyonneda}
In \cref{subsection:PWFC} we introduced the Yoneda $A_\infty$-module $\ycylv:\pwfc\lr\kmod$ associated to the cylindrization $\lcyl\sse(T^*S,\la_\std)$ of a Betti Lagrangian $L$. In a nutshell, \cref{thm:pwfcyonneda} and \cref{cor:pwfcyonnedamodule} give a description of this $A_\infty$-module on the $A_\infty$-subcategory given by the image of the functor $i_*:\mathfrak{W}(T^*\obis_{int})\lr\pwfc$, which is generated by the internal cotangent fiber. The goal of this section is to complete the description of the $A_\infty$-module $\ycylv$ on $\pwfc$. In order to describe the whole of $\ycylv:\pwfc\lr\kmod$, we construct in \cref{prop:gen_internal_infinity_fibers} a set of generators of $\pwfc$, containing the internal cotangent fiber, where we can compute the values of $\ycylv$.\\

%%%%%%%%%%%%%%%%%%%%%%%%%%%%%%%%%%%%%%%%%%%%%%%%%%%%%
%%%%%%%%%%%%%%%%%%%%%%%%%%%%%%%%%%%%%%%%%%%%%%%%%%%%%

\subsubsection{Generation by interior and infinity fibers}\label{sssec:generation_int_inf} Let $(S,\mkpts,{\bf\lknot})$ be a Betti surface, and denote by $\La_i\sse(T^\infty S,\xi_\std)$ the Legendrian link associated to the marked point $m_i\in\mkpts$. For simplicity, we shall henceforth assume that $\La_i$ has finitely many Reeb chords, each at a different angle, and that it is Reeb-positive (for instance, in the case of Stokes Legendrians without multiplicities \cref{prop:reebpositivestokes}, or the positive braid links described in \cref{rem:positivebraid}). This includes many of the most important classes, including the case where each Legendrian in ${\bf\lknot}$ is isotopic to the cylindrical closure of a positive braid. We consider the following Lagrangians in $T^*S$, all of which yield objects in $\pwfc$ once appropriately decorated and partially wrapped:

\begin{enumerate}
    \item An internal cotangent fiber $\fibre{z}$, for some point $z\in S$,\\
    
   \item The cotangent fiber $F_{m_i}$ near the puncture corresponding to $m_i$. It is given by the cotangent fiber at any point in $\obis$ arbitrarily close to the marked point $m_i\in\bis$ such that the geodesic between that point and $m_i$ does not intersect $\La_i$, i.e.~it is the cotangent fiber at infinity of $m_i$, {\it past} the Legendrian link $\La_i$. We denote $F_{\mkpts}:=\{F_1,\ldots,F_{|{\mkpts}|}\}$ the set of such fibers at the punctured infinity of $\obis$.\\
    
    \item The collection of linking disks $\mathbb{D}_i$ associated to the Legendrian $\La_i$. That is, each $\mathbb{D}_i:=\{D_{i,1},\ldots,D_{i,c_i}\}$ is the disjoint union of the linking disks of $\La_i$, as in \cite[Section 5.3]{GPSSD}, $c_i=|\pi_0(\La_i)|$, one for each component of the Legendrian link $\La_i$. See Figure \ref{fig:linking_disks}.%For simplicity, we write $D_i$ to denote a linking disk in $\mathbb{D}_i$.\\
\end{enumerate}

\begin{center}
	\begin{figure}[h!]
		\centering
		\includegraphics[scale=1.8]{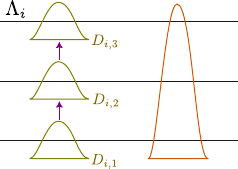}
		\caption{A front for the Legendrian $\La_i$: in this picture we assume $\La_i$ has three components, one per strand. The three linking disks $D_{i,1},D_{i,2}$ and $D_{i,3}$ for $\La_i$ are depicted, with the short Reeb chords between them highlighted in purple. The large linking disk $\EuScript{D}$, homologically obtained from coning the small linking disks along the short Reeb chords, is depicted in orange to their right.}
        \label{fig:linking_disks}
	\end{figure}
\end{center}

\noindent By \cite[Theorem 1.14]{GPSSD}, the set of Lagrangians $\{\fibre{z},\mathbb{D}_1,\ldots,\mathbb{D}_{|\mkpts|}\}$ generates $\pwfc$. By \cref{subsection:PWFC}, the spectral network $\snetwork$ allows us to describe $\ycylv$ at the generator $\fibre{z}$, by choosing $z\in\snetwork^c$ and using that the inclusion induces a fully faithful functor $\wfc\to\pwfc$. The value of $\ycylv$ at the linking disks in $\mathbb{D}_i$ is
$$\ycylv(D_i)=\Z.$$
Indeed, since $\partial^{\infty}\bcyl=\lknot$, $D_i$ intersects $\bcyl$ once. Since the same holds for a small enough positive push-off $\bcyl^{+}$ at the puncture $m_i$, $CF^\ast(D_i,\bcyl^+)=\Z$. Nevertheless, the value of $\ylagv$ on morphisms $\mbox{Hom}_{\pwfc}(D_{i,j},D_{i,k})$ between such linking disks, even for $j=k$, contains subtler information, encoding augmentations of Reeb chords at infinity which, in terms of $\snetwork$, correspond to augmented $\dfs$-trees.\footnote{These morphism groups between linking disks can themselves be difficult to describe.} In contrast, $\ycylv$ vanishes at the fibers $F_{\mkpts}$ at infinity. Therefore, a description of $\ycylv$ will follow if we can prove that $\pwfc$ is generated by $\{\fibre{z},F_{\mkpts}\}$. This is the content of the following Floer-theoretical result, which is independent of spectral networks:

\begin{proposition}[Generation by internal and infinity cotangent fibers]\label{prop:gen_internal_infinity_fibers} Let $(S,\mkpts,{\bf\lknot})$ be a Betti surface, $\pwfc$ the associated partially wrapped Fukaya category, $L\sse(T^*S,\omega_\std)$ a Betti Lagrangian and $V\in\Loc(L)$. Then $\pwfc$ is generated by $\{\fibre{z},F_{\mkpts}\}$.
\end{proposition}

\begin{proof}
Let $\fibre{z}\sse T^*S$ be an internal cotangent fiber and choose a puncture $m\in\mkpts$ with its associated fiber $F_m$ and Legendrian link $\La_m$. Denote by $\mathbb{D}_m:=\{D_{p_1},\ldots,D_{p_n}\}$ a collection of linking disks chosen such that $D_{p_i}$ locally links $\La_m$ at the point $p_i\in\La_m$ and the ordered points $p_1,\ldots,p_n\in \La_m$ are located one per each strand of $\La_m$ and all in the same angle.\footnote{We can and do assume that there is no Reeb chord at that angle.} Consider the large linking disk $\ldisk_m$ associated to $\La_m$, which is defined as in Figure \ref{fig:linking_disks}. The wrapping exact triangle in \cite[Theorem 1.10]{GPSSD} gives the exact triangle
\begin{equation}\label{eq:interior_infinity_triangle}
\fibre{z}\to F_m\to\ldisk_m
\end{equation}
in $\pwfc$. \cref{eq:interior_infinity_triangle} implies that $\{\fibre{z},F_{\mkpts}\}$ generate the same category as $\{\fibre{z},\ldisk_{m_1},\ldots,\ldisk_{m_{|{\mkpts}|}}\}$.\footnote{It also follows from \cref{eq:interior_infinity_triangle}  that $\ycylv(\ldisk_m)\cong \Z^n$.}\\

\noindent The next step is to show that this latter collection $\{\fibre{z},\ldisk_{m_1},\ldots,\ldisk_{m_{|{\mkpts}|}}\}$ generates $\pwfc$, which we do by comparing the large linking disk $\ldisk_m$ to the linking disks $D_{p_1},\ldots,D_{p_n}$ at $\La_m$. This comparison is obtained by iterating the wrapping exact triangle for short Reeb chords, which gives the homological description of $\ldisk_m$ as
$$\ldisk_m\cong \cone(\ldots \cone(\cone(D_{p_1}\to D_{p_2})[-1]\to D_{p_3})[-1])\ldots\to D_{p_n})[-1].$$
as objects in the dg-category $\mbox{Tw}(\pwfc)$ of twisted complexes of $\pwfc$. Here the morphisms $D_{p_i}\to D_{p_{i+1}}$ being coned are given by the short Reeb chords, see e.g.~Figure \ref{fig:linking_disks}. In fact, by Lemma \ref{lem:properties_linkingdisks}.(1) and (2), we actually have a quasi-isomorphism
\begin{equation}\label{eq:ldisk_directsum}
\ldisk_m\cong D_{p_1}\oplus D_{p_2}\oplus\ldots \oplus D_{p_n},
\end{equation}
again in $\mbox{Tw}(\pwfc)$. Indeed, this is a consequence of functoriality of mapping cones and homological perturbation, as follows. By \cite[Lemma 3.30]{SeidelZurich}, exact triangles are sent to exact triangles and in Section 3(e) of ibid., given a different cocycle representative with the same cohomology class, a non-canonical isomorphism is constructed between the abstract mapping cones, that are identified with Yoneda modules of mapping cones. Thus given an $A_{\infty}$-equivalence $f:CF(\ldisk_m,\ldisk_m)\to H^0(\ldisk_m,\ldisk_m)$ which induces an isomorphism $\mbox{Tw}(f):\mbox{Tw}(CF(\ldisk_m,\ldisk_m))\to \mbox{Tw}(H^0(\ldisk_m,\ldisk_m))$, we can directly use the functoriality of mapping cones to conclude that
$$\mbox{Tw}(f)(cone(\mathfrak{r}))=cone(\mbox{Tw}(f)^1(\mathfrak{r})),$$
where $\mathfrak{r}$ is the short Reeb chord and use that the morphism ${\mbox{Tw}(f)}^1(\mathfrak{r})$ vanishes.

By \cite[Theorem 1.14]{GPSSD}, $\mbox{Tw}(\pwfc)$ is generated by the interior fiber $\fibre{z}$ and the objects $D_{p_1}\oplus D_{p_2}\oplus\ldots \oplus D_{p_n}$, taken over all $\La_m$ ranging over the marked points $m\in\mkpts$. Therefore \cref{eq:ldisk_directsum} implies that the large linking disks $\ldisk_m$, ranging over the marked points $m\in\mkpts$, along with the interior fiber $\fibre{z}$, also generate $\mbox{Tw}(\pwfc)$.
\end{proof}

\begin{remark}
In \cref{prop:gen_internal_infinity_fibers}, generation is defined as having quasi-equivalent categories of modules over twisted complexes, as in e.g.~ \cite{GPSCV,GPSMS}. Given our use of cones, we must work in a triangulated framework, and we use the enhancement of twisted complexes; see also \cite[Chapter 1]{SeidelZurich}.\hfill$\Box$
\end{remark}

\noindent Note that there is no canonical choice of grading for the linking disks $D_{p_i}$. By definition, we refer to the grading on $D_{p_i}$ as the \textit{geometric} grading if $CF^\ast(D_p,\bcyl)$ is concentrated in degree $0$. For a Legendrian knot and a linking disk, this geometric grading is the grading that is invariant under contact isotopies of (the boundary of) the linking disk around the Legendrian.\\

\noindent Let us summarize useful properties of the linking disks. Intuitively, these follow from the fact that the surgery models in \cite[Section 6.1] {BEEsurgery} and \cite[Lemma 6.4]{ELVknot} show that the dynamics around a (punctured) handle can be taken to be the geodesic flow inside of the handle and the given Reeb flow outside. In more detail:

\begin{lemma}[Properties of linking disks]\label{lem:properties_linkingdisks} In the notation above, the following properties hold:\\
\begin{enumerate}
    \item (Locality) Let $m_1,m_2$ be two distinct punctures with Legendrian links $\lknot_{m_1},\lknot_{m_2}$ and $D_1,D_2$ corresponding linking disks. Then  $HW^\ast(D_1,D_2)\cong 0$.\\
    
    \item Let $\mathbb{D}$ be a collection of linking disks for the Legendrian links ${\bf\lknot}$, each equipped with the geometric grading. Then its endomorphism algebra $HW^\ast(D,D)$ is concentrated in non-negative degree and, if links are Reeb-positive, then it is concentrated in the zero degree.\\
    
    \item Let $D_{p_i},D_{p_{i+1}}$ be two linking disks at the same puncture, as above. Then the short Reeb chords $D_{p_i}\to D_{p_{i+1}}$ have cohomological grading $1$.\\

\end{enumerate}
\end{lemma}

\begin{proof}
Part (1) Locality follows from the fact that the geodesic flow near the positive unit conormal cannot escape the neighborhood of the punctures. For Part (2), we translate the surgery partial wrapping introduced in \cite{Ekholmlekili} to the partial wrapping in \cite{GPSCV}. That is, we consider an open Liouville sector $W_{\Lambda}$ obtained by attaching the disk cotangent bundle of $\lknot\times [1,\infty)$, the sectorial analogue of a punctured handle attachment. The Reeb chords of cotangent fibres in $W_{\Lambda}$ are computed using \cite[Lemma 88]{Ekholmlekili}, up to some action cut-off, from which it follows that the Reeb chords all have non-negative degrees.\footnote{Strictly speaking, the form of the metric used in ibid.~is $f(r)d\theta^2+dr^2$ for some smooth, positive function $f(r)$, with the properties as in the discussion in Section B.3 in ibid., in order to ensure that the geodesics do not travel arbitrarily deep into infinity. In our case, the boundary is a circle and so using the conical metric is sufficient.} This open Liouville sector $W_{\Lambda}$ can be truncated to give a closed Liouville sector $W_{\Lambda}(T)$ given by attaching the disk cotangent bundle of $\lknot\times [1,T]$ instead, for a large enough $T\gg 1$. The inclusion functor induced from the inclusion of Liouville sectors $W_{\Lambda}(T)\to W_{\Lambda}(T')$, $T<T'$, is cohomologically fully faithful for $T$ sufficiently large enough. Therefore, it suffices to compute the wrapped Floer cohomology in $W_{\Lambda}(T)$. Now, $W_{\Lambda}(T)$ is deformation equivalent to the closed Liouville sector obtained by removing a standard 1-jet neighborhood of $\Lambda$ in the boundary, c.f.~\cite[Example 2.15]{GPSCV}, the wrapped Floer cohomology of which is isomorphic to the partially wrapped Floer cohomology in \cite{GPSCV}, by \cite[Corollary 3.9]{GPSSD}. Therefore, partially wrapped Floer cohomology must also be concentrated in non-negative degrees.

For Part (3), we first present the full argument for the rank $2$ case. Define the grading on $D_{p_{2}}$ so that the short Reeb chord lies in $\hom^0(D_{p_i},D_{p_{2}})$. For this grading
$$\hom(\cone(D_{p_1}\to D_{p_{2}})[-1],\blag)=\hom(\D_l,\blag)$$
with the large linking disk $\D_l$, connecting sum $D_{p_1}$ and $D_{p_{2}}$, being equipped with the geometric grading. In particular, $D_{p_{2}}[-1]$ is the graded Lagrangian object equivalent to the linking disk at $p_{2}$ equipped with the geometric grading. Therefore, since $$\hom^{\ast}(D_{p_1},D_{p_{2}}[-1])=\hom^{\ast-1}(D_{p_1},D_{p_{2}})$$
the short Reeb chord is a degree $+1$ element in $\hom^{\ast}(D_{p_1},D_{p_{2}}[-1])$, as claimed. The higher rank case follows from the exactly same argument. 
\end{proof}

\subsubsection{A concluding remark}\label{sssec:Stokes} There are more topological descriptions of $\wfc$ and $\pwfc$, not involving pseudo-holomorphic curves. As stated above, $\wfc$-$\kmod$ is quasi-equivalent to the category of local systems $\Loc(S)$ and thus perfect modules $\mbox{Perf}(\wfc)$ is equivalent to its compact objects $\Loc^c(S)$. More generally, \cite[Theorem 1.1]{GPSMS} shows that the category of perfect modules $\mbox{Perf}(\pwfc)$ is equivalent to the category $\Sh_{\lknot}^c(S)$ of compact objects in $\Sh_{\lknot}(S)$, the category of sheaves with singular support contained in $\lknot$. Independently, since $S$ is a surface in the study of Floer theory and spectral networks, one can also approach the generalization from $\wfc$ to $\pwfc$ as the generalization from local systems to Stokes local systems. We refer to \cite{Boalch21_TopologyStokes} for a detailed description of the latter.\\

We conjecture that the relation between partially wrapped Fukaya categories and spectral networks, as discussed above, is compatible with Stokes local systems. In an admittedly unimpressive degree of detail, we briefly sketch what we expect, as follows. Given a Betti Lagrangian $L\sse T^*S$ and a compatible spectral network $\snetwork\sse S$, we expect the existence of the following commutative diagram:

\begin{equation}\label{diag1}
\begin{tikzcd}[row sep=tiny]
& \pwfc \\
\Loc^c(L) \ar[ur, "\EuScript{Y}"] \ar[dr, "\Phi_{\snetwork}^{St}"'] & \\
& \mbox{St}(S,\lknot) \ar[uu,"\EuScript{M}"]
\end{tikzcd}
\end{equation}

\noindent where the notions in the diagram are:\\

\begin{enumerate}
    \item $\mbox{St}(S,\La)$ is the category of Stokes local systems, as described in \cite[Section 8]{Boalch21_TopologyStokes}, with the Stokes diagrams given by $\lknot$.\footnote{The category might be needed to be dg-enhanced, derived and consider compact objects to fit the diagram above, but all meaningful content is in ibid.} Technically, ibid.~defines Stokes local systems given irregular data -- which is naturally algebraic -- but the same definitions allow for the notion of a Stokes local system associated to any Legendrian link isotopic to (the lift of) a front given by the cylindrical closure of a positive braid.\\

    \item The Yoneda functor $\EuScript{Y}$ is given by $V\mapsto\ycylv$, as discussed in \cref{subsection:PWFC} and \cref{sssec:generation_int_inf} above.\\

    \item The functor $\Phi_{\snetwork}^{St}$ is essentially the functor $\Phi_{\snetwork}$ constructed and studied in \cref{section_Floer}. The only difference is that $\Phi_{\snetwork}(V)$ was described as a local system, and $\Phi_{\snetwork}^{St}(V)$ should be a Stokes local system. As discussed in \cite{GMN12_WallCrossCoupled,GMN13_Framed}, the non-abelianization process actually produces a Stokes local system, so this adjustment on the codomain is obtained by the same methods.\\

    \item The functor $\EuScript{M}$ can be described explicitly as follows. Given a Stokes local system $V_{St}\in\mbox{St}(S,\La)$, the $A_\infty$-module $\EuScript{M}(V_{St}):\pwfc\lr\kmod$ is determined by
    $$\EuScript{M}(V_{St})(\fibre{z})=V_z,\quad \EuScript{M}(V_{St})(F_\infty)=0,$$
    where $\fibre{z}$ is an internal fiber, $F_\infty$ any fiber at infinity and $V$ is the underlying local system of $V_{St}$. By \cref{prop:gen_internal_infinity_fibers}, this determines $\EuScript{M}$ on objects. Note that for any $z,w\in S$ two $\snetwork$-adapted points, we will also have $\EuScript{M}(V_{St})(i_*[\alpha_{zw}])=V([\alpha_{zw}])$.\\
\end{enumerate}

\noindent The commutativity of Diagram (\ref{diag1}) is a categorical analogue of the results in \cref{section_Floer}, comparing the non-abelianiation functor $\Phi_\snetwork$ to the Family Floer functor $\ff$.

\begin{remark} (1) Note that \cite{Boalch21_TopologyStokes} establishes equivalent descriptions of $\mbox{St}(S,\La)$, comparing to Stokes graded and Stokes filtered local systems. The category of Stokes filtered local systems is much closer to $\Sh_{\lknot}(S)$. In contrast, the Floer-theoretic context has generators that directly give a {\it grading}, provided by the Lagrangian intersections, and not just the filtration more naturally associated to $\Sh_{\lknot}(S)$. Therefore, the passage from $\Sh_{\lknot}(S)$ to $\pwfc$ provides, in a sense, a symplectic topological viewpoint on the splitting result \cite[Theorem 1.1]{Boalch21_TopologyStokes}, which produces a canonical Stokes grading from the filtration.\\

\noindent (2) It would be desirable to also have a conceptual description of $\EuScript{M}$, not given in terms of explicit generators of $\pwfc$. Specifically, we expect $\EuScript{M}$ to be an equivalence and its inverse should be given by an enhancement of the Family Floer methods in \cref{subsection:familyfloer}.\hfill$\Box$
\end{remark}

By \cref{thm:pwfcyonneda}, Diagram (\ref{diag1}) commutes if the Legendrian $\lknot$ at infinity is Reeb-positive. This is the case, for instance, if the Legendrian links at infinity are all cylindrical closures of positive braids, which includes the case of any irregular data with no multiplicities. In this Reeb-positive case, it also follows that the non-abelianization functor is injective on objects: it injectively sends local systems on $L$ to Stokes local systems on $S$, which was conjectured by the original works \cite{GMN12_WallCrossCoupled,GMN13_Framed}. Indeed, given $V,V'\in\Loc(L)$, wrapped cohomology $HW^\ast((\bcyl,V),(\bcyl,V'))$ is isomorphic to the Morse cohomology of $\blag$ twisted by $V^{\ast}\otimes V'$ coefficients. In particular, its degree-0 group is non-zero if and only if $V$ and $V'$ are isomorphic, from which injectivity follows.
%    \newpage
%%%%%%%%%%%%%%%%%%%%%%%%%%%%%%%%%%%%%%%%%%%%%%%%%%%%%%%%%%%%%%%%%%%%%%%%
%%%%%%%%%%%%%%%%%%%%%%%%%%%%%%%%%%%%%%%%%%%%%%%%%%%%%%%%%%%%%%%%%%%%%%%%
%%%%%%%%%%%%%%%%%%%%%%%%%%%%%%%%%%%%%%%%%%%%%%%%%%%%%%%%%%%%%%%%%%%%%%%%
\section{Spectral networks and weaves}\label{section_weave}

%%%%%%%%%%%%%%%%%%%%%%%%%%%%%%%%%%%%%%%%%%%%%%%%%%%%%%%%%%%%
%%%%%%%%%%%%%%%%%%%%%%%%%%%%%%%%%%%%%%%%%%%%%%%%%%%%%%%%%%%%
%%%%%%%%%%%%%%%%%%%%%%%%%%%%%%%%%%%%%%%%%%%%%%%%%%%%%%%%%%%%

Weaves were introduced in \cite{legendrianweaves}, in the context of contact and symplectic topology. Since then, they have also been productively used in the study of cluster structures, see e.g.~ \cite{casals2023demazureweavesreducedplabic,casals2023microlocal,casals2024clusterstructuresbraidvarieties}.The object of this section is to present a first relation between weaves and spectral networks, in particular discussing \cref{thm:aug_graphs}. The necessary ingredients on weaves are presented in \cref{ssec:weaves_Betti_surfaces}, and the combinatorial construction of a spectral network $\snetwork_\ww$ from a Demazure weave $\ww$ is given in \cref{ssec:networks_Demazure_weaves}. We then conclude with \cref{ssec:explicit_comp} presenting a number of explicit examples and computations. Note that, given the rigorously established bridge between weaves and cluster algebras, cf.~\cite{casals2024clusterstructuresbraidvarieties}, this construction of $\snetwork_\ww$ provides a precise account on how cluster coordinate arise in spectral networks, in line with some predictions from supersymmetric computations, see e.g.~\cite{MR3263304}. In addition, since \cite{casals2023demazureweavesreducedplabic,casals2023microlocal,casals2024clusterstructuresbraidvarieties,} give explicit descriptions of Donaldson-Thomas transformations in terms of weaves, this connection between weaves and spectral networks therefore gives a direct method to compute the BPS spectrum generator.

%%%%%%%%%%%%%%%%%%%%%%%%%%%%%%%%%%%%%%%%%%%%%%%%%%%%%%%%%%%%
%%%%%%%%%%%%%%%%%%%%%%%%%%%%%%%%%%%%%%%%%%%%%%%%%%%%%%%%%%%%
%%%%%%%%%%%%%%%%%%%%%%%%%%%%%%%%%%%%%%%%%%%%%%%%%%%%%%%%%%%%

\subsection{Weaves on Betti surfaces}\label{ssec:weaves_Betti_surfaces} Let $(S,\mkpts,{\bf\lknot})$ be a Betti surface of rank $n$ and let $\beta_i$ be a positive (cyclic) braid word in $n$-strands whose cylindrical closure gives the link $\La_i$ at the puncture $m_i\in\mkpts$, $i\in[1,|\mkpts|]$. The definition of a weave in \cite{legendrianweaves} is adapted to the context of Betti surfaces as follows:

\begin{definition}[Weaves]\label{def:weave}
A weave $\ww$ in $(S,\mkpts,{\bf\lknot})$ is a properly embedded graph in $S$ with edges decorated by permutations in $W(G)\cong S_n$ such that:
\begin{enumerate}
    \item There are only three types of vertices for $\ww$, as depicted in Figure \ref{fig:weaves_types_vertices}:

    \begin{center}
	\begin{figure}[h!]
		\centering
		\includegraphics[scale=2.1]{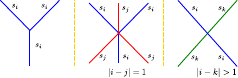}
		\caption{The three types of vertices allowed in a weave $\ww$. (Left) Trivalent vertex, where all edges are decorated with the same permutation $s_i\in W(G)$. (Center) Hexavalent vertex, with the edge decorations alternating between $s_i,s_j$ with $|i-j|=1$. (Right) Tetravalent vertex, with edge decorations as drawn with $|i-k|>1$.}
		\label{fig:weaves_types_vertices}
	\end{figure}
\end{center}
\vspace{-0.5cm}
    \item The asymptotics of the graph $\ww$ at a puncture $m_i\in\mkpts$ are such that the permutations decorating its edges coincide with the corresponding positive cyclic braid word $\beta_i$, for all $i\in[1,m]$.\hfill$\Box$
\end{enumerate}
\end{definition}

\noindent The asymptotic condition in \cref{def:weave}.(2) is drawn in Figure \ref{fig:weaves_near_boundary}. In figures, we often adopt the convention that colors \textcolor{blue}{blue} and \textcolor{red}{red} are adjacent, and \textcolor{blue}{blue} and \textcolor{DarkGreen}{green} are not adjacent, see e.g.~\cref{fig:weaves_types_vertices}. By convention, when drawing in small rank $n\in\N$, we typically use \textcolor{blue}{blue for $s_1$}, \textcolor{red}{red for $s_2$} and \textcolor{DarkGreen}{green for $s_3$}.

 \begin{center}
	\begin{figure}[h!]
		\centering
		\includegraphics[scale=2.1]{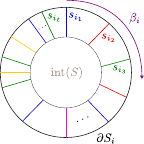}
		\caption{A weave $\ww$ near a circle $\dd S_i$ around a puncture $m_i\in\mkpts$, according to \cref{def:weave}.(2). The decorations on the edges of $\ww$ are written as $s_{i_j}$, and we use colors to emphasize that they might be different. The cyclic braid word being spelled is $\beta_i=s_{i_1}s_{i_2}s_{i_3}\ldots s_{i_\ell}$.}
		\label{fig:weaves_near_boundary}
	\end{figure}
\end{center}

In this section, we focus on the case of the Betti surface $(S^2,\{\infty\},\beta\delta(\beta))$ given by a once punctured 2-sphere $S=S^2$, with the braid around the unique puncture $\{\infty\}$ being the form $\beta\delta(\beta)$.\footnote{We can and will assume that $\delta(\beta)=w_0$, see e.g.~\cite{casals2024clusterstructuresbraidvarieties}.} In higher rank $G=\SL_n$, these Betti surfaces already lead to many interesting cases, including Bers-Nevins-Roberts \cite{BerkNevinsRoberts82_NewStokes} and all braid varieties \cite{casals2024clusterstructuresbraidvarieties}, thus all double Bruhat cells, positroids and double Bott-Samelson cases. To study weaves and spectral networks in $(S^2,\{\infty\},\beta\delta(\beta))$, we use the class of Demazure weaves introduced in \cite[Section 4]{casals2024algebraicweavesbraidvarieties}, cf.~also \cite[Section 4]{casals2024clusterstructuresbraidvarieties}. Recall that a Demazure weave $\ww:\beta\lr\delta(\beta)$ for $\beta=\s_{i_1}\cdots\s_{i_\ell}$ is a weave in $\R^2$ drawn vertically top-to-bottom such that\\
\begin{itemize}
    \item[(i)] There are two types of semi-infinite edges: north and south. The northern semi-infinite edges coincide with the vertical upwards semi-rays $\{j\}\times\R_+$ outside of a compact set, $j\in[1,\ell]$, and the souththern semi-infinite edges coincide with the vertical downwards semi-rays $\{k\}\times\R_+$, $k\in[1,\ell(\delta(\beta))$.\\

    \item[(ii)] The weave $\ww$ is never tangent to the horizontal lines $\R\times\{y\}$, $y\in\R$.\\

    \item[(iii)] The only vertices of $\ww$ are as literally as drawn in \cref{fig:weaves_types_vertices} when $\ww$ is scanned top-to-bottom. That is, the trivalent vertex cannot go from $s_i$ on top to $s_{i}s_i$ at the bottom: it must always go from $s_is_i$ at the top to $s_i$ at the bottom.\\
\end{itemize}

\noindent %An example of a Demazure weave is depicted in Figure {\RC add this}. 
Here we identify $S^2\setminus\{\infty\}\cong\R^2$ via a diffeomorphism and choose Cartesian coordinates $(x,y)\in\R^2$. Demazure weaves are a diagrammatic description of a sequence of braid words where the only allowed moves are braid moves and the nil-Hecke type move $\s_i^2\to\s_i$, the sequence starts with $\beta$, corresponding to a horizontal slice in positive large $y$-height, and ends with $\delta(\beta)$, corresponding to a horizontal slice in negative large $y$-height. For the purpose of constructing spectral networks, we use the following modification of a Demazure weave:

\begin{definition}[Bending of Demazure weaves]\label{def:bent_demazure_weave}
Let $\ww:\beta\lr\delta(\beta)$ be a Demazure weave, as depicted in Figure \ref{fig:weaves_bending} (left). By definition, the right bending of $\ww$ is the weave in $\R^2$ drawn as in Figure \ref{fig:weaves_bending} (right).\hfill$\Box$
\end{definition}

\begin{center}
	\begin{figure}[h!]
		\centering
		\includegraphics[scale=0.8]{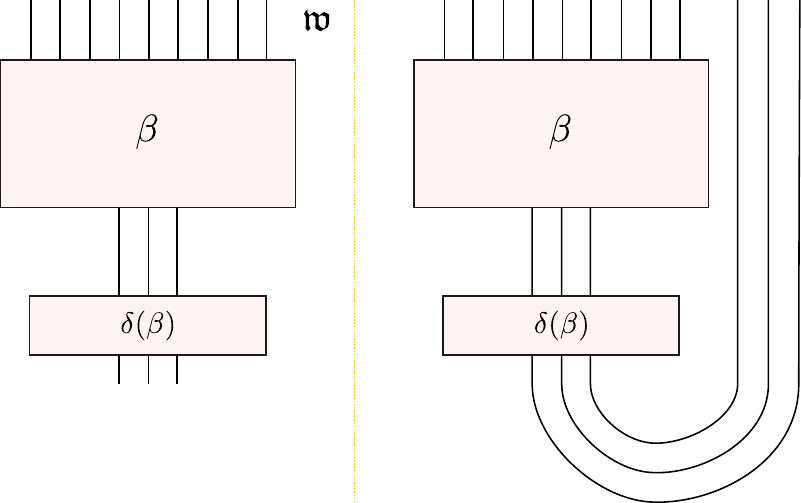}
		\caption{(Left) A general Demazure weave $\ww:\beta\to\delta(\beta)$. (Right) The right bending of $\ww$. The use of the bent weave, as depicted on the right, is that it can be seen as an embedded exact Lagrangian filling of the $(-1)$-closure of $\beta\delta(\beta)$.}
        \label{fig:weaves_bending}
	\end{figure}
\end{center}

\noindent The relevance of \cref{def:bent_demazure_weave} is that the bent Demazure weave represents a spatial wavefront whose Legendrian lift is such that its Lagrangian projection in $(\R^4,\omega_\std)$ an exact embedded Lagrangian filling of the $(-1)$-closure $\La_{\beta\delta(\beta)}\sse(\R^3,\xi_\std)$. This follows from the original construction of weaves in \cite{legendrianweaves} and the observation that bending a Demazure weave $\ww$ does not introduce any Reeb chords. It is the starting idea behind many recent results connecting symplectic topology and cluster algebras, cf.~e.g.~ \cite{casals2022conjugate,casals2023demazureweavesreducedplabic,casals2024clusterstructuresbraidvarieties,casals2024algebraicweavesbraidvarieties}. In short, the bending of a Demazure weave $\ww\sse\R^2$ describe an exact Betti Lagrangian $L_\ww$ for the Betti surface $(S^2,\{\infty\},\beta\delta(\beta))$.

\begin{remark} (1) The need for bending is to ensure there are no concave ends. A Demazure weave $\ww$, without bending, would naturally yield an exact embedded Lagrangian cobordism from $\La_{\delta(\beta)}$, in the concave end, up to $\La_{\beta}$, in the convex end. The bending effectively fills the concave end, yielding a Lagrangian filling of $\La_\beta$. We plan to study non-empty concave ends in future work as well.\\

\noindent (2) The theory of weaves, as developed in \cite{legendrianweaves}, works in an arbitrary Betti surface. For instance, \cite[Section 3.1]{legendrianweaves} builds weaves from $N$-triangulations of surfaces, whose spectral networks will match (and generalize) the level $N$ lift spectral networks from \cite{GMN14_Snakes}. Although it is possible to define a generalization of Demazure weaves to any Betti surface, with higher genus and punctures, the wealth of Demazure weaves in $S^2\times\{\infty\}$ and their associated spectral networks, as momentarily built in \cref{ssec:networks_Demazure_weaves}, already covers many interesting examples and applications, so we focus on those.\hfill$\Box$
\end{remark}

%%%%%%%%%%%%%%%%%%%%%%%%%%%%%%%%%%%%%%%%%%%%%%%%%%%%%%%%%%%%
%%%%%%%%%%%%%%%%%%%%%%%%%%%%%%%%%%%%%%%%%%%%%%%%%%%%%%%%%%%%
%%%%%%%%%%%%%%%%%%%%%%%%%%%%%%%%%%%%%%%%%%%%%%%%%%%%%%%%%%%%

\subsection{Spectral networks for Demazure weaves}\label{ssec:networks_Demazure_weaves} By \cref{ssec:weaves_Betti_surfaces}, bending a Demazure weave $\ww\sse\R^2$ gives an exact Betti Lagrangian $L_\ww\sse(T^*\R^2,\la_\std)$ for the Betti surface $(S^2,\{\infty\},\beta\delta(\beta))$. The goal is to construct a Morse spectral network $\snetwork_\ww\sse\R^2$ for $L_\ww$ from the data of the weave $\ww$.

\begin{definition}[Augmentation forest of $\ww$]\label{def:augmentation_forest}
Let $\ww\sse\R^2$ be the right bending of a Demazure weave with $\ell$ trivalent vertices. By definition, the augmentation forest $\snetwork_\w\sse \R^2$ of $\ww$ is the spectral network built as follows:

\begin{enumerate}
\item Start with an empty prespectral network $\fnetwork_0=\snetwork_0=\emptyset$. Scan the trivalent vertices of $\ww$ bottom-to-top, indexing by $i\in[1,\ell]$ in this bottom-to-top order for the vertices: at the $i$th trivalent weave vertex $v_i$ of $\ww$, create three directed flowlines as in Figure \ref{fig:weaves_specnet_construction}.(1).\\

\item Flow each of the three edge flowlines in Figure \ref{fig:weaves_specnet_construction}.(1) from $v_i$ as follows. Flowlines (a) and (b) continue following the weave edges up until their reach the top of the weave. These flowlines go straight through both 4 and 6-valent vertices and continue up on the left of a trivalent, as depicted in Figure \ref{fig:weaves_specnet_construction}.(2); the permutation labels of the flowline change by conjugation with the labels of the weave lines. The flowline (c) is extended to the right, through weave edges, changing its labeling accordingly, until it reaches a weave edge whose permutation label coincides with that of the flowline (at that stage). Right before that point (to its left), the edge (c) is extended upward following that weave edge until it reaches the top. The result of these extensions is a prespectral network $\fnetwork_i$, which is obtained by adding these flowlines to the spectral network $\snetwork_{i-1}$.\\

\item Consider the prespectral network $\fnetwork_i$ and add additional flowlines to create a consistent extension of $\fnetwork_i$. The only new flowlines that need to be created appear at the intersection of the (new) flowlines both at $v_i$ with the flowlines of $\snetwork_{i-1}$. Then declare $\snetwork_i$ to be the resulting spectral network.\\

\item Declare $\snetwork_\ww$ to be $\snetwork_\ell$, the result of applying the above process to all trivalent vertices of $\ww$, scanned bottom to top.\hfill$\Box$
\end{enumerate}
   \end{definition}

\begin{center}
	\begin{figure}[h!]
		\centering
		\includegraphics[scale=1]{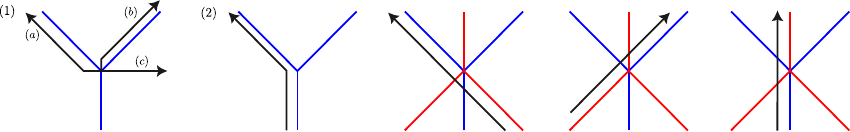}
		\caption{(1) The local model of the flowlines in the augmentation forest $\Theta_\w\sse \R^2$ near a trivalent vertex. (2) The local models for the flowline behaviour near trivalent vertices and hexavalent vertices north of their source trivalent vertex.}
        \label{fig:weaves_specnet_construction}
	\end{figure}
\end{center}

The spectral network $\snetwork_\w\sse \R^2$ in \cref{ssec:networks_Demazure_weaves} is constructed combinatorially from $\ww$, it is just a directed graph with labels, a priori without any analytic meaning or relation to $\dfs$-flowtrees. The adiabatic degeneration results in \cref{section_mnetwork} along with the spatial front associated to a weave $\ww$, as described in \cite{legendrianweaves}, imply that $\snetwork_\w$ is in fact obtained by adiabatically degenerating the pseudoholomorphic strips associated to the augmentation induced by the Lagrangian filling $L_\ww$ of the $(-1)$-closure of $\beta\delta(\beta)$. We record this fact as follows, which implies \cref{thm:aug_graphs}:

\begin{proposition}\label{prop:augforest_is_aug}
Let $\ww:\beta\to w_0$ be a Demazure weave for a braid word $\beta$ with $\delta(\beta)=w_0$. Consider the associated embedded exact Lagrangian filling $L_\ww\sse T^*\R^2$ of $\La_{\beta w_0}$ and the induced augmentation
$\e_\ww:\SA(\La_{\beta w_0})\lr k[H_1(L_\ww)]$
for its Legendrian contact dg-algebra. Then, the spectral network $\snetwork_\w\sse \R^2$ is the union of all the $\dfs$-trees obtained by adiabatically degenerating the rigid pseudo-holomorphic strips contributing to $\e_\ww$.\hfill$\Box$
\end{proposition}

\cref{def:augmentation_forest} and \cref{prop:augforest_is_aug} provide a wealth of interesting spectral networks directly related to Floer theory. In particular, given the rigorously establish existence of and weave calculus for cluster structures, they yield a versatile set of tools for computing BPS states. For instance, the Donaldson-Thomas transformations for Demazure weaves are explicitly described in \cite[Section 8.3]{casals2024clusterstructuresbraidvarieties}, cf.~also \cite[Section 5.4]{casals2023microlocal}, which readily yield the BPS spectrum generator, a.k.a.~BPS monodromy, for the associated physical system.

\begin{remark}
(1) In the meromorphic case, some of the $\dfs$-trees in $W_\ww$ with a positive puncture at a Reeb chord in $\La_{\dd_\ww}$ align with the BPS solitons described by the (non-degenerate) open finite webs in \cite[Section 3.2]{GNMSN}.\\

\noindent (2) The Lusztig cycles of a Demazure weave $\ww$ are introduced in \cite[Section 4]{casals2024clusterstructuresbraidvarieties}, see also \cite[Section 3]{casals2023microlocal}. Intuitively, in the meromorphic case, the mutable Lusztig cycles relate to $W_\ww$ in that they should correspond to finite webs representing 4d BPS states, as in \cite[Section 3.1]{GNMSN}, for a degenerate spectral network related to (but in a different phase than) the non-degenerate network $W_\ww$. This relation remains to be explored in detail.\hfill$\Box$
\end{remark}

%%%%%%%%%%%%%%%%%%%%%%%%%%%%%%%%%%%%%%%%%%%%%%%%%%%%%%%%%%%%
%%%%%%%%%%%%%%%%%%%%%%%%%%%%%%%%%%%%%%%%%%%%%%%%%%%%%%%%%%%%
%%%%%%%%%%%%%%%%%%%%%%%%%%%%%%%%%%%%%%%%%%%%%%%%%%%%%%%%%%%%

\subsection{Explicit computations of spectral networks from Demazure weaves}\label{ssec:explicit_comp} In this subsection, we discuss several cases of \cref{def:augmentation_forest} and \cref{prop:augforest_is_aug} in detail, establishing the comparison between the flowlines of spectral networks and the adiabatic degenerations of pseudoholomorphic strips.

%%%%%%%%%%%%%%%%%%%%%%%%%%%%%%%%%%%%%%%%%%%%%%%%%%%%%%%%%%%%
%%%%%%%%%%%%%%%%%%%%%%%%%%%%%%%%%%%%%%%%%%%%%%%%%%%%%%%%%%%%

\subsubsection{A 2-stranded example in detail}\label{sssec:2strands_examples_1}

Consider the braid word $\beta=\s_1^6$ in 2-strands and the Legendrian dg-algebra $\SA_\beta$ of the $(-1)$-closure of $\beta\delta(\beta)$. Let us denote its degree-0 Reeb chords by the variables $z_6,z_5,z_4,z_3,z_2,z_1,w_1$, read left-to-right, with $w_1$ corresponding to the crossing of $\delta(\beta)=\s_1$. This labeling is drawn in Figure \ref{fig:weaves_example111111_specnet}. Consider the Demazure weave $\ww:\beta\lr\delta(\beta)$ depicted in Figure \ref{fig:weaves_example111111_specnet} (left). In the language of pinching sequences, the Lagrangian filling $L_\ww$ is obtained by pinching the crossings in the Lagrangian projection associated to $z_1,z_3,z_4,z_5,z_2$, in this order. A computation, as in \cite{CasalsNg22}, shows that $L_\ww$ gives the augmentation
$$\e_\ww:\SA_\beta\lr k[H_1(L_\ww,T)]$$
uniquely determined by its non-zero values:
\begin{align*} 
\begin{split}
\e_\ww(z_6) &=  \frac{1}{s_2 s_3^2 s_4^2 s_5^2}-\frac{1}{s_5} \\ 
\e_\ww(z_5) &= s_5-\frac{1}{s_4}\\
\e_\ww(z_4) &= s_4-\frac{1}{s_3}\\
\e_\ww(z_3) &= s_3\\
\e_\ww(z_2) &= s_2-\frac{1}{s_1}-\frac{1}{s_3}+\frac{1}{s_3^2 s_4}-\frac{1}{s_3^2 s_4^2 s_5}\\
  \end{split}
\quad
  \begin{split}
\e_\ww(z_1) &= s_1\\
\e_\ww(w_1) &= \frac{1}{s_1^2 s_2}-\frac{1}{s_1}\\
\e_\ww(t_1) &= -\frac{1}{s_1 s_2 s_3 s_4 s_5}\\
\e_\ww(t_2) &= s_1 s_2 s_3 s_4 s_5.
 \end{split}
\end{align*}
where $t_1,t_2$ are marked points $T=\{t_1,t_2\}$ at the (right at the) right of $w_1$, $t_i$ in the $i$th strand, counting from the bottom, and $s_i\in H_1(L_\ww,T)$ is the relative cycle Poincar\'e dual to the trivalent vertex whose associated top-right edge corresponds to $z_i$.\\

By construction, $\e_\ww$ counts rigid pseudo-holomorphic strips and records their (relative) homology classes. Therefore, there ought to be a precise correspondence between the non-zero terms in $\e_\ww(z_i)$, and $\e_\ww(w_1)$, and the rigid flowtrees of the spectral network $\snetwork_\ww$ for $L_\ww$ which are asymptotic to $z_i$, and $w_1$. This is indeed the case: $\snetwork_\ww$ is drawn in Figure \ref{fig:weaves_example111111_specnet} (right), and each term of $\e_\ww$ above exactly corresponds to such flowtrees. For instance, the flowtrees in the spectral network that are coloured not black correspond to the terms of
$$\e_\ww(z_2) = s_2-\frac{1}{s_1}-\frac{1}{s_3}+\frac{1}{s_3^2 s_4}-\frac{1}{s_3^2 s_4^2 s_5},$$
as they are asymptotic to the Reeb chord $z_2$. Specifically, the green flowline gives the term $-s_3^{-1}$, the yellow gives $s_3^{-2}s_4$, the orange $-s_3^{-2}s_4^{-2}s_5^{-1}$, the purple gives $s_2$ and the rightmost cardinal flowline $-s_1^{-1}$.

\begin{center}
	\begin{figure}[h!]
		\centering
		\includegraphics[scale=1.8]{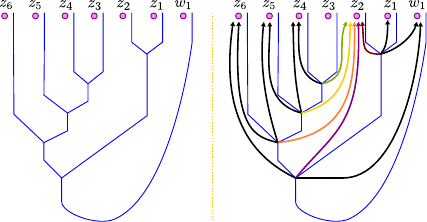}
		\caption{(Left) A Demazure weave $\ww$ for $\beta=\s_1^6$, with its Reeb chords at the (top) boundary marked with a pink dot and labeled. (Right) The spectral network $\snetwork_\ww$ compatible with $L_\ww$, as built above.}
        \label{fig:weaves_example111111_specnet}
	\end{figure}
\end{center}

%%%%%%%%%%%%%%%%%%%%%%%%%%%%%%%%%%%%%%%%%%%%%%%%%%%%%%%%%%%%
%%%%%%%%%%%%%%%%%%%%%%%%%%%%%%%%%%%%%%%%%%%%%%%%%%%%%%%%%%%%

\subsubsection{Mutation in 2-strands}\label{sssec:2strands_examples_2} Similar to the example above, this correspondence between terms of the augmentation $\e_\ww$ and the flowlines in $\snetwork_\ww$ works in the same lin for other 2-stranded braids and their Demazure weaves. The mutation in spectral networks can be analyzed in the local model of $\beta=\s_1^3$. The two sides of a weave mutations are depicted in Figure \ref{fig:weaves_example1111_specnet}, $\ww_1$ on the left and $\ww_2$ on the right.

\begin{center}
	\begin{figure}[h!]
		\centering
		\includegraphics[scale=1.8]{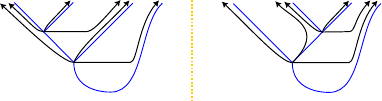}
		\caption{The two sides of a weave mutation, $\ww_1$ and $\ww_2$, and their associated spectral networks $\snetwork_{\ww_1}$ and $\snetwork_{\ww_2}$.}
        \label{fig:weaves_example1111_specnet}
	\end{figure}
\end{center}
\noindent If we label the degree-0 Reeb chords associated to the crossings of $\beta\delta(\beta)=\s_1^4$ by $z_3,z_2,z_1,w_1$, reading left-to-right, and use $T=\{t_1,t_2\}$ as above, the two augmentations
$$\e_{\ww_i}:\SA_{\sigma_1^3}\lr k[H_1(L_{\ww_i},T)],\quad i=1,2,$$
are given by

\begin{align*} 
\begin{split}
\e_{\ww_1}(z_3) &= \frac{1}{s_1 s_2^2}-\frac{1}{s_2} \\ 
\e_{\ww_1}(z_2) &= s_2\\
\e_{\ww_1}(z_1) &= s_1-\frac{1}{s_2}\\
\e_{\ww_1}(w_1) &= -\frac{1}{s_1}\\
\e_{\ww_1}(t_1) &= -\frac{1}{s_1 s_2}\\
\e_{\ww_1}(t_2) &= -s_1 s_2.
  \end{split}
\quad
  \begin{split}
\e_{\ww_2}(z_3) &= -\frac{1}{s_2} \\ 
\e_{\ww_2}(z_2) &= s_2-\frac{1}{s_1}\\
\e_{\ww_2}(z_1) &= s_1\\
\e_{\ww_2}(w_1) &= \frac{1}{s_1^2 s_2}-\frac{1}{s_1}\\
\e_{\ww_2}(t_1) &= -\frac{1}{s_1 s_2}\\
\e_{\ww_2}(t_2) &= -s_1 s_2.
 \end{split}
\end{align*}

\noindent Note that in the $z$-coordinates of the braid variety $X(\s_1^3)$, the mutable cluster variable $A_1$ for the torus chart $T_{\ww_1}\sse X(\s_1^3)$ associated to $L_{\ww_1}$ is $A_1=z_2$, and the frozen is $z_2z_3-1$. The quiver $Q_{\ww_2}$ has a unique arrow from the frozen to the mutable. After mutation, the mutable cluster variable $A_2$ for the torus chart $T_{\ww_2}\sse X(\s_1^3)$ associated to $L_{\ww_2}$ is $A_2=z_3$, and the frozen $z_3z_2-1$. The quiver $Q_{\ww_1}$ has a unique arrow from the mutable to the frozen. By the mutation formula, we must have that the mutated variable $A_1'$ of $A_1=z_2$ is such that
$$A_1A_1'=1+(z_2z_3-1)$$
and thus indeed $A_1'=A_2=z_3$. The corresponding mutation in the spectral networks $\snetwork_{\ww_1}$ and $\snetwork_{\ww_2}$, i.e.~its change in the flowtrees, is drawn in Figure \ref{fig:weaves_example1111_specnet}. The change in the flowlines, seen as adiabatic limits of pseudo-holomorphic strips, matches the difference in the augmentation $\e_{\ww_1}$ and $\e_{\ww_2}$ and, using $z$-coordinates, it becomes the standard cluster mutation.

\begin{center}
	\begin{figure}[h!]
		\centering
		\includegraphics[scale=1.3]{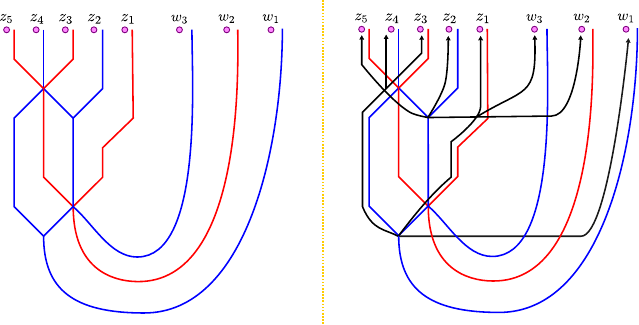}
		\caption{(Left) The weave $\ww_{bnr}$ obtained by considering the front of the Legendrian lift of the real part of the Bers-Nevins-Roberts spectral curve $\Sigma_{bnr}$. (Right) The spectral network $\snetwork_{\ww_{bnr}}$ according to \cref{ssec:networks_Demazure_weaves}.}
        \label{fig:weaves_exampleBNR_specnet}
	\end{figure}
\end{center}

%%%%%%%%%%%%%%%%%%%%%%%%%%%%%%%%%%%%%%%%%%%%%%%%%%%%%%%%%%%%
%%%%%%%%%%%%%%%%%%%%%%%%%%%%%%%%%%%%%%%%%%%%%%%%%%%%%%%%%%%%

\subsubsection{The Berk-Nevins-Roberts spectral network}\label{sssec:3strands_examples_BRN} The remarkable article \cite{BerkNevinsRoberts82_NewStokes} first introduced the {\it new} Stokes lines, born at certain intersections of two ordinary Stokes lines and studied them in the context of the WKB asymptotics, c.f.~\cref{ex:irregularclass_Legendrian}. These flowlines only appear at higher rank $n\geq3$, and mark a crucial difference between the rank-2 case and higher ranks. The specific example studied in \cite{BerkNevinsRoberts82_NewStokes} corresponds to the spectral curve
$$\Sigma_{bnr}:=\{(z,w)\in T^*\C: w^3-3w+x=0\}.$$
The weave $\ww_{bnr}$ associated to its real part $\Re(\Sigma)$ is depicted in Figure \ref{fig:weaves_exampleBNR_specnet} (left). It is already depicted with the bending, so that it is apparent that can be seen as a bent Demazure weave for the braid word $\beta=\s_2\s_1\s_2\s_1\s_2$. Note that the associated spectral network $\snetwork_{bnr}$, depicted in \cite[Figure 1]{BerkNevinsRoberts82_NewStokes} was analyzed in \cite[Section 5]{MR3322389} from the perspective of harmonic maps to buildings.\\

\begin{center}
	\begin{figure}[h!]
		\centering
		\includegraphics[scale=1.4]{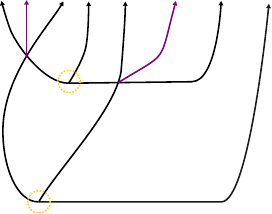}
		\caption{The spectral network $\snetwork_{\ww_{bnr}}$, which indeed coincides with the spectral network $\snetwork_{bnr}$ from \cite{BerkNevinsRoberts82_NewStokes}. In this figure, the two new Stokes line are highlighted in purple, and the two trivalent $\dfs$-vertices are circled in dashed yellow.}
        \label{fig:weaves_exampleBNR2_specnet}
	\end{figure}
\end{center}

From the Floer-theoretic perspective, we label the Reeb chords of $\beta\delta(\beta)=\s_2\s_1\s_2\s_1\s_2(\s_1\s_2\s_1)$ as in Figure \ref{fig:weaves_exampleBNR_specnet}: $z_1,z_2,z_3,z_4,z_5$ correspond to the crossings of $\beta$, reading right-to-left, and $w_1,w_2,w_3$ to those of $\delta(\beta)$, also reading right-to-left. The spectral network $\snetwork_{\ww_{bnr}}$ associated to $\ww_{bnr}$ according to \cref{ssec:networks_Demazure_weaves} is depicted in Figure \ref{fig:weaves_exampleBNR_specnet} (right). We have also drawn $\snetwork_{\ww_{bnr}}$ on its own in Figure \ref{fig:weaves_exampleBNR2_specnet}, where the two {\it new} Stokes lines, born at creation vertices, are depicted in purple. Here we can readily see that the spectral network $\snetwork_{\ww_{bnr}}$ obtained by adiabatic degeneration of pseudoholomorphic strips coincides with the original spectral network $\snetwork_{bnr}$ from \cite{BerkNevinsRoberts82_NewStokes} associated to $\Sigma_{bnr}$.

%%%%%%%%%%%%%%%%%%%%%%%%%%%%%%%%%%%%%%%%%%%%%%%%%%%%%%%%%%%%
%%%%%%%%%%%%%%%%%%%%%%%%%%%%%%%%%%%%%%%%%%%%%%%%%%%%%%%%%%%%

\subsubsection{An example of a cancellation pair}\label{sssec:3strands_examples_cancellation} Let us consider a local model for a cancellation pair coming from applying a Reidemeister III move twice to $\beta=\s_2\s_1\s_2$. This corresponds to the sequence of braid words $\s_2\s_1\s_2\to\s_1\s_2\s_1\to\s_2\s_1\s_2$. First, we will consider one Reidemeister III move $\s_2\s_1\s_2\to\s_1\s_2\s_1$ on its own, as realized by the weave $\ww_{R3}$ in Figure \ref{fig:weaves_example_cancelation_pair}. This is understood as a local move that can be inserted anywhere in a weave where there is a subword $\s_{i+1}\s_i\s_{i+1}$.

\begin{center}
	\begin{figure}[h!]
		\centering
		\includegraphics[scale=1.2]{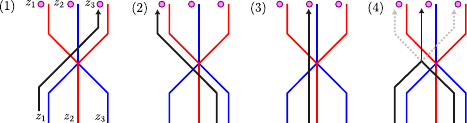}
		\caption{The four degenerations of pseudoholomorphic strips contributing to the dg-algebra morphism induced by a Reidemeister III move on a Legendrian front. In (4), the bifurcation occurs as a creation vertex from the flowlines in (1) and (2), the traces of which are drawn in dashed gray in (4).}
        \label{fig:weaves_example_cancelation_pair}
	\end{figure}
\end{center}

\noindent The embedded exact Lagrangian cobordism associated to $\ww_{R3}$ gives a dg-algebra morphism $\Psi_{R3}$ between the Legendrian contact dg-algebra of the convex end, which has the piece $\s_2\s_1\s_2$, to the concave end, which contains $\s_1\s_2\s_1$. Denote by $z_1,z_2,z_3$ the Reeb chords next to the crossings associated to $\s_2\s_1\s_2$, and keep the same notation for those in $\s_1\s_2\s_1$, as depicted in Figure \ref{fig:weaves_example_cancelation_pair}.(1). This morphism $\Psi_{R3}$ is computed in \cite[Section 4.1.2]{CasalsNg22} and, locally in this model, reads as follows:
$$\Psi_{R3}(z_1)=z_3,\quad \Psi_{R3}(z_3)=z_1,\quad \Psi_{R3}(z_2)=z_2\pm z_1z_3,$$
where we do not specify the sign, cf.~\cite{CasalsNg22} for the choices that give precise signs. Each of these four terms exists because of a pseudoholomorphic strip. Their adiabatic degenerations, in the form of flowtrees, are depicted in Figure \ref{fig:weaves_example_cancelation_pair}. Namely, Figure \ref{fig:weaves_example_cancelation_pair}.(1) explains $\Psi_{R3}(z_3)=z_1$, \ref{fig:weaves_example_cancelation_pair}.(2) explains $\Psi_{R3}(z_1)=z_3$, \ref{fig:weaves_example_cancelation_pair}.(3) is the tree corresponding to the term $z_2$ in $\Psi_{R3}(z_2)=z_2\pm z_1z_3$ and the $\dfs$-tree in \ref{fig:weaves_example_cancelation_pair}.(4) corresponds to the term $z_1z_3$.\\

Let us now perform two consecutive Reidemeister moves, inverses of each other. The corresponding weave $\ww_{c}$ is depicted in Figure \ref{fig:weaves_example_cancelation_pair2}.(1). By the candy twist move in \cite[Figure 19]{legendrianweaves}, this weave $\ww_c$ is equivalent to the trivial weave with three strands labeled $s_1,s_2,s_1$ and no crossings. Therefore, by Hamiltonian invariance of the associated exact embedded Lagrangian cobordisms, it must be that the induced dg-algebra map given by $\ww_{c}$ is the identity and there {\it has to} be a cancelation account for the term $z_1z_3$ above appearing and disappearing.

\begin{center}
	\begin{figure}[h!]
		\centering
		\includegraphics[scale=1.2]{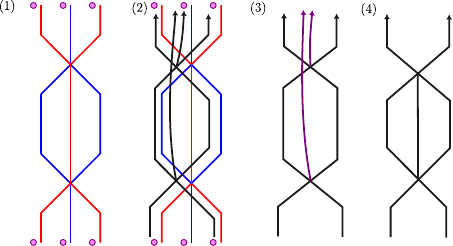}
		\caption{(1) The weave corresponding to two consecutive Reidemeister III moves $\s_2\s_1\s_2\to\s_1\s_2\s_1\to\s_2\s_1\s_2$. (2) The associated spectral network, which coincides with the adiabatic $\dfs$-tree degeneration of the pseudoholomorphic strips corresponding to the induce dg-algebra maps. (3) The spectral network on its own, with the two canceling terms highlighted in purple. (4) The result of deleting the walls of the spectral network in (3) that cancel each other. In this case, the interaction vertex that is not creation would be of {\it annihilation} type. The actual spectral network in (3) witnesses both canceling flowlines, and their cancelation is seen in the BPS index. More intuitively, the $\dfs$-tree corresponding to the upper purple flowline has a self-intersection point, unlike the tree for the purple flowline, hence its sign is opposite.}
        \label{fig:weaves_example_cancelation_pair2}
	\end{figure}
\end{center}

\noindent This cancellation is indeed seen in the spectral network $\snetwork_{\ww_c}$, which is depicted in Figure \ref{fig:weaves_example_cancelation_pair2}.(2). The two flowlines, with the same asymptotic end, that cancel each other are highlighted in purple in Figure \ref{fig:weaves_example_cancelation_pair2}.(3). In the physics literature, this configuration is sometimes depicted as in Figure \ref{fig:weaves_example_cancelation_pair2}.(4), where the parts of the walls of the canceling pair that have vanishing BPS index are not drawn: this is, with phases appropriately understood, related to the {\it setting sun} model in \cite[Section 7.4]{GNMSN}.

%%%%%%%%%%%%%%%%%%%%%%%%%%%%%%%%%%%%%%%%%%%%%%%%%%%%%%%%%%%%
%%%%%%%%%%%%%%%%%%%%%%%%%%%%%%%%%%%%%%%%%%%%%%%%%%%%%%%%%%%%

\subsubsection{A more elaborate 3-stranded example}\label{sssec:3strands_examples} Consider the 3-stranded braid word $\beta=(\s_2\s_1)^3\s_2$, so that $\beta\delta(\beta)=(\s_2\s_1)^5\s_2$ and the (-1)-closure of $\beta\delta(\beta)$ is the max-tb $(3,2)$-torus knot, i.e.~the max-tb trefoil presented in 3-strands. Consider its right inductive weave $\ww$, as depicted in Figure \ref{fig:weaves_example_3strands_specnet_1} (left), and its associated embedded exact Lagrangian filling $L_\ww$. This example illustrates how the correspondence between the flowlines of the spectral network $\snetwork_\ww$ and adiabatic degenerations of the pseudoholomorphic strips contributing to the augmentation  $\e_\ww$ induced by $L_\ww$. It is particularly helpful in understanding how terms can (and do) cancel, locally as in \cref{sssec:3strands_examples_cancellation}, in a globally interesting example.

\begin{center}
	\begin{figure}[h!]
		\centering
		\includegraphics[scale=0.9]{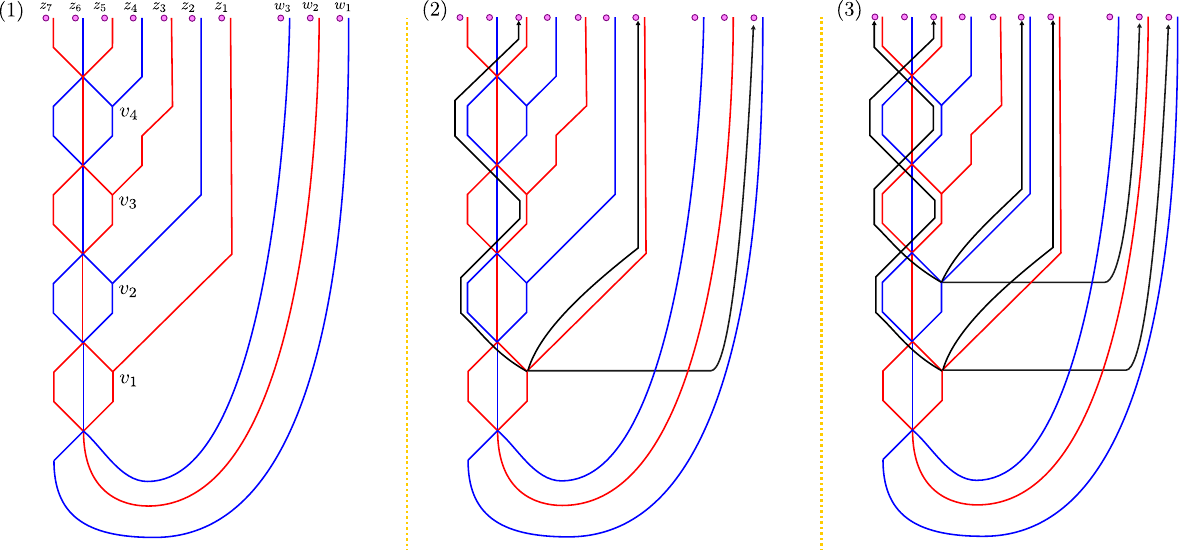}
		\caption{(1) The right inductive weave $\ww$ for $\beta=(\s_2\s_1)^3\s_2$. The four trivalent vertices are labeled $v_1,v_2,v_3,v_4$, the Reeb chords associated to crossings in $\beta$ are denoted $z_1,\ldots,z_{7}$ and those for $\delta(\beta)$ are denoted $w_1,w_2,w_3$. (2) The first prespectral network $\fnetwork_1$, consisting of the three flowlines starting at the trivalent $\dfs$-vertex $v_1$. (3) The second prespectral network $\fnetwork_2$, consisting of the flowlines starting at the two $\dfs$-vertices $v_1,v_2$. Notice that it is {\it not} consistent, as there are 4-valent vertices in $\fnetwork_2$ which should be creation vertices. A consistent extension is drawn in Figure \ref{fig:weaves_example_3strands_specnet_2}.(4).}
        \label{fig:weaves_example_3strands_specnet_1}
	\end{figure}
\end{center}

First, from the Floer-theoretic side, we label the Reeb chords as in Figure \ref{fig:weaves_example_3strands_specnet_1} (left). In these coordinates, the Lagrangian filling $L_\ww$ gives the augmentation
$$\e_\ww:\SA_\beta\lr k[H_1(L_\ww,T)]$$
which determined by its non-zero values:

\begin{align}\label{eq:augmentation_3strand_example}
\begin{split}
\e_\ww(z_7) &= -\frac{s_3}{s_2 s_4^2}-\frac{1}{s_4} \\ 
\e_\ww(z_6) &= \frac{s_2}{s_1 s_3^2}-\frac{1}{s_3}+\frac{1}{s_1 s_3 s_4} \\ 
\e_\ww(z_5) &= \frac{s_2 s_4}{s_1 s_3^2}-\frac{s_4}{s_3}\\
\e_\ww(z_4) &= s_4\\
\e_\ww(z_3) &= s_3\\
\e_\ww(z_2) &= s_2+\frac{s_3}{s_4}\\
\e_\ww(z_1) &= s_1-\frac{s_2}{s_3}-\frac{1}{s_4}
  \end{split}
\quad
  \begin{split}
\e_\ww(w_3) &= \frac{s_1}{s_2}-\frac{1}{s_3}\\
\e_\ww(w_2) &= -\frac{1}{s_2}\\
\e_\ww(w_1) &= -\frac{1}{s_1}\\
\e_\ww(t_1) &= -\frac{1}{s_1 s_3}\\
\e_\ww(t_2) &= -\frac{s_1 s_3}{s_2 s_4}.\\
\e_\ww(t_3) &= -s_2 s_4.
 \end{split}
\end{align}

where $t_1,t_2,t_3$ are marked points $T=\{t_1,t_2,t_3\}$ at the (right at the) right of $w_1$, $t_i$ in the $i$th strand, counting from the bottom. As above, $s_i\in H_1(L_\ww,T)$ denotes the relative cycle Poincar\'e dual to the trivalent vertex whose associated top-right edge corresponds to $z_i$.\\

\begin{center}
	\begin{figure}[h!]
		\centering
		\includegraphics[scale=0.9]{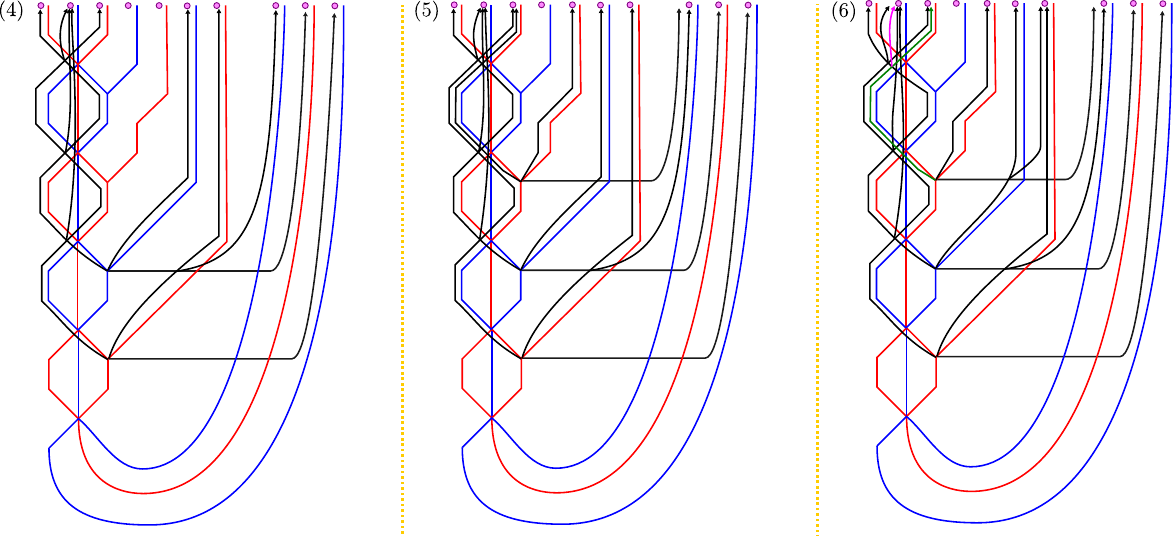}
		\caption{{Prespectral networks in the steps towards building $\snetwork_\ww$}. In (4) we added flowlines to make the network in Figure \ref{fig:weaves_example_3strands_specnet_1}.(3) consistent. In (5) the flowlines born from $v_3$ are added and (6) depicts the associated consistent extension.}
        \label{fig:weaves_example_3strands_specnet_2}
	\end{figure}
\end{center}

Second, independent of the Floer-theoretic computation, we consider the spectral network $\snetwork_\ww$ associated to the weave $\ww$, as built in \cref{ssec:networks_Demazure_weaves}. The final spectral network $\snetwork_\ww$ is drawn in Figure \ref{fig:weaves_example_3strands_specnet_4}. The process to obtain it is depicted in several steps and described as follows:

\begin{itemize}
    \item[(i)] The first prespectral network $\fnetwork_1$ is depicted in Figure \ref{fig:weaves_example_3strands_specnet_1}.(2), where the three walls of $\fnetwork_1$ are drawn, being born at the trivalent $\dfs$-vertex $v_1$.\\

    \item[(ii)] The second prespectral network $\fnetwork_2$ is depicted in Figure \ref{fig:weaves_example_3strands_specnet_1}.(3), where three additional walls are born at the $\dfs$-vertex $v_2$, in addition to those for $\fnetwork_1$ are drawn.\\

    \item[(iii)] Since the prespectral network $\fnetwork_2$ not consistent, we proceed by considering a consistent extension $\fnetwork_3$, as drawn in Figure \ref{fig:weaves_example_3strands_specnet_2}.(4). It is obtained by adding four flowlines to $\fnetwork_2$, resulting in four creation vertices. The prespectral network $\fnetwork_3$ is consistent.\\

    \item[(iv)] The next prespectral network $\fnetwork_4$ is depicted in Figure \ref{fig:weaves_example_3strands_specnet_2}.(5), where the three flowlines born in the $\dfs$-vertex $v_3$ are added. Such $\fnetwork_4$ is inconsistent, and thus we consider its consistent extension $\fnetwork_5$, which is drawn in Figure \ref{fig:weaves_example_3strands_specnet_2}.(6). The only flowline needed to be added to $\fnetwork_4$ for consistency is depicted in pink, which itself is caused by the flowline born in $v_3$ highlighted in green interacting with a flowline born in $v_2$.\\

    \item[(v)] The final step is to add the three flowlines born at $v_4$ and the additional flowlines to make the resulting prespectral network consistent. The resulting spectral network is $\snetwork_\ww$, which is drawn in Figure \ref{fig:weaves_example_3strands_specnet_3} (left), on top of the weave $\ww$. It is drawn on its own in Figure \ref{fig:weaves_example_3strands_specnet_3} (right).
\end{itemize}

\begin{center}
	\begin{figure}[h!]
		\centering
		\includegraphics[scale=1.2]{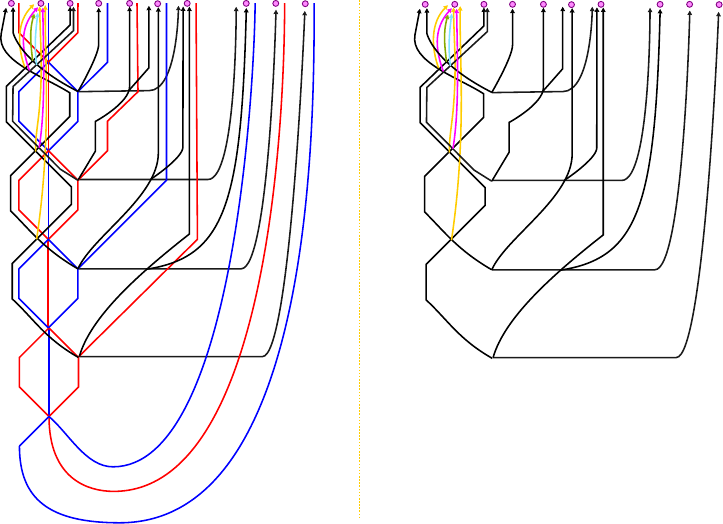}
		\caption{(Left) The spectral network $\snetwork_\ww$ on the weave $\ww$. (Right) The spectral network $\snetwork_\ww$ drawn on its own. The colors for some of the flowlines indicate that they are creation flowlines: those that share a color are created by (repeated) intersections of the same two flowlines. Some pairs of flowlines of the same color cancel each other, e.g.~the two pink flowlines cancel each other.}
        \label{fig:weaves_example_3strands_specnet_3}
	\end{figure}
\end{center}

The spectral network $\snetwork_\ww$, with all its flowlines, is drawn in Figure \ref{fig:weaves_example_3strands_specnet_3} (right), but some of these flowlines cancel each other. That is, when considering counts of BPS states or pseudoholomorphic strips, there are pairs of such flowlines where each of the two flowlines contributes the same term, but they do so with opposite signs. Specifically, this occurs in this example in two cases: the two pink flowlines in Figure \ref{fig:weaves_example_3strands_specnet_3} (right) cancel each other, and two of the yellow flowlines (the two ones born lower than the third one) also cancel each other. We have depicted the spectral network obtained by removing such flowlines in Figure \ref{fig:weaves_example_3strands_specnet_4}, where the two canceling yellow lines are drawn dashed and the two pink flowlines have been erased.

\begin{center}
	\begin{figure}[h!]
		\centering
		\includegraphics[scale=1.4]{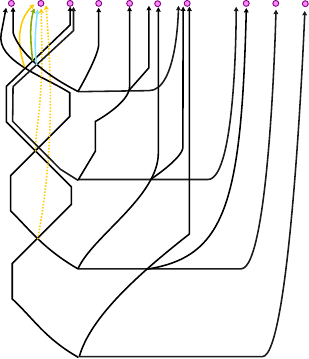}
		\caption{The spectral network $\snetwork_\ww$ associated to the weave $\ww$. The two dashed flowlines, drawn in yellow, cancel each other. The (solid) flowlines in yellow, green and blue are the flowlines born in creation vertices in the last step of the process, after adding the flowlines born in $v_4$ and making the resulting prespectral network consistent.}
        \label{fig:weaves_example_3strands_specnet_4}
	\end{figure}
\end{center}

\noindent We conclude by noting that the flowlines for the spectral network $\snetwork_\ww$, as drawn in Figure \ref{fig:weaves_example_3strands_specnet_4}, indeed match with the terms of the augmentation $\e_\ww$ as described in \cref{eq:augmentation_3strand_example}.

\begin{comment}
\subsection{Weaves and Demazure weaves}
We recall the basic notions of Legendrian weaves, which give a combinatorial recipe to build Legendrians in $\jbis$. If the Lagrangian projection is \textit{embedded}, we get exact Betti Lagrangians. Then we introduce a rich class of Legendrian weaves called Demazure weaves, following cite(yourpaper)
\subsection{Reidmeister move III and BNR}
We compute the Morse flowtrees for the exact Lagrangian cobordism given by the trace of Reidmeister move III. We also provide its \textit{holomorphic model}. We then compare these flowtrees to the interaction vertex in \cref{fig:verticesons-network}. We then study the BNR network in depth.   
\subsection{Double-streets and examples from Demazure weaves}
Given embedded weaves, there's a natural way to enhance the weave $N$-graph to a \textit{double-street} spectral network introduced in \cite{GNMSN}. We then show that their two-street resolutions can be understood as certain real-generic spectral networks.

\subsection{$n$-triangulations}
In \cite{legendrianweaves}, a free Legendrian weave corresponding to an $n$-triangulation of an ideal triangulation (without self-folding edges) was constructed. In \cite{SNandsnakes}, a topological spectral network for such an $n$-triangulation was constructed. It was then shown that non-abelianization recovers Fock-Goncharov coordinates for $PGL(n;\mathbb{C})$. We give a Morse network that is isotopic to the "Yang" spectral network in \cite{SNandsnakes}.
\end{comment}
\newpage

%\section{Wall-crossing formulas and $A_{\infty}$-bimodules}
%\input{section_algebra.tex}\label{section_algebra}
%\section{Degenerate Morse Spectral Networks}
%\section{Full Adiabatic Degeneration}

\appendix
\section{Aspects of the geometry of pseudo-holomorphic curves}

This appendix collects three technical results that are used in the manuscript: geometric boundedness of meromorphic spectral curves, a type of reverse isoperimetric inequality and a brief discussion on spin structures.

\subsection{Geometric boundedness of meromorphic spectral curves}
In this section, we show that the meromorphic spectral curves that satisfy the $O(-1)$-end conditions are geometrically bounded, in the sense of \cite[Definition 2.7]{JboundedGroman}. This is achieved by showing that the $C^\ell$-norm of the second fundamental form is uniformly bounded. We use the Sasaki metric obtained from metrics of the form $\abs{z}^{-2}\abs{dz}^2$ near the poles. This metric is geometrically bounded outside $T^{\ast}K$ for some compact subset $K\sse C$, since the metric $\abs{z}^{-2}\abs{dz}^2$ is flat. This is used in Sections \ref{section_adg}, \ref{section_Floer} and \ref{section_wfc}. The precise result reads as follows:

%We write $(\mathbb{R}^2,\omega
%_{st},g_{st})$ for $\mathbb{R}^2$ regarded as a flat K\"{a}hler manifold. 
%We need the following preliminary lemma. 
%\begin{lemma} 
%For $r\gg 1$, and some constants $c_1,c_2,\ldots, c_k)$, where $c_1\neq 0$, let $C:=(\log(r),c_1r+c_2r^{-1}+\ldots c_k r^{-k})$ be a plane curve in $\mathbb{R}^2$. Then $(r,f(r))$ is a geometrically bounded Lagrangian. 
%R\end{lemma}
%\begin{proof}
%For large $r$ enough, the velocity of $C$ is of the form $g(r)=\sqrt{c_1+O(r^{-1})}$ and so $g(r)$ is uniformly bounded below, and its derivatives are uniformly bounded above, and has $O(r^{-2})$ growth. The curvature of $C$ is given by the formula
%\[\kappa(r)=\frac{r^{-1}(O(r^{-2}))+(c_1+O(r^{-1}))r^{-2}}{g(r)^{3/2}}\]
%and so $\kappa(r)$ and its $C^k$-derivatives are $O(r^{-2+k})$. Therefore, since $\kappa(r)$ is the norm of the second fundamental form of $(f(r),g(r))$, uniform positivity of $g(r)$, the $C^k$ bound on $g(r)$ and $C^k$ bound on $f(r)$ imply that the $C^k$ norm of the second fundamental form is also uniformly bounded, and is of the form $O(r^{-1})$. This finishes the proof. 
%\end{proof}

\begin{proposition}\label{prop:geombounded}
Let $\scurve\sse T^*C$ be a meromorphic spectral curve. Then the smooth sheets of the scaled curves $\e\scurve\sse T^*C$ near infinity are uniformly geometrically bounded, uniformly with respect to $\e$, for any $\e\in\R_+$.
\end{proposition}
\begin{proof}
The argument is essentially finding an appropriate coordinate system where geometric boundedness is apparent. First, by construction, the germ of $\scurve$ at infinity can be written as
$$z\to (z,c_1t^{-k}+c_2t^{-k+1}+\ldots+c_{k-n}t^{-n}),\quad\mbox{ for some }n<k,\quad t^n=z.$$
Choose a branch cut for the logarithm and consider the flat conformal coordinate $W:=\int z^{-1}dz=\log(z)$; we can and do assume that $\abs{z}<1$, so that $W$ has \textit{negative} real part. Under this coordinate transformation, the smooth sheets of the meromorphic spectral curve transform into
\[(W,c_1\exp((-k/n+1)W)+\ldots+c_k).\] 
Set $m=k/n$ and note that $\abs{\exp((-m+1)W)}$ converges to $+\infty$ as $W\to -\infty$ because $m>1$. Also, since the metric $\abs{z}^{-2}\abs{dz}^2$ transforms into $\abs{dW}^2$, the metric on $\mathbb{C}^2_{W,dW}$ is the standard flat metric. By denoting $Q:=\exp(c_1(-m+1)W)$, and using the branch cut, we can write $\log(Q)=c_1(-m+1)W$. This coordinate $Q$ is the one we will use to prove geometric boundedness, as follows. Reparametrizing the sheets in the $Q$-coordinate we obtain
\[(b_0\log(Q),b_1 Q+O(Q^{1/(k-1)}))\]
for some $0<\alpha<1$ and  non-zero constants $b_0,b_1$. We now show that curves in $\C^2$ of the form 
\[(\log(Q),Q)\] are geometrically bounded over the entire complex plane, with the more general case with the $O(Q^{1/(k-1)})$ term follows similarly. Consider a curve of the form $(\log(Q),Q)$. If $\abs{Q}\leq 1$ then re-write $W=\log(Q)$ so that $\abs{Q}\leq 1$ is equivalent to $W$ having negative real part. The geometric boundedness of the curve for $\abs{Q}\leq 1$ then follows from the geometric boundedness of the curve $(W,e^{W})$ for $\Re{(W)}<0$.\\

\noindent If $\abs{Q}\geq 1$, note that given a holomorphic curve of the form $(f(z),g(z))$ in $\C^2$, the tangent vectors are given by $(f'(z),g'(z))$, $(if'(z),ig'(z))$ and the normal vectors are given by $(-\overline{g'(z)},\overline{f'(z)})$ and $(-i\overline{g'(z)},i\overline{f'(z)})$. Specializing to our case, the tangent vectors are given by $(1/Q,1)$ and $(i/Q, i)$, and the normal vectors are given by $(-1,1/\bar{Q})$ and $(-i,i/\bar{Q})$. The induced metric on the curve is $(\abs{Q}^{-2} + 1)\abs{dQ}^2$. Therefore, the non-vanishing Christoffel symbols are given by 
\begin{align*}
\Gamma^{Q}_{QQ}= (1+\abs{Q}^{-2})^{-1}(\bar{Q}^{-1}),\\
\Gamma^{\bar{Q}}_{\bar{Q}\bar{Q}}=(1+\abs{Q}^{-2})^{-1}(Q^{-1}),
\end{align*}
and the second fundamental form reads
$$B(\partial_z,\partial_z)= [-Q^{-2}/(1+\abs{Q}^{-1})^{3/2}](-1,1/\bar{Q}).$$

A direct computation implies that the derivatives of the second fundamental form, the Christofell symbols, and their derivatives are of the form $O(\abs{Q}^{-1})$. The norm of $\partial_z$ is uniformly positive, and the derivatives are again $O(\abs{Q}^{-1})$. Therefore, $C^\ell$ norm of the second fundamental form must be bounded as well. In addition, these bounds are uniform with respect to $\e$ because we can reparametrize
$(\log(Q),\epsilon Q)=(\log(P)-\log(\epsilon),P)$ and use the boundedness of $(\log(Q),Q)$.
\end{proof}

A consequence of \cref{prop:geombounded} is the following a priori diameter estimate for pseudoholomorphic disks with boundaries on Betti Lagrangians and conormals:

\begin{lemma}\label{lem:diametercontrol}
Let $\obis$ be a Betti surface and $\blag\sse T^*\obis$ a Betti Lagrangian. If $\blag$ is not exact, suppose that it is meromorphic near the punctures. Let $\precptset\subset \obis$ be a precompact submanifold of $\obis$ containing $\pi(\caustL)$ in its interior. Then there exists $\e_0=\e_0(M,g,\multigraph, K)\in\R_+$ such that for any $\e\in(0,\e_0)$ and $J_g$-holomorphic half-disk $u:(\halfdisk_{1},\partial^{+} \halfdisk_{1})\to (D^{\ast}M,D^{\ast}(\precptset)
\cup \e \multigraph)$ with energy less than $\e E$, the image of $u$ is contained in the disk bundle $D^{\ast}N_{CE}(\precptset)$, where $C\in\R_+$ depends only on $(M,g,\multigraph,K)$. 
\end{lemma} 
\begin{proof}
By \cref{prop:geombounded}, smooth sheeets of $\e\scurve$ are uniformly geometrically bounded near infinity. Thus an argument as in \cite[Chapter V]{Audin1994HolomorphicCI} implies that that given any enery cut-off $E\in\R_+$, there is a constant $C_0\in\R_+$ such that pseudo-holomorphic half-disks with boundary on a \textit{single} smooth sheet of $\e\scurve$ satisfy the estimate. Since $\e\scurve$ is multi-sheeted, we set $C=\max(\rho^{-1},1)\cdot C_0$ and iterate. For the monotonicity argument, we need to choose a small neighborhood of points on $\blag$ such that its intersection with $\e\multigraph$ is contractible and connected. The size of such a neighborhood will be $\rho\e$ and shrinking the size of the neighborhood affects the diameter given by monotonicity estimate by $\rho^{-1}\e^{-1}$. Since $\rho^{-1}$ is fixed, the diameter is bounded by $CE$, as required.
\end{proof}

%%%%%%%%%%%%%%%%%%%%%%%%%%%%%%%%%%%%%%%%%%%%%%%%%%%%%
%%%%%%%%%%%%%%%%%%%%%%%%%%%%%%%%%%%%%%%%%%%%%%%%%%%%%

\begin{comment}

{\YJN I apologize in advance...}

{\YJN This lemma is used. I don't know if it is proved in a measure theory book. I just proved it myself.
\begin{lemma}\label{lemma:monotoneabsolutecontinuitiy}
Let $0\leq ...f_n\leq f_{n+1}...$ be a uniformly bounded monotone sequence of absolutely continuous functions defined on $[a,b]$ such that their derivatives are also non-negative almost everywhere. Then the pointwise limit $f=\lim_{m\to \infty}f_n$ is also absolutely continuous.
	\end{lemma}
	\begin{proof}
		By the monotone convergence theorem, 
\[f(r)=\lim f_n(r)= \lim f_n(a)+\lim \int_0^r f'_n(r)dr=f(a)+\int_0^r \lim f'_n(r)dr,\]
for all $r$. Now apply the fundamental theorem of Lebesgue calculus.
	\end{proof}}
\end{comment}

%%%%%%%%%%%%%%%%%%%%%%%%%%%%%%%%%%%%%%%%%%%%%%%%%%%%%
%%%%%%%%%%%%%%%%%%%%%%%%%%%%%%%%%%%%%%%%%%%%%%%%%%%%%

\subsection{Truncated reverse isoperimetric inequality} Reverse isoperimetric inequalities in the context of pseudoholomorphic curves bounded by Lagrangians appeared in \cite{Gromanreverse}, cf.~also \cite{Duvalreverse}. As pointed out in \cref{section_adg}, we believe an inequality of this type is also required in the arguments of \cite{Morseflowtree} and, for completeness, we include a proof here.

Let $X$ be a compact almost K\"{a}hler manifold and $K$ a union of finitely many balls in $X$. Let $L\sse X$ be a Lagrangian submanifold and $r_0$ the radial injectivity radius of $L$ outside $K$. We consider $r_0$ to be small enough such that outside of $K$, the squared radial distance function $h:=\rho^2:D_{r_0}^{\ast}L\to \mathbb{R}$ is strictly plurisubharmonic with respect to $J$. For $\delta\in\R_+$, we write $N_{\delta}(K)$ for the $\delta$-neighborhood of $K$. The necessary reverse isoperimetric inequality reads as follows:

\begin{theorem}\label{thm:truncatedreverseiso}
In the notation above, there exist a choice of $\delta\in\R_+$ and a constant $T\in\R_+$ such that any $J$-holomorphic curve
$$u:(S,\partial S)\to (X,L\cup K)\quad\mbox{ with energy}\quad \frac{1}{2}\int \abs{du}^2\leq E$$ has boundary length outside $N_{2\delta}(K)$ bounded above by $TE$.

%, there exists $C>0$ such that the boundary of 

%$u^{-1}(int(K)^c)$

	%  , there exists some $0<c(X,L,K,\delta,E,r_0)<\infty$ such that any $J$-holomorphic curve $u:(S,\partial S)\to (X,L\cup K)$ with energy $<E$ and  compact has its boundary length outside of $N_{\delta}K$ bounded above by $c(X,L,K,E)$. The constant $c(X,L,K,E)=O(E)$ depends only on the geometry of $X$, the relative geometry of $L$, and some choices made with respect to $K$. 
\end{theorem}

\begin{proof}
At core, the argument uses the coarea formula to slice in terms of the level sets of the distance function $\rho$ and then use the Lagrangian boundary condition to study bounds for the length. We have structured the proof in four steps: summarizing, we will choose an appropriate bump function $\beta$ around $K$ and define functions $a(r)$ which coarsely measures the area outside $K$. Then we show that the inequality
\[ra'(r)\geq a(r)\frac{1-Dr}{1+Cr}\] for some constants $C,D\in\R_+$ that only depend on $(X,K,L,r_0)$ but not on $u$. The required bound on length will follows from that.

First, let us construct the support function $\beta:X\to \R$, which typically depends on $u$, though the resulting constants $C,D$ will not be affected. Fix a reference constant $\delta\in(0,r_0/2)$ and $\delta_u\in(0,\delta)$. We then choose a non-negative smooth function $\beta:X\to \R_{\geq 0}$ such that
\begin{enumerate}
    \item $\beta$ vanishes on $N_{\delta_0}(K)$,
    \item $\beta$ is strictly positive on $N_{\delta_0}(K)^c-\partial N_{\delta_0}(K)$,
    \item $\beta$ is equal to $1$ outside $N_{2\delta}(K)$.
\end{enumerate}
Choose $\delta_u$ above generically enough so that $u$ is transverse to $\partial N_{\delta}(K)$, and sufficiently close to $\delta$ so that the corresponding bump function $\beta$ has its derivative bounded above by $C\delta^{-1}$, for some constant $C\in\R_+$, independent of $\delta_u$. The transversality condition is needed because in Step 3 below we will need $u^{-1}(\partial N_{\delta}(K))$ to have \textit{finite} length, although its length does not contribute to our inequality.\\

Set $N:=N_{\delta_{0}}(K)$ for the $\delta_0$-neighborhood of $K$ and choose a strictly decreasing sequence $\delta_m\to  \delta_u$ so that $\rho/\beta$ is Lipshitz on the $\delta_m$-neighborhood $N_m:=N_{\delta_m}(K)$. The functions $a(r)$ and $l^{\beta}(r)$ are defined, for almost every real number, as follows.
		\begin{align}
		&	a(r):=\int_{\{\rho\leq \beta r\}\cap u\cap {N^c}} dA
		&&a_m(r):=\int_{\{\rho\leq \beta r\}\cap u\cap N_m^c} dA\\
        &l^\beta(r):= \int_{\{\rho=\beta r\}\cap u\cap N^c}\beta dl.
 		%&a_m^{\beta}(r)=\int_{\{\rho\leq \beta r\}\cap u\cap B_n^c} \beta dA,
        &&l_m^{\beta}(r):=\int_{\{\rho=\beta r\}\cap u\cap N_m^c}\beta dl.
	\end{align}%a_m^{\beta}(r)=\int_{\{\rho\leq \beta r\}\cap u\cap N_m^c} \beta dA,&&
Here, the integral is taken with respect to the induced metric on the image of $u$. The reason why $a_m(r)$ and $l_m^{\beta}(r)$ are introduced is that we may apply the co-area formula. Since the integrand is always non-negative, by the monotone convergence theorem, we can freely replace the integral and limit.\\
%Choose a strictly increasing sequence $\delta_m\to \delta$. Let $N_m=$

    {\bf Step 1.} We show that the function $a_m(r)$ is absolutely continuous and satisfies
\begin{align}\label{coarea3}
		a_m(r)\geq \int_0^{r} \frac{l_m^{\beta}(\tau)}{1+\tau C}d\tau\quad\text{and}\quad a'_m(r)\geq \frac{l_m^{\beta}(r)}{1+rC}\quad \text{ a.e}. 
	\end{align}
We deduce \cref{coarea3} from the coarea formula, as follows. Since $\beta$ is positive on the complement $N^c$, it follows that $\rho/\beta$ is Lipschitz on $N_m^c$ for any $m\in\N$.  Applying the coarea formula, we obtain
	\begin{align}\label{coarea1}
		a_m(r)=\int_{\{\rho\leq \beta r\}\cap u\cap N_m^c} dA=\int_{\{\rho\leq \beta r\}\cap u\cap N_m^c} \frac{1}{\abs{\nabla (\rho/\beta)}}\cdot {\abs{\nabla (\rho/\beta)}}dA=\int_{0}^r \int_{\{\rho=\beta \tau\}\cap u\cap N_m^c}  \frac{1}{\abs{\nabla (\rho/\beta)}} dl d\tau.
		%&=\int_{0}^r \int_{d\geq \beta r}\frac{\beta}{\abs{\nabla (d/\beta)}}d.
	\end{align}
Since $\abs{\nabla \rho}=1$, in the region $\rho=\tau\beta$ we have the estimate
	\begin{align}\label{bound1}
		\abs{\nabla(\rho/\beta)}\leq  \frac{1}{\beta}\abs{\nabla\rho}+\frac{\rho\abs{\beta'}}{\beta^2}
        \leq \frac{1}{\beta}\left(\abs{\nabla\rho}+\tau\abs{\beta'}\right)\leq \frac{1+\tau C}{\beta}.
	\end{align}Thus \eqref{coarea3} follows from from  \eqref{coarea1} and \eqref{bound1}.
    %{\YJN We are also using the following fact that given a submanifold $Y\subset X$ of a Riemannian manifold $X$, $f:X\to \R$ is a smooth function, and $f\vert_{Y}:Y\to \R$ is the restriction, $(\nabla f)^{\perp}=\nabla f_{\vert Y}, Y\subset X\to f$ where $\perp:TX\to TY$ is the orthonormal projection}
Since $a_m(r)$ is a absolutely continuous non-negative function with non-negative derivative, it follows that $a(r)=\lim_{m\to \infty}a(r)$ is also absolutely continuous. Furthermore, by the monotone convergence theorem applied with respect to $\delta_m$,
\begin{align}\label{coarea2}
		a(r)\geq \int_0^{r} \frac{l^{\beta}(\tau)}{1+\tau C}d\tau\quad\text{and}\quad a'(r)\geq  \frac{l^{\beta}(r)}{1+rC}\quad \text{ a.e} . 
	\end{align}

    {\bf Step 2.} Recall that we denote $h=\rho^2$ for the square of the distance function to $L$. Let us show the inequality
$$rl^{\beta}(r)\geq \frac{1}{2}\int_{\{\rho\leq \beta r\}\cap u\cap N^c}dd^ch.$$
The key observation is that, because $\beta$ vanishes near $\partial(N_{\delta}(K))$, we can suppress the contribution to the length coming from $u\cap \partial(N_{\delta}(K))$. For almost every $r$, we have the following:
	\begin{align}
		rl_m^{\beta}(r)&=\int_{\{\rho=\beta r\}\cap u\cap N_m^c} r\beta dl\geq \int_{\{\rho=\beta r\}\cap u\cap N_m^c} \frac{1}{2} \inner{\nabla h}{\nu}dl\label{eq:length-area1line1} \\
		&=\int_{\{\rho= \beta r\}\cap u\cap N_m^c} \frac{1}{2} d^c h=\frac{1}{2}\int_{\{\rho\leq \beta r\}\cap u\cap N_m^c} dd^c h-	
		\int_{\{\rho\leq \beta r\}\cap u\cap \partial N_m} d^c h. \label{eq:length-area1line2}
	\end{align}
    Here we have parametrized the oriented smooth curve $\{\rho=\beta r\}\cap u\cap N_m^c\}$ via $l(t)$ and took its unit normal $\nu(t)=-J(du\circ l(t))\abs{du\circ l'(t)}^{-1}$. We are also using the fact that the critical points of $u$ do not lie on $l(t)$ for generic $r$. For the inequality in \eqref{eq:length-area1line1}, we use that $\abs{\nabla \rho}=1$ and, from $\inner{\nabla h}{\nu}dl$ to $d^c h$, we used 
	\begin{align*}
		\inner{\nabla h}{\nu}dl=\inner{\nabla h}{\nu}\abs{du(l'(t))}l'(t)dt=-l'(t)dh\circ (J\circ du\circ \nu(t))=d^c h.\end{align*} The critical points of $u$ do not contribute because they are discrete, which is measure $0$. The first equality in \eqref{eq:length-area1line2} is more involved. For that, we first use Stokes' theorem to obtain
	\begin{align}\label{eq:boundaryparts}
		\int_{\{\rho\leq \beta r\}\cap u\cap N_m^c} dd^c h&=
		\int_{\{\rho\leq \beta r\}\cap u\cap \partial N_m} d^c h+\int_{\{\rho\leq\beta r\}\cap \partial u\cap N_m^c}d^c h\nonumber \\
		&+\int_{\{\rho=\beta r\}\cap u\cap N_m^c}d^c h.
	\end{align}
Here, $d^c h=0$ on $\{\rho\leq\beta r\}\cap \partial u$  since this set is contained in $L$. This explains why the second term in \eqref{eq:boundaryparts} vanishes. We now need to show that the term 	$\int_{\{\rho\leq \beta r\}\cap u\cap \partial N_m} d^c h$ converges to zero as $m\to 0$. For that, note that the set is contained in $u^{-1}(\partial N_m)=u^{-1}(\partial N_{\delta_m}(K))$, whose length is uniformly bounded for sufficiently large $m\in\N$. This is the only place where $u$ being transverse to $\partial(N_{\delta}(K))$ is used. Therefore, we have
\begin{align}
\int_{\{\rho\leq \beta r\}\cap u\cap \partial N_m} d^c h\leq \int_{\{\rho\leq \beta r\}\cap u\cap N_m^c} \frac{1}{2} \inner{\nabla h}{\nu}dl\leq \sup_{\partial{N_{\delta_m}(K)}} \abs{\beta} r_0 \cdot \text{Length}(u^{-1}(\partial N_{m}(K))). 
\end{align}
Now, since $\sup_{\partial{(N_{\delta_m}(K))}} \abs{\beta}$ converges to zero as $m\to \infty$, this term uniformly converges to zero as $m\to \infty$. The remaining term $\frac{1}{2}\int_{\{\rho\leq \beta r\}\cap u\cap N_m^c} dd^c h$ in $\eqref{eq:length-area1line2}$ converges to $\int_{{\rho\leq \beta r}\cap u\cap K^c}dd^ch$ by monotone convergence theorem, since $dd^c h$ is $J$-plurisubharmonic that the integrand is necessarily non-negative. In consequence, we obtained the claimed inequality
\begin{align}
	rl^{\beta}(r)=\lim_{m\to \infty} rl_m^{\beta}(r)\geq \frac{1}{2}\int_{\{\rho\leq \beta r\}\cap u\cap N^c} dd^c h.
\end{align}

{\bf Step 3.} Let us prove the inequality
$$\frac{1}{2}dd^c h\geq (1-Dr)\omega,\quad\mbox{ for some }D\in\R_+.$$
Indeed, for each point $p\in L$, we can trivialize the metric and bring $(X,L,J)$ to the configuration $(\mathbb{C}_{(x,y)}^n,\mathbb{R}_x^n,J)$ such that $h=\sum_{i,j} a_{ij}y^iy^j+ O(\abs{y}^3)$, $a_{ij}(x)=\delta_{ij}+O(\abs{x}^2)$, $J-i=O(\abs{y})$ and $g=g_{std}+O(\abs{y}^2)+O(\abs{x}^2)$. In particular, $\frac{1}{2}dd^c h$ restricted to $x=0$ is $\omega_{std}+O(\abs{y})$ and $g$ restricted to $x=0$ is $g_{std}+O(\abs{y}^2)$. It follows that $\omega=\omega_{std}+O(\abs{y})$ and so $\omega=dd^c h+O(\abs{y})$. For sufficiently small $y$, it follows that $dd^c h\geq (1-Dr)\omega$.\\

{\bf Step 4.} Let us combine the inequalities thus far to show that \[ra'(r)\geq a(r)\frac{1-Dr}{1+Cr}\] and conclude the proof of the truncated reverse isoperimetric inequality. Step 3 implies that $\lim rl^{\beta}(r)\geq (1-Dr)a(r)$ for some $D\in\R_+$.	Combining \eqref{coarea2} and \eqref{eq:length-area1line1}--\eqref{eq:length-area1line2}, we obtain the inequality
\begin{align}\label{eq:differentialinequality1}
	 ra'(r)\geq a(r)\frac{1-Dr}{1+C r}.
	 \end{align}
	This implies the inequality
	\begin{align}\label{eq:differentialinequality2}
	\frac{d}{dr} \log \left(a(r)\cdot \frac{C r+1}{r(1-Dr)}\right)\geq 0,
	\end{align}
	which implies that the function
	\[r\to  {\frac{a(r)}{r(1-Dr)^{-(C+D)/D}}}\]
	is non-decreasing. For $r\ll\min{(D^{-1},C^{-1})}$ small enough, choose a constant $T\in\R_+$ greater than $(1-Dr)^{-(C+D)/D}$ so that
	\begin{align}
		T\frac{a(r)}{r}\geq \lim_{s\to 0} \frac{\text{Area}(u;\{\rho\leq s\}\cap N_{2\delta}(K)^c)}{s} \frac{C s+1}{(1-sD)}.
	\end{align}
	This implies $a(r)T\geq rl$, as required.
\end{proof}

%%%%%%%%%%%%%%%%%%%%%%%%%%%%%%%%%%%%%%%%%%%%%%%%
%%%%%%%%%%%%%%%%%%%%%%%%%%%%%%%%%%%%%%%%%%%%%%%%
%%%%%%%%%%%%%%%%%%%%%%%%%%%%%%%%%%%%%%%%%%%%%%%%

\subsection{Local systems and twisted local systems}\label{ssec:lsys_vs_tlsys} In \cref{section_Floer} we directly work with twisted local systems. Since \cite{GPSCV} is presented only for (untwisted) local systems, \cref{section_wfc} is written in the context of local systems. In the framework of Betti Lagrangians, being spin, there is an equivalence between twisted local systems and local systems. To unify both approaches, we briefly set up the non-abelianization functor $\Phi_\snetwork$ in terms of spin structures and local  systems, instead of twisted local systems. The notation follows that of the manuscript.
%{\YJN Spin structures. The only place where they become relevant is passing from sphere bundle counts to the usual counts. The usual counts are necessary because of wrapped Fukaya.}
%\color{orange}
%We end this section with a discussion on how to pass from twisted local systems to genuine local systems. 
Let $\basepath\sse\obis$ be a path transverse to $\snetwork$. The difference line bundle $\beta=\widetilde{\mathfrak{s}}-\pi^{\ast}\mathfrak{s}$ can be understood as a $O(1;\mathbb{R})$-local system on the trivial bundle over $\blag-\caustL$. Such local systems are almost-flat, in that their monodromy along a small loop encircling a point in $\caustL$ is $-Id$. 

As in \cref{subsubsection:pathdetourclasses} , we will define a functor
$$\Phi^{\locsys, \beta}_{\snetwork}:\Loc_1(L)\lr\Loc_n(S).$$ We will adopt the same notational conventions as introduced in \cref{section_Floer}, using genuine paths. As before, we will write $\cotpath^i$ to denote the cotangent lifts of the paths. The analogue of \cref{def:detourpath} reads as follows.

\begin{definition}\label{def:detourpathspin}
Let $\locsys\in\Loc_1(\blag)$ be a $\GL_1(\C)$-local system on $\blag$ and $\basepath:[0,1]\to \bis$ a path transverse to $\snetwork$. Let $\spins,\spinb$ be spin structures on $\blag$ and $\obis$, respectively, and $\beta$ be their difference line bundle. 

%By definition, $\Phi^{\beta}_{\snetwork}$ is constructed by specifying
%\[\Phi^{\locsys,\beta}_{\snetwork}:\pi_1\big(\bis,\basepath(0),\basepath(1)\big)\to \mathbb{C}^{\beta,\locsys}[\pi_1\big(\scurve,\cotpath(0),\cotpath(1)\big)]\] which is defined as follows.
\begin{enumerate}[label=(\roman*)]
\item If $\basepath$ is represented by a free $\snetwork$-adapted path, then
 \[\Phi^{\locsys, \beta}_{\snetwork}(\basepath)=\sum_{i=1}^n \Phi^{\locsys,\beta}(\cotpath^{i})[\cotpath^{i}],\]
 where the sum runs over the $n$ lifts $\cotpath^{i}$ to $\blag$, $i\in [1,n].$
 \item If $\basepath$ is short and intersects $\snetwork$ at $z$, and $\ppp^u$ the unit velocity sphere bundle lift, then
\[\Phi^{\locsys\otimes \beta}_{\snetwork}(\basepath)=\sum_{i=1}^n \Phi^{\locsys\otimes \beta}(\cotpath^{i})[\cotpath^{i}]+\sum_{\soliton{z;\wall}\in \mathfrak{S}(z;\wall)} {\mu (\ccc{(z)})}\Phi^{\locsys\otimes \beta}\big(\pi\big(\ccc{(\ppp^u)}))[\pi\big(\ccc{(\ppp^u)}\big)].\]
 
 \item  If $\basepath$ is given as a composition $\basepath=\basepath_1\circ \basepath_2 \circ\ldots \circ \basepath_n$, where each $\basepath_i$ is either short or free, then
 \[\Phi^{\locsys,\beta}_{\snetwork}(\basepath)=\Phi^{\locsys, \beta}_{\snetwork}(\basepath_1)\circ \Phi^{\locsys,\beta}_{\snetwork}(\basepath_2)\circ \cdots \circ \Phi^{\locsys, \beta}_{\snetwork}(\basepath_k).\] 
 \hfill$\Box$
 \end{enumerate}
 \end{definition}

%Let $\locsys\in\Loc(\blag)$ and set
%\begin{align}\label{eq:formalsumspacespin}
%\mathbb{C}^{\beta,\locsys}[\pi_1\big(\scurve,\cotpath(0),\cotpath(1)\big)]:=\bigoplus_{i,j} \Big[\mathbb{C}[\pi_1(\scurve,\cotpath^i(0),\cotpath^j(1)]\otimes_{\mathbb{C}} \Hom_{V\otimes \beta}(\cotpath^i(0), \cotpath^j(1))\Big],
%\end{align}
% Vectors in $\mathbb{C}^{\beta,\locsys}[\pi_1\big(\scurve,\cotpath(0),\cotpath(1)\big)]$ and $\mathbb{C}^{\beta,\locsys}[\pi_1\big(\scurve,\cotpath'(0),\cotpath'(1)\big)]$ can be composed using $\circ$, the natural concatenation operation. 

The spin structure $\spins$ can be seen as a fibrewise double cover $\spins:\sphsc\to \sphsc$. Then the pull-back $\spins^{\ast}\locsys$ has monodromy $Id$ along $H$, and so using $\spins$ to choose a trivialization $\sphsc\simeq \sphsc\times U(1)_{\theta}$, and pulling back $\spins^{\ast}V$ by the $U(1)$-section $\scurve\to \sphsc$, we can regard $\spins^{\ast}V$ as a local system on $\blag$. Conversely, given a rank-$1$ flat connection $\nabla$ on the trivial bundle of $\blag$, we can use $\spins$ to trivialize $\sphsc\simeq \blag\times U(1)_{\theta}$ and consider the twisted connection $\nabla+\pi id\theta$. The pull-back under the fibrewise square map gives $\nabla+2\pi i d\theta$, which has trivial monodromy along the $U(1)$ component. The same construction recovers $\nabla$.\\

\noindent Similarly, the spin structure $\spinb$ trivializes the sphere bundle on $\obis$. Let $\basepath$ be a $\snetwork$-transverse path, and let $\ppp^{\spinb}(\basepath)$ be the lift of $\basepath$ via the trivialization $\spinb:\sphb\simeq \obis\times U(1)$. By construction, $\ppp^{\spinb}(\basepath)$ is a $\snetwork$-adapted path. We identify the stalk of $\spins^{\ast}V$, regarded as a local system on $\blag$, with that of $V$. The relation with $\Phi_\snetwork$ in \cref{section_Floer} reads:

 % There's a map ${\spins}^{\ast}:\mathbb{C}^{V}[\pi_1\big(\blag,\ppp^{\spinb}(\basepath),\ppp^{\spinb}(\basepath))]\to\mathbb{C}^{\spins^{\ast}\locsys,\beta}[\pi_1\big(\blag,\cotpath,\cotpath\big)]$ defined as follows. 
%\[\pi_{\ast}:\sum c_i[\ppp_i]\to \sum c_i \beta\big(\pi_{\ast}(\ppp_i)\big)\pi_{\ast}[\ppp_i].\]

%Th
%$\pi_1$-generators to the lift along $\spins:\sphsc\to\sphsc$ followed by projection to the base, and sending $\hom(V,V)$ coefficients to $\hom(\spins^{\ast}V\otimes \beta,\spins^{\ast}V\otimes \beta)$.

\begin{proposition}\label{prop:spintosphere}
$\tdfd(\ppp^{\spinb}(\basepath))=\Phi^{\spins^{\ast}V, \beta}_{\snetwork}(\basepath)$
\end{proposition}
\begin{proof}
It suffices to verify the formula for free paths and short paths. For free paths,
\begin{align*}
 \Phi^{\spins^{\ast}V\otimes \beta}_{\snetwork}(\basepath)=\sum_{i=1}^n \Phi^{\spins^{\ast}V\otimes \beta}(\basepath^{i})[\cotpath^{i}]=\sum_{i=1}^n \Phi^{V}(\ppp^{\spinb}(\basepath)^i)[\cotpath^{i}],
 \end{align*}
 since $\beta=\spins-\pi^{\ast}\spinb$. For short paths,
 \begin{align*}
 {\mu (\ccc{(z)})}\Phi^{\spins^{\ast}\locsys\otimes \beta}\big(\pi\big(\ccc{(\ppp^u)})\big)[\pi\big(\ccc{(\ppp^u)}\big)]=  {\mu (\ccc{(z)})}\Phi^{\locsys}\big(\ccc{(\ppp^{\spinb}(\basepath)})\big)[\pi\big(\ccc{(\ppp(\basepath)}\big)],
 \end{align*}
 since $\ccc(\ppp^u)$ and $\ccc(\ppp^{\spinb}(\basepath))$ differ by $\inner{\beta}{\partial \ccc(\ppp^u)}H$ modulo $2H$.
\end{proof}

%\begin{definition}\label{def:nonab}
\noindent Finally, given a local system $\locsys\in\Loc(\blag)$ and $V\in\loct(\blag)$ a twisted local system on $\sphsc$ such that $\spins^{\ast}V=\locsys$, we can define the non-abelianization of $\spins^{\ast}V$ to be $\spinb^{\ast}\Phi_\snetwork(V)$. This is consistent with the main body of the manuscript by \cref{prop:spintosphere}.
%\end{definition}

\printbibliography

\end{document}